\documentclass [12pt]{amsart}
\usepackage[utf8]{inputenc}
\usepackage{amsmath}
\usepackage{amssymb}
\usepackage{hyperref}
\usepackage{amsthm}
\usepackage[table]{xcolor}    % loads also »colortbl«
\usepackage{graphicx,tikz}
\usetikzlibrary{calc}

\usepackage{amsfonts}
\usepackage[utf8]{inputenc} %% mac coding
\usepackage[english]{babel}
\usepackage{bbm}
\usepackage{mathtools}
\usepackage{enumerate}
\usepackage{tikz-cd}
\usepackage{tikz-3dplot}
\usepackage{mathrsfs}
\usepackage{comment}
\usetikzlibrary{arrows}
\usetikzlibrary{shapes,snakes}

\pdfpageattr{/Group <</S /Transparency /I true /CS /DeviceRGB>>}

\newtheorem{theorem}{Theorem}[section]
\newtheorem{lemma}[theorem]{Lemma}
\newtheorem{proposition}[theorem]{Proposition}
\newtheorem{corollary}[theorem]{Corollary}
\newtheorem{conjecture}[theorem]{Conjecture}
\theoremstyle{definition}
\newtheorem{definition}[theorem]{Definition}
\theoremstyle{definition}
\newtheorem{example}[theorem]{Example}
\newtheorem{remark}[theorem]{Remark}

\newenvironment{claim}[1]{\medskip\par\noindent\emph{Claim.}\space#1}{\medskip}
\newenvironment{claimproof}[1]{\medskip\par\noindent\emph{Proof of the claim.}\space#1}{\leavevmode\unskip\penalty9999 \hbox{}\nobreak\hfill\quad\hbox{$\blacksquare$}\medskip}

\def\C{\mathbb{C}}
\def\R{\mathbb{R}}

\def\Z{\mathbb{Z}}

\def\<{{\langle}}
\def\>{{\rangle}}
\def\l{{\lambda}}
\def\m{{\mu}}

\def\Acal{{\mathcal{A}}}
\def\Bcal{{\mathcal{B}}}
\def\Ccal{{\mathcal{C}}}
\def\Dcal{{\mathcal{D}}}

\def\Fcal{{\mathcal{F}}}
\def\Gcal{{\mathcal{G}}}

\def\Qbf{{\mathbf Q}}

\def\weakMap{\leadsto}
\def\e{{\epsilon}}
\def\Conv{{ \operatorname{Conv}}}

\def\WS{\mathcal{S}}
\def\PC{\mathbf{A}}
\def\VC{\mathbf{V}}
\def\Zon{\Zcal}
\def\Tiling{\mathfrak{T}}
\def\TPTiling{\widetilde{\PTiling}}
\def\TPTFamily{\TPTiling_*}
\def\Face{\tau}
\def\v{\mathbf{v}}
\def\x{\mathbf{x}}
\def\Hyp{H}
\def\rk{ \operatorname{rank}}
\def\cork{ \operatorname{corank}}

\def\conv{ \operatorname{Conv}}

\def\Vert{ \operatorname{Vert}}
\def\ipar{{(i)}}

\def\UP{ \operatorname{UP}}
\def\DOWN{ \operatorname{DOWN}}
\def\proj{ \operatorname{proj}}
\def\maxgap{ \operatorname{maxgap}}
\def\IC{ \operatorname{IC}}

\def\WhCl{{\Wcal}}
\def\BlCl{{\Bcal}}
\def\Bound{{\partial}}

\def\Mcal{{\mathcal{M}}}
\def\Mcalu{{\underline{\mathcal{M}}}}
\def\Mcaltilde{\widetilde{\mathcal{M}}}
\def\Lcal{{\mathcal{L}}}
\def\Tcal{{\mathcal{T}}}
\def\Ind{{\mathrm{Ind}}}
\def\Dcal{{\mathcal{D}}}
\def\Zcal{{\mathcal{Z}}}
\def\Wcal{{\mathcal{W}}}
\def\Bcal{{\mathcal{B}}}

\def\Cu{\underline{C}}
\def\Xu{\underline{X}}
\def\Yu{\underline{Y}}

\def\cube{\,\tikz\draw[scale=.25] (0,0,0) rectangle (1,1,0) (0,1,0) -- (0,1,-1) -- (1,1,-1) --(1,1,0) (1,0,0) -- (1,0,-1) -- (1,1,-1);} 

\def\dom{{\Dcal}}
\def\ind{\i}
\newcommand{\mut}[2]{\mu_{#1}(#2)}
\def\mutgraph{{\Gcal_\mu}}

\def\pad{\mathrm{pad}}

\def\empu{}

\newcommand{\reorient}[2]{{_{-#1}#2}}

\numberwithin{equation}{section}

\newcounter{todocnt} %TODO: add to template

\def\qqq{\qq^{\oplus 2}}
\def\qqb{\bar{\qq}}

\def\qb{\bar\q}
\begin{document}

\title[Purity and separation for oriented matroids]{Purity and separation \\ for oriented matroids}

%\title{Zonotopal tilings and purity for oriented matroids}

\author{Pavel Galashin}

\author{Alexander Postnikov}

\begin{abstract}
Leclerc and Zelevinsky, motivated by the study of quasi-commuting quantum flag minors,  
introduced the notions of {\it strongly separated\/} and {\it weakly separated\/} collections.  These notions are closely related to the theory of {\it cluster algebras,} to the combinatorics of the {\it double Bruhat cells,} and to the {\it totally positive Grassmannian.} 

A key feature, called the {\it purity phenomenon,} is that every maximal by inclusion strongly (resp., weakly) separated collection of subsets in $[n]$ has the same cardinality.  

In this paper, we extend these notions and define {\it $\mathcal{M}$-separated collections} for any oriented matroid $\mathcal{M}$.  

We show that maximal by size $\mathcal{M}$-separated collections are in bijection with fine zonotopal tilings (if $\mathcal{M}$ is a realizable oriented matroid), or with one-element liftings of $\mathcal{M}$ in general position (for an arbitrary oriented matroid).

We introduce the class of {\it pure oriented matroids\/} for which the purity phenomenon holds: an oriented matroid $\mathcal{M}$ is pure if $\mathcal{M}$-separated collections form a pure simplicial complex, i.e., any maximal by inclusion $\mathcal{M}$-separated collection is also maximal by size.

We pay closer attention to several special classes of oriented matroids: oriented matroids of rank $3$, graphical oriented matroids, and uniform oriented matroids. We classify pure oriented matroids in these cases.  An oriented matroid of rank $3$ is pure if and only if it is a {\it positroid\/}
(up to reorienting and relabeling its ground set). A graphical oriented matroid is pure if and only if its underlying graph is an {\it outerplanar graph,} that is, a subgraph of a triangulation of an $n$-gon.

We give a simple conjectural characterization of pure oriented matroids by forbidden minors and prove it for the above classes of matroids (rank $3$, graphical, uniform).
\end{abstract}

\address{Department of Mathematics, University of California, Los Angeles, CA 90095, USA}
\email{{\href{mailto:galashin@math.ucla.edu}{galashin@math.ucla.edu}}}

\address{Department of Mathematics, Massachusetts Institute of Technology, Cambridge, MA, 02139, USA.}
\email{{\href{mailto:apost@mit.edu}{apost@mit.edu}}}


\date{\today}
\maketitle

\setcounter{tocdepth}{1}
\tableofcontents

\section{Introduction}

In 1998, Leclerc and Zelevinsky~\cite{LZ} defined \emph{strongly separated} and
\emph{weakly separated} collections.  Several variations of these notions were
studied in \cite{DKK10, DKK14,OPS,  FG, Galashin}.  
The main goal of the present paper is to introduce the notion of
\emph{$\Mcal$-separation} 
in the general framework of \emph{oriented matroids}, which extends the
previous cases, and study its properties and 
its relationship with \emph{zonotopal tilings}.

\medskip

The notions of strongly and weakly separated collections originally appeared in
\cite{LZ} motivated 
by the study of the $q$-deformation $\Qbf_q[\Fcal]$ of the coordinate
ring of the flag variety $\Fcal$.  They also appeared in the study
\cite{Scott,Scott2} of the \emph{cluster algebra}~\cite{FZ,FZ2,FZ3,FZ4}
structure on the Grassmannian. 
They are closely related 
to the combinatorics of the \emph{totally positive Grassmannian}
and \emph{plabic graphs}, see
\cite{Postnikov, OPS}.

The study of zonotopal tilings is a popular topic in
combinatorics. The celebrated Bohne-Dress theorem~\cite{Bohne} gives a
correspondence between zonotopal tilings and one-element liftings of oriented
matroids.  
Fine zonotopal tilings of the $2n$-gon (also known as \emph{rhombus tilings}) correspond to commutation classes 
of reduced decompositions of the longest element in the symmetric group $S_n$.
More generally, Ziegler \cite{Ziegler} proved that fine zonotopal 
tilings of cyclic zonotopes correspond to
elements of Manin-Shekhtman's \emph{higher Bruhat orders} \cite{MS,VK}.  
Zonotopal
tilings were studied in the context of the \emph{generalized Baues problem} for
cubes, see \cite{BKS,Reiner,Bohne} and~\cite[Section~7.2]{Book}, which was
recently answered in the negative by Liu~\cite{Liu}.

In this paper, we connect these two areas of research.

\medskip

Let $I$ and $J$ be two subsets of the set $[n]:=\{1,2,\dots,n\}$. 
Leclerc and Zelevinsky~\cite{LZ} proved that
two quantum flag minors $[I]$ and $[J]$ in $\Qbf_q[\Fcal]$ 
quasi-commute if and only if $I$ and $J$ are weakly separated,
and that the product $[I][J]$ is invariant under the
involution on $\Qbf_q[\Fcal]$ that sends $q$ to $q^{-1}$
if and only if $I$ and $J$ are strongly separated.
Scott~\cite{Scott,Scott2} showed that two sets $I$ and $J$ 
of the same cardinality are weakly separated if
and only if the corresponding Pl\"ucker coordinates can appear together in the
same cluster in the cluster algebra of the Grassmannian.

Leclerc and Zelevinsky showed that any maximal \emph{by inclusion} strongly
separated collection of subsets of $[n]$ has size 
$$
{n\choose 0}+{n\choose
1}+{n\choose 2}
$$ 
and conjectured the same for weakly separated collections.
This \emph{purity conjecture} was proved independently 
in \cite{DKK10} and \cite{OPS}.

Another related purity result is that, for fixed $k\leq n$, any
maximal \emph{by inclusion} weakly separated collection 
of $k$-element subsets of $[n]$ has size
$$
k(n-k)+1.
$$
% This result is closely related to the combinatorics
% of the \emph{totally positive Grassmannian}, 
% see \cite{Postnikov}.
It was shown in \cite{OPS} that maximal by inclusion collections 
of weakly separated $k$-element subsets of $[n]$ 
are in bijection with \emph{plabic graphs} from \cite{Postnikov} associated
with parametrizations of the top cell of the totally positive Grassmannian.

Several other similar purity phenomena have been recently discovered,
see \cite{DKK14, FG, Galashin}. In \cite{Galashin},
the notion of~\emph{chord separation} related to that of weak separation 
was introduced. It was shown in \cite{Galashin} 
that any maximal \emph{by inclusion} chord separated collection has size 
\[{n\choose
0}+{n\choose 1}+{n\choose 2}+{n\choose 3}\] and 
is associated with the set of vertices of 
a fine zonotopal tiling of the 
$3$-dimensional cyclic zonotope $\Zon_{C^{n,3}}$.
% $\Zon_{\Cyclic(n,3)}$. 

\medskip

In this paper, we extend these various versions of separation
and purity to oriented matroids.
For any oriented matroid $\Mcal$, we define the notion of
\emph{$\Mcal$-separation} for collections of subsets of the ground set of
$\Mcal$.  
For an oriented matroid associated with a vector configuration
$\VC=(\v_1,\dots,\v_n)$, where $\v_1,\dots,\v_n\in \R^d$, 
we call this notion \emph{$\VC$-separation}.

For example, for \emph{alternating} oriented matroids of rank $2$ and $3$
(associated with cyclic vector configurations in $\R^2$ and $\R^3$),
the notion of $\Mcal$-separation
is equivalent, respectively, to strong separation from \cite{LZ} 
and chord separation from \cite{Galashin}.

Let $\Zon_\VC$ be the \emph{zonotope} associated with 
a vector configuration $\VC$, defined as the Minkowski 
sum of the line intervals $[0,\v_1],\dots,[0,\v_n]$.
A \emph{fine zonotopal tiling} of $\Zon_\VC$ is a subdivision
of the zonotope into parallelotopes.
Each fine zonotopal tiling $\Tiling$ is naturally equipped with a family
$\Vert(\Tiling)$ of subsets of $[n]$ that label the vertices of the tiling.

We prove that maximal \emph{by size} $\VC$-separated
collections are in bijection with fine zonotopal tilings $\Tiling$ of $\Zon_\VC$;
namely, they are precisely the collections $\Vert(\Tiling)$
of vertex labels of tilings.
The size of such a $\VC$-separated collection equals the number
of \emph{independent sets} of the associated oriented matroid.

If all maximal \emph{by inclusion} $\Mcal$-separated collections have the
same cardinality, we call the oriented matroid $\Mcal$ \emph{pure}.
We give a complete description of pure oriented matroids
in the following cases:
\begin{enumerate} 
\item oriented matroids of rank $3$, 
\item \emph{graphical} oriented matroids, and 
\item \emph{uniform} oriented matroids.  
\end{enumerate} 

For the first class, the pure oriented matroids 
are precisely all oriented matroids obtained from 
\emph{positroids} of rank $3$ 
by relabeling and reorienting the ground set.

For the second class, we show that an undirected graph $G$ gives rise to a
pure oriented matroid if and only if $G$ is \emph{outerplanar}. 

For the third class, we show that all pure uniform vector configurations
either have \emph{rank} at most $3$ or \emph{corank} at most $1$.

For an arbitrary oriented matroid $\Mcal$, we give a conjectural 
criterion for purity in terms of forbidden minors of $\Mcal$.

\medskip

Here is the general outline of the paper. In Section~\ref{sect:separation_def}, we define the notion of \emph{separation} for vector configurations and oriented matroids and discuss its relationship with zonotopal tilings and one-element liftings. In Section~\ref{sect:motivation}, we describe some known motivating examples
of the purity phenomenon. We then recall several simple operations on oriented matroids in Section~\ref{sec:simple-oper-orient}. In Section~\ref{sect:main}, we state our main results
 regarding the purity phenomenon for vector configurations and oriented matroids. Next, we give some technical background on zonotopal tilings and oriented matroids in Section~\ref{sect:background}. The rest of the paper is mainly concerned with proving theorems from Section~\ref{sect:main}. In Section~\ref{sect:max_by_sz}, we prove the results regarding maximal \emph{by size} $\VC$-separated collections,
including Theorem~\ref{thm:max_size_vc} that gives a simple bijection between
maximal \emph{by size} $\VC$-separated collections and vertex label collections
of fine zonotopal tilings of the zonotope corresponding to $\VC$. Next, we
concentrate on pure vector configurations/oriented matroids. In
Section~\ref{sect:pure_OM}, we show that the property of being a pure oriented
matroid is preserved under various oriented matroid operations. We prove the
purity phenomenon for outerplanar graphs in Section~\ref{sect:graph}, where we also give enumerative results on the number of maximal $G$-separated collections. We then proceed to
showing the purity of totally nonnegative rank $3$ vector configurations in
Section~\ref{sect:rk_3}. Finally, we give the remaining proofs of our various classification 
 results for pure oriented matroids in Section~\ref{sect:classif}.

\subsection*{Acknowledgments}
We are indebted to the anonymous referee for their extremely careful reading of the first version of this manuscript and for pointing out several gaps in our arguments. We also thank Steven Karp and Melissa Sherman-Bennett for their comments on the text. This work was supported by an Alfred P. Sloan Research Fellowship and by the National Science Foundation under Grants No.~DMS-1954121 and No.~DMS-2046915.

\section{Separation, purity, and zonotopal tilings}
\label{sect:separation_def}

In this section, we introduce the notion of separation and purity  for vector configurations 
and oriented matroids and formulate some of our results.  

%Later, in Section~\ref{sect:main_matroids}, we will extend 
%these constructions in a more general setting of oriented matroids.

\medskip

Let $[n]:=\{1,\dots,n\}$. Denote by $2^E$ the set of all subsets of a set $E$.
%and let $2^{[n]}$ be the set of all subsets in $[n]$.

\subsection{Separation for vector configurations}

\begin{definition}
Let $\VC = (\v_1,\dots,\v_n)$ be a finite configuration of vectors $\v_1,\dots,\v_n\in\R^d$.

Two subsets $I,J\subset [n]$ are called \emph{$\VC$-separated} if there exists a linear function 
$h:\R^d\to\R$ such that $h(\v_i)>0$ for $i\in I\setminus J$, and 
$h(\v_j)<0$ for $j\in J\setminus I$. 

A collection $\WS\subset 2^{[n]}$ of subsets of $[n]$ is called
a \emph{$\VC$-separated collection} if any two of its elements are $\VC$-separated.
\end{definition} 

\subsection{Separation for oriented matroids}

The above notion of $\VC$-separation depends only on the oriented matroid associated
with a vector configuration $\VC$.  One can extend this notion to any oriented matroid
as follows.

First, recall the definition of oriented matroids; see 
Section~\ref{sect:OM} for more details. A \emph{signed subset} $X$ of a set $E$ is a pair $X=(X^+,X^-)$ of disjoint
subsets $X^+, X^-$ of $E$.  The \emph{support} $\Xu$ of $X$ is the (usual) set $\Xu:=X^+\sqcup X^-$.
The empty signed subset is $\varnothing=(\emptyset,\emptyset)$.
For a signed subset $X=(X^+,X^-)$, let $-X:=(X^{-},X^{+})$.

\begin{definition}[{\cite[Definition~3.2.1]{Book}}]\label{dfn:OM}
An \emph{oriented matroid} $\Mcal$ is a pair $\Mcal=(E,\Ccal)$, where $E$ is a
set, called the \emph{ground set}, and $\Ccal$ is a collection of signed subsets of $E$,
called \emph{circuits}, that satisfy the following axioms:
 \begin{enumerate}
  \item[(C0)] $\varnothing\not\in\Ccal$.
  \item[(C1)] For all $X\in \Ccal$, we have $-X\in \Ccal$.
  \item[(C2)]\label{item:C2} For all $X,Y\in\Ccal$, if $\Xu\subset\Yu$, then $X=Y$ or $X=-Y$.
  \item[(C3)]\label{item:C3} %(weak elimination) 
For all $X,Y\in\Ccal$, $X\neq -Y$, and $e\in X^+\cap Y^-$, there is 
             a $Z\in\Ccal$ such that 
  \[Z^+\subset (X^+\cup Y^+)\setminus\{e\}\quad\text{and}\quad 
Z^-\subset (X^-\cup Y^-)\setminus\{e\}.\]
\end{enumerate}

A subset $I\subset E$ is called \emph{independent} if there is no circuit $X\in \Ccal$ such that $\Xu\subset I$. The \emph{rank} of $\Mcal$, denoted $\rk(\Mcal)$, is the maximal size of an independent subset.
\end{definition}

A vector configuration $\VC=(\v_1,\dots,\v_n)$, defines 
the \emph{associated oriented matroid} $\Mcal_\VC$ on the ground set $E=[n]$ such that
a nonempty signed subset $X$ of $[n]$ is a circuit of $\Mcal_\VC$ if and only if there exists
a linear dependence $\sum_{i\in \Xu} c_i \,\v_i=0$ with $c_i>0$, for all $i\in X^+$,
and $c_j<0$, for all $j\in X^-$, and any proper subset of vectors 
$\v_i$, $i\in \Xu$, is linearly independent.

\begin{definition}
%For the reader familiar with the oriented matroid terminology (see Section~\ref{sect:OM}), 
%let us preview a more general definition (see Definition~\ref{def:M_separation}).
For an oriented matroid $\Mcal=(E,\Ccal)$, we say that two sets
$I,J\subset E$ are \emph{$\Mcal$-separated} if there is no circuit
$X\in \Ccal$ such that $X^+\subset I\setminus J$ and $X^-\subset J\setminus I$.

A collection $\WS\subset 2^{E}$ of subsets of $E$ is
called an \emph{$\Mcal$-separated collection} if any two of its elements are $\Mcal$-separated.
\end{definition}

The following lemma is an easy exercise for the reader.

\begin{lemma}
Let $\Mcal=\Mcal_\VC$ be the oriented matroid associated with a vector configuration $\VC$.
Then a collection $\WS\subset 2^E$ is $\VC$-separated if and only if it is $\Mcal$-separated.
\end{lemma}

\subsection{Zonotopal tilings}

Recall that the \emph{Minkowski sum} of two (or more) sets $A,B\subset\R^d$
is the set $A+B :=\{a+b\mid a\in A,\, b\in B\}$.

For a vector configuration $\VC=(\v_1,\dots,\v_n)$, the corresponding \emph{zonotope}
$\Zon_\VC$ is defined as the Minkowski sum of the line segments $[0,\v_i]$:
$$
\Zon_\VC:=[0,\v_1]+\cdots+[0,\v_n].
$$ 

Equivalently, the zonotope $\Zon_\VC$ is the image $p(\cube_n)$ of the 
standard $n$-hypercube 
$$
\cube_n :=[0,e_1]+\cdots+[0,e_n] \subset \R^n
%\{(x_1,\dots,x_n)\in\R^n\mid 0\leq x_i\leq 1,\textrm{ for }i=1,\dots,n\}
$$
under the projection 
$$
p:\R^n\to \R^d\quad
\textrm{such that }
p(e_i)=\v_i, \textrm{ for }i=1,\dots,n, 
$$
where $e_1,\dots,e_n$ are the standard coordinate vectors in $\R^n$.

\begin{definition}\label{dfn:fine_z_tiling_intro}
A \emph{fine zonotopal tiling} of $\Zon_\VC$ is a cubical subcomplex $\Tiling$ 
of the $n$-hypercube $\cube_n$,
i.e., a collection of faces $F$ of $\cube_n$ closed under taking subfaces,
such that the projection $p$ induces a homeomorphism between $\bigcup_{F\in\Tiling}F$ and 
the zonotope $\Zon_\VC$.
\end{definition}

A fine zonotopal tiling $\Tiling$ gives a polyhedral subdivision $p(\Tiling)$
of the zonotope $\Zon_\VC$ with faces $p(F)$, for $F\in\Tiling$.  Each face
$p(F)$ of the subdivision $p(\Tiling)$ is a parallelotope whose edges are
parallel translations of some vectors $\v_i$.

Faces $F=\cube_X$ of the hypercube $\cube_n$ can be labeled by signed subsets 
$X=(X^+,X^-)$ of $[n]$ as follows:
$$
\cube_X:=\sum_{i\in X^+} e_i + \sum_{j\in [n]\setminus \underline X} [0,e_j].
$$

Thus every face $p(\cube_X)$ of the subdivison $p(\Tiling)$ of the zonotope $\Zon_\VC$
is naturally labeled by the signed subset $X$ of $[n]$. In what follows, we identify a face $F=\cube_X$ of the cube with the signed subset $X$ that labels it.

Clearly, $\dim \cube_X =\dim p(\cube_X) = n-|\underline X|$. 

\begin{remark}
According to our definition, a fine zonotopal 
tiling $\Tiling$ contains more information than 
the polyhedral subdivision $p(\Tiling)$ of the zonotope $\Zon_\VC$.
Namely, it also includes the labeling of the faces 
of the subdivision $p(\Tiling)$ by signed sets $X$.
It is possible that two different tilings $\Tiling_1$ and $\Tiling_2$
produce the same subdivision $p(\Tiling_1)=p(\Tiling_2)$ of the 
zonotope.

However, if the vectors $\v_1,\dots,\v_n$ of the configuration $\VC$ are non-zero
and not collinear to each other, it is not hard to show
(an exercise for the reader) that the subdivision $p(\Tiling)$ of $\Zon_\VC$ 
uniquely defines the labeling of its faces by signed sets.  In this case, we can 
identify a fine zonotopal tiling $\Tiling$ with the corresponding polyhedral subdivision
$p(\Tiling)$ of $\Zon_\VC$.
\end{remark}

Clearly, a face $\cube_X$ of the hypercube $\cube_n$ is a vertex 
if and only if $\underline X=[n]$.  So we can label vertices of $\cube_n$ 
by usual subsets $I=X^+\subset [n]$. 
%where $S=(I,[n]\setminus I)$.

For a fine zonotopal tiling $\Tiling$ of $\Zon_\VC$, let 
$\Vert(\Tiling)\subset 2^{[n]}$ be the collection of labels $I$ of vertices of $\Tiling$,
i.e., 
\begin{equation}\label{eq:tiling_vertices_intro}
\Vert(\Tiling):= \{I\in 2^{[n]}\mid 
\cube_{(I,[n]\setminus I)} \in \Tiling \}.
\end{equation}

We say that a $\VC$-separated collection $\WS$ is \emph{maximal by size} if
its cardinality $|\WS|$ is maximal among all $\VC$-separated collections.

Our first main result on $\VC$-separation identifies maximal by size $\VC$-separated 
collections with fine zonotopal tilings of $\Zon_\VC$.

\begin{theorem}\label{thm:max_size_vc}
Let $\VC\subset\R^d$ be a vector  configuration. Then the map $\Tiling\mapsto
\Vert(\Tiling)$ is a bijection between fine zonotopal tilings of $\Zon_\VC$ and
maximal \emph{by size} $\VC$-separated collections of subsets of $[n]$. 

Any such collection has size $|\Ind(\VC)|$, where $\Ind(\VC)$ denotes the collection of
linearly independent subsets of $\VC$.  
\end{theorem}

Actually, we will prove a stronger result, Theorem~\ref{thm:max_size_matroid},
concerning an arbitrary oriented matroid $\Mcal$.  According to the Bohne-Dress
theorem (see Theorem~\ref{thm:bohne}), fine zonotopal tilings of 
the zonotope $\Zon_\VC$
are canonically identified with \emph{one-element liftings of 
the oriented matroid $\Mcal_\VC$ in general position}. 
Thus one can view the notion of ``one-element liftings in general position'' 
as an extension of the notion of ``fine zonotopal tilings'' 
to an arbitrary oriented matroid $\Mcal$. 
Theorem~\ref{thm:max_size_matroid} gives
a bijection between these objects and maximal \emph{by size} $\Mcal$-separated
collections.

\subsection{Pure oriented matroids}

We will distinguish between two different notions of maximality of $\Mcal$- or $\VC$-separated 
collections: maximal \emph{by size} (appearing in the previous theorem) and 
maximal \emph{by inclusion}.

\begin{definition}
For an oriented matroid $\Mcal$, an $\Mcal$-separated collection $\WS$ 
is called \emph{maximal by inclusion} if $\WS$ is
not properly contained in any other $\Mcal$-separated collection. Similarly, for a vector configuration $\VC$, 
we define \emph{maximal by inclusion} $\VC$-separated
collections. 
% An $\Mcal$-separated collection $\WS$ is called \emph{maximal by size} if its
% cardinality $|\WS|$ is maximal among all $\Mcal$-separated collections.
% Similarly, for a vector configuration $\VC$, 
% we define maximal by inclusion and maximal by size $\VC$-separated
% collections.
\end{definition}

Clearly, 
all $\Mcal$-separated collections in $2^{E}$ form an abstract simplicial complex, i.e.,
any subset of an $\Mcal$-separated collection is also $\Mcal$-separated.

Recall that a simplicial complex is called \emph{pure} if any simplex 
in it is a face of a top-dimensional simplex in this simplicial complex.

\begin{definition}
We say that an oriented matroid $\Mcal$ is \emph{pure} if any maximal
\emph{by inclusion} $\Mcal$-separated collection is also maximal \emph{by size}.
Equivalently, an oriented matroid $\Mcal$ is pure if all 
$\Mcal$-separated collections form a pure simplicial complex.

We also say that 
a vector configuration $\VC$ is \emph{pure} if the associated oriented
matroid $\Mcal_\VC$ is pure.
\end{definition}

Clearly, if $\VC$ is pure then we can replace the phrase ``maximal \emph{by size}'' 
in Theorem~\ref{thm:max_size_vc} with ``maximal \emph{by inclusion}''.

\subsection{Mutation-closed domains}\label{sec:mutat-clos-doma}

\begin{definition} \label{dfn:mutation_graph_intro}
The \emph{mutation graph}  
of an oriented matroid $\Mcal$ is a simple undirected graph
on the vertex set $2^E$ such that two subsets $I,J \subset E$ are connected by an edge 
if and only if the signed set
$(I\setminus J, J\setminus I)$ is a circuit of $\Mcal$. See Figure~\ref{fig:icos_dodec} for an example of an oriented matroid for which the connected components of the mutation graph are the $1$-skeleta of the icosahedron and the dodecahedron.

We say that a subset $\dom\subset 2^E$ is a 
\emph{mutation-closed domain} for $\Mcal$ if $\dom$ is a union of
connected components of the mutation graph of $\Mcal$.

We say that a mutation-closed domain $\dom\subset 2^E$ is \emph{$\Mcal$-pure}
if all $\Mcal$-separated collections $\WS\subset \dom$ form a pure 
simplicial complex.  
In other words, $\dom$ is $\Mcal$-pure if and only if
any $\Mcal$-separated collection $\WS\subset \dom$,
which is maximal \emph{by inclusion} among all $\Mcal$-separated collections
that belong to $\dom$, is also maximal \emph{by size} among all such collections.
\end{definition}

See Section~\ref{subsection:mutation_closed_big} for more details
related to this definition. Note that a concept equivalent to the mutation graph appeared
in~\cite[Section~4]{Gioan1}. 

\begin{remark}\label{rmk:coll_domain}
Both $\Mcal$-separated collections $\WS$ and mutation-closed domains 
$\dom$ are subsets of $2^E$.  Strictly speaking, both terms
``collection'' and ``domain'' mean a ``set of subsets'' of $E$.
However, we usually use the term \emph{collection} when we talk about
$\Mcal$-separated collections.  On the other hand, \emph{domains} 
 need not be $\Mcal$-separated.
Typically, we will fix a domain $\dom$ and study all $\Mcal$-separated
collections $\WS$ inside $\dom$. The same convention is used in~\cite{DKK14}.
\end{remark}

\begin{conjecture}
\label{conj:mutation_closed_pure}
If $\Mcal$ is a pure oriented matroid then any mutation-closed domain
$\dom$ for $\Mcal$ is an $\Mcal$-pure domain.
\end{conjecture}

We will prove this conjecture in two important cases.

There are local transformations of fine zonotopal tilings of $\Mcal_\VC$
called \emph{flips}.
Using the Bohne-Dress correspondence between zonotopal tilings and
liftings, one can extend the definition of flips
to one-element liftings in general position of any oriented matroid $\Mcal$.
An oriented matroid $\Mcal$ is called 
\emph{flip-connected} if any two one-element liftings of $\Mcal$
in general position are connected by a sequence of flips.
See Section~\ref{subsection:mutation_closed_big} 
and Definition~\ref{def:flip_oriented_matroid} for more details.

\begin{proposition}
\label{prop:graphicalandflipconnected}
Conjecture~\ref{conj:mutation_closed_pure} is true
for graphical oriented matroids and also for flip-connected 
oriented matroids.
\end{proposition}

The graphical case is proved in Section~\ref{sect:graph}.
The case of flip-connected matroids is proved in Section~\ref{subsection:mutation_closed_big}.

\section{Motivating examples}\label{sect:motivation}

In this section, we describe several results about strong, weak, and chord
separation that have been proved in~\cite{LZ,OPS,Galashin,DKK10,DKK14}.

\subsection{Alternating oriented matroids}\label{sec:altern-orient-matr}

A \emph{cyclic vector configuration} is a vector
configuration $\VC= (\v_1,\dots,\v_n)$ such that
all maximal $d\times d$ minors of the $d\times n$ matrix with columns 
$\v_1,\dots,\v_n$ are strictly positive.
For example, for the \emph{moment curve}
$$
\v(t) = (1,t,t^2,\dots,t^{d-1})\in \R^d, \quad t\in \R,
$$
the vector configuration 
$(\v(t_1),\dots,\v(t_n))$, for $0<t_1<\dots<t_n$, is cyclic.
(In this case, the maximal minors are given by positive
Vandermonde determinants.)

\begin{remark}
A \emph{cyclic polytope} is the convex hull of the endpoints
of vectors $\v_i$ in a cyclic vector configuration.
According to~\cite{Postnikov}, for fixed $n$ and $d$,
cyclic vector configurations 
represent points of the \emph{totally positive Grassmannian}
$Gr_{d,n}^{>0}$.
\end{remark}

It is not hard to see that all cyclic configurations 
of $n$
vectors in $\R^d$ define the same oriented matroid,
called the \emph{alternating oriented matroid} $C^{n,d}$.

We will call a zonotope $\Zon_{\VC}$ associated
with a cyclic vector configuration $\VC$ of $n$ vectors in $\R^d$ a \emph{cyclic zonotope} and denote it by $\Zon_{C^{n,d}}$. The combinatorial structure of cyclic polytopes and 
cyclic zonotopes depends only on 
$n$ and $d$ and is independent of the choice of vectors in
a cyclic vector configuration.

The following description of circuits of alternating oriented matroids
is well known and not hard to prove.  It explains why these oriented 
matroids are called ``alternating''.

\def\odd{{\operatorname{odd}}}
\def\even{{\operatorname{even}}}

\begin{lemma}
\label{lem:alternating_matroid}
The circuits of the alternating oriented matroid $C^{n,d}$ are exactly
all signed subsets of $[n]$ of the form
$(I^\odd,I^\even)$  
or $(I^\even,I^\odd)$,
where 
$I=\{i_1<i_2<\dots<i_{d+1}\}$ is any $(d+1)$-element subset of $[n]$,
$I^\odd:=\{i_1,i_3,i_5,\dots\}$,
and
$I^\even:=\{i_2,i_4,i_6,\dots\}$.
\end{lemma}

\begin{figure}
  % \makebox[\textwidth]{
  \resizebox{1.0\textwidth}{!}{
\begin{tabular}{cc}
 \begin{tikzpicture}
  \draw[->] (0,0) -- (-2,2);
  \draw[->] (0,0) -- (-1,2);
  \draw[->] (0,0) -- (0,2);
  \draw[->] (0,0) -- (1,2);
  \draw[->] (0,0) -- (2,2);
  \draw (-2,2) node[anchor=south]{$\v_1$};
  \draw (-1,2) node[anchor=south]{$\v_2$};
  \draw (0,2) node[anchor=south]{$\v_3$};
  \draw (1,2) node[anchor=south]{$\v_4$};
  \draw (2,2) node[anchor=south]{$\v_5$};
 \end{tikzpicture}
 &
\tdplotsetmaincoords{80}{10}
\begin{tikzpicture}[scale=2,tdplot_main_coords]
    \foreach\x in {1,2,...,6}{
      \draw[->] (0,0,0) -- ({cos(\x*60)},{sin(\x*60)},1) node[anchor=south]{$\v_\x$};
      \fill[opacity=0.1,blue] (0,0,1) -- ({cos(\x*60)},{sin(\x*60)},1) -- ({cos(\x*60+60)},{sin(\x*60+60)},1)--cycle;
    }
    \def\y{1.3}
    \draw[blue] (\y,\y,1)  -- (\y,-\y,1)node[anchor=west,black]{$z=1$} -- (-\y,-\y,1) -- (-\y,\y,1)  --cycle;
\end{tikzpicture}
\\
  A cyclic vector configuration  & A cyclic vector configuration \\
  representing $C^{5,2}$.  & representing $C^{6,3}$.
\end{tabular}
}
\caption{Cyclic vector configurations in $\R^2$ and $\R^3$.}
%The picture on the right is taken from~\cite{Galashin}. 
\label{fig:cyclic} 
\end{figure}
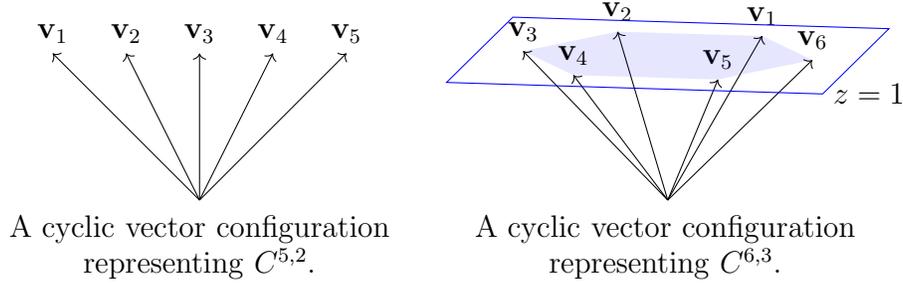

% Here is a well studied property of alternating oriented matroids.

\begin{theorem}[{\cite[Theorem~4.1(G)]{Ziegler}}] 
\label{th:Ziegler_flips}
All alternating oriented matroids $C^{n,d}$ are flip-connected.
\end{theorem}

\begin{remark}
Ziegler~\cite{Ziegler} identified fine zonotopal 
tilings of cyclic zonotopes with elements 
of the \emph{higher Bruhat orders} \cite{MS, VK}.  
The flip connectedness of $C^{n,d}$ is equivalent to 
the connectedness of the corresponding higher Bruhat order poset.
\end{remark}

\subsection{Strong separation}

Leclerc and Zelevinsky~\cite{LZ} defined \emph{strong separation} 
as follows. 

\begin{definition}
Two sets $I,J\subset[n]$ are called \emph{strongly separated}
if there are no three elements $i<j<k\in [n]$ such that $i,k\in I\setminus J$ 
and $j\in J\setminus I$, or vice versa.
\end{definition}

%Here by $S-T=S\setminus T$ we denote the set-theoretic difference of sets $S$
%and $T$. 

A collection $\WS\subset 2^{[n]}$ of subsets of $[n]$ is
\emph{strongly separated} if any two of its sets are strongly separated from
each other. Such a collection is called maximal \emph{by inclusion} if it is
not properly contained in any other strongly separated collection. 

\begin{theorem}[{\cite[Theorem 1.6]{LZ}}]\label{thm:purity_ss}
Any maximal \emph{by inclusion} strongly separated collection $\WS\subset 2^{[n]}$ 
is also maximal \emph{by size}:
	\[|\WS|={n\choose 0}+{n\choose 1}+{n\choose 2}.\]
	Such collections are in bijection with rhombus tilings of a convex $2n$-gon, see Figure~\ref{fig:tiling_2d}.
\end{theorem}

For example, Figure~\ref{fig:tiling_2d} shows a maximal
by inclusion strongly separated collection
with ${5\choose 0} + {5\choose 1} + {5\choose 2} =16$ elements. 

%\def\Complex{\Delta} Whenever every maximal \emph{by inclusion} strongly
%(weakly, chord) separated collection is also maximal \emph{by size}, we say
%that the \emph{purity phenomenon} holds. This name is motivated by the fact
%that the flag simplicial complex $\Complex$ with vertex set $2^{[n]}$ and edge
%set \[\{(S,T)\mid S,T\subset [n];\quad S\text{ is strongly separated from
%}T\}\] is \emph{pure}, that is, every face of $\Complex$ is contained in a
%top-dimensional face. 

\begin{figure}
\scalebox{0.6}{
\begin{tikzpicture}
\node[draw,ellipse,black,fill=white] (node0) at (0.00,0.00) {$\emptyset$};
\node[draw,ellipse,black,fill=white] (node1) at (-3.80,1.20) {$1$};
\node[draw,ellipse,black,fill=white] (node5) at (3.80,1.20) {$5$};
\node[draw,ellipse,black,fill=white] (node12) at (-5.70,2.40) {$1\,2$};
\node[draw,ellipse,black,fill=white] (node13) at (-3.80,2.40) {$1\,3$};
\node[draw,ellipse,black,fill=white] (node15) at (0.00,2.40) {$1\,5$};
\node[draw,ellipse,black,fill=white] (node35) at (3.80,2.40) {$3\,5$};
\node[draw,ellipse,black,fill=white] (node45) at (5.70,2.40) {$4\,5$};
\node[draw,ellipse,black,fill=white] (node123) at (-5.70,3.60) {$1\,2\,3$};
\node[draw,ellipse,black,fill=white] (node135) at (0.00,3.60) {$1\,3\,5$};
\node[draw,ellipse,black,fill=white] (node235) at (1.90,3.60) {$2\,3\,5$};
\node[draw,ellipse,black,fill=white] (node345) at (5.70,3.60) {$3\,4\,5$};
\node[draw,ellipse,black,fill=white] (node1234) at (-3.80,4.80) {$1\,2\,3\,4$};
\node[draw,ellipse,black,fill=white] (node1235) at (-1.90,4.80) {$1\,2\,3\,5$};
\node[draw,ellipse,black,fill=white] (node2345) at (3.80,4.80) {$2\,3\,4\,5$};
\node[draw,ellipse,black,fill=white] (node12345) at (0.00,6.00) {$1\,2\,3\,4\,5$};
\node[draw,ellipse,black,fill=white] (node0) at (0.00,0.00) {$\emptyset$};
\node[draw,ellipse,black,fill=white] (node1) at (-3.80,1.20) {$1$};
\node[draw,ellipse,black,fill=white] (node5) at (3.80,1.20) {$5$};
\node[draw,ellipse,black,fill=white] (node12) at (-5.70,2.40) {$1\,2$};
\node[draw,ellipse,black,fill=white] (node13) at (-3.80,2.40) {$1\,3$};
\node[draw,ellipse,black,fill=white] (node15) at (0.00,2.40) {$1\,5$};
\node[draw,ellipse,black,fill=white] (node35) at (3.80,2.40) {$3\,5$};
\node[draw,ellipse,black,fill=white] (node45) at (5.70,2.40) {$4\,5$};
\node[draw,ellipse,black,fill=white] (node123) at (-5.70,3.60) {$1\,2\,3$};
\node[draw,ellipse,black,fill=white] (node135) at (0.00,3.60) {$1\,3\,5$};
\node[draw,ellipse,black,fill=white] (node235) at (1.90,3.60) {$2\,3\,5$};
\node[draw,ellipse,black,fill=white] (node345) at (5.70,3.60) {$3\,4\,5$};
\node[draw,ellipse,black,fill=white] (node1234) at (-3.80,4.80) {$1\,2\,3\,4$};
\node[draw,ellipse,black,fill=white] (node1235) at (-1.90,4.80) {$1\,2\,3\,5$};
\node[draw,ellipse,black,fill=white] (node2345) at (3.80,4.80) {$2\,3\,4\,5$};
\node[draw,ellipse,black,fill=white] (node12345) at (0.00,6.00) {$1\,2\,3\,4\,5$};
\draw[line width=0.04mm,black] (node0) -- (node1);
\draw[line width=0.04mm,black] (node0) -- (node5);
\draw[line width=0.04mm,black] (node1) -- (node12);
\draw[line width=0.04mm,black] (node1) -- (node13);
\draw[line width=0.04mm,black] (node1) -- (node15);
\draw[line width=0.04mm,black] (node5) -- (node15);
\draw[line width=0.04mm,black] (node5) -- (node35);
\draw[line width=0.04mm,black] (node5) -- (node45);
\draw[line width=0.04mm,black] (node12) -- (node123);
\draw[line width=0.04mm,black] (node13) -- (node123);
\draw[line width=0.04mm,black] (node13) -- (node135);
\draw[line width=0.04mm,black] (node15) -- (node135);
\draw[line width=0.04mm,black] (node35) -- (node135);
\draw[line width=0.04mm,black] (node35) -- (node235);
\draw[line width=0.04mm,black] (node35) -- (node345);
\draw[line width=0.04mm,black] (node45) -- (node345);
\draw[line width=0.04mm,black] (node123) -- (node1234);
\draw[line width=0.04mm,black] (node123) -- (node1235);
\draw[line width=0.04mm,black] (node135) -- (node1235);
\draw[line width=0.04mm,black] (node235) -- (node1235);
\draw[line width=0.04mm,black] (node235) -- (node2345);
\draw[line width=0.04mm,black] (node345) -- (node2345);
\draw[line width=0.04mm,black] (node1234) -- (node12345);
\draw[line width=0.04mm,black] (node1235) -- (node12345);
\draw[line width=0.04mm,black] (node2345) -- (node12345);
\node[draw,ellipse,black,fill=white] (node0) at (0.00,0.00) {$\emptyset$};
\node[draw,ellipse,black,fill=white] (node1) at (-3.80,1.20) {$1$};
\node[draw,ellipse,black,fill=white] (node5) at (3.80,1.20) {$5$};
\node[draw,ellipse,black,fill=white] (node12) at (-5.70,2.40) {$1\,2$};
\node[draw,ellipse,black,fill=white] (node13) at (-3.80,2.40) {$1\,3$};
\node[draw,ellipse,black,fill=white] (node15) at (0.00,2.40) {$1\,5$};
\node[draw,ellipse,black,fill=white] (node35) at (3.80,2.40) {$3\,5$};
\node[draw,ellipse,black,fill=white] (node45) at (5.70,2.40) {$4\,5$};
\node[draw,ellipse,black,fill=white] (node123) at (-5.70,3.60) {$1\,2\,3$};
\node[draw,ellipse,black,fill=white] (node135) at (0.00,3.60) {$1\,3\,5$};
\node[draw,ellipse,black,fill=white] (node235) at (1.90,3.60) {$2\,3\,5$};
\node[draw,ellipse,black,fill=white] (node345) at (5.70,3.60) {$3\,4\,5$};
\node[draw,ellipse,black,fill=white] (node1234) at (-3.80,4.80) {$1\,2\,3\,4$};
\node[draw,ellipse,black,fill=white] (node1235) at (-1.90,4.80) {$1\,2\,3\,5$};
\node[draw,ellipse,black,fill=white] (node2345) at (3.80,4.80) {$2\,3\,4\,5$};
\node[draw,ellipse,black,fill=white] (node12345) at (0.00,6.00) {$1\,2\,3\,4\,5$};
\end{tikzpicture}}

\caption{\label{fig:tiling_2d}Vertex labels of a rhombus tiling of a convex
$2\times 5$-gon form a maximal by inclusion (and by size) 
strongly separated collection.}
\end{figure}
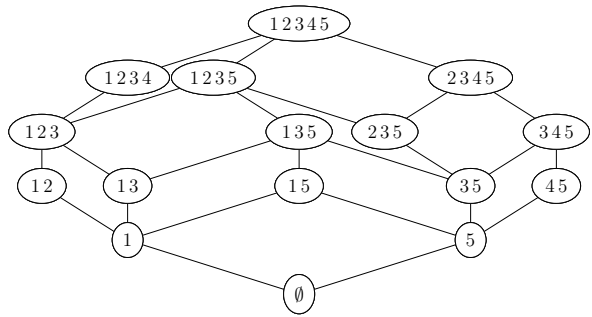

It follows directly from the definitions and Lemma~\ref{lem:alternating_matroid}
that strong separation is equivalent to $\Mcal$-separation
for the rank $2$ alternating oriented matroid $\Mcal=C^{n,2}$.

\begin{lemma}
A collection $\WS\subset 2^{[n]}$ is strongly separated if and only if
$\WS$ is $C^{n,2}$-separated. 
\end{lemma}

In this case, the cyclic zonotope $\Zon_{C^{n,2}}$  is a
centrally symmetric $2n$-gon.   Fine zonotopal tilings 
of this zonotope are exactly rhombus tilings of the $2n$-gon.

Notice that the number of independent sets of the rank $2$ alternating
oriented matroid $C^{n,2}$, 
which are all subsets of $[n]$ with at most 2 elements,
is exactly ${n\choose 0}+{n\choose 1}+{n\choose 2}$.

\subsection{Chord separation}

Before we discuss weak separation, let us first talk about a
related notion of \emph{chord separation}, 
which was recently defined in~\cite{Galashin} as follows.

\begin{definition}
Two sets $I,J\subset [n]$ 
are \emph{chord separated} 
if and only if there do not exist numbers 
$1\leq i<j<k<l\leq n$ such that $i,k\in I\setminus J$ and 
$j,l\in J\setminus I$, or vice versa.
\end{definition}

\begin{theorem}[{\cite[Theorem~1.2]{Galashin}}]
\label{thm:purity_chord} 
Any
maximal \emph{by inclusion} chord separated collection $\WS\subset 2^{[n]}$ has
size
\[ {n\choose 0}+ {n\choose 1}+{n\choose 2}+{n\choose 3}.\]
Such collections are in bijection with fine zonotopal tilings of the three-dimensional cyclic zonotope $\Zon_{C^{n,3}}$.
\end{theorem}

Figure~\ref{fig:chord} shows an example of a zonotopal tiling
of $\Zon_{\VC}$ whose vertex labels form a maximal \empu{by inclusion}
chord separated collection $\WS\subset 2^{[n]}$ for $n=5$.  This
collection has size ${5\choose 0}+{5\choose 1}+{5\choose 2} +
{5\choose 3}=26$, as predicted by Theorem~\ref{thm:purity_chord}.

\begin{figure}
\scalebox{0.6}{
\begin{tikzpicture}
\node[draw,ellipse,black,fill=white] (node0) at (0.00,0.00) {$\emptyset$};
\node[draw,ellipse,black,fill=white] (node1) at (5.60,1.75) {$1$};
\node[draw,ellipse,black,fill=white] (node2) at (1.73,2.45) {$2$};
\node[draw,ellipse,black,fill=white] (node3) at (-4.53,2.18) {$3$};
\node[draw,ellipse,black,fill=white] (node4) at (-4.53,1.32) {$4$};
\node[draw,ellipse,black,fill=white] (node5) at (1.73,1.05) {$5$};
\node[draw,ellipse,black,fill=white] (node12) at (7.33,5.33) {$1\,2$};
\node[draw,ellipse,black,fill=white] (node13) at (1.07,5.06) {$1\,3$};
\node[draw,ellipse,black,fill=white] (node15) at (7.33,3.93) {$1\,5$};
\node[draw,ellipse,black,fill=white] (node23) at (-2.80,5.76) {$2\,3$};
\node[draw,ellipse,black,fill=white] (node34) at (-9.06,4.63) {$3\,4$};
\node[draw,ellipse,black,fill=white] (node35) at (-2.80,4.36) {$3\,5$};
\node[draw,ellipse,black,fill=white] (node45) at (-2.80,3.50) {$4\,5$};
\node[draw,ellipse,black,fill=white] (node123) at (2.80,9.07) {$1\,2\,3$};
\node[draw,ellipse,black,fill=white] (node125) at (9.06,7.94) {$1\,2\,5$};
\node[draw,ellipse,black,fill=white] (node135) at (2.80,7.67) {$1\,3\,5$};
\node[draw,ellipse,black,fill=white] (node145) at (2.80,6.81) {$1\,4\,5$};
\node[draw,ellipse,black,fill=white] (node234) at (-7.33,8.64) {$2\,3\,4$};
\node[draw,ellipse,black,fill=white] (node235) at (-1.07,8.37) {$2\,3\,5$};
\node[draw,ellipse,black,fill=white] (node345) at (-7.33,7.24) {$3\,4\,5$};
\node[draw,ellipse,black,fill=white] (node1234) at (-1.73,11.52) {$1\,2\,3\,4$};
\node[draw,ellipse,black,fill=white] (node1235) at (4.53,11.25) {$1\,2\,3\,5$};
\node[draw,ellipse,black,fill=white] (node1245) at (4.53,10.39) {$1\,2\,4\,5$};
\node[draw,ellipse,black,fill=white] (node1345) at (-1.73,10.12) {$1\,3\,4\,5$};
\node[draw,ellipse,black,fill=white] (node2345) at (-5.60,10.82) {$2\,3\,4\,5$};
\node[draw,ellipse,black,fill=white] (node12345) at (-0.00,12.57) {$1\,2\,3\,4\,5$};
\coordinate (wnode0n1) at (0.93,2.13);
\coordinate (wnode0n2) at (0.93,1.66);
\coordinate (wnode0n3) at (-2.44,1.52);
\coordinate (wnode1n1) at (5.24,4.77);
\fill [opacity=0.2,black] (node12.center) -- (node13.center) -- (node23.center) -- cycle;
\coordinate (bnode123n1) at (1.87,5.38);
\fill [opacity=0.2,black] (node13.center) -- (node15.center) -- (node35.center) -- cycle;
\coordinate (bnode135n1) at (1.87,4.45);
\coordinate (wnode3n1) at (-1.51,5.06);
\coordinate (wnode5n1) at (0.58,3.93);
\coordinate (wnode3n2) at (-4.89,4.92);
\fill [opacity=0.2,black] (node34.center) -- (node35.center) -- (node45.center) -- cycle;
\coordinate (bnode345n1) at (-4.89,4.16);
\fill [opacity=0.2,black] (node123.center) -- (node125.center) -- (node135.center) -- cycle;
\coordinate (bnode1235n1) at (4.89,8.23);
\fill [opacity=0.2,black] (node123.center) -- (node135.center) -- (node235.center) -- cycle;
\coordinate (bnode1235n2) at (1.51,8.37);
\coordinate (wnode23n1) at (-1.87,8.70);
\coordinate (wnode15n1) at (4.89,7.48);
\fill [opacity=0.2,black] (node135.center) -- (node145.center) -- (node345.center) -- cycle;
\coordinate (bnode1345n1) at (-0.58,7.24);
\coordinate (wnode35n1) at (-1.87,7.76);
\fill [opacity=0.2,black] (node234.center) -- (node235.center) -- (node345.center) -- cycle;
\coordinate (bnode2345n1) at (-5.24,8.09);
\fill [opacity=0.2,black] (node1234.center) -- (node1235.center) -- (node2345.center) -- cycle;
\coordinate (bnode12345n1) at (-0.93,11.20);
\fill [opacity=0.2,black] (node1235.center) -- (node1245.center) -- (node1345.center) -- cycle;
\coordinate (bnode12345n2) at (2.44,10.59);
\fill [opacity=0.2,black] (node1235.center) -- (node1345.center) -- (node2345.center) -- cycle;
\coordinate (bnode12345n3) at (-0.93,10.73);
\draw[line width=0.03mm,black] (node1) -- (node2);
\draw[line width=0.03mm,black] (node1) -- (node3);
\draw[line width=0.03mm,black] (node1) -- (node5);
\draw[line width=0.03mm,black] (node2) -- (node1);
\draw[line width=0.03mm,black] (node2) -- (node3);
\draw[line width=0.03mm,black] (node3) -- (node1);
\draw[line width=0.03mm,black] (node3) -- (node2);
\draw[line width=0.03mm,black] (node3) -- (node4);
\draw[line width=0.03mm,black] (node3) -- (node5);
\draw[line width=0.03mm,black] (node4) -- (node3);
\draw[line width=0.03mm,black] (node4) -- (node5);
\draw[line width=0.03mm,black] (node5) -- (node1);
\draw[line width=0.03mm,black] (node5) -- (node3);
\draw[line width=0.03mm,black] (node5) -- (node4);
\draw[line width=0.03mm,black] (node12) -- (node13);
\draw[line width=0.03mm,black] (node12) -- (node15);
\draw[line width=0.03mm,black] (node12) -- (node23);
\draw[line width=0.03mm,black] (node13) -- (node12);
\draw[line width=0.03mm,black] (node13) -- (node15);
\draw[line width=0.03mm,black] (node13) -- (node23);
\draw[line width=0.03mm,black] (node13) -- (node35);
\draw[line width=0.03mm,black] (node15) -- (node12);
\draw[line width=0.03mm,black] (node15) -- (node13);
\draw[line width=0.03mm,black] (node15) -- (node35);
\draw[line width=0.03mm,black] (node15) -- (node45);
\draw[line width=0.03mm,black] (node23) -- (node12);
\draw[line width=0.03mm,black] (node23) -- (node13);
\draw[line width=0.03mm,black] (node23) -- (node34);
\draw[line width=0.03mm,black] (node23) -- (node35);
\draw[line width=0.03mm,black] (node34) -- (node23);
\draw[line width=0.03mm,black] (node34) -- (node35);
\draw[line width=0.03mm,black] (node34) -- (node45);
\draw[line width=0.03mm,black] (node35) -- (node13);
\draw[line width=0.03mm,black] (node35) -- (node15);
\draw[line width=0.03mm,black] (node35) -- (node23);
\draw[line width=0.03mm,black] (node35) -- (node34);
\draw[line width=0.03mm,black] (node35) -- (node45);
\draw[line width=0.03mm,black] (node45) -- (node15);
\draw[line width=0.03mm,black] (node45) -- (node34);
\draw[line width=0.03mm,black] (node45) -- (node35);
\draw[line width=0.03mm,black] (node123) -- (node125);
\draw[line width=0.03mm,black] (node123) -- (node135);
\draw[line width=0.03mm,black] (node123) -- (node234);
\draw[line width=0.03mm,black] (node123) -- (node235);
\draw[line width=0.03mm,black] (node125) -- (node123);
\draw[line width=0.03mm,black] (node125) -- (node135);
\draw[line width=0.03mm,black] (node125) -- (node145);
\draw[line width=0.03mm,black] (node135) -- (node123);
\draw[line width=0.03mm,black] (node135) -- (node125);
\draw[line width=0.03mm,black] (node135) -- (node145);
\draw[line width=0.03mm,black] (node135) -- (node235);
\draw[line width=0.03mm,black] (node135) -- (node345);
\draw[line width=0.03mm,black] (node145) -- (node125);
\draw[line width=0.03mm,black] (node145) -- (node135);
\draw[line width=0.03mm,black] (node145) -- (node345);
\draw[line width=0.03mm,black] (node234) -- (node123);
\draw[line width=0.03mm,black] (node234) -- (node235);
\draw[line width=0.03mm,black] (node234) -- (node345);
\draw[line width=0.03mm,black] (node235) -- (node123);
\draw[line width=0.03mm,black] (node235) -- (node135);
\draw[line width=0.03mm,black] (node235) -- (node234);
\draw[line width=0.03mm,black] (node235) -- (node345);
\draw[line width=0.03mm,black] (node345) -- (node135);
\draw[line width=0.03mm,black] (node345) -- (node145);
\draw[line width=0.03mm,black] (node345) -- (node234);
\draw[line width=0.03mm,black] (node345) -- (node235);
\draw[line width=0.03mm,black] (node1234) -- (node1235);
\draw[line width=0.03mm,black] (node1234) -- (node2345);
\draw[line width=0.03mm,black] (node1235) -- (node1234);
\draw[line width=0.03mm,black] (node1235) -- (node1245);
\draw[line width=0.03mm,black] (node1235) -- (node1345);
\draw[line width=0.03mm,black] (node1235) -- (node2345);
\draw[line width=0.03mm,black] (node1245) -- (node1235);
\draw[line width=0.03mm,black] (node1245) -- (node1345);
\draw[line width=0.03mm,black] (node1345) -- (node1235);
\draw[line width=0.03mm,black] (node1345) -- (node1245);
\draw[line width=0.03mm,black] (node1345) -- (node2345);
\draw[line width=0.03mm,black] (node2345) -- (node1234);
\draw[line width=0.03mm,black] (node2345) -- (node1235);
\draw[line width=0.03mm,black] (node2345) -- (node1345);
\node[draw,ellipse,black,fill=white] (node0) at (0.00,0.00) {$\emptyset$};
\node[draw,ellipse,black,fill=white] (node1) at (5.60,1.75) {$1$};
\node[draw,ellipse,black,fill=white] (node2) at (1.73,2.45) {$2$};
\node[draw,ellipse,black,fill=white] (node3) at (-4.53,2.18) {$3$};
\node[draw,ellipse,black,fill=white] (node4) at (-4.53,1.32) {$4$};
\node[draw,ellipse,black,fill=white] (node5) at (1.73,1.05) {$5$};
\node[draw,ellipse,black,fill=white] (node12) at (7.33,5.33) {$1\,2$};
\node[draw,ellipse,black,fill=white] (node13) at (1.07,5.06) {$1\,3$};
\node[draw,ellipse,black,fill=white] (node15) at (7.33,3.93) {$1\,5$};
\node[draw,ellipse,black,fill=white] (node23) at (-2.80,5.76) {$2\,3$};
\node[draw,ellipse,black,fill=white] (node34) at (-9.06,4.63) {$3\,4$};
\node[draw,ellipse,black,fill=white] (node35) at (-2.80,4.36) {$3\,5$};
\node[draw,ellipse,black,fill=white] (node45) at (-2.80,3.50) {$4\,5$};
\node[draw,ellipse,black,fill=white] (node123) at (2.80,9.07) {$1\,2\,3$};
\node[draw,ellipse,black,fill=white] (node125) at (9.06,7.94) {$1\,2\,5$};
\node[draw,ellipse,black,fill=white] (node135) at (2.80,7.67) {$1\,3\,5$};
\node[draw,ellipse,black,fill=white] (node145) at (2.80,6.81) {$1\,4\,5$};
\node[draw,ellipse,black,fill=white] (node234) at (-7.33,8.64) {$2\,3\,4$};
\node[draw,ellipse,black,fill=white] (node235) at (-1.07,8.37) {$2\,3\,5$};
\node[draw,ellipse,black,fill=white] (node345) at (-7.33,7.24) {$3\,4\,5$};
\node[draw,ellipse,black,fill=white] (node1234) at (-1.73,11.52) {$1\,2\,3\,4$};
\node[draw,ellipse,black,fill=white] (node1235) at (4.53,11.25) {$1\,2\,3\,5$};
\node[draw,ellipse,black,fill=white] (node1245) at (4.53,10.39) {$1\,2\,4\,5$};
\node[draw,ellipse,black,fill=white] (node1345) at (-1.73,10.12) {$1\,3\,4\,5$};
\node[draw,ellipse,black,fill=white] (node2345) at (-5.60,10.82) {$2\,3\,4\,5$};
\node[draw,ellipse,black,fill=white] (node12345) at (-0.00,12.57) {$1\,2\,3\,4\,5$};
\node[scale=2,anchor=west] (SIGN_0) at (0.70,-0.12) {$z=0$};
\node[scale=2,anchor=west] (SIGN_1) at (7.00,1.75) {$z=1$};
\node[scale=2,anchor=north west] (SIGN_2) at (8.73,5.33) {$z=2$};
\node[scale=2,anchor=west] (SIGN_3) at (10.46,7.94) {$z=3$};
\node[scale=2,anchor=north west] (SIGN_4) at (5.93,11.25) {$z=4$};
\node[scale=2,anchor=west] (SIGN_5) at (1.40,12.57) {$z=5$};
\end{tikzpicture}}

\caption{\label{fig:chord}The horizontal sections of a zonotopal tiling of $\Zon_{C^{5,3}}$ by planes $z=k$, for $k=0,1,\dots, 5$, are dual to trivalent plabic graphs.}
\end{figure}
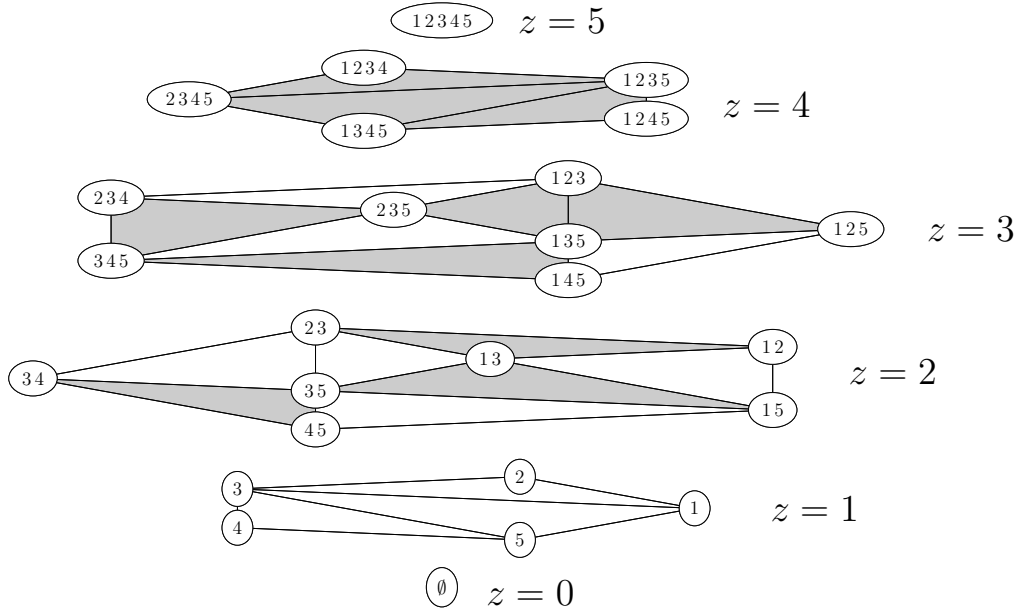

It follows directly from the definitions and Lemma~\ref{lem:alternating_matroid}
that chord separation is equivalent to $\Mcal$-separation
for the rank $3$ alternating oriented matroid $\Mcal=C^{n,3}$:

\begin{lemma}
A collection $\WS\subset 2^{[n]}$ is chord separated if and only if
$\WS$ is $C^{n,3}$-separated. 
\end{lemma}

Notice that the number of independent sets in $C^{n,3}$, 
which are all subsets of $[n]$ with at most 3 elements,
is exactly ${n\choose 0}+{n\choose 1}+{n\choose 2}+{n\choose 3}$.

Using our matroidal terminology, 
Theorems~\ref{thm:purity_ss} and~\ref{thm:purity_chord} imply
the following result.

\begin{corollary}
\label{cor:c23narepure}
The rank $2$ and $3$ alternating oriented matroids $C^{n,2}$ and $C^{n,3}$ 
are pure.
\end{corollary}

\subsection{Weak separation}

Leclerc and Zelevinsky \cite{LZ} also introduced \emph{weak separation}, which
is a more subtle notion than strong separation.

For two sets
$I,J\subset[n]$, we say that $I$ \emph{surrounds} $J$ if their set-theoretic
difference $I\setminus J$ 
can be partitioned as a disjoint union of two sets $I_1$ and
$I_2$ so that $I_1 \prec (J\setminus I) \prec I_2$.
Here, for two sets $A$ and $B$ of integers, the notation $A\prec B$ means
that any element of $A$ is less than any element of $B$.

\begin{definition}[{\cite{LZ}}]
	Two sets $I,J\subset[n]$ 
are \emph{weakly separated} if 
	\begin{enumerate}
		\item $|I|\leq |J|$ and $I$ surrounds $J$, or
		\item $|J|\leq |I|$ and $J$ surrounds $I$.
	\end{enumerate}
\end{definition}

Clearly, if $|I|=|J|$ then $I$ and $J$ are weakly separated if
and only if they are chord separated. 
However, sets $I$ and $J$ of different cardinalities can be chord separated but not weakly separated.

Leclerc-Zelevinsky's purity conjecture \cite[Conjecture~1.5]{LZ}
for weak separation was independently proved in~\cite[Theorem~A]{DKK10}
and~\cite[Theorem~1.3]{OPS} (in a more general version for \emph{positroids}).
Let us formulate two special cases of this result.

Let ${[n]\choose k}$ be the set of all $k$-element subsets of $[n]$.
Thus $2^{[n]}$ is the disjoint union of the sets ${[n]\choose k}$, $k=0,1,\dots,n$.

\begin{theorem}[{\cite[Theorem~1.3]{OPS}, \cite[Theorem~A]{DKK10}}] 
\label{thm:purity_ws}\leavevmode
\begin{enumerate}[\normalfont (1)]
\item\label{item:purity_nchoosek}  Every maximal \emph{by inclusion} weakly separated collection $\WS\subset {[n]\choose k}$ of $k$-element subsets of $[n]$ is also maximal \emph{by size}: 
  \[|\WS|=k(n-k)+1.\]
\item\label{item:purity_2_n} Every maximal \emph{by inclusion} weakly separated collection 
$\WS\subset 2^{[n]}$ is also maximal \emph{by size}: 
	\[|\WS|={n\choose 0}+{n\choose 1}+{n\choose 2}.\]
\end{enumerate}
\end{theorem}

This theorem was proved in~\cite{OPS} by constructing a bijection between maximal
\empu{by inclusion} weakly separated collections $\WS\subset {[n]\choose k}$
and \emph{reduced plabic graphs} introduced in~\cite{Postnikov}
in the study of the totally nonnegative Grassmannian.

Figure~\ref{fig:plabic} shows an example
of a weakly separated collection of $k$-element subsets of $[n]$,
for $k=3$ and  $n=6$:
\[\WS=\{123,126,156,236,136,146,346,234,345,456\}\subset {[6]\choose 3}.\]
It consists of  $|\WS|=k(n-k)+1=3\times 3+1=10$ elements. Here
we abbreviate a subset $\{a,b,c\}\subset [n]$ by $abc$.

\begin{figure}
\scalebox{1.0}{
\begin{tikzpicture}[yscale=0.8,every node/.style={scale=0.6}]
\node[draw,circle,black,fill=white] (node123) at (2.80,3.64) {$1\,2\,3$};
\node[draw,circle,black,fill=white] (node126) at (5.60,0.00) {$1\,2\,6$};
\node[draw,circle,black,fill=white] (node136) at (2.80,0.00) {$1\,3\,6$};
\node[draw,circle,black,fill=white] (node146) at (1.40,-1.82) {$1\,4\,6$};
\node[draw,circle,black,fill=white] (node156) at (2.80,-3.64) {$1\,5\,6$};
\node[draw,circle,black,fill=white] (node234) at (-2.80,3.64) {$2\,3\,4$};
\node[draw,circle,black,fill=white] (node236) at (1.40,1.82) {$2\,3\,6$};
\node[draw,circle,black,fill=white] (node345) at (-5.60,0.00) {$3\,4\,5$};
\node[draw,circle,black,fill=white] (node346) at (-2.80,0.00) {$3\,4\,6$};
\node[draw,circle,black,fill=white] (node456) at (-2.80,-3.64) {$4\,5\,6$};
\draw [opacity=0.2,fill=black,black,thick] (node123.center)-- (node126.center)-- (node136.center)-- (node236.center) -- cycle;
\draw [opacity=0.2,fill=black,black,thick] (node136.center)-- (node146.center)-- (node346.center) -- cycle;
\draw [opacity=0.2,fill=black,black,thick] (node146.center)-- (node156.center)-- (node456.center) -- cycle;
\draw [opacity=0.2,fill=black,black,thick] (node234.center)-- (node236.center)-- (node346.center) -- cycle;
\draw [opacity=0.2,fill=black,black,thick] (node345.center)-- (node346.center)-- (node456.center) -- cycle;
\draw [opacity=0.2,black,thick] (node123.center)-- (node126.center)-- (node136.center)-- (node236.center) -- cycle;
\draw [opacity=0.2,black,thick] (node136.center)-- (node146.center)-- (node346.center) -- cycle;
\draw [opacity=0.2,black,thick] (node146.center)-- (node156.center)-- (node456.center) -- cycle;
\draw [opacity=0.2,black,thick] (node234.center)-- (node236.center)-- (node346.center) -- cycle;
\draw [opacity=0.2,black,thick] (node345.center)-- (node346.center)-- (node456.center) -- cycle;
\draw [opacity=0.2,black,thick] (node126.center)-- (node136.center)-- (node146.center)-- (node156.center) -- cycle;
\draw [opacity=0.2,black,thick] (node123.center)-- (node234.center)-- (node236.center) -- cycle;
\draw [opacity=0.2,black,thick] (node234.center)-- (node345.center)-- (node346.center) -- cycle;
\draw [opacity=0.2,black,thick] (node136.center)-- (node236.center)-- (node346.center) -- cycle;
\draw [opacity=0.2,black,thick] (node146.center)-- (node346.center)-- (node456.center) -- cycle;
\coordinate (node12) at (5.04,2.18);
\coordinate (node16) at (3.15,-1.36);
\coordinate (node23) at (0.47,3.03);
\coordinate (node34) at (-3.73,1.21);
\coordinate (node36) at (0.47,0.61);
\coordinate (node45) at (-5.04,-2.18);
\coordinate (node46) at (-1.40,-1.82);
\coordinate (node56) at (-0.00,-4.36);
\coordinate (node1234) at (0.00,4.36);
\coordinate (node1256) at (5.04,-2.18);
\coordinate (node1346) at (0.47,-0.61);
\coordinate (node1456) at (0.47,-3.03);
\coordinate (node2345) at (-5.04,2.18);
\coordinate (node2346) at (-1.40,1.82);
\coordinate (node3456) at (-3.73,-1.21);

\def\dy{-0.5}
\def\dx{0.7}
\draw [opacity=0.4,black,thick] (node123.center)-- (node136.center);
\draw[blue] (5.04,2.18).. controls (4.20,1.82) .. ({3.15+\dx},{1.36+\dy});
% \draw[blue] (0.47,3.03).. controls (2.10,2.73) ..  ({3.15+\dx},{1.36+\dy});
% \draw[blue] (0.47,0.61).. controls (2.10,0.91) ..  ({3.15+\dx},{1.36+\dy});
% \draw[blue] (5.04,2.18).. controls (4.20,1.82) .. ({3.15-\dx},{1.36-\dy});
\draw[blue] (0.47,3.03).. controls (2.10,2.73) ..  ({3.15-\dx},{1.36-\dy});
\draw[blue] (0.47,0.61).. controls (2.10,0.91) ..  ({3.15-\dx},{1.36-\dy});
% \draw[blue] (3.15,-1.36).. controls (4.20,0.00) ..  ({3.15-\dx},{1.36-\dy});
\draw[blue] ({3.15+\dx},{1.36+\dy}).. controls ({3.15},{1.66}) ..  ({3.15-\dx},{1.36-\dy});
\draw[blue,fill=blue] ({3.15+\dx},{1.36+\dy}) circle (3pt);
\draw[blue,fill=blue] ({3.15-\dx},{1.36-\dy}) circle (3pt);

\def\ddy{-0.7}
\def\ddx{0.1}
\draw [opacity=0.4,black,thick] (node146.center)-- (node126.center);
\draw[blue] ({3.15-\ddx},{-1.36-\ddy}).. controls (4.20,0.00) ..  ({3.15+\dx},{1.36+\dy});
\draw[blue] ({3.15-\ddx},{-1.36-\ddy}).. controls (2.10,-0.91) .. (0.47,-0.61);

\draw[blue] ({3.15+\ddx},{-1.36+\ddy}).. controls (2.10,-2.73) .. (0.47,-3.03);
\draw[blue] ({3.15+\ddx},{-1.36+\ddy}).. controls (4.20,-1.82) .. (5.04,-2.18);
% \draw[blue] (5.04,-2.18).. controls (4.20,-1.82) .. (3.15,-1.36);
% \draw[blue] (0.47,-0.61).. controls (2.10,-0.91) .. (3.15,-1.36);
% \draw[blue] (0.47,-3.03).. controls (2.10,-2.73) .. (3.15,-1.36);
\draw[blue]  ({3.15+\ddx},{-1.36+\ddy}).. controls ({3.65},{-1.06}) ..   ({3.15-\ddx},{-1.36-\ddy});
\draw[blue,fill=white] ({3.15+\ddx},{-1.36+\ddy}) circle (3pt);
\draw[blue,fill=white] ({3.15-\ddx},{-1.36-\ddy}) circle (3pt);

\draw[blue] (0.47,3.03).. controls (0.00,3.64) .. (0.00,4.36);
\draw[blue] (0.47,3.03).. controls (-0.70,2.73) .. (-1.40,1.82);
\draw[blue] (-3.73,1.21).. controls (-4.20,1.82) .. (-5.04,2.18);
\draw[blue] (-3.73,1.21).. controls (-4.20,0.00) .. (-3.73,-1.21);
\draw[blue] (-3.73,1.21).. controls (-2.80,1.82) .. (-1.40,1.82);
\draw[blue] (0.47,0.61).. controls (-0.70,0.91) .. (-1.40,1.82);
\draw[blue] (0.47,0.61).. controls (0.00,0.00) .. (0.47,-0.61);
\draw[blue] (-5.04,-2.18).. controls (-4.20,-1.82) .. (-3.73,-1.21);
\draw[blue] (-1.40,-1.82).. controls (-0.70,-0.91) .. (0.47,-0.61);
\draw[blue] (-1.40,-1.82).. controls (-2.80,-1.82) .. (-3.73,-1.21);
\draw[blue] (-1.40,-1.82).. controls (-0.70,-2.73) .. (0.47,-3.03);
\draw[blue] (-0.00,-4.36).. controls (-0.00,-3.64) .. (0.47,-3.03);
\draw[blue] (0.00,4.36).. controls (0.00,3.64) .. (0.47,3.03);
\draw[blue] (0.47,-0.61).. controls (-0.70,-0.91) .. (-1.40,-1.82);
\draw[blue] (0.47,-0.61).. controls (0.00,0.00) .. (0.47,0.61);
\draw[blue] (0.47,-3.03).. controls (-0.00,-3.64) .. (-0.00,-4.36);
\draw[blue] (0.47,-3.03).. controls (-0.70,-2.73) .. (-1.40,-1.82);
\draw[blue] (-5.04,2.18).. controls (-4.20,1.82) .. (-3.73,1.21);
\draw[blue] (-1.40,1.82).. controls (-0.70,2.73) .. (0.47,3.03);
\draw[blue] (-1.40,1.82).. controls (-0.70,0.91) .. (0.47,0.61);
\draw[blue] (-1.40,1.82).. controls (-2.80,1.82) .. (-3.73,1.21);
\draw[blue] (-3.73,-1.21).. controls (-4.20,0.00) .. (-3.73,1.21);
\draw[blue] (-3.73,-1.21).. controls (-2.80,-1.82) .. (-1.40,-1.82);
\draw[blue] (-3.73,-1.21).. controls (-4.20,-1.82) .. (-5.04,-2.18);
\draw[blue,fill=white] (node12) circle (3pt);
\draw[blue,fill=white] (node23) circle (3pt);
\draw[blue,fill=white] (node34) circle (3pt);
\draw[blue,fill=white] (node36) circle (3pt);
\draw[blue,fill=white] (node45) circle (3pt);
\draw[blue,fill=white] (node46) circle (3pt);
\draw[blue,fill=white] (node56) circle (3pt);
\draw[blue,fill=white] (node1234) circle (3pt);
\draw[blue,fill=white] (node1256) circle (3pt);
\draw[blue,fill=blue] (node1346) circle (3pt);
\draw[blue,fill=blue] (node1456) circle (3pt);
\draw[blue,fill=white] (node2345) circle (3pt);
\draw[blue,fill=blue] (node2346) circle (3pt);
\draw[blue,fill=blue] (node3456) circle (3pt);
\node[draw,circle,black,fill=white] (node123) at (2.80,3.64) {$1\,2\,3$};
\node[draw,circle,black,fill=white] (node126) at (5.60,0.00) {$1\,2\,6$};
\node[draw,circle,black,fill=white] (node136) at (2.80,0.00) {$1\,3\,6$};
\node[draw,circle,black,fill=white] (node146) at (1.40,-1.82) {$1\,4\,6$};
\node[draw,circle,black,fill=white] (node156) at (2.80,-3.64) {$1\,5\,6$};
\node[draw,circle,black,fill=white] (node234) at (-2.80,3.64) {$2\,3\,4$};
\node[draw,circle,black,fill=white] (node236) at (1.40,1.82) {$2\,3\,6$};
\node[draw,circle,black,fill=white] (node345) at (-5.60,0.00) {$3\,4\,5$};
\node[draw,circle,black,fill=white] (node346) at (-2.80,0.00) {$3\,4\,6$};
\node[draw,circle,black,fill=white] (node456) at (-2.80,-3.64) {$4\,5\,6$};
\end{tikzpicture}}

	\caption{\label{fig:plabic}A (trivalent) \emph{plabic graph} is shown in blue. The face labels form a maximal \empu{by inclusion} weakly separated collection of $k$-element subsets of $[n]$, for $k=3$ and $n=6$.}
\end{figure}
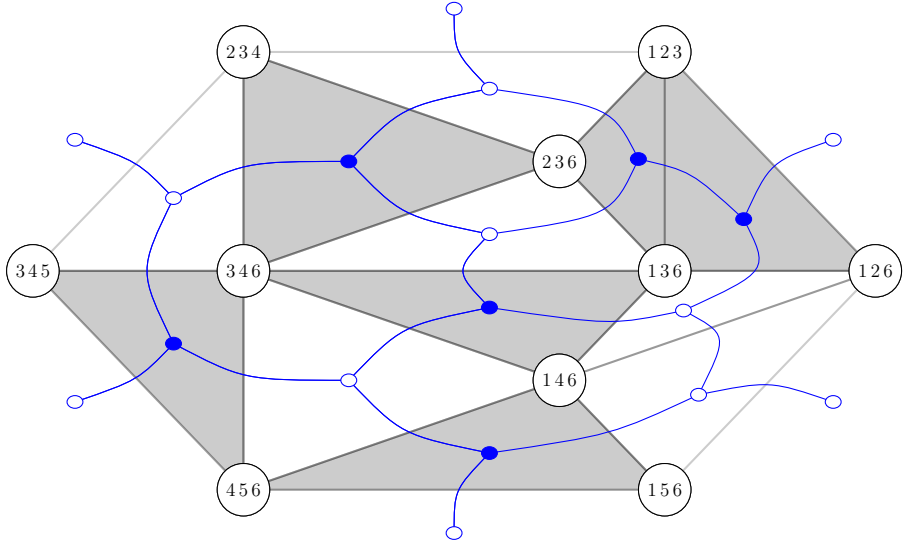

%Note that all sets in ${[n] \choose k}$ have the same size in which case 

\begin{remark}
Part~\eqref{item:purity_2_n} of Theorem~\ref{thm:purity_ws},
concerning weakly separated 
collections $\WS\subset 2^{[n]}$ of subsets of various
cardinalities, was 
deduced in \cite{OPS} from part~\eqref{item:purity_nchoosek}, concerning collections of
subsets of the same cardinality,
using a simple \emph{padding construction} as follows. 

\def\WShat{\widehat{\WS}_0}
\def\WSnot{\WS_0}
Let $\pad: 2^{[n]} \to {[2n]\choose n}$ be the injective map
given by $\pad(I)= I \cup \{2n,2n-1,\dots,n + |I|+1\}$,
for $I\subset [n]$.
It is easy to see \cite[Lemma~12.7]{OPS} that $\WS$
is a weakly separated collection in $2^{[n]}$ if and only if
its image $\pad(\WS)$ is a weakly separated collection in $[2n]\choose n$.
Moreover, according to \cite{OPS}, maximal by inclusion
weakly separated collections $\WS\subset 2^{[n]}$ correspond
to maximal by inclusion weakly separated collections 
$\tilde\WS \subset {[2n]\choose n}$
of $n$-element subsets in $[2n]$ that contain some fixed collection $\WShat$. The correspondence is given explicitly by $\WS\mapsto \pad(\WS)\sqcup\WSnot$ (disjoint union) for a slightly smaller fixed collection $\WSnot\subset\WShat$.  One can take, for example,
\begin{equation*}
  \begin{split}
    \WSnot&:= \{[a,n]\cup [b,c]\mid 0< a\leq n <b\leq c< 2n,\ |[a,n]\cup[b,c]|=n\};\\
    \WShat&:= \{[a,n]\cup [b,c]\mid 0< a\leq n <b\leq c\leq 2n,\ |[a,n]\cup[b,c]|=n\}.\\
  \end{split}
\end{equation*}
This shows that the original Leclerc-Zelevinsky's notion of weak
separation essentially reduces to the notion of weak separation 
(equivalently, chord separation) for collections of subsets of the same 
cardinality.
\end{remark}

Let us show how weakly separated collections $\WS\subset
{[n]\choose k}$ fit into our general setup of oriented 
matroids.

Observe that ${[n]\choose k}$ is a mutation-closed domain
for the alternating matroid $\Mcal = C^{n,3}$ of rank $3$.
Theorem~\ref{thm:purity_ws} implies that this is an 
$\Mcal$-pure domain.

\begin{corollary}
\label{cor:Cn3isMpure}
Let $\Mcal=C^{n,3}$ be the alternating oriented matroid of rank $3$.
The mutation-closed domain $[n]\choose k$,
for $k=0,1,\dots,n$, 
is an $\Mcal$-pure domain.
\end{corollary}

Note that, in view of Proposition~\ref{prop:graphicalandflipconnected}
and Theorem~\ref{th:Ziegler_flips},
Corollary~\ref{cor:Cn3isMpure} formally 
follows from Corollary~\ref{cor:c23narepure}.
However, the proof of purity of $C^{n,3}$ given in \cite{Galashin}
relies on the $C^{n,3}$-purity of all domains $[n]\choose k$ that was proven in~\cite{OPS,DKK10}.

\section{Simple operations on oriented matroids}\label{sec:simple-oper-orient}

There are several simple operations on oriented matroids
that do not affect purity.

\subsection{Relabeling and adding/removing loops and coloops}
Clearly, an oriented matroid obtained from a pure oriented matroid
by relabeling the elements of the ground set is again pure.

The following lemma is straightforward. See Section~\ref{sect:background} for the definitions of loops and coloops and Section~\ref{sect:pure_OM} for the proof.

\begin{lemma}\label{lemma:adding_removing_zeros}
Let $\Mcal$ be a pure oriented matroid.
Then any oriented matroid obtained from $\Mcal$ by adding (or removing) 
loops and coloops is pure.
\end{lemma}

\subsection{Adding/removing parallel elements}

There is another simple operation on oriented matroids $\Mcal$,
the operation of adding parallel elements.
Let $e\subset E$ be an element of the ground
set of $\Mcal$.  Let $\Mcal'$ be the oriented matroid on the ground set
$E'=E\cup \{e'\}$ (where $e'\not\in E$) whose set of circuits
contains exactly all circuits of $\Mcal$, all circuits of $\Mcal$ 
with the element $e$ replaced by $e'$, and also the circuits
given by the signed sets $(\{e\},\{e'\})$ 
and  $(\{e'\},\{e\})$.
If $\Mcal$ is an oriented matroid associated with a vector configuration $\VC$,
this operation means that we add an extra copy of some vector $\v_i$ to 
$\VC$.
We say that an oriented matroid is 
\emph{obtained by adding parallel elements} from $\Mcal$ if it is 
obtained by a sequence of such operations.

The following result is easy to formulate but (surprisingly) hard to prove,
see Lemma~\ref{lemma:parallel}.

\begin{lemma}\label{lemma:parallel_intro}
Let $\Mcal$ be an oriented matroid and $\Mcal'$ be any 
oriented matroid obtained from $\Mcal$ by adding parallel elements.
Then $\Mcal$ is pure 
if and only if $\Mcal'$ is pure.
\end{lemma}

\subsection{Reorientations}

Let us also describe the operation of \emph{reorientation,} 
defined on signed sets and on oriented matroids as follows.

For a signed subset $X=(X^+,X^-)$ of $E$ and an element $f\in E$, we write
\begin{equation}\label{eq:X_f}
X_f=
  \begin{cases}
    1, &\text{if $f\in X^+$,}\\
    -1, &\text{if $f\in X^-$,}\\
    0, &\text{if $f\in E\setminus \Xu$.}\\
  \end{cases}
\end{equation}

  For $e\in E$, let $\reorient{e}{X}=X'$ be the signed subset of $E$ such that
$X'_e = -X_e$, and 
$X'_f = X_f$, for $f\ne e$. 

For an oriented matroid $\Mcal=(E,\Ccal)$, let $\reorient{e}{\Mcal}$
be the oriented matroid on the same ground set $E$ with circuits $\reorient{e}{X}$, for $X\in \Ccal$.

If $\Mcal$ is the oriented matroid associated with a vector configuration 
$\VC=(\v_1,\dots,\v_n)$, then $\reorient{i}{\Mcal}$ is the oriented matroid
associated with the vector configuration 
$(\v_1,\dots,\v_{i-1},-\v_i,\v_{i+1},\dots,\v_n)$.

Let us define related operations on $\Mcal$-separated collections.
For a usual subset $I\subset E$, let 
\begin{equation}\label{eq:reorient_set}
\reorient{e}{I} := 
\begin{cases}
  I\setminus\{e\}, &\text{if $e\in I$,}\\
  I\cup\{e\}, &\text{if $e\notin I$.}\\
\end{cases}
\end{equation}

For a collection $\WS\subset 2^E$, let $\reorient{e}{\WS}$ be the collection 
of subsets $\reorient{e}{I}$, for $I\in \WS$.

The following lemma follows directly from the definitions.

\begin{lemma}
\label{lem:signreversal}  
Let $\Mcal=(E,\Ccal)$  be an oriented matroid, and let $e\in E$.
Then $\WS$ is an $\Mcal$-separated collection if and only if $\reorient{e}{\WS}$ is
an $\reorient{e}{\Mcal}$-separated collection.

Thus $\Mcal$ is pure if and only if $\reorient{e}{\Mcal}$ is pure.
\end{lemma}

Let us say that two oriented matroids $\Mcal$ and $\Mcal'$ are \emph{isomorphic} if they can be obtained from each other by a sequence of reorientations followed by a relabeling of the ground set. Lemma~\ref{lem:signreversal} implies that, for two isomorphic oriented matroids $\Mcal$ and $\Mcal'$, there is a natural  (inclusion- and cardinality-preserving) one-to-one correspondence between $\Mcal$-separated collections and $\Mcal'$-separated collections. In particular, we get the following result.

\begin{proposition}
For two isomorphic oriented matroids $\Mcal$ and $\Mcal'$, $\Mcal$ is pure if and only if $\Mcal'$ is pure.
\end{proposition}

\section{Main results on purity}\label{sect:main}

\subsection{Purity of matroids of rank $2$ or corank $1$}

We prove the following easy claim in Section~\ref{sect:classif}.

\begin{proposition}\label{prop:rk_2_cork_1}
Any oriented matroid $\Mcal$ such that $\rk(\Mcal)\leq 2$ or
$\cork(\Mcal)\leq 1$ is pure.
\end{proposition}

\subsection{Purity of rank $3$ oriented matroids}

%Recall that in this
%section we assume that $\VC$ is a vector configuration that consists of nonzero
%vectors and does not contain pairs of collinear vectors. We later show in
%that these conditions do not affect the purity of
%$\VC$.

\begin{definition}[cf.~\cite{Postnikov}]
  An oriented matroid $\Mcal$ of rank $d$ on the ground set $[n]$ is a \emph{positroid} if it can be represented by the columns of a $d\times n$ matrix of rank $d$ all of whose $d\times d$ minors are nonnegative.\footnote{In~\cite{Postnikov}, positroids were defined as (unoriented) matroids. But they can be naturally endowed with a structure of an oriented matroid.}
\end{definition}

According to~\cite{Postnikov}, full rank $d\times n$ matrices with nonnegative $d\times d$ minors represent points of the \emph{totally nonnegative Grassmannian} $Gr_{d,n}^{\geq 0}$. It comes equipped with a CW decomposition into cells labeled by positroids.

The following result generalizes Theorem~\ref{thm:purity_chord}.

\begin{theorem}\label{thm:purity_3d}
Let $\Mcal$ be an oriented matroid of rank $3$.
Then the following are equivalent:
 \begin{enumerate}[{\normalfont (1)}]
  \item \label{item:purity_3d_is_pure} $\Mcal$ is pure.
  \item\label{item:3_x_n_matrix} %$\Mcal$ is a positroid,\footnote{More precisely, $\Mcal$ is isomorphic to a \emph{positively oriented matroid}, see Remark~\ref{rmk:positroids}.}
    $\Mcal$ is isomorphic to a positroid.
  \item\label{item:purity_3d_polygon} $\Mcal$ is represented by
a vector configuration $\VC$ such that, after a suitable rescaling 
of vectors (by nonzero scalars) and removing zero vectors, 
the endpoints of vectors in $\VC$ lie in the same affine plane and belong to the boundary of a convex $m$-gon for some $3\leq m\leq |\VC|$. 
 \end{enumerate}
\end{theorem}

\begin{remark}
We distinguish between being ``a positroid'' and  ``isomorphic to a positroid''. The property of being a positroid depends on the ordering of the 
elements of the ground set and is not invariant under reorientations. On the other hand, purity is invariant under relabeling and reorienting
the ground set of $\Mcal$.
% Theorem~\ref{thm:purity_3d} says that pure oriented matroids of rank $3$
% are exactly all oriented matroids obtained from positroids of rank $3$
% by some reorientation followed by relabeling the elements of the ground set.
\end{remark}

The equivalence of conditions~\eqref{item:3_x_n_matrix} and~\eqref{item:purity_3d_polygon} in Theorem~\ref{thm:purity_3d} is a simple well known fact. We prove that~\eqref{item:purity_3d_is_pure} implies~\eqref{item:3_x_n_matrix} in Theorem~\ref{thm:purity_vc}, and the converse is shown in Section~\ref{sect:classif}.

\subsection{Graphical oriented matroids}\label{sec:graph-orient-matr}

Let $\vec G$ be a
directed graph\footnote{We denote directed graphs by $\vec
G$ and undirected graphs by $G$.} with vertex set $[d]$ and with $n$ edges, then the
corresponding vector configuration $\VC_{\vec G}\subset \R^d$ consists of
vectors $e_i-e_j$, where $i\to j$ is an edge of $\vec G$, and the vectors
$e_i$, $1\leq i\leq d$, are the standard coordinate vectors in $\R^d$. 

The \emph{graphical oriented matroid} $\Mcal_{\vec G}$ is 
the oriented matroid associated with the vector configuration 
$\VC_{\vec G}$.

According to Lemma~\ref{lem:signreversal} on reorientations,
the property of $\Mcal_{\vec G}$ being pure does not depend on the orientation
of the edges of $\vec G$. 
We briefly explain how our definitions translate to undirected
graphs. 

Consider an undirected graph $G$ without loops or parallel edges. We say that two total orientations $O_1$ and $O_2$ of $G$ are \emph{$G$-separated} if there does not exist a cycle $C$ of $G$ such that $C$ is directed in both $O_1$ and $O_2$ but in the opposite ways. In particular, acyclic orientations of $G$ are $G$-separated from all other total orientations of $G$. We say that $G$ is \emph{pure} if the size of any maximal \emph{by inclusion} collection of pairwise $G$-separated total orientations of $G$ equals the number of \emph{forests} of $G$.

\begin{definition}[{\cite{CH}}]
An undirected graph $G$ is called \emph{outerplanar} if $G$ can be drawn in the plane without self-intersections and so that every vertex is incident to the exterior face of $G$.
\end{definition}

\begin{theorem}[\cite{CH}]\label{thm:outerplanar_minors}
	Given an undirected graph $G$, the following conditions are equivalent:
\begin{enumerate}[{\normalfont (1)}]
 \item $G$ is outerplanar;
 \item $G$ is a subgraph of the $1$-skeleton of a triangulation of a convex $m$-gon;
 \item $G$ does not contain $K_4$ or $K_{2,3}$ as a minor.
\end{enumerate}
\end{theorem}

Here $K_4$ denotes the complete graph with $4$ vertices and $K_{2,3}$ denotes the complete bipartite graph with $2+3$ vertices. The last condition in Theorem~\ref{thm:outerplanar_minors} is analogous to the celebrated theorems of Kuratowski~\cite{Kuratowski} and Wagner~\cite{Wagner} for planar graphs.

Here is our main result on the purity of graphical oriented matroids, which is proved in Section~\ref{sect:graph}.

\begin{theorem}\label{thm:outerplanar}
An undirected graph $G$ is pure
(i.e., the graphical oriented matroid $\Mcal_{\vec G}$ is pure
for any orientation $\vec G$ of $G$)
if and only if $G$ is outerplanar.
\end{theorem}

Note that the zonotope associated with the vector configuration $\VC_{\vec K_4}$ for the directed graph $\vec K_4$ with edge set $\{i\to j\mid 1\leq i<j\leq 4\}$ is the three-dimensional permutohedron. Figures~\ref{fig:K_23} and~\ref{fig:K_4} show that the graphs $K_{2,3}$ and $K_4$ are not pure. This is explained in more detail in Section~\ref{sect:graph}.

\begin{figure}
\centering
\begin{tabular}{ccc}
 \begin{tikzpicture}[scale=1]
  \node[draw,ellipse] (A) at (0.5,1) {$ $};
  \node[draw,ellipse] (B) at (1.5,1) {$ $};
  \node[draw,ellipse] (C) at (0,0) {$ $};
  \node[draw,ellipse] (D) at (1,0) {$ $};
  \node[draw,ellipse] (E) at (2,0) {$ $};
  \draw[->, red, line width=0.3mm] (C) -- (A);
  \draw[->, red, line width=0.3mm] (A) -- (E);
  \draw[->, red, line width=0.3mm] (E) -- (B);
  \draw[->, red, line width=0.3mm] (B) -- (C);
  \draw[->, black, line width=0.1mm] (A)--(D);
  \draw[->, black, line width=0.1mm] (B)--(D);
 \end{tikzpicture}
&
 \begin{tikzpicture}[scale=1]
  \node[draw,ellipse] (A) at (0.5,1) {$ $};
  \node[draw,ellipse] (B) at (1.5,1) {$ $};
  \node[draw,ellipse] (C) at (0,0) {$ $};
  \node[draw,ellipse] (D) at (1,0) {$ $};
  \node[draw,ellipse] (E) at (2,0) {$ $};
  \draw[->, red, line width=0.3mm] (A) -- (C);
  \draw[->, red, line width=0.3mm] (C) -- (B);
  \draw[->, red, line width=0.3mm] (B) -- (D);
  \draw[->, red, line width=0.3mm] (D) -- (A);
  \draw[->, black, line width=0.1mm] (A)--(E);
  \draw[->, black, line width=0.1mm] (B)--(E);
 \end{tikzpicture}
&
 \begin{tikzpicture}[scale=1]
  \node[draw,ellipse] (A) at (0.5,1) {$ $};
  \node[draw,ellipse] (B) at (1.5,1) {$ $};
  \node[draw,ellipse] (C) at (0,0) {$ $};
  \node[draw,ellipse] (D) at (1,0) {$ $};
  \node[draw,ellipse] (E) at (2,0) {$ $};
  \draw[->, red, line width=0.3mm] (A) -- (D);
  \draw[->, red, line width=0.3mm] (D) -- (B);
  \draw[->, red, line width=0.3mm] (B) -- (E);
  \draw[->, red, line width=0.3mm] (E) -- (A);
  \draw[->, black, line width=0.1mm] (A)--(C);
  \draw[->, black, line width=0.1mm] (B)--(C);
 \end{tikzpicture}\\
 
 \begin{tikzpicture}[scale=1]
  \node[draw,ellipse] (A) at (0.5,1) {$ $};
  \node[draw,ellipse] (B) at (1.5,1) {$ $};
  \node[draw,ellipse] (C) at (0,0) {$ $};
  \node[draw,ellipse] (D) at (1,0) {$ $};
  \node[draw,ellipse] (E) at (2,0) {$ $};
  \draw[->, red, line width=0.3mm] (C) -- (A);
  \draw[->, red, line width=0.3mm] (A) -- (E);
  \draw[->, red, line width=0.3mm] (E) -- (B);
  \draw[->, red, line width=0.3mm] (B) -- (C);
  \draw[->, black, line width=0.1mm] (D)--(A);
  \draw[->, black, line width=0.1mm] (D)--(B);
 \end{tikzpicture}
&
 \begin{tikzpicture}[scale=1]
  \node[draw,ellipse] (A) at (0.5,1) {$ $};
  \node[draw,ellipse] (B) at (1.5,1) {$ $};
  \node[draw,ellipse] (C) at (0,0) {$ $};
  \node[draw,ellipse] (D) at (1,0) {$ $};
  \node[draw,ellipse] (E) at (2,0) {$ $};
  \draw[->, red, line width=0.3mm] (A) -- (C);
  \draw[->, red, line width=0.3mm] (C) -- (B);
  \draw[->, red, line width=0.3mm] (B) -- (D);
  \draw[->, red, line width=0.3mm] (D) -- (A);
  \draw[->, black, line width=0.1mm] (E)--(A);
  \draw[->, black, line width=0.1mm] (E)--(B);
 \end{tikzpicture}
&
 \begin{tikzpicture}[scale=1]
  \node[draw,ellipse] (A) at (0.5,1) {$ $};
  \node[draw,ellipse] (B) at (1.5,1) {$ $};
  \node[draw,ellipse] (C) at (0,0) {$ $};
  \node[draw,ellipse] (D) at (1,0) {$ $};
  \node[draw,ellipse] (E) at (2,0) {$ $};
  \draw[->, red, line width=0.3mm] (A) -- (D);
  \draw[->, red, line width=0.3mm] (D) -- (B);
  \draw[->, red, line width=0.3mm] (B) -- (E);
  \draw[->, red, line width=0.3mm] (E) -- (A);
  \draw[->, black, line width=0.1mm] (C)--(A);
  \draw[->, black, line width=0.1mm] (C)--(B);
 \end{tikzpicture}
\end{tabular}

 \caption{\label{fig:K_23} A $K_{2,3}$-separated collection of total orientations that is maximal \emph{by inclusion} but not \emph{by size}. For each total orientation, the unique cycle that it orients is shown in red.}
\end{figure}
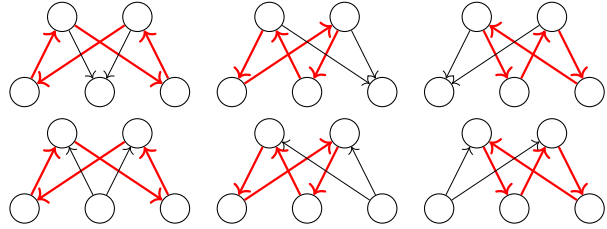

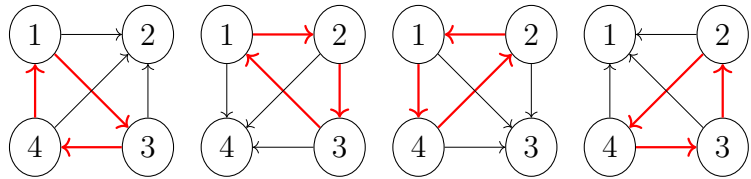
\begin{figure}
\centering
\begin{tabular}{cccc}
 \begin{tikzpicture}[scale=1.5]
  \node[draw,ellipse] (A) at (0,1) {$1$};
  \node[draw,ellipse] (B) at (1,1) {$2$};
  \node[draw,ellipse] (C) at (1,0) {$3$};
  \node[draw,ellipse] (D) at (0,0) {$4$};
  \draw[->, red, line width=0.3mm] (A) -- (C);
  \draw[->, red, line width=0.3mm] (C) -- (D);
  \draw[->, red, line width=0.3mm] (D) -- (A);
  \draw[->, black, line width=0.1mm] (A)--(B);
  \draw[->, black, line width=0.1mm] (C)--(B);
  \draw[->, black, line width=0.1mm] (D)--(B);
 \end{tikzpicture}
&
 \begin{tikzpicture}[scale=1.5]
  \node[draw,ellipse] (A) at (0,1) {$1$};
  \node[draw,ellipse] (B) at (1,1) {$2$};
  \node[draw,ellipse] (C) at (1,0) {$3$};
  \node[draw,ellipse] (D) at (0,0) {$4$};
  \draw[->, red, line width=0.3mm] (A) -- (B);
  \draw[->, red, line width=0.3mm] (B) -- (C);
  \draw[->, red, line width=0.3mm] (C) -- (A);
  \draw[->, black, line width=0.1mm] (A)--(D);
  \draw[->, black, line width=0.1mm] (B)--(D);
  \draw[->, black, line width=0.1mm] (C)--(D);
 \end{tikzpicture}
&
 \begin{tikzpicture}[scale=1.5]
  \node[draw,ellipse] (A) at (0,1) {$1$};
  \node[draw,ellipse] (B) at (1,1) {$2$};
  \node[draw,ellipse] (C) at (1,0) {$3$};
  \node[draw,ellipse] (D) at (0,0) {$4$};
  \draw[->, red, line width=0.3mm] (A) -- (D);
  \draw[->, red, line width=0.3mm] (D) -- (B);
  \draw[->, red, line width=0.3mm] (B) -- (A);
  \draw[->, black, line width=0.1mm] (A)--(C);
  \draw[->, black, line width=0.1mm] (D)--(C);
  \draw[->, black, line width=0.1mm] (B)--(C);
 \end{tikzpicture}
&
 \begin{tikzpicture}[scale=1.5]
  \node[draw,ellipse] (A) at (0,1) {$1$};
  \node[draw,ellipse] (B) at (1,1) {$2$};
  \node[draw,ellipse] (C) at (1,0) {$3$};
  \node[draw,ellipse] (D) at (0,0) {$4$};
  \draw[->, red, line width=0.3mm] (B) -- (D);
  \draw[->, red, line width=0.3mm] (D) -- (C);
  \draw[->, red, line width=0.3mm] (C) -- (B);
  \draw[->, black, line width=0.1mm] (B)--(A);
  \draw[->, black, line width=0.1mm] (D)--(A);
  \draw[->, black, line width=0.1mm] (C)--(A);
 \end{tikzpicture}
\end{tabular}

 \caption{\label{fig:K_4} A $K_4$-separated collection of total orientations that is not contained in any maximal \emph{by size} $K_4$-separated collection.}
\end{figure}

%We generalize the above results and definitions to oriented matroids in
%Section~\ref{sect:main_matroids}. For the rest of this section, we assume some
%knowledge of the oriented matroid terminology which we review later in
%Section~\ref{sect:background}. 

\subsection{Uniform oriented matroids}

% For an oriented matroid $\Mcal$, let $\rk(\Mcal)$ and $\cork(\Mcal)$ denote its \emph{rank} and \emph{corank} respectively.

An oriented matroid $\Mcal$ of rank $d$ is called \emph{uniform} 
if
$$
\{\Xu \mid X\textrm{ if a circuit of } \Mcal\} = {[n]\choose d+1}.
$$
If $\Mcal$ is associated with a vector configuration $\VC$,
uniformity means that the vectors in $\VC$ are in general position.

The following theorem gives a complete characterization of pure uniform oriented matroids.

\begin{theorem}\label{thm:uniform_classification}
 Suppose $\Mcal$ is a uniform oriented matroid.
Then $\Mcal$ is pure if and only if
 \begin{enumerate}[{\normalfont (1)}]
\item \label{item:uniform_classification:rk_2}
$\rk(\Mcal)\leq 2$, or 
\item \label{item:uniform_classification:cork_1}
$\cork(\Mcal)\leq 1$, or
\item\label{item:pure_3} $\rk(\Mcal)=3$ 
and $\Mcal$ is isomorphic to the alternating matroid $C^{n,3}$. 
 \end{enumerate}
\end{theorem}

This result is proved in Section~\ref{sect:classif}. It implies in particular that there are no
pure uniform oriented matroids $\Mcal$ with 
   $\rk(\Mcal)\geq 4$ and $\cork(\Mcal)\geq 2$.

\subsection{Arbitrary oriented matroids}

Let us give a general conjecture that, according to our 
computer experiments, characterizes the class of pure oriented matroids.

\begin{conjecture}\label{conj:weak_minors}
 An oriented matroid $\Mcal$ is pure if and only if all of its six-element
minors are pure. Explicitly, it is pure if and only if one cannot obtain 
the graphical oriented matroids $\Mcal_{\vec K_4}$ and $\Mcal_{\vec
  K_{2,3}}$ from $\Mcal$ by taking minors and rank-preserving weak maps.

In particular, if $\Mcal_1\weakMap\Mcal_2$ is a rank-preserving weak map and $\Mcal_1$ is a pure oriented matroid then $\Mcal_2$ is a pure oriented matroid as well.
\end{conjecture}

We refer the reader to  Section~\ref{sect:background} for the definition of a rank-preserving weak map.

A simple corollary to Theorem~\ref{thm:purity_3d} and Proposition~\ref{prop:rk_2_cork_1} is that Conjecture~\ref{conj:weak_minors} holds when $\rk(\Mcal_1)\leq 3$, since a weak map image of a positroid of rank $3$ is again a positroid (see Lemma~\ref{lemma:positroids}). As an illustration to Conjecture~\ref{conj:weak_minors}, we list all pure and non-pure oriented matroids on $6$ elements of rank $4$ and corank $2$ in Figure~\ref{fig:rank_4_corank_2}. Using the \emph{oriented matroid database}~\cite{Finschi}, we have also computationally verified Conjecture~\ref{conj:weak_minors} for all oriented matroids with at most $8$ elements.

In Proposition~\ref{prop:if_pure_then_minors_are_pure}, we prove one part of Conjecture~\ref{conj:weak_minors}, namely, that a minor of a pure oriented matroid is again pure. 

In the three classes of oriented matroids that we discussed above
(rank $3$, graphical, uniform),
Conjecture~\ref{conj:weak_minors} agrees
with Theorems~\ref{thm:purity_3d}, \ref{thm:outerplanar},
and~\ref{thm:uniform_classification}.

%As we will see in Section~\ref{sect:classif},
%Theorem~\ref{thm:uniform_classification} agrees with
%Conjecture~\ref{conj:weak_minors}.

% 
% 
% We answer this question in Theorem~\ref{thm:uniform_classification} in
% the generality of oriented matroids. In order to restate it for vector
% configurations, we need the notions of \emph{rank} and \emph{corank}. Given a
% vector configuration $\VC\subset\R^d$, its \emph{rank} denoted $\rk\VC$ is
% the dimension of the linear span of $\VC$. Its \emph{corank} equals
% $\cork(\VC)=|\VC|-\rk(\VC)$. We say that a vector configuration $\VC$ is
% \emph{uniform} if every $\rk(\VC)$ vectors of $(\VC)$ are linearly
% independent. Finally, we say that a vector configuration of rank $3$ is
% \emph{totally nonnegative} if after a suitable rescaling, the endpoints of
% vectors of $\VC$ all belong to the boundary of a $2$-dimensional convex
% polygon. For example, the configuration $\Cyclic(n,3)$ is totally
% nonnegative, see Figure~\ref{fig:cyclic} (right). 
% 
% \begin{theorem}\label{pure_vc_classification} Suppose $\VC$ is a vector
% configuration consisting of nonzero vectors. Then: \begin{enumerate}
% \item\label{item:pure_1_2_vc} If $\rk(\VC)\leq 2$ or $\cork(\VC)\leq 1$ then
% $\VC$ is pure; \item\label{item:pure_3_vc} if $\rk(\VC)=3$ then $\VC$ is pure
% if and only if it is totally nonnegative; \item\label{item:pure_4_vc} if
% $\VC$ is uniform and $\rk(\VC)\geq 4$, $\cork(\VC)\geq 2$ then $\VC$ is
% \emph{not} pure.  \end{enumerate} \end{theorem}

\begin{remark}
 We have already mentioned that maximal \emph{by size} strongly and weakly separated collections correspond to clusters in certain cluster algebras. When $\Mcal$ is a uniform oriented matroid, one can define \emph{mutations} on maximal by size $\Mcal$-separated collections in a way similar to how they are defined in the case of strong and weak separation. A natural question arises: for which uniform oriented matroids $\Mcal$ do maximal \emph{by size} $\Mcal$-separated collections form clusters in a cluster algebra? We do not know the answer to this question, but Theorem~\ref{thm:uniform_classification} implies that strong and weak separation are essentially the only two cases where the corresponding oriented matroid is both uniform and pure. However, mutation-closed domains for uniform oriented matroids provide more possibilities for purity, see, e.g., Example~\ref{ex:icos_dodec}.
\end{remark}

\section{Background on zonotopal tilings and oriented matroids}
\label{sect:background}

In this section, we fix notation and recall some notions from~\cite{Book}.

\subsection{Sets and signed vectors}\label{subsect:notation}
% We denote by $[n]$ the set $\{1,2,\dots,n\}$. For any set $A$ and an integer $k$, by ${A\choose k}$ we denote the collection of all $k$-element subsets of $A$ and by $2^A$ we denote the collection of all subsets of $A$. By $|A|$ we denote the cardinality of $A$.

From now on, we denote the set-theoretic difference of two sets $S,T$ by $S-T$ rather than $S\setminus T$, following the conventions of~\cite{LZ,DKK10}. For a set $S$ and an element $e\not\in S$, we denote $Se=S\cup\{e\}$. In particular, the use of $Se$ indicates that $e\not\in S$. On the other hand, we denote $S\cup e:=S\cup\{e\}$ and  $S-e:=S-\{e\}$ regardless of whether $e$ belongs to $S$ or not.

Generalizing~\eqref{eq:reorient_set}, for two sets $S,T\subset E$, denote by $\reorient{T}{S}$ their \emph{symmetric difference}:
\[\reorient{T}{S}:=(S-T)\cup (T-S).\]

% A \emph{signed subset} $X$ of the \emph{ground set} $E$ is a pair $(X^+,X^-)$ of disjoint subsets of $E$. 
% We define the \emph{support} $\Xu$ of $X$ by $\Xu:=X^+\sqcup X^-$.

We abbreviate $+1$ and $-1$ by $+$ and $-$ respectively. Let $X=(X^+,X^-)$ be a signed subset of the \emph{ground set} $E$. The \emph{zero set} $X^0$ of $X$ is the complement of its support, $X^0=E-\Xu$. We denote the collection of all signed subsets of $E$ by $\{+,-,0\}^E$ and recall that for each element $f\in E$, $X_f\in\{+,-,0\}$ is defined by~\eqref{eq:X_f}.
% we write 
% \[X_e=\begin{cases}
%        +, &\text{ if $e\in X^+$};\\
%        -, &\text{ if $e\in X^-$};\\
%        0, &\text{ if $e\in X^0$}.
%      \end{cases}\]

We say that two signed sets $X,Y\in\{+,-,0\}^E$ are \emph{orthogonal} if either of the following holds:
\begin{itemize}
 \item there exist two elements $e,f\in E$ such that $X_e=Y_f=+$ and $X_f=Y_e=-$, or
 \item for every $e\in E$, either $X_e=0$ or $Y_e=0$.
\end{itemize}

In this case we write $X\perp Y$. 

We introduce a partial order $\leq$ on $\{+,-,0\}$ by $0<+$ and $0<-$, while leaving $+$ and $-$ incomparable. This induces an order on $\{+,-,0\}^E$: for $X,Y\in\{+,-,0\}^E$, we write $X\leq Y$ if for all $e\in E$, $X_e\leq Y_e$. 

For two signed sets $X,Y\in\{+,-,0\}^E$, their \emph{composition} $X\circ Y\in\{+,-,0\}^E$ is defined by
\[(X\circ Y)_e=\begin{cases}
       X_e, &\text{ if $X_e\neq 0$};\\
       Y_e, &\text{ otherwise}.
               \end{cases}\]

\subsection{Zonotopal tilings}\label{sect:z_tilings}

% In this section, we give a background on zonotopal tilings that mostly follows~\cite{Galashin}. 

A \emph{vector configuration} $\VC=(\v_e)_{e\in E}$ is a finite subset of $\R^d$ indexed by the elements of some \emph{ground set} $E$.

% to be the \emph{Minkowski sum}
% \[\Zon_\VC:=\sum_{e\in E}[0,\v_e].\]
% Here $[0,\v_e]=\{t\v_e\mid 0\leq t\leq 1\}$ denotes the line segment connecting $\v_e$ to the origin and the \emph{Minkowski sum} of two subsets $A$ and $B$ of $\R^d$ is defined by $A+B=\{a+b\mid a\in A,b\in B\}$.

For a signed set $X\in\{+,-,0\}^E$, we denote by $\Face_X=p(\cube_X)$ the following zonotope:
% \[\Face_X:=\sum_{i\in X^+} v_i+\sum_{j\not\in\Xu} [0,v_j].\]
\[\Face_X:=\sum_{e\in E}\begin{cases}
                           \v_e, &\text{ if $e\in X^+$};\\
                           0, &\text{ if $e\in X^-$};\\
                           [0,\v_e], &\text{ otherwise}.                           
                          \end{cases}
\]

Let us now give an alternative definition of a zonotopal tiling in a slightly different language.
\begin{definition}\label{dfn:tilings_realizable}
A collection $\Tiling$ of signed subsets of $E$ is called a \emph{zonotopal tiling} of $\Zon_\VC$ if and only if the following conditions hold:
\begin{itemize}
\item ${\displaystyle\Zon_\VC=\bigcup_{X\in\Tiling} \Face_X}$;
  \item for any $X\in\Tiling$ and any $Z\geq X$, we have $Z\in\Tiling$;
 \item\label{item:proper_subd} for any two $X,Y\in\Tiling$, either the intersection $\Face_X\cap\Face_Y$ is empty or there exists $Z\in\Tiling$ such that $Z\leq X,Y$ (i.e., $\Face_Z$ is a proper face of $\Face_X$ and $\Face_Y$) and 
 \[\Face_X\cap\Face_Y=\Face_Z.\]
\end{itemize}
\end{definition}

A zonotopal tiling $\Tiling$ is called \emph{fine} if for every $X\in\Tiling$, the vectors $\v_e$, $e\in X^0$, are linearly independent. In particular, all the top-dimensional tiles of $\Tiling$ must be parallelotopes. It is easy to see that this definition is equivalent to Definition~\ref{dfn:fine_z_tiling_intro}. Indeed, every cubical subcomplex of $\cube_{|E|}$ from Definition~\ref{dfn:fine_z_tiling_intro} satisfies the above properties. Conversely, given a collection $\Tiling$ of signed subsets of $E$ satisfying the three properties above, it is easy to show that it also defines a cubical subcomplex satisfying Definition~\ref{dfn:fine_z_tiling_intro}. To see that, note that we get a continuous bijection from the subcomplex $\bigcup_X \cube_X$ of $\cube_{|E|}$ to $\Zon_\VC$, and any such bijection is a homeomorphism since it maps a compact space to a Hausdorff space.

For a fine zonotopal tiling $\Tiling$, its \emph{set of vertices} is defined as
\[\Vert(\Tiling):=\{X^+\mid X\in\Tiling\text{ such that } \Xu=E\}\subset 2^E.\]
This is a slight modification of~\eqref{eq:tiling_vertices_intro}.

\subsection{Oriented matroids}\label{sect:OM}

An oriented matroid is a notion that has several cryptomorphic descriptions; for example, see Definition~\ref{dfn:OM}. 
% \begin{definition}[{\cite[Definition~3.2.1]{Book}}]\label{dfn:OM}
%  A collection $\Ccal$ of signed subsets of $E$ is a collection of \emph{circuits} of an oriented matroid if and only if it satisfies the following axioms:
%  \begin{enumerate}
%   \item[(C0)] $\varnothing\in\Ccal$,
%   \item[(C1)]$\Ccal=-\Ccal$,
%   \item[(C2)]\label{item:C2} for all $X,Y\in\Ccal$, if $\Xu\subset\Yu$ then $X=Y$ or $X=-Y$,
%   \item[(C3)]\label{item:C3} for all $X,Y\in\Ccal$, $X\neq -Y$, and $e\in X^+\cap Y^-$ there is a $Z\in\Ccal$ such that 
%   \[Z^+\subset (X^+\cup Y^+)-e\quad\text{and}\quad Z^-\subset (X^-\cup Y^-)-e.\]
%  \end{enumerate}
% \end{definition}
The axiom~\hyperref[item:C3]{(C3)} is called the \emph{weak elimination axiom}. The set $E$ is called the \emph{ground set} of $\Mcal$ and throughout the text we denote the ground set of $\Mcal$ by $E$ unless told otherwise. Given an oriented matroid $\Mcal$ with circuits $\Ccal(\Mcal)$, define its collection $\Lcal(\Mcal)$ of \emph{covectors} by 
\[\Lcal(\Mcal)=\{X\in\{+,-,0\}^E\mid X\perp Y\quad \forall\, Y\in\Ccal\}.\]
The collection $\Ccal^*(\Mcal)$ of \emph{cocircuits} of $\Mcal$ is the set of minimal non-zero elements of $\Lcal(\Mcal)$ with respect to the $\leq$ order from Section~\ref{subsect:notation}. Next, $\Tcal(\Mcal)$ denotes the collection of all maximal elements of $\Lcal(\Mcal)$. Such elements are called \emph{maximal covectors} or \emph{topes}. The \emph{dual matroid} $\Mcal^*$ of $\Mcal$ is the oriented matroid whose set of circuits equals $\Ccal^*(\Mcal)$. 

We denote by $\Ind(\Mcal)$ the collection of \emph{independent sets of $\Mcal$}, where a set is \emph{independent} if it does not contain the support of any circuit of $\Mcal$. The maximal by inclusion independent sets are called \emph{bases} of $\Mcal$ and the collection of all bases of $\Mcal$ is denoted $\Bcal(\Mcal)$. They all have the same size, which we call the \emph{rank} of $\Mcal$ and denote $\rk(\Mcal)$. If every $\rk(\Mcal)$-element subset of $E$ is a basis then $\Mcal$ is called \emph{uniform}. The \emph{corank} of $\Mcal$ is $\cork(\Mcal):=|E|-\rk(\Mcal)$. 

Every vector configuration $\VC\subset \R^r$ determines an oriented matroid $\Mcal_\VC$ whose circuits are the sign vectors of minimal linear dependencies of $\VC$. We call $\Mcal_\VC$ \emph{the oriented matroid of linear dependencies of $\VC$} (also called \emph{the oriented matroid associated with $\VC$} in the earlier sections). 

Recall that, for a directed graph $\vec G$, the \emph{graphical} oriented matroid $\Mcal_{\vec G}$ is the oriented matroid of linear dependencies of the vector configuration $\VC_{\vec G}$ defined in Section~\ref{sect:separation_def}. 

Another structure that defines an oriented matroid is a \emph{chirotope}. Given an oriented matroid $\Mcal$ of rank $r$, its \emph{chirotope} is a certain mapping $\chi:E^r\to\{+,-,0\}$ which can be obtained from $\Ccal(\Mcal)$ and vice versa using~\cite[Theorem~3.5.5]{Book}. If $\Mcal$ is an oriented matroid associated to a vector configuration $\VC$ then $\chi(i_1,\dots,i_r)$ is equal to $0$ unless the vectors $\v_{i_1},\dots,\v_{i_r}$ form a basis of $\R^r$, in which case the sign of $\chi(i_1,\dots,i_r)$ equals the sign of the determinant of the matrix with rows $\v_{i_1},\dots, \v_{i_r}$. An oriented matroid is called \emph{positively oriented} if there is a total order $\prec$ on $E$ such that for any $i_1\prec i_2\prec \dots\prec i_r\in E$, $\chi(i_1,\dots,i_r)\in\{0,+\}$. One example of a positively oriented matroid is the \emph{alternating matroid} $C^{n,r}$, see Section~\ref{sec:altern-orient-matr}.
% defined for every two integers $r\leq n$. Its ground set is $[n]$ and its set of circuits is given by
% \[\Ccal(C^{n,r})=\pm\{ (\{i_1,i_3,\dots\},\{i_2,i_4,\dots\})\mid i_1<i_2<\dots<i_{r+1}\in[n]\}.\]
In this case, we have $\chi(i_1,\dots,i_r)=+$ for any $i_1<i_2<\dots<i_r\in [n]$.

\begin{remark}\label{rmk:positroids}
	A closely related notion is that of a \emph{positroid} which is a matroid coming from a totally nonnegative matrix of~\cite{Postnikov}. Positroids have been introduced in~\cite{Postnikov}, and it was shown in~\cite{ARW} that every positively oriented matroid is realizable. Thus these objects are essentially the same.
\end{remark}

Given an oriented matroid $\Mcal$ and a set $A\subset E$, the \emph{reorientation} $\reorient{A}{\Mcal}$ of $\Mcal$ on $A$ is another oriented matroid whose set of circuits is given by
\[\Ccal(\reorient{A}{\Mcal})=\{\left(\reorient{(A\cap \Xu)}{(X^+)},\reorient{(A\cap \Xu)}{(X^-)}\right)\mid X\in\Ccal(\Mcal)\}.\]
Two oriented matroids that differ by a reorientation are called \emph{reorientation equivalent}. Two oriented matroids $\Mcal_1$ and $\Mcal_2$ on ground sets $E_1$ and $E_2$ are called \emph{isomorphic} if there is a bijection $\phi:E_1\to E_2$ and a subset $A\subset E_2$ such that the oriented matroids $\phi(\Mcal_1)$ and $\reorient{A}{\Mcal_2}$ are equal (i.e., have the same collections of circuits). In this case, we write $\Mcal_1\cong\Mcal_2$.

An element $e\in E$ is called a \emph{loop} of $\Mcal$ if $\{e\}\in\Ccal(\Mcal)$. It is called a \emph{coloop} of $\Mcal$ if $\{e\}\in\Ccal^*(\Mcal)$. 

An oriented matroid is called \emph{acyclic} if $(E,\emptyset)\in\Tcal(\Mcal)$, that is, if it has a positive covector. Clearly, every loopless oriented matroid is isomorphic to an acyclic oriented matroid.

Two elements $e,f\in E$ are called \emph{parallel} (resp., antiparallel) if $(\{e\},\{f\})\in\Ccal(\Mcal)$ (resp., $(\{e,f\},\emptyset)\in\Ccal(\Mcal)$). An oriented matroid is called \emph{simple} if it has no loops and parallel or antiparallel elements. 

If an element $e\in E$ is not a coloop of $\Mcal$ then the oriented matroid $\Mcal-e$ is defined by 
\begin{equation}\label{eq:matroid_deletion}
\Ccal(\Mcal-e)=\{(X^+,X^-)\mid X\in\Ccal(\Mcal):\ e\not\in\Xu\}. 
\end{equation}
Similarly, if $e\in E$ is not a loop of $\Mcal$ then the oriented matroid $\Mcal/e$ is defined by 
\begin{equation}\label{eq:matroid_contraction}
\Ccal(\Mcal/e)=\min\nolimits_<\{(X^+-e,X^--e)\mid X\in\Ccal(\Mcal)\},
\end{equation}
where $\min_<$ denotes the collection of all minimal signed sets with respect to the order $<$ from Section~\ref{subsect:notation}. The \emph{restriction} $\Mcal\mid_A$ of $\Mcal$ to $A\subset E$ is defined as $\Mcal-(E-A)$. The \emph{rank} and \emph{nullity} of $A\subset E$ are defined as the rank and corank of $\Mcal\mid_A$.

For two oriented matroids $\Mcal_1$ and $\Mcal_2$ of the same rank, we say that there is a \emph{rank-preserving weak map} $\Mcal_1\weakMap\Mcal_2$ if for every signed circuit $X$ of $\Mcal_1$, there exists a signed circuit $Y$ of $\Mcal_2$ such that $Y\leq X$ (see~\cite[Proposition~7.7.5]{Book} for other equivalent formulations).

\begin{definition}\label{dfn:lifting}
 Given an oriented matroid $\Mcal$, its \emph{one-element lifting} $\Mcaltilde$ is another oriented matroid on the ground set $Eg$ such that $\Mcaltilde/g=\Mcal$. A result due to Las Vergnas (see \cite[Proposition~7.1.4]{Book} or \cite{LasVergnas}) shows that for each one-element lifting of $\Mcal$ there is a unique function 
\[\sigma:\Ccal(\Mcal)\to\{+,-,0\}\]
such that for every circuit $Y$ of $\Mcal$, $(Y,\sigma(Y))\in\Ccal(\Mcaltilde)$. 
Here $(Y,\sigma(Y))$ denotes the signed set $Y=(Y^+,Y^-)$ with $g$ added to $Y^+$ if $\sigma(Y)=+$ and to $Y^-$ if $\sigma(Y)=-$. We call such a function $\sigma$ a \emph{colocalization}, and if the image of $\sigma$ lies in $\{+,-\}$ then we say that $\sigma$ is a \emph{colocalization in general position}. In this case, we also say that $\Mcaltilde$ is a \emph{one-element lifting of $\Mcal$ in general position}. 
\end{definition}

We next review a theorem of Las Vergnas~\cite{LasVergnas} that gives a characterization of one-element liftings. 
\begin{definition}
	Consider an oriented matroid $\Mcal$ and a map $\sigma:\Ccal(\Mcal)\to\{+,-,0\}$. For any subset $A\subset E$ of nullity $2$, the restriction of $\Mcal$ to $A$ is isomorphic (up to removing parallel elements of the dual matroid) to the alternating matroid $C^{m,m-2}$. The $2m$ circuits of $C^{m,m-2}$ have a natural cyclic order 
% 	\[(C_1,C_2,\dots,C_m,-C_1,-C_2,\dots,-C_m)\]	
	on them (see Lemma~\ref{lemma:altn_description}). Using~\cite[Figure~7.1.6]{Book}, we say that the restriction of $\sigma$ to the circuits of $\Mcal\mid_A$ is 
	\begin{itemize}
		\item \emph{of Type I} if its values on the circuits of $C^{m,m-2}$ are all zeroes;
		\item \emph{of Type II} if its values on the circuits of $C^{m,m-2}$, up to a cyclic shift, are $(+,\dots,+,0,-,\dots,-,0)$, where the number of plus signs equals the number of minus signs and equals $m-1$;
		\item \emph{of Type III} if its values on the circuits of $C^{m,m-2}$, up to a cyclic shift, are $m$ plus signs followed by $m$ minus signs.
	\end{itemize}
% 
% 	
% 	We say that a map $\sigma:\Ccal(\Mcal)\to\{+,-\}$ is \emph{of Type III} if its restriction to any corank $2$ subset $A$ of $E$ is of Type III.
\end{definition}
See Lemma~\ref{lemma:altn_description} for a detailed description of the circuits of $C^{m,m-2}$. 

\begin{theorem}[{\cite{LasVergnas},\cite[Theorem~7.1.8]{Book}}]\label{thm:LasVergnas}\leavevmode
	\begin{itemize}
		\item Given a map $\sigma:\Ccal(\Mcal)\to\{+,-\}$, its restriction to every nullity $2$ subset of $E$ is of Type III if and only if $\sigma$ is a colocalization in general position.
		\item Given a map $\sigma:\Ccal(\Mcal)\to\{+,-,0\}$, its restriction to every nullity $2$ subset of $E$ is of Type I, II, or III if and only if $\sigma$ is a colocalization, not necessarily in general position.
	\end{itemize}

\end{theorem}

\begin{proposition}[{\cite[Proposition~7.1.4]{Book}}]\label{prop:circuits_coloc}
Given a colocalization $\sigma$ of $\Mcal$, the collection of circuits of the corresponding one-element lifting $\Mcaltilde$ is described by
\[\Ccal(\Mcaltilde)=\{(Y,\sigma(Y)): Y\in\Ccal(\Mcal)\}\cup \{Y^1\circ Y^2\},\]
where the second set runs over all pairs $Y^1,Y^2\in\Ccal(\Mcal)$ such that 
\begin{itemize}
 \item $\sigma(Y^1)=-\sigma(Y^2)\neq 0$;
 \item $Y^1,Y^2<Y^1\circ Y^2$;
 \item $\rk^\ast(\underline{Y^1\circ Y^2})=2$, where $\rk^\ast(A)=|A|-\rk(\Mcal)+\rk(A)$ denotes the rank of $A\subset E$ in $\Mcal^\ast$.
\end{itemize} 
\end{proposition}

For the rest of this paper, \emph{we assume that all colocalizations and one-element liftings are in general position unless told otherwise}.

We now review the Bohne-Dress theorem that gives a connection between zonotopal
tilings and oriented matroid liftings.

\begin{theorem}[{\cite{Bohne},\cite[Theorem~2.2.13]{Book}}]
\label{thm:bohne}
	Let $\VC$ be a vector configuration and let $\Zon_\VC$ be the associated zonotope. Let $\Mcal_\VC$ be the oriented matroid of linear dependencies of $\VC$. Then there is a canonical bijection between fine zonotopal tilings of $\Zon_\VC$ and one-element liftings of $\Mcal_\VC$ in general position.
\end{theorem}

This bijection can be described explicitly in terms of covectors of the lifting
of $\Mcal_\VC$, we refer the reader to \cite{Book} for the details.

\section{Maximal by size $\Mcal$-separated collections}\label{sect:max_by_sz}
In this section, we develop some initial properties of maximal \emph{by size} $\Mcal$-separated collections and use them to prove a strengthening of Theorem~\ref{thm:max_size_vc}, see Theorem~\ref{thm:max_size_matroid} below.

Recall that a collection $\WS\subset 2^E$ is called \emph{$\Mcal$-separated} if any two subsets $S,T\in\WS$ are $\Mcal$-separated, i.e., there is no circuit $C=(C^+,C^-)\in\Ccal(\Mcal)$ such that $C^+\subset S-T$ and $C^-\subset T-S$. For an $\Mcal$-separated collection $\WS$, we define a map $\sigma_\WS:\Ccal\to\{+,-,0\}$ as follows: for every circuit $Y\in\Ccal$, we set $\sigma_\WS(Y):=0$ unless there is a set $S\in\WS$ satisfying 
\[Y^+\subset S;\quad Y^-\subset (E- S)\]
or the other way around:
\[Y^-\subset S;\quad Y^+\subset (E- S).\]
In the first case (resp., in the second case) we set $\sigma_\WS(Y):=+$ (resp., $\sigma_\WS(Y):=-$) and we say that $S$ \emph{orients $Y$ positively} (resp., \emph{negatively}). Thus if a collection $\WS\subset 2^E$ is $\Mcal$-separated then $\sigma_\WS(Y)$ is well defined for all $Y\in\Ccal$.% and we say that $\WS$ \emph{orients $Y$ positively} if $\sigma_\WS(Y)=+$ and \emph{negatively} if $\sigma_\WS(Y)=-$. 

\begin{definition}
Given a colocalization $\sigma:\Ccal\to\{+,-\}$, we define $\WS(\sigma)\subset 2^E$ to be the collection of all subsets $S\subset E$ such that for every circuit $Y\in\Ccal$ that $S$ orients positively (resp., negatively), we have $\sigma(Y)=+$ (resp., $\sigma(Y)=-$).
\end{definition}

% We now restate several theorems from Section~\ref{sect:main} in terms of oriented matroids.

% tdo: find the reference for the fact that $|\Tcal(\Mcaltilde)|=2|\Ind(\Mcal)|$ \cite[Proposition 4.12]{Santos} and \cite[Theorem 4.14]{Santos}
% \begin{remark}\label{remark:independent}
% We will assume from now on that for any one-element lifting $\Mcaltilde$ of $\Mcal$ we have $|\Tcal(\Mcaltilde)|=2|\Ind|$. I couldn't find that in the literature, but combining \cite[Proposition 4.12]{Santos} and \cite[Theorem 4.14]{Santos} we see that the number of bases of $\Mcal$ equals the number of cocircuits(?) of $\Mcaltilde$, so the other thing should follow easily from the same Lawrence polytope construction.
% \end{remark}

We now restate Theorem~\ref{thm:max_size_vc} in the oriented matroid language:

\begin{theorem}
\label{thm:max_size_matroid}
 Let $\Mcal$ be any oriented matroid. Then the map $\WS\mapsto\sigma_\WS$ is a bijection (with inverse $\sigma\mapsto\WS(\sigma)$) between maximal \emph{by size} $\Mcal$-separated collections of subsets of $E$ and one-element liftings of $\Mcal$ in general position. Every such collection has size $|\Ind(\Mcal)|$.
\end{theorem}
The fact that Theorem~\ref{thm:max_size_matroid} generalizes Theorem~\ref{thm:max_size_vc} follows from the Bohne-Dress theorem (Theorem~\ref{thm:bohne}).

\begin{figure}
\begin{tabular}{|c|c|}\hline

\begin{tabular}{c}
\scalebox{0.6}{
\begin{tikzpicture}
\node[draw,ellipse,black,fill=white] (node0) at (0.00,0.00) {$\emptyset$};
\node[draw,ellipse,black,fill=white] (node1) at (-3.78,1.60) {$1$};
\node[draw,ellipse,black,fill=white] (node3) at (0.00,1.60) {$3$};
\node[draw,ellipse,black,fill=white] (node5) at (1.80,1.60) {$5$};
\node[draw,ellipse,black,fill=white] (node12) at (-5.40,3.20) {$1\,2$};
\node[draw,ellipse,black,fill=white] (node13) at (-3.78,3.20) {$1\,3$};
\node[draw,ellipse,black,fill=white] (node35) at (1.80,3.20) {$3\,5$};
\node[draw,ellipse,black,fill=white] (node123) at (-5.40,4.80) {$1\,2\,3$};
\node[draw,ellipse,black,fill=white] (node134) at (-3.78,4.80) {$1\,3\,4$};
\node[draw,ellipse,black,fill=white] (node135) at (-1.98,4.80) {$1\,3\,5$};
\node[draw,ellipse,black,fill=white] (node345) at (1.80,4.80) {$3\,4\,5$};
\node[draw,ellipse,black,fill=white] (node1234) at (-5.40,6.40) {$1\,2\,3\,4$};
\node[draw,ellipse,black,fill=white] (node1345) at (-1.98,6.40) {$1\,3\,4\,5$};
\node[draw,ellipse,black,fill=white] (node2345) at (0.18,6.40) {$2\,3\,4\,5$};
\node[draw,ellipse,black,fill=white] (node12345) at (-3.60,8.00) {$1\,2\,3\,4\,5$};
\node[draw,ellipse,black,fill=white] (node0) at (0.00,0.00) {$\emptyset$};
\node[draw,ellipse,black,fill=white] (node1) at (-3.78,1.60) {$1$};
\node[draw,ellipse,black,fill=white] (node3) at (0.00,1.60) {$3$};
\node[draw,ellipse,black,fill=white] (node5) at (1.80,1.60) {$5$};
\node[draw,ellipse,black,fill=white] (node12) at (-5.40,3.20) {$1\,2$};
\node[draw,ellipse,black,fill=white] (node13) at (-3.78,3.20) {$1\,3$};
\node[draw,ellipse,black,fill=white] (node35) at (1.80,3.20) {$3\,5$};
\node[draw,ellipse,black,fill=white] (node123) at (-5.40,4.80) {$1\,2\,3$};
\node[draw,ellipse,black,fill=white] (node134) at (-3.78,4.80) {$1\,3\,4$};
\node[draw,ellipse,black,fill=white] (node135) at (-1.98,4.80) {$1\,3\,5$};
\node[draw,ellipse,black,fill=white] (node345) at (1.80,4.80) {$3\,4\,5$};
\node[draw,ellipse,black,fill=white] (node1234) at (-5.40,6.40) {$1\,2\,3\,4$};
\node[draw,ellipse,black,fill=white] (node1345) at (-1.98,6.40) {$1\,3\,4\,5$};
\node[draw,ellipse,black,fill=white] (node2345) at (0.18,6.40) {$2\,3\,4\,5$};
\node[draw,ellipse,black,fill=white] (node12345) at (-3.60,8.00) {$1\,2\,3\,4\,5$};
\draw[line width=0.04mm,black] (node0) -- (node1);
\draw[line width=0.04mm,black] (node0) -- (node3);
\draw[line width=0.04mm,black] (node0) -- (node5);
\draw[line width=0.04mm,black] (node1) -- (node12);
\draw[line width=0.04mm,black] (node1) -- (node13);
\draw[line width=0.04mm,black] (node3) -- (node13);
\draw[line width=0.04mm,black] (node3) -- (node35);
\draw[line width=0.04mm,black] (node5) -- (node35);
\draw[line width=0.04mm,black] (node12) -- (node123);
\draw[line width=0.04mm,black] (node13) -- (node123);
\draw[line width=0.04mm,black] (node13) -- (node134);
\draw[line width=0.04mm,black] (node13) -- (node135);
\draw[line width=0.04mm,black] (node13) -- (node123);
\fill [opacity=0.2,blue] (node13.center)-- (node123.center)-- (node1234.center)-- (node134.center) -- cycle;
\draw[line width=0.04mm,black] (node13) -- (node134);
\fill [opacity=0.2,blue] (node13.center)-- (node134.center)-- (node1345.center)-- (node135.center) -- cycle;
\draw[line width=0.04mm,black] (node35) -- (node135);
\draw[line width=0.04mm,black] (node35) -- (node345);
\draw[line width=0.04mm,black] (node35) -- (node135);
\fill [opacity=0.2,blue] (node35.center)-- (node135.center)-- (node1345.center)-- (node345.center) -- cycle;
\draw[line width=0.04mm,black] (node123) -- (node1234);
\draw[line width=0.04mm,black] (node134) -- (node1234);
\draw[line width=0.04mm,black] (node134) -- (node1345);
\draw[line width=0.04mm,black] (node135) -- (node1345);
\draw[line width=0.04mm,black] (node345) -- (node1345);
\draw[line width=0.04mm,black] (node345) -- (node2345);
\draw[line width=0.04mm,black] (node1234) -- (node12345);
\draw[line width=0.04mm,black] (node1345) -- (node12345);
\draw[line width=0.04mm,black] (node2345) -- (node12345);
\node[draw,ellipse,black,fill=white] (node0) at (0.00,0.00) {$\emptyset$};
\node[draw,ellipse,black,fill=white] (node1) at (-3.78,1.60) {$1$};
\node[draw,ellipse,black,fill=white] (node3) at (0.00,1.60) {$3$};
\node[draw,ellipse,black,fill=white] (node5) at (1.80,1.60) {$5$};
\node[draw,ellipse,black,fill=white] (node12) at (-5.40,3.20) {$1\,2$};
\node[draw,ellipse,black,fill=white] (node13) at (-3.78,3.20) {$1\,3$};
\node[draw,ellipse,black,fill=white] (node35) at (1.80,3.20) {$3\,5$};
\node[draw,ellipse,black,fill=white] (node123) at (-5.40,4.80) {$1\,2\,3$};
\node[draw,ellipse,black,fill=white] (node134) at (-3.78,4.80) {$1\,3\,4$};
\node[draw,ellipse,black,fill=white] (node135) at (-1.98,4.80) {$1\,3\,5$};
\node[draw,ellipse,black,fill=white] (node345) at (1.80,4.80) {$3\,4\,5$};
\node[draw,ellipse,black,fill=white] (node1234) at (-5.40,6.40) {$1\,2\,3\,4$};
\node[draw,ellipse,black,fill=white] (node1345) at (-1.98,6.40) {$1\,3\,4\,5$};
\node[draw,ellipse,black,fill=white] (node2345) at (0.18,6.40) {$2\,3\,4\,5$};
\node[draw,ellipse,black,fill=white] (node12345) at (-3.60,8.00) {$1\,2\,3\,4\,5$};
\end{tikzpicture}}
\\

(a) The collection $\WS$
\\

\end{tabular}

&

\begin{tabular}{c}

\begin{tabular}{c}
\scalebox{0.6}{
\begin{tikzpicture}
\node[draw,ellipse,black,fill=white] (node0) at (0.00,0.00) {$\emptyset$};
\node[draw,ellipse,black,fill=white] (node1) at (-3.78,1.60) {$1$};
\node[draw,ellipse,black,fill=white] (node3) at (0.00,1.60) {$3$};
\node[draw,ellipse,black,fill=white] (node5) at (1.80,1.60) {$5$};
\node[draw,ellipse,black,fill=white] (node12) at (-5.40,3.20) {$1\,2$};
\node[draw,ellipse,black,fill=white] (node13) at (-3.78,3.20) {$1\,3$};
\node[draw,ellipse,black,fill=white] (node35) at (1.80,3.20) {$3\,5$};
\node[draw,ellipse,black,fill=white] (node123) at (-5.40,4.80) {$1\,2\,3$};
\node[draw,ellipse,black,fill=white] (node135) at (-1.98,4.80) {$1\,3\,5$};
\node[draw,ellipse,black,fill=white] (node235) at (0.18,4.80) {$2\,3\,5$};
\node[draw,ellipse,black,fill=white] (node1235) at (-3.60,6.40) {$1\,2\,3\,5$};
\node[draw,ellipse,black,fill=white] (node0) at (0.00,0.00) {$\emptyset$};
\node[draw,ellipse,black,fill=white] (node1) at (-3.78,1.60) {$1$};
\node[draw,ellipse,black,fill=white] (node3) at (0.00,1.60) {$3$};
\node[draw,ellipse,black,fill=white] (node5) at (1.80,1.60) {$5$};
\node[draw,ellipse,black,fill=white] (node12) at (-5.40,3.20) {$1\,2$};
\node[draw,ellipse,black,fill=white] (node13) at (-3.78,3.20) {$1\,3$};
\node[draw,ellipse,black,fill=white] (node35) at (1.80,3.20) {$3\,5$};
\node[draw,ellipse,black,fill=white] (node123) at (-5.40,4.80) {$1\,2\,3$};
\node[draw,ellipse,black,fill=white] (node135) at (-1.98,4.80) {$1\,3\,5$};
\node[draw,ellipse,black,fill=white] (node235) at (0.18,4.80) {$2\,3\,5$};
\node[draw,ellipse,black,fill=white] (node1235) at (-3.60,6.40) {$1\,2\,3\,5$};
\draw[line width=0.04mm,black] (node0) -- (node1);
\draw[line width=0.04mm,black] (node0) -- (node3);
\draw[line width=0.04mm,black] (node0) -- (node5);
\draw[line width=0.04mm,black] (node1) -- (node12);
\draw[line width=0.04mm,black] (node1) -- (node13);
\draw[line width=0.04mm,black] (node3) -- (node13);
\draw[line width=0.04mm,black] (node3) -- (node35);
\draw[line width=0.04mm,black] (node5) -- (node35);
\draw[line width=0.04mm,black] (node12) -- (node123);
\draw[line width=0.04mm,black] (node13) -- (node123);
\draw[line width=0.04mm,black] (node13) -- (node135);
\draw[line width=0.04mm,black] (node35) -- (node135);
\draw[line width=0.04mm,black] (node35) -- (node235);
\draw[line width=0.04mm,black] (node123) -- (node1235);
\draw[line width=0.04mm,black] (node135) -- (node1235);
\draw[line width=0.04mm,black] (node235) -- (node1235);
\node[draw,ellipse,black,fill=white] (node0) at (0.00,0.00) {$\emptyset$};
\node[draw,ellipse,black,fill=white] (node1) at (-3.78,1.60) {$1$};
\node[draw,ellipse,black,fill=white] (node3) at (0.00,1.60) {$3$};
\node[draw,ellipse,black,fill=white] (node5) at (1.80,1.60) {$5$};
\node[draw,ellipse,black,fill=white] (node12) at (-5.40,3.20) {$1\,2$};
\node[draw,ellipse,black,fill=white] (node13) at (-3.78,3.20) {$1\,3$};
\node[draw,ellipse,black,fill=white] (node35) at (1.80,3.20) {$3\,5$};
\node[draw,ellipse,black,fill=white] (node123) at (-5.40,4.80) {$1\,2\,3$};
\node[draw,ellipse,black,fill=white] (node135) at (-1.98,4.80) {$1\,3\,5$};
\node[draw,ellipse,black,fill=white] (node235) at (0.18,4.80) {$2\,3\,5$};
\node[draw,ellipse,black,fill=white] (node1235) at (-3.60,6.40) {$1\,2\,3\,5$};
\end{tikzpicture}}
\\

(b) The collection $\WS-4$
\\

\end{tabular}

\\\hline

\begin{tabular}{c}

\\

\scalebox{0.6}{
\begin{tikzpicture}
\node[draw,ellipse,black,fill=white] (node13) at (-3.60,0.00) {$1\,3$};
\node[draw,ellipse,black,fill=white] (node35) at (1.80,0.00) {$3\,5$};
\node[draw,ellipse,black,fill=white] (node123) at (-5.22,0.00) {$1\,2\,3$};
\node[draw,ellipse,black,fill=white] (node135) at (-1.80,0.00) {$1\,3\,5$};
\node[draw,ellipse,black,fill=white] (node13) at (-3.60,0.00) {$1\,3$};
\node[draw,ellipse,black,fill=white] (node35) at (1.80,0.00) {$3\,5$};
\node[draw,ellipse,black,fill=white] (node123) at (-5.22,0.00) {$1\,2\,3$};
\node[draw,ellipse,black,fill=white] (node135) at (-1.80,0.00) {$1\,3\,5$};
\draw[line width=0.04mm,black] (node13) -- (node123);
\draw[line width=0.04mm,black] (node13) -- (node135);
\draw[line width=0.04mm,black] (node35) -- (node135);
\node[draw,ellipse,black,fill=white] (node13) at (-3.60,0.00) {$1\,3$};
\node[draw,ellipse,black,fill=white] (node35) at (1.80,0.00) {$3\,5$};
\node[draw,ellipse,black,fill=white] (node123) at (-5.22,0.00) {$1\,2\,3$};
\node[draw,ellipse,black,fill=white] (node135) at (-1.80,0.00) {$1\,3\,5$};
\end{tikzpicture}}
\\

(c) The collection $\WS/4$
\\

\end{tabular}

\\

\end{tabular}

\\\hline

\end{tabular}

 \caption{\label{fig:deletion_contraction} Deletion-contraction recurrence for maximal \emph{by size} $\Mcal$-separated collections.}
\end{figure}
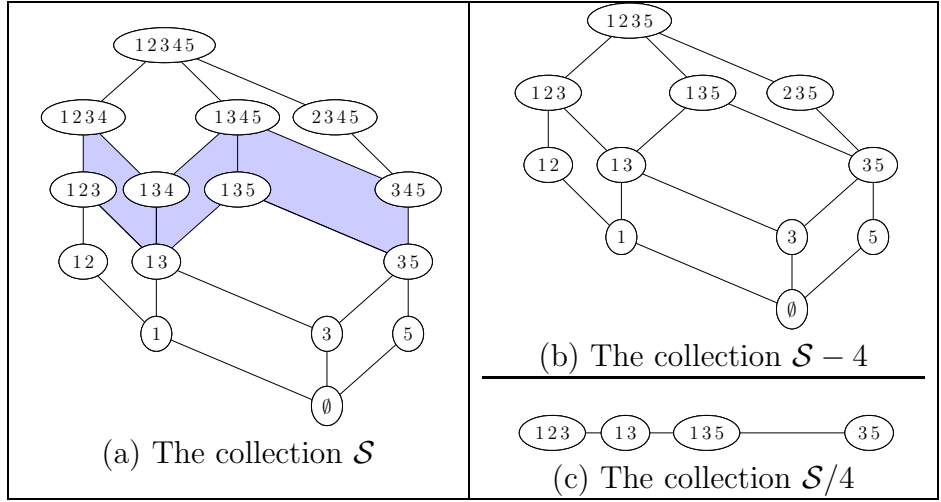

\begin{definition}\label{dfn:ws_delete_contract}
Suppose we are given an oriented matroid $\Mcal$, a collection $\WS\subset 2^E$, and an element $e\in E$. Define collections $\WS-e$  and $\WS/e$  of subsets of $E-e$ as follows:
\begin{equation}\label{eq:ws_delete_contract}
\begin{split}
\WS-e&=\{S\subset (E-e)\mid S\in\WS\ \textrm{ or }\  (S\cup e)\in\WS\}=\{S-e\mid S\in\WS\};\\
\WS/e&=\{S\subset (E-e)\mid S\in\WS\textrm{ and } (S\cup e)\in\WS\}.
\end{split}
\end{equation}
\end{definition}

\begin{remark}
After reading the first version of this manuscript, Steven Karp pointed out to us that some of our constructions are very similar to those related to the notion of \emph{VC-dimension} studied in the early seventies in the extremal combinatorics literature, see~\cite{Shelah, Sauer, VCold}. Later, a bijective correspondence between uniform oriented matroids and collections of sets of maximal possible size with a fixed VC-dimension was given by G\"artner and Welzl in~\cite{VCmatroid}. In particular, our Definition~\ref{dfn:ws_delete_contract} can be found in~\cite[Definition~5]{VCmatroid} and the idea of the proof of our Theorem~\ref{thm:max_size_matroid} is quite similar to the one used in the proof of~\cite[Theorem~18]{VCmatroid}. The authors of~\cite{VCmatroid} are giving a characterization to possible collections of topes of a uniform oriented matroid in this language. However, we are giving a characterization to one-element liftings of a \emph{fixed} oriented matroid $\Mcal$. Thus we do not see a way to deduce any of our results from the results of G\"artner-Welzl or vice versa. We are extremely grateful to Steven for bringing these papers to our attention.
\end{remark}

\begin{proposition}\label{prop:recursion}
 Let $\Mcal$ be an oriented matroid and suppose that a collection $\WS\subset 2^E$ is $\Mcal$-separated. Then for any $e\in\ E$, the collection $\WS-e$ is an $(\Mcal-e)$-separated collection (if $e$ is not a coloop) and the collection $\WS/e$ is an $(\Mcal/e)$-separated collection (if $e$ is not a loop).
\end{proposition}
\begin{proof}
 We start with $\WS-e$. We need to check that every circuit $C=(C^+,C^-)$ of $\Mcal-e$ is not oriented in the opposite directions by $\WS-e$. Recall from~\eqref{eq:matroid_deletion} that $C\in\Ccal(\Mcal)$ is a circuit of $\Mcal-e$ if and only if $e\not\in\Cu$. Thus if $S\in \WS$ is such that $S-e$ orients $C$ in $\Mcal-e$ then $S$ orients $C$ in $\Mcal$ as well, so it is impossible to have another set $T\in\WS$ orienting $C$ the opposite way.
 
 The proof for $\WS/e$ will be a bit harder. Recall from~\eqref{eq:matroid_contraction} that if $C$ is a circuit of $\Mcal/e$ then there is a circuit $C_1$ of $\Mcal$ such that $C=C_1-e$ (i.e., $C$ is the restriction of $C_1$ to $E-e$). Assume moreover that there are sets $S,T\in\WS/e$ orienting $C$ in the opposite directions:
 \[C^+\subset S,\ C^-\cap S=\emptyset,\quad C^-\subset T,\ C^+\cap T=\emptyset.\] 
 Note that by the definition of $\WS/e$, we have $S,S\cup e, T, T\cup e\in \WS$. 

 \def\Cuone{\underline{C_1}}
 If $e\not\in\Cuone$, that is, if $C_1=C$, then $S$ and $T$ are not $\Mcal$-separated, so let us assume without loss of generality that $C_1^-=C^-\cup e, C_1^+=C^+$. But then it is clear that the sets $S$ and $T\cup e$ are not $\Mcal$-separated. We get a contradiction again and thus we have shown that $\WS-e$ (resp., $\WS/e$) is an $(\Mcal-e)$-separated (resp., $(\Mcal/e)$-separated) collection of sets. 
\end{proof}

\begin{proposition}\label{prop:ws_tutte_recurrence}
	Let $\Mcal$ be an oriented matroid and suppose that $e\in E$ is neither a loop nor a coloop in $\Mcal$. Then for any $\Mcal$-separated collection $\WS$, we have
	\begin{equation}\label{eq:ws_tutte_recurrence}
	|\WS|=|\WS-e|+|\WS/e|.
	\end{equation}
\end{proposition}
\begin{proof}
	This is obvious from Definition~\ref{dfn:ws_delete_contract}.
\end{proof}

\subsection{Mutation-closed domains}
\label{subsection:mutation_closed_big}

Given a set $S\subset E$ and a circuit $C\in\Ccal(\Mcal)$ of an oriented matroid $\Mcal$ such that $S$ orients $C$, we define a new set 
\[\mut{C}{S}:=\reorient{C}{S}\]
to be the symmetric difference of $S$ and $\Cu$. We call $\mut{C}{S}$ the \emph{mutation} of $S$ along $C$. Thus for example if $S$ orients $C$ positively then 
\[\mut{C}{S}=(S\setminus C^+)\cup C^-.\]

Recall that we introduced the \emph{mutation graph} of $\Mcal$ in Definition~\ref{dfn:mutation_graph_intro} with vertex set $2^E$ and an edge from $S$ to $\mut{C}{S}$ for any circuit $C$ oriented by $S$. We denote this undirected graph by $\mutgraph(\Mcal)$. 

% The following notion for oriented matroids was introduced first in~\cite[Section~4]{Gioan1}.
% \begin{definition}
% 	The \emph{mutation graph}  of $\Mcal$ is a simple undirected graph $\mutgraph(\Mcal)$ with vertex set $2^E$ and two sets $S$ and $T$ are connected by an edge in $\mutgraph(\Mcal)$ if and only if there is a circuit $C$ of $\Mcal$ such that $S$ orients $C$ and $\mut{C}{S}=T$.
% \end{definition}

% Following~\cite{DKK10}, we give the following definition. Let us say that a \emph{domain} is a collection $\dom\subset 2^E$ of subsets of $E$. As we have already mentioned in Remark~\ref{rmk:coll_domain}, a collection and a domain are the same thing, however, unlike collections, domains are usually not supposed to be $\Mcal$-separated. 

% Given an oriented matroid $\Mcal$ and a domain $\dom\subset 2^E$, we say that
% $\dom$ is a \emph{mutation-closed domain for $\Mcal$} if whenever $S\in \dom$
% orients some circuit $C\in\Ccal(\Mcal)$ then $\mut{C}{S}\in\dom$. Equivalently,
% $\dom$ is a mutation-closed domain if and only if $\dom$ is a union of
% connected components of the mutation graph $\mutgraph(\Mcal)$.

Recall the definition of a \emph{mutation-closed domain} from Section~\ref{sec:mutat-clos-doma}. For example,
$2^E$ and $\emptyset$ are always mutation-closed domains. If $\Mcal$ is
\emph{balanced}, i.e., for every circuit $C=(C^+,C^-)$ of $\Mcal$ we have $|C^+|=|C^-|$, then clearly $E\choose k$ is a mutation-closed domain for any $k\leq |E|$. 
% We say that $\dom$ is \emph{mutation-connected} if any two sets in $\dom$ can be obtained from one another by a sequence of mutations, i.e., if $\dom$ is a single connected component of $\mutgraph(\Mcal)$.

% Let us now state more precisely the definition of an \emph{$\Mcal$-pure domain}.

% \begin{definition}
% 	Given a domain $\dom$, we say that $\WS\subset 2^E$ is \emph{a collection inside $\dom$} if $\WS\subset\dom$. We say that such $\WS$ is a \emph{maximal by size $\Mcal$-separated collection inside $\dom$} if it has the maximal size among all $\Mcal$-separated collections inside $\dom$. If any maximal \emph{by inclusion} $\Mcal$-separated collection inside $\dom$ is also maximal \emph{by size}  then we say that $\dom$ is an \emph{$\Mcal$-pure domain}.
% \end{definition}

\begin{example}\label{ex:icos_dodec}
  Suppose that $\Mcal$ is the oriented matroid associated with the vector configuration shown in Figure~\ref{fig:icos_dodec} (top). In terms of \emph{affine point configurations}, we have $\Mcal=\IC(6,3,1)$ in Figure~\ref{fig:non_postiroids}. The ground set for $\Mcal$ is $E=\{1,2,3,4,5,6\}$. Thus $\mutgraph(\Mcal)$ has $64$ vertices. It turns out that $32$ of them are isolated, and the other $32$ of them form two connected components $\dom_i$ and $\dom_d$ that are $1$-skeleta of the icosahedron and the dodecahedron respectively, see Figure~\ref{fig:icos_dodec} (bottom). The mutation-closed domain $\dom_i$ consisting of the vertex labels of the icosahedron is $\Mcal$-pure: $\Mcal$-separated collections inside $\dom_i$ form a pure $2$-dimensional simplicial complex which is, coincidentally, again the boundary of an icosahedron (this complex is dual to the \emph{cluster complex} of~\cite{FZy}).  However, the mutation-closed domain $\dom_d$ corresponding to the dodecahedron connected component is not $\Mcal$-pure. 

  % Additionally, the number of $\Mcal$-separated collections inside $\dom_i$ is $20$ and the ``cluster complex'' formed by connecting two of them that differ in only one set is again the $1$-skeleton of the dodecahedron. It remains an open problem to construct a cluster algebra analog for this pure mutation-closed domain.
\end{example}

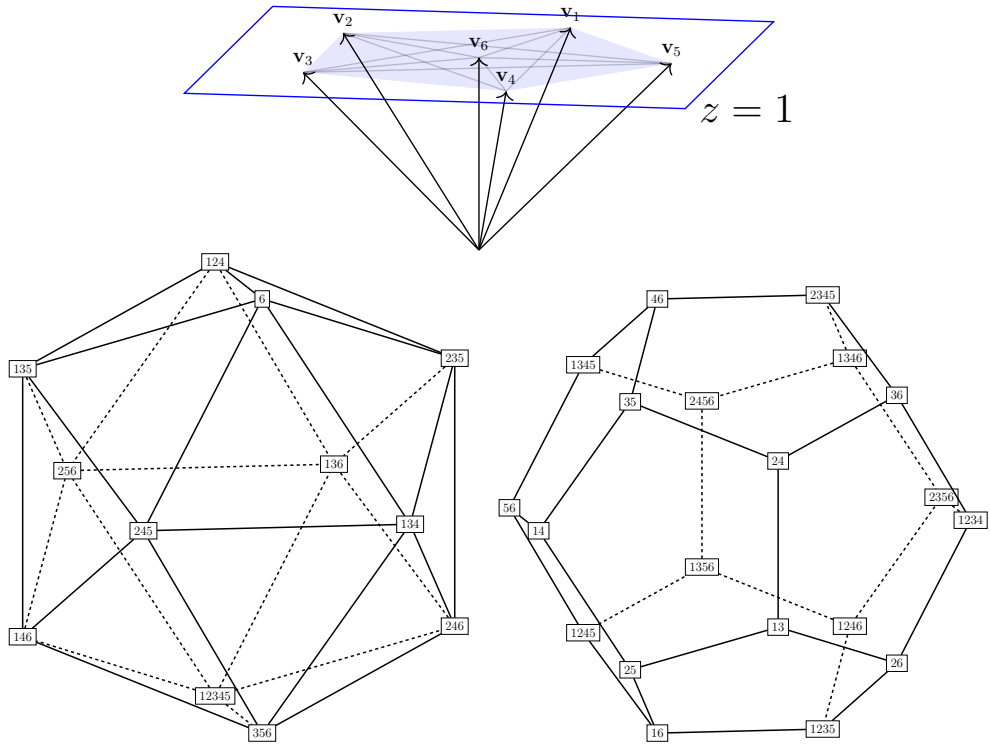
\begin{figure}
\tdplotsetmaincoords{80}{10}
\scalebox{1.3}{
\begin{tikzpicture}[scale=2,tdplot_main_coords]
    \foreach\x in {1,2,...,5}{
      \fill[opacity=0.1,blue] (0,0,1) -- ({cos(\x*72)},{sin(\x*72)},1) -- ({cos(\x*72+72)},{sin(\x*72+72)},1)--cycle;
    }
    \foreach\x in {1,2,...,5}{
      \draw[->] (0,0,0) -- ({cos(\x*72)},{sin(\x*72)},1) node[anchor=south,scale=0.5]{$\v_\x$};
      \draw[opacity=0.2] ({cos(\x*72)},{sin(\x*72)},1) -- (0,0,1);
      \draw[opacity=0.2] ({cos(\x*72)},{sin(\x*72)},1) -- ({cos(\x*72+144)},{sin(\x*72+144)},1);
    }
    
    \draw[->] (0,0,0) -- (0,0,1) node[anchor=south,scale=0.5]{$\v_6$};
    
    \def\y{1.3}
    \draw[blue] (\y,\y,1)  -- (\y,-\y,1)node[anchor=west,black]{$z=1$} -- (-\y,-\y,1) -- (-\y,\y,1)  --cycle;
\end{tikzpicture}
}
\tdplotsetmaincoords{98}{10}
\makebox[1.0\textwidth]{
\begin{tabular}{cc}
\scalebox{0.4}{
\begin{tikzpicture}[scale=4.5,tdplot_main_coords]
\node[draw,black,fill=white] (N0) at (0.00,-1.00,-1.62) {$12345$};
\node[draw,black,fill=white] (N1) at (0.00,-1.00,1.62) {$124$};
\node[draw,black,fill=white] (N2) at (0.00,1.00,-1.62) {$356$};
\node[draw,black,fill=white] (N3) at (0.00,1.00,1.62) {$6$};
\node[draw,black,fill=white] (N4) at (-1.62,0.00,-1.00) {$146$};
\node[draw,black,fill=white] (N5) at (1.62,0.00,-1.00) {$246$};
\node[draw,black,fill=white] (N6) at (-1.62,0.00,1.00) {$135$};
\node[draw,black,fill=white] (N7) at (1.62,0.00,1.00) {$235$};
\node[draw,black,fill=white] (N8) at (-1.00,-1.62,0.00) {$256$};
\node[draw,black,fill=white] (N9) at (-1.00,1.62,0.00) {$245$};
\node[draw,black,fill=white] (N10) at (1.00,-1.62,0.00) {$136$};
\node[draw,black,fill=white] (N11) at (1.00,1.62,0.00) {$134$};
\draw[line width=0.5mm,dashed] (N0) -- (N2);
\draw[line width=0.5mm,dashed] (N0) -- (N4);
\draw[line width=0.5mm,dashed] (N0) -- (N5);
\draw[line width=0.5mm,dashed] (N0) -- (N8);
\draw[line width=0.5mm,dashed] (N0) -- (N10);
\draw[line width=0.5mm] (N1) -- (N3);
\draw[line width=0.5mm] (N1) -- (N6);
\draw[line width=0.5mm] (N1) -- (N7);
\draw[line width=0.5mm,dashed] (N1) -- (N8);
\draw[line width=0.5mm,dashed] (N1) -- (N10);
\draw[line width=0.5mm] (N2) -- (N4);
\draw[line width=0.5mm] (N2) -- (N5);
\draw[line width=0.5mm] (N2) -- (N9);
\draw[line width=0.5mm] (N2) -- (N11);
\draw[line width=0.5mm] (N3) -- (N6);
\draw[line width=0.5mm] (N3) -- (N7);
\draw[line width=0.5mm] (N3) -- (N9);
\draw[line width=0.5mm] (N3) -- (N11);
\draw[line width=0.5mm] (N4) -- (N6);
\draw[line width=0.5mm,dashed] (N4) -- (N8);
\draw[line width=0.5mm] (N4) -- (N9);
\draw[line width=0.5mm] (N5) -- (N7);
\draw[line width=0.5mm,dashed] (N5) -- (N10);
\draw[line width=0.5mm] (N5) -- (N11);
\draw[line width=0.5mm,dashed] (N6) -- (N8);
\draw[line width=0.5mm] (N6) -- (N9);
\draw[line width=0.5mm,dashed] (N7) -- (N10);
\draw[line width=0.5mm] (N7) -- (N11);
\draw[line width=0.5mm,dashed] (N8) -- (N10);
\draw[line width=0.5mm] (N9) -- (N11);
\end{tikzpicture}}
&
\scalebox{0.4}{
\begin{tikzpicture}[scale=4.5,tdplot_main_coords]
\node[draw,black,fill=white] (N0) at (-1.00,-1.00,-1.00) {$1245$};
\node[draw,black,fill=white] (N1) at (-1.00,-1.00,1.00) {$1345$};
\node[draw,black,fill=white] (N2) at (-1.00,1.00,-1.00) {$25$};
\node[draw,black,fill=white] (N3) at (-1.00,1.00,1.00) {$35$};
\node[draw,black,fill=white] (N4) at (1.00,-1.00,-1.00) {$1246$};
\node[draw,black,fill=white] (N5) at (1.00,-1.00,1.00) {$1346$};
\node[draw,black,fill=white] (N6) at (1.00,1.00,-1.00) {$26$};
\node[draw,black,fill=white] (N7) at (1.00,1.00,1.00) {$36$};
\node[draw,black,fill=white] (N8) at (0.00,-1.62,-0.62) {$1356$};
\node[draw,black,fill=white] (N9) at (0.00,-1.62,0.62) {$2456$};
\node[draw,black,fill=white] (N10) at (0.00,1.62,-0.62) {$13$};
\node[draw,black,fill=white] (N11) at (0.00,1.62,0.62) {$24$};
\node[draw,black,fill=white] (N12) at (-0.62,0.00,-1.62) {$16$};
\node[draw,black,fill=white] (N13) at (0.62,0.00,-1.62) {$1235$};
\node[draw,black,fill=white] (N14) at (-0.62,0.00,1.62) {$46$};
\node[draw,black,fill=white] (N15) at (0.62,0.00,1.62) {$2345$};
\node[draw,black,fill=white] (N16) at (-1.62,-0.62,0.00) {$56$};
\node[draw,black,fill=white] (N17) at (-1.62,0.62,0.00) {$14$};
\node[draw,black,fill=white] (N18) at (1.62,-0.62,0.00) {$2356$};
\node[draw,black,fill=white] (N19) at (1.62,0.62,0.00) {$1234$};
\draw[line width=0.5mm,dashed] (N0) -- (N8);
\draw[line width=0.5mm] (N0) -- (N12);
\draw[line width=0.5mm] (N0) -- (N16);
\draw[line width=0.5mm,dashed] (N1) -- (N9);
\draw[line width=0.5mm] (N1) -- (N14);
\draw[line width=0.5mm] (N1) -- (N16);
\draw[line width=0.5mm] (N2) -- (N10);
\draw[line width=0.5mm] (N2) -- (N12);
\draw[line width=0.5mm] (N2) -- (N17);
\draw[line width=0.5mm] (N3) -- (N11);
\draw[line width=0.5mm] (N3) -- (N14);
\draw[line width=0.5mm] (N3) -- (N17);
\draw[line width=0.5mm,dashed] (N4) -- (N8);
\draw[line width=0.5mm,dashed] (N4) -- (N13);
\draw[line width=0.5mm,dashed] (N4) -- (N18);
\draw[line width=0.5mm,dashed] (N5) -- (N9);
\draw[line width=0.5mm,dashed] (N5) -- (N15);
\draw[line width=0.5mm,dashed] (N5) -- (N18);
\draw[line width=0.5mm] (N6) -- (N10);
\draw[line width=0.5mm] (N6) -- (N13);
\draw[line width=0.5mm] (N6) -- (N19);
\draw[line width=0.5mm] (N7) -- (N11);
\draw[line width=0.5mm] (N7) -- (N15);
\draw[line width=0.5mm] (N7) -- (N19);
\draw[line width=0.5mm,dashed] (N8) -- (N9);
\draw[line width=0.5mm] (N10) -- (N11);
\draw[line width=0.5mm] (N12) -- (N13);
\draw[line width=0.5mm] (N14) -- (N15);
\draw[line width=0.5mm] (N16) -- (N17);
\draw[line width=0.5mm,dashed] (N18) -- (N19);
\end{tikzpicture}}
\\

\end{tabular}
}
\caption{\label{fig:icos_dodec} A vector configuration realizing the oriented matroid $\Mcal$ from Example~\ref{ex:icos_dodec} (top). Two non-trivial connected components of $\mutgraph(\Mcal)$ (bottom). Here we abbreviate the set $\{1,2,3,5\}$ as $1235$, etc.}
\end{figure}

\begin{conjecture}\label{conj:domains} 
	Let $\WS$ be a maximal \emph{by size} $\Mcal$-separated collection inside $2^E$ and let $\dom$ be a mutation-closed domain. Then $\WS\cap \dom$ is a maximal \emph{by size} $\Mcal$-separated collection inside $\dom$.
\end{conjecture}
For the case when $\Mcal$ is a graphical oriented matroid, we prove Conjecture~\ref{conj:domains} in Section~\ref{sect:graph}. We also show in Proposition~\ref{prop:polytopal} that the connected components of $\mutgraph(\Mcal)$ are $1$-skeleta of polytopes just as in Example~\ref{ex:icos_dodec}. Note however that for $\Mcal=\IC(6,3,15)$ in Figure~\ref{fig:positroids}, one of the components of $\mutgraph(\Mcal)$ is not a $1$-skeleton of a polytope.

\begin{proposition}
Conjecture~\ref{conj:domains} implies Conjecture~\ref{conj:mutation_closed_pure}.
\end{proposition}

% \begin{corollary}
% 	If Conjecture~\ref{conj:domains} holds for some pure oriented matroid $\Mcal$ then all mutation-closed domains $\dom\subset 2^E$ are $\Mcal$-pure domains.
% \end{corollary}
\begin{proof}
  Let $\Mcal$ be a pure oriented matroid, $\dom$ be a mutation-closed domain, and $\WS$ be a maximal \emph{by inclusion} $\Mcal$-separated collection inside $\dom$. Since $\Mcal$ is pure, $\WS$ is contained in some maximal \emph{by inclusion} (and thus, \emph{by size}) $\Mcal$-separated collection $\WS'$ inside $2^E$ so the result follows.
\end{proof}

Note that Conjecture~\ref{conj:domains} is much more general than Conjecture~\ref{conj:mutation_closed_pure} as it applies to \emph{all}, not necessarily pure, oriented matroids. We also include a proof of Conjecture~\ref{conj:domains} in an important special case.

\begin{definition}
\label{def:flip_oriented_matroid}
	   Consider an oriented matroid $\Mcal$. We say that two colocalizations $\sigma,\sigma':\Ccal(\Mcal)\to\{+,-\}$ of $\Mcal$ in general position \emph{differ by a flip} if there exists $W\in\Ccal(\Mcal)$ such that  
	   \[\sigma(W)=-\sigma'(W),\quad\text{and}\quad \sigma(T)=\sigma'(T)\quad \text{for all $T\neq\pm W$}.\]
	   We say that $\Mcal$ is \emph{flip-connected} if any two colocalizations of $\Mcal$ in general position can be connected to each other by a sequence of flips.\footnote{This is equivalent to the \emph{extension space} of $\Mcal^\ast$ being connected.}
\end{definition}

% By a result of Ziegler~\cite[Theorem~4.1(G)]{Ziegler}, the alternating matroid $C^{n,r}$ is flip-connected.

We now give a proof of the ``flip-connected'' part of Proposition~\ref{prop:graphicalandflipconnected}. In the proof of Proposition~\ref{prop:flip_conn} and Lemma~\ref{lemma:algorithm} below, we rely on the result of Theorem~\ref{thm:max_size_matroid} which will be proven later in Section~\ref{sec:proof_thm_max_sz_matroid}. (Proposition~\ref{prop:flip_conn} and Lemma~\ref{lemma:algorithm} are not used in the rest of the paper.)

\begin{proposition}\label{prop:flip_conn}
	Suppose that an oriented matroid $\Mcal$ is pure and flip-connected. Then Conjecture~\ref{conj:domains} is valid for $\Mcal$. In particular, any mutation-closed domain $\dom\subset 2^E$ is $\Mcal$-pure.
\end{proposition}
\begin{proof}
	Let $\WS$ be a maximal \emph{by size} $\Mcal$-separated collection inside $2^E$ and let $\dom$ be a mutation-closed domain. We need to show that $\WS_0:=\WS\cap \dom$ is a maximal \emph{by size} $\Mcal$-separated collection inside $\dom$. Suppose that this is not the case and thus there exists an $\Mcal$-separated collection $\WS_1$ inside $\dom$ satisfying $|\WS_1|>|\WS_0|$. Since $\Mcal$ is pure, $\WS_1$ is contained in a maximal \emph{by size} $\Mcal$-separated collection $\WS_2$ inside $2^E$. Let $\sigma$ and $\sigma_2$ be the colocalizations in general positions that correspond (via Theorem~\ref{thm:max_size_matroid}) to $\WS$ and $\WS_2$ respectively. We know that $\WS\neq\WS_2$ and thus $\sigma\neq\sigma_2$. Using the flip-connectedness of $\Mcal$, we get that $\sigma$ can be connected to $\sigma_2$ by a sequence of flips. It remains to show that if two colocalizations $\sigma$ and $\sigma'$ differ by a flip then $|\WS(\sigma)\cap\dom|=|\WS({\sigma'})\cap\dom|$ for any mutation-closed domain $\dom$. The result will follow almost immediately from the following lemma.
	\begin{lemma}\label{lemma:algorithm}
		Suppose that $\sigma$ and $\sigma'$ differ by a flip and let $W\in\Ccal(\Mcal)$ be such that $\sigma(W)=+$ and $\sigma'(W)=-$. Then the collections $\WS(\sigma)$ and $\WS(\sigma')$ are related as follows:
		\begin{equation}\label{eq:mutations}
			\begin{split}
		\WS(\sigma)\cap\WS(\sigma')&=\{S\in\WS(\sigma)\mid S\text{ does not orient $W$}\};\\
		\WS(\sigma')-\WS(\sigma)&=\{\mut{W}{S}\mid S\in\WS(\sigma)\text{ orients $W$ positively}\}.
		\end{split}
		\end{equation}
	\end{lemma}
	\begin{proof}[Proof of the lemma]
	Let us define a collection $\WS'$ by~\eqref{eq:mutations}, i.e.,
		\begin{equation*}
		\begin{split}
		\WS':=&\{S\in\WS(\sigma)\mid S\text{ does not orient $W$}\}\bigsqcup \\
	    &\{\mut{W}{S}\mid S\in\WS(\sigma)\text{ orients $W$ positively}\}.
		\end{split}
		\end{equation*}
	We claim that $\WS'=\WS(\sigma')$, and it suffices to show that any element of $\WS'$ orients every circuit of $\Mcal$ in accordance with $\sigma'$, because the size of $\WS'$ is already equal to the size of $\WS(\sigma)$ (by Theorem~\ref{thm:max_size_matroid}). So suppose that there exists a set $T'\in\WS'$ and a circuit $C\in\Ccal(\Mcal)$ such that $\sigma'(C)=\sigma(C)=+$ but $T'$ orients $C$ negatively (we are using here that clearly $C\neq \pm W$). It must be the case that $T'=\mut{W}{T}$ for some $T\in\WS(\sigma)$. In particular, we may assume that $W^+\subset T$ and $W^-\cap T=\emptyset$. After reorienting all the elements of $E-T$ in $\Mcal$, we may assume that $T=E$ and $W\in \{0,+\}^E$. 
	
	Let $\Mcaltilde$ and $\Mcaltilde'$ be the one-element liftings of $\Mcal$ corresponding to $\sigma$ and $\sigma'$ respectively. Thus $(W,-)$ and $(C,+)$ are circuits of $\Mcaltilde'$. As usual, we will denote the ground sets of $\Mcaltilde$ and $\Mcaltilde'$ by $Eg$.
	
	We denote $R:=W^0$ and $S:=E-R$. Thus $W(R)=0^R$ and $W(S)=+^S$. Here $W(R)$ is the restriction of $W$ to $R$, $+^S$ is the signed vector $(S,\emptyset)$, etc. Since $C$ is oriented negatively by $T'=T-S$, we have 
	\[C(S)\in\{+,0\}^S;\quad C(R)\in\{-,0\}^R.\]
	We know that $\sigma'(C)=\sigma(C)=+$, so in particular $C\neq \pm W$. We are going to describe a certain algorithm. As an input, it takes a circuit $X\in\Ccal(\Mcal)$ such that
	\begin{itemize}
		\item $X^+\neq \emptyset$,
		\item $X\neq \pm W$,
		\item $\sigma(X)=\sigma'(X)=+$, and
		\item $X(R)\in\{-,0\}^R$.
	\end{itemize}
	As an output, it produces a circuit $Y\in\Ccal(\Mcal)$ such that
	\begin{itemize}
		\item $Y\neq \pm W$,
		\item $\sigma(Y)=\sigma'(Y)=+$,
		\item $Y(R)\in\{-,0\}^R$, and
		\item $Y^+\subsetneq X^+$.
	\end{itemize}
	Note that one can iteratively apply this algorithm starting with $X=C$ until eventually we have $Y^+=\emptyset$. This leads to a contradiction since $Y$ is then oriented negatively by $T=E\in \WS$ even though $\sigma(Y)=+$.
	
	Let us describe the steps of the algorithm.
	\begin{enumerate}
		\item Choose an element $s\in E$ such that $X_s=+$. Since $X(R)\in\{-,0\}^R$, we have $s\in S$ and thus $W_s=+$.
		\item Since $\sigma'(W)=-$, we have that both $(X,+)$ and $(-W,+)$ are circuits of $\Mcaltilde'$. In particular, $X_s=+$ and $-W_s=-$, so apply Axiom~\hyperref[item:C3]{(C3)} to produce a circuit $(Z,\epsilon)\in\Ccal(\Mcaltilde')$ for some $\epsilon\in\{+,0\}$ such that 
		\[Z_s=0;\quad Z^+\subset X^+;\quad Z(R)\in\{-,0\}^R.\]
		\item If $\epsilon=+$ then output $Y:=Z$. Because $W_s=+$ and $Z_s=0$, we have $Z\neq \pm W$, thus in particular $\sigma(Z)=\sigma'(Z)=+$. 
		\item If $\epsilon=0$ then by Proposition~\ref{prop:circuits_coloc}, there exists a pair of circuits $Y,Y'$ of $\Mcal$ such that $\sigma'(Y)=+$, $\sigma'(Y')=-$, and $Y,Y'\leq Y\circ Y'=Z$. The latter implies that $Y_s=0$ so $Y^+\subsetneq X^+$. Since $Y_s= 0$, we get $Y\neq \pm W$ and so $\sigma(Y)=\sigma'(Y)=+$. Finally, since $Y\leq Z$, we get $Y(R)\in\{-,0\}^R$.
	\end{enumerate}
	We have constructed the desired algorithm which, as we explained earlier, contradicts the existence of $C$. This finishes the proof of the lemma.
	\end{proof}
	
	Using Lemma~\ref{lemma:algorithm}, it is now easy to deduce Proposition~\ref{prop:flip_conn}. Recall that the only thing left to show was that if $\sigma$ and $\sigma'$ differ by a flip then $|\WS(\sigma)\cap\dom|=|\WS({\sigma'})\cap\dom|$ for any mutation-closed domain $\dom$. Indeed, the collections $\WS(\sigma)$ and $\WS(\sigma')$ are related by~\eqref{eq:mutations} which gives an obvious bijection 
	\[T\mapsto \begin{cases}
		      \mut{W}{T},&\text{if $T$ orients $W$ positively};\\
		      T,&\text{otherwise}
	           \end{cases}
	\] 
	between the sets $\WS(\sigma)\cap\dom$ and $\WS({\sigma'})\cap\dom$. We are done with the proof of Proposition~\ref{prop:flip_conn}.
\end{proof}

\def\altsixfour{(C^{6,2})^*}
\def\altn{(C^{n,2})^*}
\subsection{The structure of the alternating matroid of corank $2$}\label{sect:altn}
In this section, we describe explicitly the circuits, colocalizations, and the mutation graph of the alternating matroid $C^{n,n-2}$ of corank $2$. It is more convenient to describe another oriented matroid $\altn$.
% \[\altn:=(C^{n,2})^*.\]
It is easy to see that $\altn$ is isomorphic to the alternating matroid $C^{n,n-2}$. We leave the following lemma as an exercise for the reader.

\begin{lemma}
The oriented matroid $(C^{n,d})^\ast$ is isomorphic to $C^{n,n-d}$. They are obtained from each other by reorienting the set $\{2,4,\dots\}\subset [n]$.
\end{lemma}

Recall that $[n]$ denotes the set $\{1,2,\dots,n\}$. For two integers $i$ and $j$, we define $[i,j]\subset \Z$ to be the set of all $k\in\Z$ satisfying $i\leq k\leq j$. In particular, $[i,j]=\emptyset$ if $i>j$.

For $1\leq k\leq n$, let $C_k$ be the signed set given by
\begin{equation}\label{eq:circuits_altn}
C_k:=([1,k-1],[k+1,n]).
\end{equation}

\begin{lemma}\label{lemma:altn_description}
The circuits of $\altn$ are $C_1$, $C_2$, $\dots$, $C_n$, $-C_1$, $-C_2$, $\dots$, $-C_n$. The only pairs $(I,J)$ of non $\altn$-separated sets are $([1,l],[m+1,n])$ for $l,m=0,\dots,n$ with $|l-m|\leq 1$.
\end{lemma}
\begin{proof}
  The circuits of $\altn$ are the cocircuits of $C^{n,2}$ which are clearly given by~\eqref{eq:circuits_altn}
 (see Figure~\ref{fig:cyclic}). 
 
 Thus two sets $I,J\subset [n]$ are \emph{not} $\altn$-separated if and only if there exists $k\in [n]$ such that they orient $C_k$ in the opposite ways. If, say, $I$ orients $C_k$ positively and $J$ orients $C_k$ negatively then we have
\[I-J\preceq \{k\}\preceq J-I,\quad\textrm{and}\quad (I-J)\cup \{k\}\cup (J-I)=[n].\]
 Here $S\preceq T$ means that every element of $S$ is less than or equal to every element of $T$. Thus a subset $I$ of $[n]$ is $\altn$-separated from all other subsets of $[n]$ unless it has the form $[k]$ or $[k+1,n]$ for some $k\in [n]$. The empty set induces a positive orientation on $C_1$ and a negative orientation on $C_n$, the set $[n]$ induces a negative orientation of $C_1$ and a positive orientation of $C_n$. For $k\in [n-1]$, the set $[k]$ (resp., $[k+1,n]$) induces a positive (resp., negative) orientation on $C_k$ and $C_{k+1}$. Thus all pairs $(I,J)$ of non-$\altn$-separated subsets of $[n]$ are exactly the pairs listed in the statement.
\end{proof}

See Figure~\ref{fig:mut_graph_5} for an example for $n=5$.

% \begin{lemma}\label{lemma:altn_components}
% 	If $n$ is odd then the edges of $\mutgraph(\altn)$ form a union of two $n$-cycles. If $n$ is even then $\mutgraph(\altn)$ is a $2n$-cycle.
% \end{lemma}
% \begin{proof}
% 	This follows by inspection from Lemma~\ref{lemma:altn_description}. 
% \end{proof}
% For example, one connected component $\dom$ of $\mutgraph(\altn)$ for $n=5$ consists of sets $[1,0],[2,5],[1,2],[4,5],[1,4]$. Note also that for any $n$, the graph $\mutgraph(\altn)$ contains a lot of isolated vertices that correspond to sets which do not orient any circuits of $\altn$.

\begin{figure}
\scalebox{0.7}{
	\begin{tikzpicture}[scale=4]
		\node[draw,ellipse] (10) at (0,1) {$[1,0]$};
		\node[draw,ellipse] (11) at (1,1) {$[1,1]$};
		\node[draw,ellipse] (12) at (2,1) {$[1,2]$};
		\node[draw,ellipse] (13) at (3,1) {$[1,3]$};
		\node[draw,ellipse] (14) at (4,1) {$[1,4]$};
		\node[draw,ellipse] (15) at (0,0) {$[1,5]$};
		\node[draw,ellipse] (25) at (1,0) {$[2,5]$};
		\node[draw,ellipse] (35) at (2,0) {$[3,5]$};
		\node[draw,ellipse] (45) at (3,0) {$[4,5]$};
		\node[draw,ellipse] (55) at (4,0) {$[5,5]$};

		\draw[color=black] (14) to [in=15, out=165] node[pos=0.2,above] {$C_5$} node[pos=0.8,above] {$-C_5$} (10);
		\draw[color=black] (55) to [in=-15, out=-165] node[pos=0.2,below] {$-C_5$} node[pos=0.8,below] {$C_5$} (15);
		
		\draw[color=black] (10) to node[pos=0.2,above] {$C_1$} node[pos=0.9,above,inner sep=8pt] {$-C_1$} (25);
		\draw[color=black] (11) to node[pos=0.2,above] {$C_2$} node[pos=0.9,above,inner sep=8pt] {$-C_2$} (35);
		\draw[color=black] (12) to node[pos=0.2,above] {$C_3$} node[pos=0.9,above,inner sep=8pt] {$-C_3$} (45);
		\draw[color=black] (13) to node[pos=0.2,above] {$C_4$} node[pos=0.9,above,inner sep=8pt] {$-C_4$} (55);
		
		\draw[color=black] (11) to node[pos=0.2,above,inner sep=6pt] {$C_1$} node[pos=0.9,above,inner sep=10pt] {$-C_1$} (15);
		\draw[color=black] (12) to node[pos=0.2,above,inner sep=6pt] {$C_2$} node[pos=0.9,above,inner sep=10pt] {$-C_2$} (25);
		\draw[color=black] (13) to node[pos=0.2,above,inner sep=6pt] {$C_3$} node[pos=0.9,above,inner sep=10pt] {$-C_3$} (35);
		\draw[color=black] (14) to node[pos=0.2,above,inner sep=6pt] {$C_4$} node[pos=0.9,above,inner sep=10pt] {$-C_4$} (45);
		
% 		\draw[color=black] (14) to [in=15, out=165] (10);
% 		\draw[color=black] (55) to [in=-15, out=-165] (15);
% 		
% 		\draw[color=black] (10) to (25);
% 		\draw[color=black] (11) to (35);
% 		\draw[color=black] (12) to  (45);
% 		\draw[color=black] (13) to  (55);
% 		
% 		\draw[color=black] (11) to (15);
% 		\draw[color=black] (12) to (25);
% 		\draw[color=black] (13) to (35);
% 		\draw[color=black] (14) to (45);
% 		
		\draw[color=red] (10) to (15);
		\draw[color=red] (11) to (25);
		\draw[color=red] (12) to (35);
		\draw[color=red] (13) to (45);
		\draw[color=red] (14) to (55);
	\end{tikzpicture}}
\caption{\label{fig:mut_graph_5} Two subsets $S$ and $T$ are connected by an edge in this graph if and only if they are not $\altn$-separated for $n=5$. They are connected by a black edge labeled by $\pm C_k$ if and only if $T=\mut{C_k}{S}$. The red edges do not belong to $\mutgraph(\altn)$.}
\end{figure}
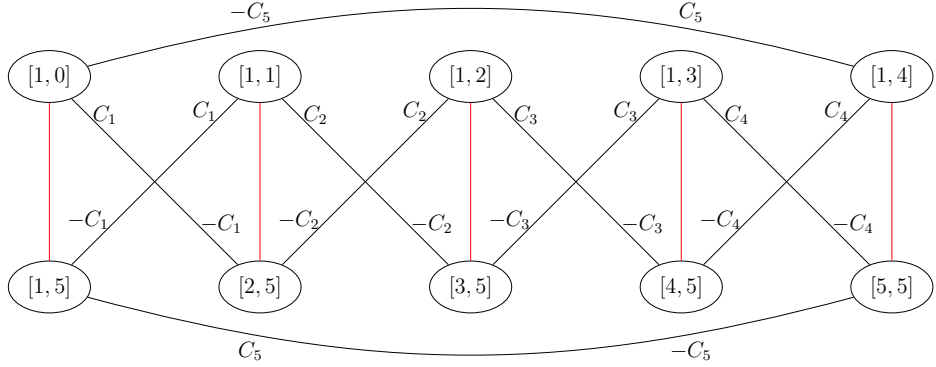

\def\q{\epsilon}
\def\qq{\epsilon}
\begin{definition}
	Let $\WS\subset 2^{[n]}$ be an $\altn$-separated collection. Define $\qq(\WS)=(\q_1,\q_2,\dots,\q_n)\in \{+,-,0\}^{[n]}$ by 
	\[\q_k= \begin{cases}
	        	+,&\text{if $[1,k-1]\in\WS$};\\
	        	-,&\text{if $[k,n]\in\WS$};\\
	        	0,&\text{otherwise}.
	        \end{cases}\]
\end{definition}
Since $\WS$ is $\altn$-separated, we have $\q_k=0$ if and only if $\WS$ contains neither $[1,k-1]$ nor $[k,n]$. 

\begin{lemma}\label{lemma:qq}
	Take any $\altn$-separated collection $\WS$ and consider the signed vector $\qq(\WS)=(\q_1,\q_2,\dots,\q_n)$. For any $k=1,2,\dots,n-1$, if $\q_k=-\q_{k+1}$ then $\q_k=\q_{k+1}=0$. The same holds for $\q_1$ and $-\q_n$, that is, if $\q_1=\q_n$ then $\q_1=\q_n=0$.
\end{lemma}
\begin{proof}
	This also follows by inspection from Lemma~\ref{lemma:altn_description}. 
\end{proof}
Lemma~\ref{lemma:qq} implies that $\qq(\WS)$ has at least one zero. We say that $\qq(\WS)$ is \emph{maximal} if it has exactly one zero.

\begin{theorem}\label{thm:altn}
	For any maximal \emph{by size} $\altn$-separated collection $\WS$, $\sigma_\WS$ is a colocalization of $\altn$ in general position. 
\end{theorem}
\begin{proof}
	One easily observes that if $\WS$ is maximal \emph{by size} then $\qq(\WS)$ is maximal, which, in turn, implies that $\sigma_\WS$ has Type III (cf. Theorem~\ref{thm:LasVergnas}).
	% One easily observes that $\WS$ is maximal \emph{by size} if and only if $\WS$ is maximal \emph{by inclusion} and $\qq(\WS)$ is maximal. On the other hand, it follows from Theorem~\ref{thm:LasVergnas} that $\sigma_\WS$ is a colocalization in general position if and only if it has Type III, which is clearly equivalent to $\qq(\WS)$ being maximal. This finishes the proof of the theorem.
\end{proof}

% 
% \begin{theorem}\label{thm:altn}
% 	Let $\dom$ be a mutation-closed domain for $\altn$. Then for any maximal \emph{by size} $\altn$-separated collection $\WS$ inside $\dom$, $\sigma_\WS$ is a colocalization of $\altn$ (that is not necessarily in general position). If $\dom=2^{[n]}$ then $\sigma_\WS$ is a colocalization in general position.
% \end{theorem}
% \begin{proof}
% 	First let us assume that $\dom=2^{[n]}$. Then one easily observes that $\WS$ is maximal \emph{by size} if and only if $\qq(\WS)$ is maximal. On the other hand, it follows from Theorem~\ref{thm:LasVergnas} that $\sigma_\WS$ is a colocalization in general position if and only if it has Type III, which is clearly equivalent to $\qq(\WS)$ being maximal. This shows the second claim.
% 	
% 	To show the first claim, the only case left to consider is when $n=2m+1$ is odd and $\dom$ is a connected component of $\mutgraph(\altn)$. Thus the restriction of $\mutgraph(\altn)$ to $\dom$ is an $n$-cycle by Lemma~\ref{lemma:altn_components}. By Lemma~\ref{lemma:altn_description}, two sets $S,T\in\dom$ are $\altn$-separated if and only if they are connected by an edge in $\mutgraph(\altn)$. Therefore if $\WS$ is maximal \emph{by size} then it must have size $m$. Thus $\qq(\WS)$ has the form either 
% 	\[(0,+,0,+,0,\dots,+,0,0,-,0,-,\dots,0,-)\] 
% 	or 
% 	\[(+,0,+,0,\dots,+,0,0,-,0,-,\dots,0,-,0).\] 
% 	It follows that $\sigma_\WS$ has Type II in each of these two cases.
% \end{proof}

\subsection{Colocalizations and complete collections}

An $\Mcal$-separated collection $\WS$ is called \emph{complete} if the image of $\sigma_\WS$ lies in $\{+,-\}$ (that is, for every circuit, there is a set in $\WS$ that orients it). We collect the properties of complete collections in the following proposition:

\begin{proposition}\label{prop:complete_properties}
 Let $\Mcal$ be an oriented matroid on the ground set $E$ and suppose that $\WS$ is an $\Mcal$-separated collection. Then:
 \begin{enumerate}[\normalfont (1)]
  \item\label{item:complete_deletion} if $\WS$ is complete then for every $e\in E$, $\WS-e$ is a complete $(\Mcal-e)$-separated collection;\footnote{In contrast, $\WS/e$ may even be empty for a complete $\Mcal$-separated collection $\WS$.}
  \item\label{item:complete_colocalization} if $\WS$ is complete and $\rk(\Mcal)\leq 3$ then $\sigma_\WS$ is a colocalization in general position.
 \end{enumerate}
\end{proposition}
\begin{proof}

 The claim~(\ref{item:complete_deletion}) is obvious from the definitions of complete, $\WS-e$, and $\Mcal-e$. To show~(\ref{item:complete_colocalization}), note that any nullity $2$ subset $A\subset E$ contains at most $n\leq 5$ elements. To check that $\sigma_\WS$ is a colocalization, we need to consider all possible nullity $2$ subsets and restrict $\sigma_\WS$ to the circuits of the corresponding alternating matroids isomorphic to $\altn$ (where $n\leq 5$). 
 By Lemmas~\ref{lemma:altn_description} and~\ref{lemma:qq}, we can describe the restriction $\WS(A)$ of $\WS$ to $A$ by a sequence $\qq(\WS(A))=(\q_1,\dots,\q_n)\in\{+,-,0\}^n$. Let  $\qqq(\WS(A)):=(\q_1,\dots,\q_n,-\q_1,\dots,-\q_n)\in\{+,-,0\}^{2n}$. 
By Lemma~\ref{lemma:qq}, no two adjacent nonzero entries of $\qqq(\WS(A))$ have opposite signs. If $\qqq(\WS(A))$ contains two consecutive zeros then $\WS(A)$ is not complete, which by~(\ref{item:complete_deletion}) implies that $\WS$ is not complete. Let $\qb_1:=\q_1$, $\qb_{2n}:=\q_{2n}$, and for $1<k<2n$, let $\qb_k$ be defined as follows: if $\q_{k-1}=\q_{k+1}\neq0$ and $\q_k=0$ then we set $\qb_k:=\q_{k-1}$, otherwise we set $\qb_k:=\q_k$. Thus if $\qb_k=0$ for some $1<k<2n$ then $\q_{k-1},\q_{k+1}$ must be nonzero and have opposite signs. Since $n\leq 5$, it follows that $\qqb(\WS(A)):=(\qb_1,\dots,\qb_n)$ has the form $(+^{k-1},0,-^{n-k})$ or $(-^{k-1},0,+^{n-k})$ for some $k\in [n]$. Indeed, the smallest example when $\qqb(\WS(A))$ satisfies all the listed properties but does not have the desired form happens for $n=6$ where we can have $\qqb(\WS(A))=(+,0,-,0,+,0)$. Since $\qq(\WS(A))$ and $\qqb(\WS(A))$ give rise to the same orientation of the circuits of $\Mcal\mid_A$, we find that $\sigma_\WS$ is a colocalization in general position.
\end{proof}

We now formulate a basic fact on existence of colocalizations in general position:

\begin{lemma}[{\cite[Proposition~7.2.2]{Book}}]\label{lemma:coloc_exists}
	Let $\Mcal$ be an oriented matroid. 
	\begin{enumerate}[\normalfont (i)]
		\item For any colocalization $\sigma$ of $\Mcal$, there exists a colocalization $\sigma'$ of $\Mcal$ in general position satisfying $\sigma\leq \sigma'$.
		\item In particular, there exists at least one colocalization of $\Mcal$ in general position.\qed
	\end{enumerate}
\end{lemma}

We now give an alternative description of $\WS(\sigma)$ in terms of the one-element lifting defined by $\sigma$.
\begin{lemma}
Given a colocalization $\sigma:\Ccal\to\{+,-,0\}$ (not necessarily in general position) corresponding to a one-element lifting $\Mcaltilde$ of $\Mcal$, we have
\begin{equation}\label{eq:colocalization}
\WS(\sigma)=\{S\subset E\mid (S-L,(Eg-S)-L)\in\Tcal(\Mcaltilde)\},
\end{equation}
where $L\subset Eg$ denotes the set of loops of $\Mcaltilde$.
\end{lemma}
\begin{proof}
  We denote the right hand side of~\eqref{eq:colocalization} by $\WS'$. We first show $\WS(\sigma)\subset\WS'$. Let $S\in\WS(\sigma)$ be a set. We need to show that the signed vector $T=(S-L,(Eg-S)-L)$ is a tope of $\Mcaltilde$. Note that the support of $T$ is $Eg-L$ and thus it suffices to prove that $T$ is a covector of $\Mcaltilde$. This is equivalent to saying that $T$ is orthogonal to every circuit of $\Mcaltilde$, so let $X\in\Ccal(\Mcaltilde)$ be such a circuit. By Proposition~\ref{prop:circuits_coloc}, either $X$ has the form $(Y,\sigma(Y))$ for some $Y\in\Ccal(\Mcal)$ or we have $X=Y^1\circ Y^2$ for some $Y^1,Y^2\in\Ccal(\Mcal)$ such that $\sigma(Y^1)=-\sigma(Y^2)\neq0$. It is clear that $T\perp X$ if $X$ is a loop of $\Mcaltilde$ so suppose that $\Xu\cap L=\emptyset$. Assume that $X$ has the form $(Y,\sigma(Y))$ for some $Y\in\Ccal(\Mcal)$. If $S$ does not orient $Y$ then $(S,E-S)$ is orthogonal to $Y$ and thus $T$ is orthogonal to $(Y,\sigma(Y))$. Otherwise, if $S$ orients $Y$, say, positively, then $\sigma(Y)=+$ so $X=(Y^+g,Y^-)$ is easily checked to be orthogonal to $T$. The case $\sigma(Y)=-$ is completely analogous and we are left with the case $X=Y^1\circ Y^2$  so that $\sigma(Y^1)=-\sigma(Y^2)\neq0$. Recall from Proposition~\ref{prop:circuits_coloc} that we have $Y^1,Y^2<X$. If $S$ does not orient $Y^i$ for either $i=1$ or $i=2$ then clearly $T$ is orthogonal to $X$. On the other hand, if $S$ orients both of them then we may assume that it orients $Y^1$ positively and $Y^2$ negatively which also implies that $T$ is orthogonal to $X$. We have shown the inclusion $\WS(\sigma)\subset\WS'$.
	
	To show $\WS(\sigma)\supset\WS'$, suppose that $S\in \WS'$ and thus $T=(S-L,(Eg-S)-L)$ is a tope of $\Mcaltilde$. Let $C\in\Ccal(\Mcal)$ be a circuit that is oriented by $S$, say, positively. This means that $C^+\subset S$ and $C^-\subset E-S$ so the only way for $(C,\sigma(C))$ to be orthogonal to $T$ is if $\sigma(C)=+$. The case when $S$ orients $C$ negatively is again analogous so we are done with the proof.
\end{proof}

\begin{lemma}\label{lemma:extension}
 Let $\Mcal$ be an oriented matroid and suppose that $\sigma$ is a colocalization of $\Mcal$ in general position. Then
 \begin{enumerate}[\normalfont (1)]
  \item \label{item:extension_complete} the collection $\WS(\sigma)$ is a complete $\Mcal$-separated collection;
  \item \label{item:extension} if $\WS$ is any $\Mcal$-separated collection satisfying $\sigma_\WS\leq \sigma$ then $\WS\subset\WS(\sigma)$.
%  \item\label{item:max_by_size_complete} if a collection $\WS\subset 2^E$ is maximal \emph{by size} then $\WS$ is complete.
 \end{enumerate}
\end{lemma}
\begin{proof}
 By definition, $\WS(\sigma)$ contains all sets $S\subset E$ that orient each circuit $C\in\Ccal$ in accordance with $\sigma$. Since $\sigma_\WS\leq \sigma$, this happens for all sets in $\WS$, and thus $\WS\subset\WS(\sigma)$, which proves~(\ref{item:extension}). Since the topes orient all the circuits by~\cite[Proposition~3.8.2]{Book}, we get that $\WS(\sigma)$ is complete which proves~\eqref{item:extension_complete}.% It is maximal by size because it corresponds to a colocalization in general position and thus by Theorem~\ref{thm:max_size_matroid} has precisely $|\Ind(\Mcal)|$ elements which proves~\eqref{item:extension_complete}.
%  
%  Finally, if $\WS$ is maximal \emph{by size} then by Theorem~\ref{thm:max_size_matroid} it corresponds to a colocalization in general position and thus is complete by the first claim. This shows~\eqref{item:max_by_size_complete}.
\end{proof}

\subsection{Proof of Theorem~\ref{thm:max_size_matroid}}\label{sec:proof_thm_max_sz_matroid}
Before we proceed to the proof, we recall some basic facts about Tutte polynomials of unoriented matroids, see, e.g.,~\cite{Crapo}.
\begin{proposition}\label{prop:Tutte}
 Let $\Mcal$ be an oriented matroid, and $\Mcalu$ its underlying matroid with Tutte polynomial $T_\Mcalu(x,y)$. Then we have $|\Ind(\Mcal)|=T_\Mcalu(2,1)$ and $|\Bcal(\Mcal)|=T_\Mcalu(1,1)$. Moreover, for every element $e\in E$ which is neither a loop nor a coloop, we have 
 \[T_\Mcalu(x,y)=T_{\Mcalu-e}(x,y)+T_{\Mcalu/e}(x,y).\]
 In particular, 
 \begin{equation*}
   \begin{split}
\pushQED{\qed} 
 |\Ind(\Mcal)|&=|\Ind(\Mcal-e)|+|\Ind(\Mcal/e)|,\text{ and}\\
   |\Bcal(\Mcal)|&=|\Bcal(\Mcal-e)|+|\Bcal(\Mcal/e)|.\qedhere
\popQED
   \end{split}
   % ,\text{ and}\\
%  |\Tcal(\Mcal)|&=&|\Tcal(\Mcal-e)|+|\Tcal(\Mcal/e)|.
 \end{equation*}

\end{proposition}

% We are now ready to prove Theorem~\ref{thm:max_size_matroid}:

% \begin{theorem}\label{thm:max_size_matroid}
%  Let $\Mcal$ be any oriented matroid. Then the map $\WS\mapsto\sigma_\WS$ is a bijection (with inverse $\sigma\mapsto\WS(\sigma)$) between maximal \emph{by size} $\Mcal$-separated collections of subsets and colocalizations of $\Mcal$, or, equivalently, one-element liftings of $\Mcal$. Any maximal \emph{by size} $\Mcal$-separated collection has size $|\Ind(\Mcal)|$. 
% \end{theorem}
\begin{proof}[Proof of Theorem~\ref{thm:max_size_matroid}]
 We are going to show three statements that will together imply the theorem.
 \begin{enumerate}[(a)]
  \item\label{item:coloc_gp_collec_max_sz} If $\sigma:\Ccal\to\{+,-\}$ is a colocalization in general position then $|\WS(\sigma)|=|\Ind(\Mcal)|$.
  \item\label{item:collec_sz_bound} For any $\Mcal$-separated collection $\WS$, $|\WS|\leq |\Ind(\Mcal)|$.
  \item\label{item:collec_colloc} If $|\WS|=|\Ind(\Mcal)|$ then $\sigma_\WS$ is a colocalization in general position.
 \end{enumerate}

To show (\ref{item:coloc_gp_collec_max_sz}), note that $|\WS(\sigma)|=|\Tcal(\Mcaltilde)|/2$ by~\eqref{eq:colocalization} so we need to show $|\Tcal(\Mcaltilde)|/2=|\Ind(\Mcal)|$. As it follows from the discussion before \cite[Proposition~3.8.3]{Book}, if $e$ is neither a loop nor a coloop in some \emph{simple} oriented matroid $\Mcal'$ then 
\[|\Tcal(\Mcal')|=|\Tcal(\Mcal'-e)|+|\Tcal(\Mcal'/e)|.\]
Since this formula only works for simple oriented matroids $\Mcal'$, the number $|\Tcal(\Mcal')|$ is not an evaluation of the Tutte polynomial. However, note that $\Mcaltilde$ is a lifting of $\Mcal$ in general position so $\Mcaltilde$ is indeed simple. Therefore we can choose some element $e\in E$ that is neither a loop nor a coloop in $\Mcal$ and apply the deletion-contraction recurrence to $\Mcaltilde$. Clearly $\Mcaltilde/e$ (resp., $\Mcaltilde-e$) will be a lifting in general position of $\Mcal/e$ (resp., of $\Mcal-e$). Thus it suffices to check the equality $|\Tcal(\Mcaltilde)|/2=|\Ind(\Mcal)|$ only for oriented matroids $\Mcal$ that consist of loops and coloops. Suppose that there are $a$ loops and $b$ coloops in $\Mcal$. Then $|\Ind(\Mcal)|=2^b$. On the other hand, $\sigma$ maps every loop of $\Mcal$ to either a $+$ or a $-$, and $\WS(\sigma)$ consists of the $2^b$ sets whose restriction to the set of loops is fixed and determined by $\sigma$. We are done with the base case and therefore with~(\ref{item:coloc_gp_collec_max_sz}).
 
Now we prove part~(\ref{item:collec_sz_bound}) by induction on $|E|$. The base of induction is the case when every element of $\Mcal$ is either a loop or a coloop, and in this case the statement of the theorem holds by an argument similar to the one above. To do the induction step, take any $\Mcal$-separated collection $\WS$ and consider any element $e\in E$ that is neither a loop nor a coloop. By the induction hypothesis combined with Propositions~\ref{prop:ws_tutte_recurrence} and~\ref{prop:Tutte}, we have 
\begin{equation}\label{eq:deletion_contraction_both_parts}
  \begin{split}
|\WS-e|&\leq |\Ind(\Mcal-e)|;\\
|\WS/e|&\leq |\Ind(\Mcal/e)|;\\
|\WS| &= |\WS-e|+|\WS/e|;\\
|\Ind(\Mcal)| &= |\Ind(\Mcal-e)|+|\Ind(\Mcal/e)|.
  \end{split}
\end{equation}
The fact that $|\WS|\leq |\Ind(\Mcal)|$ follows.

Finally, we prove (\ref{item:collec_colloc}) by induction on $|E|$. The base of induction is the empty oriented matroid $\Mcal$ with $E=\emptyset$. It has exactly one maximal by size $\Mcal$-separated collection $\WS=\{\emptyset\}$ which corresponds to the ``empty'' colocalization of $\Mcal$ via $\WS\mapsto \sigma_\WS$ and vice versa. 

Now we do the induction step. Let $\Mcal$ be any oriented matroid and let $\WS$ be any maximal \emph{by size} $\Mcal$-separated collection. Suppose $e$ is a loop of $\Mcal$. Then either all elements of $\WS$ contain $e$ or all elements of $\WS$ do not contain $e$. In other words, maximal \emph{by size} $\Mcal$-separated collections are in two-to-one correspondence with maximal \emph{by size} $(\Mcal-e)$-separated collections. Similarly, each colocalization $\sigma$ of $\Mcal$ in general position either sends $e$ to $+$ or to $-$, and induces a colocalization of $\Mcal-e$. Thus in the case when $e$ is a loop, the induction step is clear.

If $e$ is a coloop of $\Mcal$ then it is not contained in any circuit so for all $S\in\WS$ we have $S-e,S\cup e\in\WS$. The induction step is clear here as well, and we are left with the case when $\Mcal$ has no loops and no coloops.

Let $\WS$ be an $\Mcal$-separated collection satisfying $|\WS|=|\Ind(\Mcal)|$, and let $e$ be any element of $\Mcal$ (it is neither a loop nor a coloop). It follows that $\WS$ (resp., $\WS-e$ and $\WS/e$) is a maximal \emph{by size} $\Mcal$-separated (resp., $(\Mcal-e)$-separated and $(\Mcal/e)$-separated) collection. 
% Then by Lemma~\ref{lemma:coloc_exists} together with statements~\eqref{item:coloc_gp_collec_max_sz} and~\eqref{item:collec_sz_bound} shown above, $|\WS|=|\Ind(\Mcal)|$, and moreover, $\WS-e$ (resp., $\WS/e$) is a maximal \emph{by size} $(\Mcal-e)$-separated (resp., $(\Mcal/e)$-separated) collection of sets. 
 Therefore by the induction hypothesis, both $\WS-e$ and $\WS/e$ correspond to colocalizations in general position of the corresponding oriented matroids. 

We would like to show that every circuit $C=(C^+,C^-)$ of $\Mcal$ is oriented by $\WS$. Suppose first that there is an element $e\in E-\Cu$. Then we know that $\WS-e$ orients $C$ by part~\eqref{item:extension_complete} of Lemma~\ref{lemma:extension}. Thus the same is true for $\WS$. Otherwise we have $\Cu=E$ which by the incomparability axiom~\hyperref[item:C2]{(C2)} in Definition~\ref{dfn:OM} implies that $C$ and $-C$ are the only circuits of $\Mcal$. In this case obviously $\WS$ orients $C$ so we are done. 

Since $\WS$ orients every circuit of $\Mcal$, the image of $\sigma_\WS$ lies in $\{+,-\}$. In this case it is clear that the maps $\WS\mapsto\sigma_\WS$ and $\sigma\mapsto\WS(\sigma)$ are inverse to each other. Since the image of $\sigma_\WS$ lies in $\{+,-\}$, it is in general position, but we still need to show that it is a colocalization.

Assume that $\sigma_\WS$ is not a colocalization. It means that there is a subset $A\subset E$ of nullity $2$ for which the restriction of $\sigma_\WS$ is not of Type III. If $A\subsetneq E$ then we can delete any element in $E-A$ and get a contradiction with the induction hypothesis. Thus we may assume that $\Mcal$ has corank $2$, and hence is realizable by \cite[Corollary 8.2.3]{Book}. We contract some elements until $\Mcal^*$ has no pairs of parallel elements while preserving the fact that $\sigma_\WS$ is not a colocalization. After that, $\Mcal$ must be isomorphic $\altn$ (where $n=|E|$). By Theorem~\ref{thm:altn}, $\sigma_\WS$ is a colocalization of $\altn$ in general position because $\WS$ is a maximal \emph{by size} $\altn$-separated collection. This shows~\eqref{item:collec_colloc}.

Now we explain how Theorem~\ref{thm:max_size_matroid} follows from~\eqref{item:coloc_gp_collec_max_sz},~\eqref{item:collec_sz_bound}, and~\eqref{item:collec_colloc}. Take any oriented matroid $\Mcal$. By Lemma~\ref{lemma:coloc_exists}, there exists a colocalization $\sigma$ of $\Mcal$ in general position. By~\eqref{item:coloc_gp_collec_max_sz}, the size of $\WS(\sigma)$ is $|\Ind(\Mcal)|$. By~\eqref{item:collec_sz_bound}, we get that the maximal size of an $\Mcal$-separated collection is exactly equal to $|\Ind(\Mcal)|$. Finally, by~\eqref{item:collec_colloc}, any $\Mcal$-separated collection with $|\Ind(\Mcal)|$ elements corresponds to a colocalization of $\Mcal$ in general position. We are done with the proof of Theorem~\ref{thm:max_size_matroid}.
% Let $\WS$ be any maximal \emph{by size} $\altn$-separated collection. By Lemma~\ref{lemma:qq}, $\WS$ can be represented by a signed set $Q=(q_1,\dots,q_n)\in\{+,-,0\}^n$ where we set $q_k=+$ if $\WS$ contains $[k]$, $q_k=-$ if $\WS$ contains $[n]-[k]$ and $q_k=0$ otherwise. It follows from Lemma~\ref{lemma:altn_description} that $Q$ has at least one zero between any two opposite signs and that $Q$ has at least one zero, because if $Q$ consists of all plus signs then $Q$ contains both $[1]$ and $[n]$ which are not $C^{n,n-2}$-separated, and similarly for minus signs. If $Q$ has at least two zeroes then $\WS$ is not going to have the maximal size. Thus $Q$ has exactly one sign change, more specifically, $Q$ has a form $(+^{k-1},0,-^{n-k})$ or $(-^{k-1},0,+^{n-k})$ for some $k\in [n]$. But this is precisely the case when $\sigma_\WS$ is a colocalization! Thus we are done with the proof of Theorem~\ref{thm:max_size_matroid}.
\end{proof}

As a corollary, we get that a maximal \emph{by size} $\Mcal$-separated collection orients every circuit of $\Mcal$:
\begin{corollary}\label{cor:complete}
	Let $\WS$ be a maximal \emph{by size} $\Mcal$-separated collection. Then $\WS$ is complete.
\end{corollary}

\def\gin{{\ni g}}
\def\gnot{{\not\ni g}}
\def\ein{{\ni e}}
\def\enot{{\not\ni e}}
\def\fin{{\ni f}}
\def\fnot{{\not\ni f}}

\subsection{Zonotopal tilings for oriented matroids}\label{subsect:tilings}

The Bohne-Dress theorem (Theorem~\ref{thm:bohne}) provides a bijection between zonotopal tilings of a zonotope and one-element liftings of the corresponding realizable oriented matroid $\Mcal$. Unlike zonotopal tilings, the notion of one-element liftings easily generalizes to the non-realizable case. The goal of this section is to define and develop some basic properties of zonotopal tilings for arbitrary oriented matroids. Since these tilings turn out to be identical to the covectors of one-element liftings of $\Mcal$, we give only brief proofs, and view this section as the basement for the proof of Lemma~\ref{lemma:parallel_intro}.
 
\def\tile{\Face}
\def\tiling{\Tiling}
\def\sp{{ \operatorname{sp}}}

\begin{definition}
Any nonempty boolean interval $\tile\subset 2^E$ is called a \emph{tile}. Equivalently, given a signed vector $X\in\{+,-,0\}^E$, the corresponding \emph{tile} denoted $\tile_X\subset 2^E$ is defined by
 \[\tile_X:=\{S\subset E\mid X^+\subset S,\ X^-\cap S=\emptyset\}.\]
 
 The elements of $\tile$ are called its \emph{vertices}. The set $X^0$ is called the \emph{spanning set} of $\tile$ and denoted $\sp(\tile)$, and $X$ is called the \emph{signed vector} of $\tile$. The \emph{dimension} of $\tile$ is 
 \[\dim(\tile):=\rk(\sp(\tile)),\]
 and given an oriented matroid $\Mcal$, $\tile$ is called \emph{top-dimensional} if $\dim(\tile)=\rk(\Mcal)$. 
\end{definition}
For an element $e\in E$ and a tile $\tile\subset 2^E$, define another tile $\tile-e\subset 2^{E-e}$ by
\[\tile-e:=\{S-e\mid S\in\tile\}.\]
% For each tile $\tile\subset 2^{E-e}$, there are three possible tiles $\tile'$ satisfying $\tile'-e=\tile$. The first of them is $\tile$ itself viewed as a subset of $2^E$, the second one is 
% \[\tile e:=\{Se\mid S\in \tile\},\]
% and the third one is $\tile\times e:=\tile\cup\tile e$. For a tile $\tile\subset 2^E$, define 
% \[\reorient{e}{\tile}=\{\reorient{e}{S}\mid S\in\tile\}.\]

We define fine zonotopal tilings for an oriented matroid $\Mcal$ as collections of tiles coming from a maximal \emph{by size} $\Mcal$-separated collection. 
\begin{definition}
 Given a maximal \emph{by size} $\Mcal$-separated collection $\WS$, the corresponding \emph{fine zonotopal tiling} $\tiling(\WS)$ is defined to be the collection of all tiles $\tile$ all of whose vertices belong to $\WS$:
 \[\tiling(\WS)=\{\tile\subset\WS\mid \tile\text{ is a tile}\}.\]
\end{definition}

The following is a simple extension of~\cite[Proposition~2.2.11]{Book} to the non-realizable case.
\begin{proposition}\label{prop:tilings_covectors}
 Let $\WS$ be a maximal \emph{by size} $\Mcal$-separated collection, and let $\Mcaltilde$ be the one-element lifting of $\Mcal$ in general position corresponding to $\WS$ via Theorem~\ref{thm:max_size_matroid}. Then the map $\tile_X\mapsto X$ is a bijection from $\tiling(\WS)$ to $\Lcal_g^-(\Mcaltilde)$ defined by 
 \[\Lcal_g^-(\Mcaltilde):=\{X=(X^+,X^-)\in \{+,-,0\}^E\mid (X^+,X^-g)\in\Lcal(\Mcaltilde)\}.\]
\end{proposition}
\begin{proof}

Fix a tile $\tile\in\tiling(\WS)$ and let $X\in\{+,-,0\}^E$ be such that $\tile=\tile_X$. We need to show that $(X^+,X^-g)\in\Lcal(\Mcaltilde)$. Suppose that this is not the case, and therefore there exists a circuit $Y$ of $\Mcaltilde$ not orthogonal to $(X^+,X^-g)$. 

Proposition~\ref{prop:circuits_coloc} suggests considering two cases for $Y$. Suppose that $Y=(C,\sigma(C))$ for some circuit $C=(C^+,C^-)\in\Ccal(\Mcal)$. In this case, we may assume that $\sigma(C)=-$ and thus $C^+\subset X^+\cup X^0$, $C^-\subset X^-\cup X^0$. It is easy to see that there is a set $S\in\tile$ such that $C^+\subset S$ and $C^-\cap S=\emptyset$. But then $\sigma(C)=+$, because it cannot happen that $\Cu\subset X^0$ by property~(\ref{item:tiles_sp_independent}) in Theorem~\ref{thm:tiles} below since $X^0$ is necessarily an independent set of $\Mcal$. This gives a contradiction. %Thus $Y=(C,\sigma(C))=(C^+g,C^-)\in\Ccal(\Mcaltilde)$ is orthogonal to $(X^+,X^-g)$, because they disagree on $g$ and agree at least on some element from $\Xu\cap\Cu$, a contradiction. 

The second case is $Y=Y_1\circ Y_2$ where $Y_1$ and $Y_2$ satisfy the assumptions of Proposition~\ref{prop:circuits_coloc}. We get that $Y_1^+,Y_2^+\subset X^+\cup X^0, Y_1^-,Y_2^-\subset X^-\cup X^0$. Similarly to the above, this implies $\sigma(Y_1)=\sigma(Y_2)=+$ which contradicts the assumption $\sigma(Y_1)\neq \sigma(Y_2)$ from Proposition~\ref{prop:circuits_coloc}. Thus we have shown that if $\tile_X\in \tiling(\WS)$ then $(X^+,X^-g)\in\Lcal_g^-(\Mcaltilde)$. 

Conversely, suppose $(X^+,X^-g)\in\Lcal(\Mcaltilde)$ for some $X=(X^+,X^-)\in\{+,-,0\}^E$. We claim that for every $J\subset X^0$, $X^+\cup J\in\WS$. Suppose that this is not the case, and thus for some $J$, $X^+\cup J\not\in\WS$. Since $\WS$ is maximal \emph{by size}, there must be a circuit $C=(C^+,C^-)\in\Ccal(\Mcal)$ with $\sigma(C)=-$ and $C^+\subset X^+\cup J$, $C^-\subset X^-\cup X^0$. But then $(C,\sigma(C))=(C^+,C^-g)\in\Ccal(\Mcaltilde)$ is not orthogonal to $(X^+,X^-g)\in\Lcal(\Mcaltilde)$. 
\end{proof}

\def\tileGraph{\mathcal{G}}

It is easy to see that for realizable $\Mcal$, our notion of $\tiling(\WS)$ coincides with the notion of a fine zonotopal tiling from Definition~\ref{dfn:tilings_realizable}. 

Define the \emph{tile graph} $\tileGraph(\WS)$ to be an undirected graph with vertex set $\WS$ and two vertices $S$ and $T$ connected by an edge in $\tileGraph(\WS)$ if and only if $S=\reorient{e}{T}$ for some $e\in E$. 
% 
% \begin{definition}
%  Let $\Mcal$ be an oriented matroid. To every maximal \emph{by size} $\Mcal$-separated collection $\WS$ we associate a certain regular CW-complex $\tileComplex(\WS)$ which we identify with a subcomplex of the cube $[0,1]^E$. First, each subset of $E$ can be identified with a vertex of the cube. Then each tile $\tile\in\tiling(\WS)$ naturally corresponds to a face $F_\tile$ of $[0,1]^E$ by just setting 
%  \[F_\tile=\conv(\{S\mid S\in\tile\}),\] 
%  so $\tileComplex(\WS)$ consists of all faces $F_\tile$ where $\tile\in\tiling(\WS)$. By definition, its \emph{$d$-skeleton} $\tileComplex^{(d)}(\WS)$ consists of all faces of dimension at most $d$. In particular, $\tileComplex^{(1)}(\WS)$ is a graph whose vertices correspond to sets in $\WS$ and two sets $S,T\in\WS$ are connected by an edge if and only if their symmetric difference consists of one element: $|T-S|=1$ and $S\subset T$ or vice versa.
% \end{definition}
% It is easy to see that when $\Mcal$ is obtained from a vector configuration $\VC$, $\tileComplex(\WS)$ defined above is just the CW complex that comes from the zonotopal subdivision of $\Zon_\VC$ that corresponds to $\WS$.

The following theorem summarizes some basic properties of fine zonotopal tilings for oriented matroids:

\begin{theorem}\label{thm:tiles}
 Let $\Mcal$ be an oriented matroid and consider any maximal \emph{by size} $\Mcal$-separated collection $\WS$.
 \begin{enumerate}[{\normalfont (1)}]
  \item\label{item:tiles_sp_independent} For any tile $\tile\in\tiling(\WS)$, the set $\sp(\tile)$ is independent: \[\sp(\tile)\in\Ind(\Mcal).\]
  \item\label{item:tiles_independent_surjection} For any $I\in\Ind(\Mcal)$, there is a tile $\tile\in\tiling(\WS)$ with $\sp(\tile)=I$.
  \item\label{item:tiles_bases_bijection} For any basis $B\in\Bcal(\Mcal)$, there is a unique tile $\tile\in\tiling(\WS)$ with $\sp(\tile)=B$. Thus the map $\tile\mapsto\sp(\tile)$ provides a bijection between top-dimensional tiles of $\tiling(\WS)$ and bases of $\Mcal$.
  \item\label{item:tiles_purity} Every tile in $\tiling(\WS)$ is contained in a top-dimensional tile.
%   \item\label{item:tiles_pseudomanifold} Every codimension $1$ tile in $\tiling(\WS)$ is contained in either $1$ or $2$ top-dimensional tiles.
%   \item\label{item:tiles_strongly_conn} TDO STRONGLY CONNECTED AND PSEUDOMANIFOLD
  \item\label{item:tiles_graph_distance} For any two sets $S,T\in\WS$, the graph distance between the corresponding vertices in $\tileGraph(\WS)$ equals the size of their symmetric difference, that is, $|S-T|+|T-S|$. Thus the graph $\tileGraph(\WS)$ is embedded isometrically into the cube $[0,1]^E=\cube_{|E|}$. 
  \item\label{item:tiles_contract} If $e\in E$ is not a loop then 
  \[\tiling(\WS/e)=\{\tile-e\mid \tile\in\tiling(\WS): e\in\sp(\tile)\}.\]
  \item\label{item:tiles_remove} If $e\in E$ is not a coloop then 
  \[\tiling(\WS-e)=\{\tile-e\mid \tile\in\tiling(\WS)\}.\]
 \end{enumerate}
\end{theorem}
\begin{proof}
% We first prove parts~(\ref{item:tiles_sp_independent})-(\ref{item:tiles_bases_bijection}), then parts~(\ref{item:tiles_graph_distance}) and~(\ref{item:tiles_contract}). After that, we show parts~(\ref{item:tiles_purity}) and~(\ref{item:tiles_remove}).
We deduce most of the properties from their known counterparts in the language of covectors of $\Mcaltilde$.

%1==================================================================
\emph{Proof of} (\ref{item:tiles_sp_independent}): This is obvious: after some reorientation we may assume that $\tile$ is just the family of all subsets of $\sp(\tile)$, and if $\sp(\tile)$ contains a circuit $C=(C^+,C^-)$ then $\WS$ is not $\Mcal$-separated because both $C^+$ and $C^-$ belong to it.
%2==================================================================
 
 \emph{Proof of} (\ref{item:tiles_independent_surjection}): This follows from part~(\ref{item:tiles_bases_bijection}) below.
%3==================================================================
 
 \emph{Proof of} (\ref{item:tiles_bases_bijection}): 
 For every basis $B$ of $\Mcal$, there is a unique pair of opposite cocircuits $\pm C=\pm(C^+,C^-)\in\Ccal^*(\Mcaltilde)$ such that $C^0=B$. Here $\Mcaltilde$ is the one-element lifting in general position corresponding to $\WS$. We are done by Proposition~\ref{prop:tilings_covectors}.

 \emph{Proof of} (\ref{item:tiles_purity}): This follows from the fact that the order complex of the \emph{big face lattice} $\Fcal_{big}(\Mcal)^{op}$ is isomorphic to the face lattice of a PL regular cell decomposition of the $(r-1)$-sphere, see~\cite[Corollary~4.3.4]{Book}. This implies that this complex is pure, and it is easy to see that the subcomplex consisting of $\Lcal^-_g(\Mcaltilde)$ is pure as well. 
%7==================================================================
 
 \emph{Proof of} (\ref{item:tiles_graph_distance}): By Proposition~\ref{prop:tilings_covectors}, $\tileGraph(\WS)$ is just a subgraph (in fact, a \emph{halfspace}) of the \emph{tope graph} of $\Mcaltilde$, see \cite[Definition~4.2.1]{Book}. Thus the result follows immediately from~\cite[Proposition~4.2.3]{Book}. Note that since we are working with colocalizations \emph{in general position} only, the corresponding one-element lifting $\Mcaltilde$ will be a \emph{simple} oriented matroid and thus \cite[Section~4.2]{Book} applies directly. 
%8==================================================================
 
 \emph{Proof of} (\ref{item:tiles_contract}): This is completely obvious from the definitions: $\tiling(\WS)$ contains a tile $\tile$ with $e\in\sp(\tile)$ if and only if $\tiling(\WS/e)$ contains a tile $\tile-e$.

 \emph{Proof of} (\ref{item:tiles_remove}): Follows from Proposition~\ref{prop:tilings_covectors} and the deletion formula for covectors~\cite[Proposition~3.7.11]{Book}.
 
%  First, clearly, the right hand side is a subset of the left hand side, so we only need to show that for each tile $\tile$ in $\tiling(\WS-e)$ there is a tile $\tile'\in\tiling(\WS)$ with $\tile=\tile'-e$. 
%  
%  Each top-dimensional tile $\tile'\in\tiling(\WS)$ such that $e\not\in\sp(\tile')$ gives rise to a top-dimensional tile $\tile'-e$ in $\tiling(\WS-e)$, and 
%  $\sp(\tile')=\sp(\tile'-e)$ is a basis of $\Mcal-e$ as well. By Proposition~\ref{prop:Tutte}, the number of bases of $\Mcal-e$ equal the number of bases of $\Mcal$ that do not contain $e$. Now we can just apply~(\ref{item:tiles_bases_bijection}) to see that $\tiling(\WS-e)$ has no other top-dimensional tiles. If $\tile$ is not a top-dimensional tile in $\tiling(\WS-e)$ then by~(\ref{item:tiles_purity}) it is contained in a top-dimensional tile $\tile''\in\tiling(\WS-e)$ which as we have already seen comes from a top-dimensional tile $\tile'\in\tiling(\WS)$. Thus $\tile'-e$ contains $\tile$ and we are done with the proof of the theorem.
\end{proof}

\section{Pure oriented matroids}\label{sect:pure_OM}

The rest of the paper will be concerned with the question of which oriented matroids $\Mcal$ are \emph{pure}, that is, have the property that every maximal \emph{by inclusion} $\Mcal$-separated collection is also maximal \emph{by size}.
% 
% \begin{lemma}\label{lemma:isomorphism}
%  If two oriented matroids $\Mcal_1$ and $\Mcal_2$ on ground sets $E_1$ and $E_2$ are isomorphic then $\Mcal_1$ is pure if and only if $\Mcal_2$ is pure.
% \end{lemma}
% \begin{proof}
%  This is obvious from the definition of an isomorphism which we now recall: $\Mcal_1$ and $\Mcal_2$ are \emph{isomorphic} if there is a bijection $\phi:E_1\to E_2$ and a subset $A\subset E_2$ such that the oriented matroids $\phi(\Mcal_1)$ and $\reorient{A}{\Mcal_2}$ are equal. Now, for any sets $B,C\subset E_1$, it is clear that $B$ and $C$ are $\Mcal_1$-separated if and only if $\reorient{A}{\phi(B)}$ and $\reorient{A}{\phi(C)}$ are $\Mcal_2$-separated. Since $B\mapsto \reorient{A}{\phi(B)}$ is a bijection between $2^{E_1}$ and $2^{E_2}$, the result follows.
% \end{proof}

Let us start by proving a slight strengthening of Lemma~\ref{lemma:adding_removing_zeros}:
\begin{lemma}\label{lemma:loop_coloop}
 Let $\Mcal$ be an oriented matroid and suppose that an element $e$ belongs to the ground set $E$ of $\Mcal$. 
 \begin{enumerate}[\normalfont (1)]
  \item If $e$ is a loop then $\Mcal$ is pure if and only if $\Mcal-e$ is pure.
  \item If $e$ is a coloop then $\Mcal$ is pure if and only if $\Mcal/e$ is pure.
 \end{enumerate}
\end{lemma}
\begin{proof}
 Let $\WS\subset 2^E$ be a maximal \emph{by inclusion} $\Mcal$-separated collection. If $e$ is a loop then $(\{e\},\emptyset)\in\Ccal(\Mcal)$ so either all elements of $\WS$ contain $e$ or all elements of $\WS$ do not contain $e$. And then $\WS$ is a maximal \emph{by inclusion} $\Mcal$-separated collection if and only if $\WS-e$ is a maximal \emph{by inclusion} $(\Mcal-e)$-separated collection. Thus the first claim follows.
 
 Assume now that $e$ is a coloop of $\Mcal$, so none of the circuits of $\Mcal$ contain $e$. In this case, for every $S\subset E-e$, we have $S\in\WS$ if and only if $Se\in\WS$. And then $\WS$ is a maximal \emph{by inclusion} $\Mcal$-separated collection if and only if $\WS/e$ is a maximal \emph{by inclusion} $(\Mcal/e)$-separated collection. This finishes the proof of the lemma.
\end{proof}

In order to reduce to simple oriented matroids, we need to exclude parallel elements as well. We reformulate Lemma~\ref{lemma:parallel_intro} as follows.
\begin{lemma}\label{lemma:parallel}
 Let $\Mcal$ be an oriented matroid. Suppose that $e,f\in E$ are parallel to each other. Then 
 \[\text{$\Mcal$ is pure }\Leftrightarrow \text{ $\Mcal-e$ is pure }\Leftrightarrow\text{ $\Mcal-f$ is pure}.\]
\end{lemma}
We are not going to use this result in what follows and postpone its proof until Section~\ref{subsect:parallel_proof}.

It turns out that the property of being pure is preserved under taking minors:
\begin{proposition}\label{prop:if_pure_then_minors_are_pure}
 Suppose an oriented matroid $\Mcal$ is pure, and let $e\in E$ belong to its ground set. Then 
 \begin{itemize}
 	\item if $e$ is not a coloop then $\Mcal-e$ is pure;
 	\item if $e$ is not a loop then $\Mcal/e$ is pure.
 \end{itemize}
\end{proposition}
\begin{proof}
  Suppose that $e$ is not a coloop and let $\WS\subset 2^{E-e}$ be any maximal \emph{by inclusion} $(\Mcal-e)$-separated collection. Let us view $\WS$ as a collection of subsets of $E$ rather than $E-e$. Then $\WS$ is still clearly an $\Mcal$-separated collection. Let $\WS'$ be any maximal by inclusion $\Mcal$-separated collection that contains $\WS$. Since $\Mcal$ is pure, $\WS'$ is maximal \emph{by size} as well and thus by Theorem~\ref{thm:max_size_matroid} has size $|\Ind(\Mcal)|$. But then by~\eqref{eq:deletion_contraction_both_parts}, $\WS'-e$ must have size $|\Ind(\Mcal-e)|$. On the other hand, it contains $\WS$ so $\WS$ is contained in a maximal by size $(\Mcal-e)$-separated collection $\WS'-e$, so $\Mcal-e$ is pure.

 %corresponds to a colocalization of $\Mcal$ in general position. In particular, $\WS'$ is complete, and thus $\sigma_{\WS'}$ restricts to a colocalization $\sigma_{\WS'-e}$ of $\Mcal-e$. Therefore we have found a colocalization $\sigma_{\WS'-e}$ of $\Mcal-e$ in general position which clearly satisfies $\sigma_{\WS'-e}\geq \sigma_\WS$, so by Lemma~\ref{lemma:extension} $\WS$ is contained in a maximal by size $\Mcal-e$-separated collection, so $\Mcal-e$ is pure. 
 
 Now assume that $e$ is not a loop and let $\WS\subset 2^{E-e}$ be any maximal \emph{by inclusion} $(\Mcal/e)$-separated collection. Consider a collection $\WS'\subset 2^E$ defined by
 \[\WS':=\{S\mid S\in\WS\}\cup\{Se\mid S\in\WS\}.\]
 By the definition of $\Mcal/e$, it is clear that $\WS'$ is an $\Mcal$-separated collection: if $C\in\Ccal$ is a circuit of $\Mcal$ then $C-e$ contains a circuit of $\Mcal/e$ and thus if two sets $S,T$ are not $\Mcal$-separated then $S-e$ and $T-e$ are not $(\Mcal/e)$-separated. But now we can again extend $\WS'$ to a maximal \emph{by inclusion} $\Mcal$-separated collection $\WS''$ which is therefore maximal \emph{by size}. Again, by~\eqref{eq:deletion_contraction_both_parts}, the collection $\WS''/e$ has size $|\Ind(\Mcal/e)|$ and since it contains $\WS$, we get that $\Mcal/e$ is pure.
\end{proof}

\def\gin{{\ni g}}
\def\gnot{{\not\ni g}}
\def\ein{{\ni e}}
\def\enot{{\not\ni e}}
\def\fin{{\ni f}}
\def\fnot{{\not\ni f}}

\subsection{Proof of Lemma~\ref{lemma:parallel}}\label{subsect:parallel_proof}

% The next proposition shows that there is always a \emph{local} obstacle to add an extra edge to a vertex in $\tileComplex(\WS)$.

\begin{proposition}\label{prop:local}
 Let $\WS$ be a maximal \emph{by size} $\Mcal$-separated collection, and let $S\in\WS$ and $e\in E$ be such that $\reorient{e}{S}\not\in\WS$. Then there exists a unique tile $\tile\in\tiling(\WS)$ such that 
 \begin{enumerate}[(a)]
  \item\label{item:tile1} $S\in\tile\subset\WS$, and
  \item\label{item:tile2} $(Ie\cap Se,I-S)$ is a circuit of $\Mcal$, where  $I:=\sp(\tile)$.
 \end{enumerate}
 In particular, the set $\reorient{I}{S}\in\WS$ is not $\Mcal$-separated from $Se$.
\end{proposition}
\begin{proof}
 After a suitable reorientation we may assume that $S=\emptyset$. In this case, we need to show that there exists a unique independent set $I$ such that all of its subsets belong to $\WS$ and $(I,\{e\})$ is a circuit of $\Mcal$.
 
 First we show uniqueness. This is clear: if $(I,\{e\})$ and $(J,\{e\})$ are both circuits of $\Mcal$ then by the weak elimination axiom \hyperref[item:C3]{(C3)}, there is a circuit $(I',J')$ of $\Mcal$ with $I'\subset I$ and $J'\subset J$. But this is impossible because both $I'$ and $J'$ belong to $\WS$. We have shown that if such $I$ exists, it is unique.
 
 Now we show existence by induction on $|E|$. If $|E|\leq 1$ then the statement holds trivially. Let $k\in E$ be any element such that $\{k\}\in\WS$ (so $k$ is not a loop). Such $k$ exists by Theorem~\ref{thm:tiles}, part~(\ref{item:tiles_purity}). We know that $\{e\}\not\in\WS$ and thus $\WS/k$ does not contain $\{e\}$ but it clearly contains $\emptyset$. By the induction hypothesis, there is a tile $\tile'\in\tiling(\WS/k)$ that consists of all subsets of some set $I'\in\Ind(\Mcal/k)$ such that $(I',\{e\})$ is a circuit of $\Mcal/k$. Therefore either one of $(I',\{e\})$, $(I'k,\{e\})$, $(I',\{e,k\})$ is a circuit of $\Mcal$. By Theorem~\ref{thm:tiles}, part~(\ref{item:tiles_contract}), $\WS$ contains all subsets of $I'k$. And thus if $(I',\{e\})$ or $(I'k,\{e\})$ is a circuit of $\Mcal$, we are done. The only case left is when $(I',\{e,k\})$ is a circuit of $\Mcal$.
 
 Consider now $\WS-k$ instead. By the induction hypothesis, there is a tile $\tile''\in\tiling(\WS-k)$ that consists of all subsets of some set $I''\in\Ind(\Mcal-k)$ such that $(I'',\{e\})$ is a circuit of $\Mcal-k$. But then $(I'',\{e\})$ is a circuit of $\Mcal$ as well. By Theorem~\ref{thm:tiles}, part~(\ref{item:tiles_remove}), there is a tile $\tile\in\tiling(\WS)$ such that $\tile-k=\tile''$. If $\emptyset\in\tile$ then we are done. Otherwise, all subsets in $\tile$ have to contain $k$, and so $\sp(\tile)=I''$. 
 
 So far we have the following information:
 \begin{itemize}
  \item $(I',\{e,k\})$ is a circuit of $\Mcal$;
  \item $(I'',\{e\})$ is a circuit of $\Mcal$;
  \item all subsets of $I'$ belong to $\WS$;
  \item all subsets of $I''k$ that contain $k$ belong to $\WS$.
 \end{itemize}
 These four pieces lead to a contradiction: apply the circuit elimination axiom to the two circuits to get a circuit $C$ of $\Mcal$ with $C^+\subset I'$ and $C^-\subset I''k$, and then we have $C^+\in\WS$ and  $C^-\cup k\in\WS$, but these two sets are not $\Mcal$-separated. This contradicts the assertion that all subsets in $\tile$ contain $k$ and thus we are done with the proof of the proposition.
\end{proof}

\begin{lemma}\label{lemma:gnot}
 Let $\Mcal$ be an oriented matroid on the ground set $Eg$ such that $\Mcal-g$ is pure, and let $\WS$ be a maximal \emph{by size} $\Mcal$-separated collection. Define 
 \[\WS^\gin:=\{T\in\WS:\ g\in T\};\quad \WS^\gnot:=\{T\in\WS:\ g\not\in T\}.\]
 Let $S\subset E$ be any subset.
\begin{itemize}
 \item Assume $S\not\in\WS^\gnot$. Then $S$ is $\Mcal$-separated from $\WS^\gnot$ if and only if $Sg$ is;
 \item Assume $Sg\not\in\WS^\gin$. Then $S$ is $\Mcal$-separated from $\WS^\gin$ if and only if $Sg$ is.
\end{itemize}
\end{lemma}
\begin{proof}
 Since replacing each set by its complement preserves the notion of $\Mcal$-separation, we only need to show the first claim. Moreover, if $Sg$ is $\Mcal$-separated from $\WS^\gnot$ then obviously the same is true for $S$, because if a circuit is oriented differently by $S$ and $\WS^\gnot$ then it is also oriented differently by $Sg$ and $\WS^\gnot$. So suppose that $S$ is $\Mcal$-separated from $\WS^\gnot$, but $Sg$ is not, and thus there is a  signed set $X=(X^+,X^-)$ such that the circuit $C=(X^+g,X^-)$ of $\Mcal$ is oriented negatively by $\WS^\gnot$ and positively by $Sg$. Let $T\in \WS^\gnot$ be any set that orients $C$ negatively. We have 
 \[X^-\subset T-Sg,\quad X^+g\subset Sg-T.\]
 By Theorem~\ref{thm:tiles}, part~(\ref{item:tiles_independent_surjection}), there exists a set $R\in\WS$ that orients $X$ positively, because $\Xu$ is an independent set of $\Mcal$. If $g\in R$ then $\WS$ would not be $\Mcal$-separated, because $R$ would orient $C$ positively, and we know that $T\in\WS$ orients $C$ negatively. Thus $g\not\in R$ and by definition, $R\in\WS^\gnot$. 
 
 Since $S$ is $\Mcal$-separated from $\WS^\gnot$, it is also $(\Mcal-g)$-separated from $\WS^\gnot$. And because of the assumption that $\Mcal-g$ is pure, $\WS^\gnot\cup\{S\}$ is contained in a maximal \emph{by size} $(\Mcal-g)$-separated collection, which we denote $\WS_0$. By Theorem~\ref{thm:tiles}, part~(\ref{item:tiles_graph_distance}), the subgraph of $\tileGraph(\WS_0)$ induced on all vertices that orient $X$ positively is connected. We want to show that the path 
 \[R=R_0,R_1,\dots,R_t=S\]
 in this subgraph that connects $R$ and $S$ passes through the \emph{boundary} of $\WS^\gnot$, that is, contains at least one set from $\WS/g\subset\WS^\gnot$. We know that $R_0=R\not\in\WS/g$ because otherwise $Rg$ would belong to $\WS$ so it would orient $C$ positively. We claim that for every set $U$ from $\WS^\gnot$ that does not belong $\WS/g$, the set of edges adjacent to $U$ in $\tileGraph(\WS_0)$ is the same as the set of edges adjacent to $U$ in $\tileGraph(\WS)$. If we manage to do so then the only way to get connected to $S\not\in\WS^\gnot$ would be to pass through the boundary $\WS/g$ of $\WS^\gnot$. And the way to show our claim is to consider the collection $\WS-g$. By Theorem~\ref{thm:tiles}, part~(\ref{item:tiles_remove}), there are no tiles containing $U$ in $\Tiling(\WS-g)$ other than the ones from $\Tiling(\WS)$. And then Proposition~\ref{prop:local} applied to $\WS-g$ and $U$ immediately shows that there can be no other edges from $U$ in any tiling that contains all tiles adjacent to $U$, in particular, in $\Tiling(\WS_0)$. Thus there will be some set $R_i$ that belongs to the boundary of $\WS^\gnot$, that is, to $\WS/g$, and thus $R_ig$ belongs to $\WS$ but it orients $C$ positively while we started with the set $T\in\WS$ that orients $C$ negatively. We are done with the proof of the lemma and we are now ready to apply this technical result to prove Lemma~\ref{lemma:parallel}. 
\end{proof}

\begin{proof}[Proof of Lemma~\ref{lemma:parallel}.]

 If $\Mcal$ is pure then by Proposition~\ref{prop:if_pure_then_minors_are_pure}, $\Mcal-f$ is pure.
 
 Suppose now $\Mcal-f\cong\Mcal-e$ is pure. Since $e$ and $f$ are parallel, $(\{e\},\{f\})\in\Ccal(\Mcal)$ and no other circuit contains both $e$ and $f$ in its support. Moreover, for any circuit $(C^+e,C^-)$, $(C^+f,C^-)$ is also a circuit and vice versa. Let $Eef$ be the ground set of $\Mcal$.
 
 Let $\WS\subset 2^{Eef}$ be a maximal \emph{by inclusion} $\Mcal$-separated collection. Then for any two sets $S,T\in\WS$, it is not the case that $e\in S-T, f\in T-S$. Thus without loss of generality we may assume that for all $S\in\WS$, if $f\in S$ then $e\in S$. 
 
%  We need to show that $\WS$ orients every circuit of $\Mcal$. Let $C\in\Ccal(\Mcal)$ be any circuit and assume first that $f\not\in\Cu$. Then $C\in\Ccal(\Mcal-f)$ as well. 
%  
%  Consider the collection $\WS-e$, which is $\Mcal-e$-separated. 

We would like to show that $\WS$ is maximal \emph{by size}, and we are going to do so by considering collections $\WS-e$ and $\WS-f$. Since the oriented matroids $\Mcal-e$ and $\Mcal-f$ are isomorphic, we introduce another oriented matroid $\Mcal'$ on the ground set $Eg$ that is isomorphic to both of them. Define the maps $\phi:Ee\to Eg$ and $\psi: Ef\to Eg$ that send $e$ and $f$ to $g$ respectively and restrict to the identity map on $E$. Define collections $\WS_0,\WS_1,\WS_2\subset 2^E$ as follows:
\[\WS_0=\{S\in\WS\mid e,f\not\in S\};\quad \WS_1=\{S-e\mid S\in\WS: e\in S,\ f\not\in S\};\]
\[\WS_2=\{S-\{e,f\}\mid S\in\WS: e,f\in S\}.\]
Then
\[\WS=\WS_0\sqcup\WS_1e\sqcup\WS_2ef,\]
where by definition $\WS_1e:=\{Se\mid S\in\WS_1\}$ and $\WS_2ef:=\{Sef\mid S\in\WS_2\}$.

\def\mine{{\setminus e}}
\def\minf{{\setminus f}}

Denote $\WS^\minf:=\phi(\WS-f)$ and $\WS^\mine:=\psi(\WS-e)$. Then 
\[\WS^\minf=\WS_0\cup \WS_1g\cup \WS_2g;\quad \WS^\mine=\WS_0\cup \WS_1\cup \WS_2g.\]

Our temporary goal is to prove that both collections $\WS^\mine$ and $\WS^\minf$ are maximal \emph{by inclusion} $\Mcal'$-separated collections. 

%If we manage to do that, then by purity of $\Mcal'$, both of them will be maximal by size. In this case by Lemma~\ref{lemma:extension}, part (\ref{item:max_by_size_complete}) both of them will be complete. 

 Suppose first that $\WS^\minf=\WS_0\cup \WS_1g\cup \WS_2g$ is not maximal \emph{by inclusion}, and consider a set $S\not\in\WS^\minf$ which is $\Mcal'$-separated from $\WS^\minf$. If $g\not\in S$ then we claim that $S$ is in fact $\Mcal$-separated from $\WS$. Indeed, clearly $S$ is $\Mcal$-separated from all subsets in $\WS_0$ and $\WS_1e$. Thus we only need to show that $S$ is $\Mcal$-separated from all subsets in $\WS_2ef$. Let $Tef\in\WS_2ef$ be such a subset, so $T\in\WS_2$. Since $Tg\in\WS_2g$ is $\Mcal'$-separated from $S$, we see that both $Te$ and $Tf$ are $\Mcal$-separated from $S$. But then $Tef$ must be $\Mcal$-separated from $S$ as well. We have shown that if $S\not\in\WS^\minf$ is $\Mcal'$-separated from $\WS^\minf$ then $g\in S$. 

% we know that there is a circuit $C=(C^+f,C^-)\in\Ccal(\Mcal)$ with $C^+f\subset (Tef-S)$ and $C^-\subset (S-Tef)$. But then $(C^+e,C^-)$ is also a circuit of $\Mcal$ as well as of $\Mcal-f$, and clearly $S$ and $Te$ orient it in the opposite directions. We get a contradiction since $Te\in\WS$ and $Tg\in\WS^\minf$, and thus if $e\not\in S$ then $S$ is $\Mcal$-separated from $\WS$. But $\WS$ was maximal \emph{by inclusion} and therefore it has to contain $S$, in which case $S\in\WS^\minf$, and it contradicts our initial assumption. The conclusion is, if $S\not\in\WS^\minf$ is $\Mcal'$-separated from $\WS^\minf$ then $g\in S$. 

Let $\WS'\supset\WS^\minf$ be a maximal \emph{by inclusion} $\Mcal'$-separated collection of subsets that contains $\WS^\minf$. Since $\Mcal'$ is assumed to be pure, it means that $\WS'$ is complete and maximal \emph{by size}. By the previous observation, $\WS'$ differs from $\WS^\minf$ only in subsets that contain $g$. In particular, $(\WS')^\gnot=(\WS^\minf)^\gnot=\WS_0$. Using Lemma~\ref{lemma:gnot}, we get the following:

\begin{claim} 
Assume $S\not\in\WS_0$. Then $S$ is $\Mcal'$-separated from $\WS_0$ if and only if $Sg$ is.
\end{claim}

A completely analogous argument applied to $\WS^\mine$ shows the following:

\begin{claim} 
Assume $S\not\in\WS_2$. Then $S$ is $\Mcal'$-separated from $\WS_2g$ if and only if $Sg$ is.
\end{claim}

Let us now return to the collections $\WS'\supset\WS^\minf$. We know that $\WS'$ may contain some extra subsets, but all of them have to contain $g$. Let $Sg$ be such a subset, so $Sg\not\in\WS^\minf$ but $Sg$ is $\Mcal'$-separated from $\WS^\minf$. Recall that $\WS^\minf=\WS_0\cup \WS_1g\cup \WS_2g$. Therefore $S\not\in \WS_2$ and $Sg$ is $\Mcal'$-separated from $\WS_2g$. By the second claim, $S$ is $\Mcal'$-separated from $\WS_2g$. Also, since $Sg$ is $\Mcal'$-separated from $\WS_1g$, it must be true that $S$ is $\Mcal'$-separated from $\WS_1$. Finally, if $Sg$ is $\Mcal'$-separated from $\WS_0$ then of course so is $S$. To sum up, $S$ is $\Mcal'$-separated from 
$\WS_0\cup\WS_1\cup\WS_2g=\WS^\mine$. Then we claim that $Se$ is $\Mcal$-separated from our initial collection $\WS$. Indeed:
\begin{itemize}
 \item $Se$ is $\Mcal$-separated from $\WS_0$ because $Sg$ is $\Mcal'$-separated from $\WS_0\subset\WS^\minf$;
 \item $Se$ is $\Mcal$-separated from $\WS_1e$ because $S$ is $\Mcal'$-separated from $\WS_1\subset\WS^\mine$;
 \item $Se$ is $\Mcal$-separated from $\WS_2ef$ because $S$ is $\Mcal'$-separated from $\WS_2g\subset\WS^\mine$.
\end{itemize}
Thus $Se\in\WS$ which contradicts the fact that $Sg\not\in\WS^\minf$. It follows that $\WS^\minf$ is maximal \emph{by inclusion}. 
Since replacing each set in $\WS$ with its complement swaps the roles of $e$ and $f$, we find that $\WS^\mine$ is also maximal \emph{by inclusion}.
 More precisely, $\reorient{Eef}{\WS}$ is a maximal \emph{by inclusion} $\Mcal$-separated collection with the roles of $e$ and $f$ swapped. By what we have just shown, $(\reorient{Eef}{\WS})^\mine=\reorient{E}{\WS_0}g\sqcup \reorient{E}{\WS_1}g\sqcup \reorient{E}{\WS_2}$ is a maximal \emph{by inclusion} $\Mcal'$-separated collection. Applying the map $\reorient{Eg}{(\cdot)}$ to this collection, we obtain $\WS^\mine$.
% It remains to observe that $\reorient{Eg}{((\reorient{Eef}{\WS})^\mine)}=\WS^\mine$.

We now have the following situation. The collections $\WS^\minf$ and $\WS^\mine$ are maximal \emph{by inclusion} $\Mcal'$-separated, and since $\Mcal'$ is assumed to be pure, both of them are maximal \emph{by size} and complete (see Corollary~\ref{cor:complete}). We would like to show that $\WS$ is a complete $\Mcal$-separated collection. It is clear that it orients all circuits of $\Mcal$ except for possibly $C=(\{e\},\{f\})$. If $\WS$ orients $C$ then it is complete, so assume $\sigma_\WS(C)=0$. We are going to extend the map $\sigma_\WS$ to a map $\sigma$ by setting $\sigma(C)=+$. To check that $\sigma$ is still a colocalization, we need to see why $\sigma$ has Type III for all possible subsets $A\subset Eef$ of nullity $2$. If $A$ does not contain either $e$ or $f$ then we are done because in this case on $A$, we have $\sigma=\sigma_\WS$. If $A$ contains both $e$ and $f$ then $C$ is a cocircuit of the oriented matroid $(\Mcal\mid_A)^*$ of rank $2$. A simple case analysis shows that $A$ can be partitioned into three subsets $A=S\sqcup T\sqcup \{e,f\}$ so that the cocircuits of $(\Mcal\mid_A)^*$ written in the cyclic order are 
\[(\{e\},\{f\}),(Se,T),(Sf,T),(\{f\},\{e\}),(T,Se),(T,Sf),\]
or
\[(\{e\},\{f\}),(S,T),(\{f\},\{e\}),(T,S).\]
In the second case, the Type III assumption holds regardless of the value $\sigma((S,T))$. In the first case however, we want to show that the corresponding values of $\sigma$ are not equal to $(+,-,+,-,+,-)$. In other words, we need to see why if $\sigma((Se,T))=-$ then $\sigma((Sf,T))=-$. But this is clear because if $R\in \WS$ orients $(Se,T)$ negatively then $T\subset R$ and $Se\cap R=\emptyset$, so, in particular, $e\not\in R$ and thus $f\not\in R$, so $R$ orients the circuit $(Sf,T)$ negatively as well. This shows that $\sigma$ is a colocalization, and thus $\WS$ is contained in a complete collection by Lemma~\ref{lemma:extension}, part~(\ref{item:extension}). This complete collection has to be maximal \emph{by size} by Theorem~\ref{thm:max_size_matroid}. We are done with the proof of Lemmas~\ref{lemma:parallel_intro} and~\ref{lemma:parallel}. 
\end{proof}

\section{The graphical case} 
\label{sect:graph}

\subsection{Pure graphs}
In this section, we prove Theorem~\ref{thm:outerplanar} by investigating which undirected graphs are pure.
\def\Cyc{{ \operatorname{Cyc}}}
\def\Conn{{ \operatorname{Conn}}}
\newcommand\GOR[1]{{\vec G_{O_{#1}}}}
\def\GORO{\GOR{}}

%By Lemma~\ref{lemma:isomorphism}, the property of a matroid being pure does not change when we reorient some elements of a matroid. Let $G$ be an undirected graph without loops or parallel edges. 

Let $G$ be an undirected graph (possibly with loops or parallel edges). Fix some total orientation $O$ of $G$. We let $\GORO$ be the directed graph where each edge of $G$ is oriented according to $O$. We say that $G$ is \emph{pure} if the oriented matroid $\Mcal_{\GORO}$ is pure, and since for different orientations $O$ of $G$, the oriented matroids $\Mcal_{\GORO}$ differ by a reorientation, the property of $G$ being pure does not depend on the choice of $O$. In Section~\ref{sect:main}, we explained how to translate the notion of $\Mcal_{\GORO}$-separation to the language of total orientations of $G$, and we start now by proving that this is indeed the case. 

The ground set of $\Mcal_{\GORO}$ is precisely the set $E$ of edges of $G$. There is a simple bijection $\alpha$ between total orientations $O'$ of $G$ and subsets of $E$, namely, for a total orientation $O'$ of $G$, $\alpha(O')$ is the set of edges of $G$ where $O$ and $O'$ disagree. 

We let $\Cyc(G)$ denote the set of cycles of $G$, and by a \emph{cycle} we always mean a non-self-intersecting undirected cycle of $G$ viewed as a subset of edges of $G$. For each cycle $C$ of $G$, there are two orientations of $C$ that make it into a directed cycle which we denote $O_+(C)$ and $O_-(C)$. For example, if $G$ is planar then $O_+(C)$ and $O_-(C)$ can denote the clockwise and counterclockwise orientations of $C$. For general graphs $G$, there is no canonical way to decide which of the two orientations is positive and which is negative so we just fix some arbitrary choice of $O_+(C)$ and $O_-(C)$ for all elements $C\in\Cyc(G)$.

Recall the definition of $G$-separation from Section~\ref{sect:main}:
\begin{definition}
 We say that two total orientations $O_1$ and $O_2$ of $G$ are \emph{$G$-separated} if there does not exist a cycle $C$ of $G$ such that the restrictions of $O_1$ and $O_2$ on $G$ are $O_+(C)$ and $O_-(C)$  or vice versa.
\end{definition}

\begin{proposition}
 Two total orientations $O_1$ and $O_2$ of $G$ are $G$-separated if and only if the sets $\alpha(O_1)$ and $\alpha(O_2)$ are $\Mcal_{\GORO}$-separated.
\end{proposition}
\begin{proof}
 Note that the circuits of $\Mcal_\GORO$ correspond to the cycles of $G$. More precisely, let $C^+$ and $C^-$ be subsets of the edges of $C$ defined as follows: an edge $e$ of $C$ belongs to $C^\pm$ if and only if the reference orientation $O$ agrees with $O_\pm(C)$ on $e$. Then $(C^+,C^-)$ is a circuit of $\Mcal_\GORO$. Moreover, given a total orientation $O'$ of $G$, the set $\alpha(O')$ orients $(C^+,C^-)$ positively (resp., negatively) if and only if $O'$ agrees with $O_-(C)$ (resp., with $O_+(C)$) on $C$. This shows that $O_1$ and $O_2$ are $G$-separated if and only if there is no circuit of $\Mcal_\GORO$ that they orient in the opposite ways so the result follows.
\end{proof}

It also follows from the description of the circuits of $\Mcal_\GORO$ that the independent subsets of $\Mcal_\GORO$ correspond to \emph{forests} of $G$, that is, to subsets of the edges of $G$ without cycles.

\begin{proof}[Proof of Theorem~\ref{thm:outerplanar}]
 We first show that the graphs $K_4$ and $K_{2,3}$ are not pure. The graph $K_{2,3}$ contains $54$ forests but just $46$ acyclic total orientations. Thus every maximal \emph{by size} $K_{2,3}$-separated collection contains these $46$ acyclic orientations together with $8$ other orientations. Figure~\ref{fig:K_23} contains $6$ orientations of $K_{2,3}$ that are $K_{2,3}$-separated from each other but there is no other orientation $K_{2,3}$-separated from all of them that would not be acyclic. This shows that $K_{2,3}$ is not pure.

 Now let us concentrate on $K_4$. By Corollary~\ref{cor:complete}, in a pure oriented matroid $\Mcal$, any $\Mcal$-separated collection $\WS$ is contained in a complete $\Mcal$-separated collection. Equivalently, there is a colocalization $\sigma$ of $\Mcal$ in general position satisfying $\sigma\geq \sigma_\WS$. We will construct a collection $\WS$ for $K_4$  such that there does not exist a colocalization $\sigma$ satisfying $\sigma\geq\sigma_\WS$. 
 
 \def\coloc{{\gamma}}
 \def\coltree{{\tilde\coloc}}
Let us start with translating the notion of a colocalization to undirected outerplanar graphs. 
\begin{definition}\label{dfn:graph_coloc}
 Given an outerplanar graph $G$, a \emph{$G$-colocalization} $\coloc:\Cyc(G)\to\{+,-\}$ is a map that assigns an orientation to each cycle of $G$ so that an extra condition below is satisfied. Consider any cycle $C$ such that there is an edge $e$ of $G$ that is not in $C$ but connects two vertices of $C$. Then the union of $C$ and $e$ contains two more cycles which we denote $C_1$ and $C_2$. Let $O_+(C),O_+(C_1),$ and $O_+(C_2)$ be chosen so that they all agree on the edges of $C$. Then the condition on $\coloc$ is that either $\coloc(C)=\coloc(C_1)$ or $\coloc(C)=\coloc(C_2)$ (or both). 
\end{definition}
It is easy to see that for an outerplanar graph $G$, $\coloc$ is a $G$-colocalization if and only if the corresponding map  $\sigma:\Ccal(\Mcal_\GORO)\to \{+,-\}$ defined by $\sigma((C^+,C^-))=\coloc(C)$ is a colocalization of $\Mcal_\GORO$ in general position. Indeed, the non-trivial nullity $2$ subsets of the edges of $G$ correspond to connected subgraphs $G'$ of $G$ with $|E(G')|=|V(G')|+1$, where $E(G')$ and $V(G')$ are the sets of edges and vertices of $G'$, respectively. After removing leaves (which are coloops of the corresponding oriented matroid), we get that $G'$ is a union of three paths with the same start and the same end and non-overlapping interiors. If each path has at least two edges then $G$ is not outerplanar, so we may assume that one of the paths is a single edge $e$, and so $G'$ is a union of a cycle and $e$. And then the condition of $\sigma$ being of Type III is precisely the extra condition on $\coloc$ in Definition~\ref{dfn:graph_coloc}.

We are now ready to prove that the (outerplanar) graph $K_4$ is not pure. Consider the collection $\WS$ of four total orientations of $K_4$ in Figure~\ref{fig:K_4}.

We claim that $\WS$ is $K_4$-separated but does not belong to any maximal \emph{by size} $K_4$-separated collection. It is clear that $\WS$ is $K_4$-separated because for each cycle of $K_4$, there is at most one orientation in $\WS$ that orients it. 

Suppose that there exists a maximal \emph{by size} $K_4$-separated collection containing $\WS$. Thus there must be a $K_4$-colocalization $\coloc$ satisfying $\coloc\geq \sigma_\WS$. Let $C$ be the cycle with vertices $(1,2,3,4)$ listed in this order, and let $O_+(C)=(1\to2\to3\to4\to1)$ and $O_-(C)=(1\to4\to3\to2\to1)$ be its two possible orientations. There are two edges in $K_4$ that do not belong to $C$, namely, $(1,3)$ and $(2,4)$. Applying Definition~\ref{dfn:graph_coloc} to the union of $C$ and $(1,3)$ yields that $\coloc(C)$ has to be positive while applying Definition~\ref{dfn:graph_coloc} to the union of $C$ and $(2,4)$ yields that $\coloc(C)$ has to be negative. We get a contradiction and thus $K_4$ is not pure.

Recall from Section~\ref{sec:graph-orient-matr} that $G$ is assumed to be \emph{simple}, i.e., to have no loops or parallel edges. (Non-simple graphs may be easily treated using Lemmas~\ref{lemma:loop_coloop} and~\ref{lemma:parallel}.) By Theorem~\ref{thm:outerplanar_minors}, outerplanar graphs are exactly the graphs that do not contain $K_4$ and $K_{2,3}$ as minors. Therefore by Proposition~\ref{prop:if_pure_then_minors_are_pure}, if a graph $G$ is not outerplanar then it is not pure. It remains to show purity for outerplanar graphs. By Theorem~\ref{thm:outerplanar_minors} again, every such graph is a subgraph of a triangulation of an $n$-gon, so again by Proposition~\ref{prop:if_pure_then_minors_are_pure} it suffices to show purity for triangulations. 

\begin{definition}\label{dfn:tree}
Let $G$ be the $1$-skeleton of a triangulation of an $n$-gon. Construct a plane tree $T=T(G)$ as follows: put a vertex of $T$ inside each triangular face of $G$ and connect two of them by an edge in $T$ if and only if the corresponding triangular faces share a diagonal. Define $\Conn(T)$ to be the set of all connected induced subgraphs of $T$. There is an obvious bijective correspondence $\tau$ between  $\Conn(T)$ and $\Cyc(G)$, see Figure~\ref{fig:triangulation}. 
\end{definition}

Given two subgraphs $T_1,T_3\in\Conn(T)$, we call them \emph{vertex-disjoint} if their vertex sets $V(T_1),V(T_3)$ are disjoint, in which case their \emph{union} is the induced subgraph of $T$ on $V(T_1)\sqcup V(T_3)$.

\begin{definition}\label{dfn:bad_triple}
For two vertex-disjoint connected subgraphs $T_1,T_3\in\Conn(T)$ whose union $T_2$ is also connected, we say that $(T_1,T_2,T_3)$ is a \emph{Las Vergnas triple}. Given a map $\coltree:\Conn(T)\to\{+,-,0\}$ and a Las Vergnas triple $(T_1,T_2,T_3)$, we say that $(T_1,T_2,T_3)$ is a \emph{bad triple for $\coltree$} if \[\coltree(T_2)\neq 0,\quad \coltree(T_2)\neq\coltree(T_1),\quad \coltree(T_2)\neq\coltree(T_3).\]
\end{definition}

Let us explain the motivation for this terminology. For each cycle $C$ of $G$, we choose $O_+(C)$ and $O_-(C)$ so that $O_+(C)$ is oriented clockwise. Therefore a $G$-colocalization $\coloc$ becomes a function $\coltree:\Conn(T)\to\{+,-\}$ (defined by $\coltree=\coloc\circ\tau$) such that there are no bad triples for $\coltree$. This is true because every Las Vergnas triple comes precisely from a union of a cycle of $G$ and an edge connecting two of its vertices as in Definition~\ref{dfn:graph_coloc}. 

Now consider any $G$-separated collection $\WS$ of total orientations of $G$ and let $\coloc_\WS:\Cyc(G)\to\{+,-,0\}$ be the corresponding cycle signature. Let $\coltree_\WS=\coloc_\WS\circ\tau:\Conn(T)\to\{+,-,0\}$. We would like to show that there is a $G$-colocalization $\coloc\geq\coloc_\WS$, or equivalently, that there is a map $\coltree:\Conn(T)\to\{+,-\}$ that has no bad triples and satisfies $\coltree\geq\coltree_\WS$. Since $\WS$ is $G$-separated, there are no bad triples for $\coltree_\WS$. Indeed, suppose that $(T_1,T_2,T_3)$ is a bad triple for $\coltree_\WS$. Then $\coltree_\WS(T_2)\neq 0$ which means that there is a total orientation $O'\in\WS$ that orients $\tau(T_2)$. There is an edge $e$ of $G$ such that the other two cycles of $\tau(T_2)\cup e$ are precisely $\tau(T_1)$ and $\tau(T_3)$. Since $O'$ has to orient $e$, it follows that it orients either $\tau(T_1)$ or $\tau(T_3)$ with the same sign as $\tau(T_2)$. Thus $\coltree_\WS$ has no bad triples.

Let $\coltree:\Conn(T)\to\{+,-,0\}$ be the maximal with respect to the $\geq$ order map such that $\coltree\geq\coltree_\WS$ and such that there are no bad triples for $\coltree$. We claim that the image of  $\coltree$ lies in $\{+,-\}$, which is the same as saying that it is a $G$-colocalization in general position. Suppose that this is not the case and choose the minimal by inclusion connected subgraph $T'$ of $T$ such that $\coltree(T')=0$. We would like to show that there is $\e\in\{+,-\}$ such that setting $\coltree(T'):=\e$ introduces no bad triples. First of all, regardless of $\e$, we can get no bad triples of the form $(T',T_2,T_3)$. This is true because if $\coltree(T_2)\neq 0$ then we necessarily have $\coltree(T_2)=\coltree(T_3)$, otherwise $(T',T_2,T_3)$ would be a bad triple for $\coltree$. Thus we need to only make sure that we will get no bad triples of the form $(T_1,T',T_3)$ after setting $\coltree(T'):=\e$. Since $T'$ is minimal, we have $\coltree(T_1),\coltree(T_3)\neq 0$. If $\coltree(T_1)\neq\coltree(T_3)$ then $(T_1,T',T_3)$ cannot be a bad triple for $\coltree$ after we set $\coltree(T')=\e$. Thus the only possible restrictions on $\e$ come from triples $(T_1,T',T_3)$ such that $\coltree(T_1)=\coltree(T_3)\neq 0$. If for all such triples the sign of $\coltree(T_1)=\coltree(T_3)$ is the same then we can just choose $\e$ to be this sign. Otherwise there must exist two triples $(T_1,T',T_3)$ and $(T_5,T',T_7)$ such that 
\begin{equation}\label{eq:coltree}
\coltree(T_1)=\coltree(T_3)=+,\quad \coltree(T_5)=\coltree(T_7)=-.
\end{equation}

We need to show that this is impossible. Note that $T_1$ and $T_3$ are obtained from $T'$ by removing some edge $e$. Similarly, $T_5$ and $T_7$ are obtained from $T'$ by removing some other edge $f$. If we remove $e$ and $f$ from $T'$, it will split into three connected components, and after a possible switching of indices we may assume that one of them is $T_1$ and the other one is $T_7$. Let us denote the remaining connected component $T_4$. Thus $T_1\cup T_4=T_5$ and $T_4\cup T_7=T_3$ and $T_1,T_4,T_7$ are vertex-disjoint. By~\eqref{eq:coltree} and since $(T_1,T_5,T_4)$ and $(T_4,T_3,T_7)$ do not form bad triples for $\coltree$, we must have $\coltree(T_3)=\coltree(T_4)=\coltree(T_5)$ which yields a contradiction. We are done with the proof of Theorem~\ref{thm:outerplanar}.
\end{proof}

\subsection{Mutation-closed domains for graphical oriented matroids}
In this section, we investigate the mutation graph $\mutgraph(\Mcal_\GORO)$ and prove Conjecture~\ref{conj:domains} in the graphical case. We also point out the relationship with the results of Gioan~\cite{Gioan1,Gioan2} as well as recent developments of Backman-Baker-Yuen~\cite{BBY} on the Jacobian group of a graph. 

Recall that for an oriented matroid $\Mcal$, two sets $S,T\subset E$ are \emph{related by a mutation} if there exists a circuit $C$ of $\Mcal$ oriented by $S$ and such that $T=(S-C^+)\cup C^-$ in which case we write $T=\mut{C}{S}$ and have an edge between $S$ and $T$ in the \emph{mutation graph} $\mutgraph(\Mcal)$ of $\Mcal$. Translating this to the language of total orientations of an undirected graph $G$ yields the following notion which is due to Gioan~\cite{Gioan1}.
\begin{definition}
	We say that two orientations $O_1$ and $O_2$ of $G$ are related by a \emph{cycle reversal} if there is an undirected cycle $C$ of $G$ such that $O_1$ and $O_2$ orient it in the opposite ways and agree on all other edges of $G$. We let $\mutgraph(G)$ be a simple graph whose vertices are total orientations of $G$ and two such orientations are connected by an edge if and only if they are related by a cycle reversal.
\end{definition}

\def\indeg{ \operatorname{indeg}}
\begin{definition}
	Let $V$ be the vertex set of $G$ and consider a total orientation $O$ of $G$. We define the \emph{indegree sequence} of $O$ to be a map $\indeg_O:V\to \Z$ associating to each vertex $v\in V$ the number $\indeg_O(v)$ of edges pointing towards $v$ in $O$.
\end{definition}

We recall some of the results of Gioan:
\begin{proposition}[{\cite[Proposition~4.10]{Gioan1}}]
	Two orientations $O_1$ and $O_2$ of $G$ belong to the same connected component of $\mutgraph(G)$ if and only if $\indeg_{O_1}(v)=\indeg_{O_2}(v)$ for all $v\in V$.\qed
\end{proposition}

\begin{proposition}
	Let $G$ be an undirected graph. The following quantities are equal:
	\begin{itemize}
		\item the number of connected components of $\mutgraph(G)$;
		\item the number of forests of $G$;
		\item the number of independent sets of $\Mcal_\GORO$;
		\item the size of a maximal \emph{by size} $G$-separated collection;
		\item the value $T_G(2,1)$ where $T_G$ is the Tutte polynomial of $G$.
	\end{itemize}
\end{proposition}
\begin{proof}
	The fact that the number of indegree sequences equals $T_G(2,1)$ is given in~\cite[Corollary~4.11]{Gioan1} and the rest follows from various known interpretations of $T_G(2,1)$ combined with our results from Section~\ref{sect:max_by_sz}.
\end{proof}

Our next result follows from the proof of~\cite[Proposition~4.10]{Gioan1}:
\begin{proposition}\label{prop:graph_not_G_separated_indegree}
	If two orientations $O_1$ and $O_2$ of $G$ belong to the same connected component of $\mutgraph(G)$ then $O_1$ and $O_2$ are not $G$-separated.
\end{proposition}
\begin{proof}
	Indeed, as Gioan shows in the proof of~\cite[Proposition~4.10]{Gioan1}, if $O_1$ and $O_2$ have the same indegree sequence then there is a cycle in $G$ that they orient in the opposite ways. This finishes the proof.
\end{proof}

Thus we have the following situation: every maximal \emph{by size} $G$-separated collection $\WS$ of total orientations has $T_G(2,1)$ elements, there are $T_G(2,1)$ connected components in $\mutgraph(G)$ and each connected component contains at most one element of $\WS$. Thus each component contains exactly one element of $\WS$. 
\begin{corollary}
	Conjecture~\ref{conj:domains} is valid for graphical oriented matroids.
\end{corollary}
\begin{proof}
	Indeed, for any mutation-closed domain $\dom$ consisting of $k$ connected components of $\mutgraph(G)$, any $G$-separated collection inside $\dom$ contains at most $k$ elements, but the restrictions of maximal \emph{by size} $G$-separated collections inside $2^E$ to $\dom$ contain exactly $k$ elements. 
\end{proof}

\begin{corollary}
	An undirected graph $G$ is outerplanar  if and only if for any mutation-closed domain $\dom$, every maximal \emph{by inclusion} $G$-separated collection of total orientations is also maximal \emph{by size}.
\end{corollary}

\begin{remark}
	Several of our results for the graphical case can be found in~\cite{BBY}. For example, Proposition~\ref{prop:graph_not_G_separated_indegree} is analogous to~\cite[Proposition~3.3.3]{BBY}. One major difference is that in~\cite{BBY}, the authors consider only \emph{coherent} colocalizations (see the next section for the definition) which correspond to \emph{regular} zonotopal tilings. Another difference is that the main focus of~\cite{BBY} as well as~\cite{Gioan1} are \emph{cycle-cocycle reversal} classes rather than cycle reversal classes that we consider. The latter correspond to the vertices of the zonotopal tiling while the former correspond to the top-dimensional tiles of the same tiling.
\end{remark}

\def\ptp{P}
\def\vect{x}
We now would like to show that the connected components of $\mutgraph(G)$ are actually $1$-skeleta of polytopes. Choose some reference orientation $O$ of $G$. To every total orientation $O'$ of $G$ we now associate a vector $\vect=\vect_{O'}\in\R^E$ as follows. For every edge $e\in E$, we put 
\begin{equation}\label{eq:vect_O}
\begin{cases}
	\vect(e)=+1,&\text{if $O$ and $O'$ agree on $e$};\\
	\vect(e)=-1,&\text{otherwise.}
\end{cases}
\end{equation}
 For a collection $\dom$ of total orientations of $G$, we define $\ptp(\dom)\subset\R^E$ to be the convex hull of $\vect_{O'}$ for all $O'\in \dom$. Since this polytope is contained in the cube $\Conv(\{+1,-1\}^E)\subset\R^E$ and for any total orientation $O'$, $\vect_{O'}$ is a vertex of this cube, it follows that for every orientation $O'\in \dom$, $\vect_{O'}$ is a vertex of $\ptp(\dom)$.

\def\vecty{y}
\begin{proposition}\label{prop:polytopal}
Let $\dom$ be a single connected component of $\mutgraph(G)$. Then two orientations $O_1,O_2\in\dom$ are connected by a cycle reversal if and only if $\vect_{O_1}$ and $\vect_{O_2}$ form an edge of $\ptp(\dom)$. Thus the restriction of $\mutgraph(G)$ to $\dom$ is the $1$-skeleton of $\ptp(\dom)$.
\end{proposition}
\begin{proof}
	Note that $\ptp(\dom)$ lies in the affine subspace $W$ of $\R^E$ consisting of total orientations with fixed indegree. More precisely, define the linear map $\phi:\R^E\to\R^V$ by $\phi(e)=\vecty_v$ where $e$ is directed in $O$ towards $v$ and $(\vecty_u)_{u\in V}$ is a basis of $\R^V$. Then fixing the indegree corresponds to taking the preimage of a single point under $\phi$ which yields an affine subspace of $\R^E$. Now, we claim that every edge of $\mutgraph(G)$ restricted to $\dom$ is a one-dimensional face of $\ptp(\dom)$ and vice versa. Take any edge of $\mutgraph(G)$ connecting two orientations $O_1$ and $O_2$. Then there is a unique cycle $C$ of $G$ where $O_1$ and $O_2$ disagree. One easily observes that the line segment 
	\[[\vect_{O_1},\vect_{O_2}]:=\{t\vect_{O_1}+(1-t)\vect_{O_2}\mid 0\leq t\leq 1\}\]
	maximizes the linear function $\l$ on $\R^E$ defined as follows: for $e\in E$, put $\l(e)\in\{+1,0,-1\}$ to be 
	\[\l(e)= \begin{cases}
	         	+1,&\text{if $\vect_{O_1}(e)=\vect_{O_2}(e)=+1$;}\\
	         	-1,&\text{if $\vect_{O_1}(e)=\vect_{O_2}(e)=-1$;}\\
	         	0,&\text{otherwise.}\\
	         \end{cases}\]
	This shows that every edge of $\mutgraph(G)$ is an edge of $\ptp(\dom)$. To show the converse, suppose that $O_1,O_2\in\dom$ are such that the line segment $[\vect_{O_1},\vect_{O_2}]$ is an edge of $\ptp(\dom)$. Since the indegrees of $O_1$ and $O_2$ are the same, the set of edges of $G$ where they disagree is a union of several cycles $C_1,C_2,\dots,C_k$. If it consists of just one cycle then we are done, and otherwise we get that $[\vect_{O_1},\vect_{O_2}]$ belongs to a $k$-dimensional face of $G$ spanned by the edges of $\ptp(\dom)$ corresponding to reversing just a single cycle $C_i$ for $1\leq i\leq k$. This finishes the proof.
\end{proof}

\begin{comment}
\begin{remark}
	According to our computations, it is not always the case that the connected components of $\mutgraph(\Mcal)$ are $1$-skeleta of polytopes for a general oriented matroid $\Mcal$. There are however some interesting examples of oriented matroids for which this happens. For example, for the the (non-pure) oriented matroid $\Mcal=\IC(6,3,1)$ from Figure~\ref{fig:non_postiroids}, there are $64$ vertices in $\mutgraph(\Mcal)$, $32$ of them form isolated connected components, and the rest $32$ of them split two connected components with $12$ and $20$ vertices that are $1$-skeleta of the \emph{icosahedron} and the \emph{dodecahedron} respectively. Moreover, the mutation-closed domain $\dom$ for the icosahedron is \emph{pure}: every maximal \emph{by inclusion} $\Mcal$-separated collection inside $\dom$ has exactly $3$ elements. There are $20$ such collections, and connecting two of them by an edge if they differ in one element again produces the $1$-skeleton of the dodecahedron.
\end{remark}
\end{comment}

\subsection{Enumerating maximal $G$-separated collections}
In this section, we show that for some outerplanar graphs $G$, the total number of maximal (by size or by inclusion) $G$-separated collections is given by a simple multiplicative formula. 

Let $G$ be a triangulation of an $(n+2)$-gon. Recall from Definition~\ref{dfn:tree} that for each such graph $G$ there is an associated tree $T=T(G)$ on $n$ vertices such that the cycles of $G$ correspond to connected subgraphs of $T$. The former set is denoted $\Cyc(G)$ and the latter set is denoted $\Conn(T)$. 

\def\leaf{\ell}
For each $a,b\geq 0$, define the tree $T_{a,b}$ with $a+b+2$ vertices to be the path of length $a+b+1$ on the vertices $(-a,-a+1,\dots,-1,0,1,\dots,b-1,b)$ with an extra leaf $\leaf$ attached to $0$. An example of $T_{2,5}$ (also known as the Dynkin diagram of affine type $\hat E_8$) together with the corresponding outerplanar graph $G$ is shown in Figure~\ref{fig:triangulation}. When $a=0$ or $b=0$ then $T_{a,b}$ is just a path. Note that this family of trees includes all $ADE$ Dynkin diagrams of finite type as well as some other graphs.

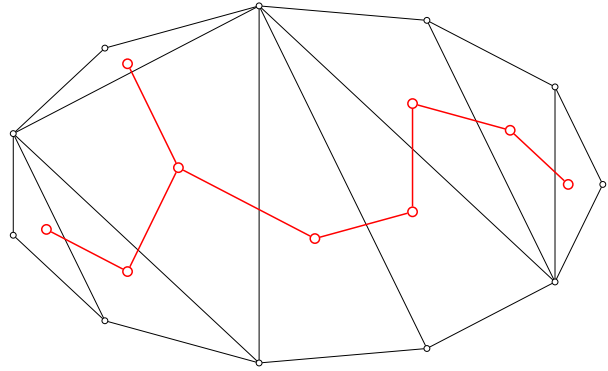
\begin{figure}
	
	\scalebox{0.4}{
	\begin{tikzpicture}[xscale=10,yscale=6]
		\node[draw,circle,scale=0.5] (A) at ({0*360/11}:1) {};
		\node[draw,circle,scale=0.5] (B) at ({1*360/11}:1) {};
		\node[draw,circle,scale=0.5] (C) at ({2*360/11}:1) {};
		\node[draw,circle,scale=0.5] (D) at ({3*360/11}:1) {};
		\node[draw,circle,scale=0.5] (E) at ({4*360/11}:1) {};
		\node[draw,circle,scale=0.5] (F) at ({5*360/11}:1) {};
		\node[draw,circle,scale=0.5] (G) at ({6*360/11}:1) {};
		\node[draw,circle,scale=0.5] (H) at ({7*360/11}:1) {};
		\node[draw,circle,scale=0.5] (I) at ({8*360/11}:1) {};
		\node[draw,circle,scale=0.5] (J) at ({9*360/11}:1) {};
		\node[draw,circle,scale=0.5] (K) at ({10*360/11}:1) {};
		\draw (A) -- (B) -- (C) -- (D) -- (E) -- (F) -- (G) -- (H) -- (I) -- (J) -- (K) -- (A) ;
		\draw (B)--(K)--(C);
		\draw (J)--(D)--(I)--(F)--(H);
		\draw (F)--(D);
		\draw (K) -- (D);
		\coordinate (ABK) at ($.33*(A)+.33*(B)+.33*(K)$);
		\coordinate (BKC) at ($.33*(B)+.33*(K)+.33*(C)$);
		\coordinate (KCJ) at ($.33*(K)+.33*(C)+.33*(D)$);
		\coordinate (CJD) at ($.33*(K)+.33*(J)+.33*(D)$);
		\coordinate (DJI) at ($.33*(D)+.33*(J)+.33*(I)$);
		\coordinate (DIF) at ($.33*(D)+.33*(I)+.33*(F)$);
		\coordinate (DFE) at ($.33*(D)+.33*(F)+.33*(E)$);
		\coordinate (FIH) at ($.33*(F)+.33*(I)+.33*(H)$);
		\coordinate (FHG) at ($.33*(F)+.33*(H)+.33*(G)$);
		\node[draw,circle,scale=0.8,red,line width=0.5mm] (u1) at (ABK) {};
		\node[draw,circle,scale=0.8,red,line width=0.5mm] (u2) at (BKC) {};
		\node[draw,circle,scale=0.8,red,line width=0.5mm] (u3) at (KCJ) {};
		\node[draw,circle,scale=0.8,red,line width=0.5mm] (u4) at (CJD) {};
		\node[draw,circle,scale=0.8,red,line width=0.5mm] (u5) at (DJI) {};
		\node[draw,circle,scale=0.8,red,line width=0.5mm] (u6) at (DIF) {};
		\node[draw,circle,scale=0.8,red,line width=0.5mm] (u7) at (DFE) {};
		\node[draw,circle,scale=0.8,red,line width=0.5mm] (u8) at (FIH) {};
		\node[draw,circle,scale=0.8,red,line width=0.5mm] (u9) at (FHG) {};
		\draw[red,line width=0.5mm] (u1) -- (u2) -- (u3) -- (u4) -- (u5) -- (u6) -- (u8) -- (u9);
		\draw[red,line width=0.5mm] (u6) -- (u7);
	\end{tikzpicture}}

\caption{\label{fig:triangulation}A triangulation $G$ of an $11$-gon is shown in black. The corresponding tree $T(G)=T_{2,5}=\hat E_8$ is shown in red.}

\end{figure}

The motivation for the following definition will be clear later, see Proposition~\ref{prop:coherent}.
\begin{definition}
	We say that $G$ is \emph{all-coherent} if there exist $a,b\geq 0$ such that $T(G)=T_{a,b}$.
\end{definition}

The main result of this section is the following enumeration of maximal $G$-separated collections:
\begin{theorem}\label{thm:coherent}
	Let $G$ be a triangulation of an $n+2$-gon and suppose that $G$ is all-coherent: $T(G)=T_{a,b}$ for some $a,b\geq 0$ with $n=a+b+2$. Then the number of maximal (by size $=$ by inclusion) $G$-separated collections equals 
	\begin{equation}\label{eq:coherent}
		2(n+1)\frac{n!}{(n-a)!}\frac{n!}{(n-b)!}.
	\end{equation}
      \end{theorem}

\begin{example}
Taking $G$ to be a triangulation of a square yields $T(G)=T_{a,b}$ for $a=b=0$ and $n=2$. Thus $T_{a,b}$ consists of two vertices and there are eight possible maps $\Conn(T)\to\{+,-\}$. Six of them give $G$-colocalizations while the other two have a bad triple. Each of the six $G$-colocalizations corresponds to a maximal $G$-separated collection, in agreement with~\eqref{eq:coherent}.
\end{example}

\begin{proof}
	We will show that if $G$ is all-coherent then the number of maximal $G$-separated collections equals to the number of regions of a certain \emph{hyperplane arrangement} (see~\cite{Stanley} for the background)  called the \emph{secondary arrangement}.
	
	\def\second{\Acal}
	\begin{definition}
		Let $\Mcal$ be an oriented matroid with ground set $E$ and let $\sigma:\Ccal(\Mcal)\to\{+,-,0\}^E$ be a colocalization in general position. We say that $\sigma$ is \emph{coherent} if there exists a function $\l:E\to \R$ such that for any circuit $C=(C^+,C^-)$ of $\Mcal$, we have $\l(C)\neq 0$ and $\sigma(C)=+$ if and only if $\l(C)>0$. Here 
		\[\l(C)=\sum_{e\in C^+} \l(e)-\sum_{e\in C^-} \l(e).\]
		Let  $\second(\Mcal)$ be the arrangement of hyperplanes in $\R^E=\{\l:E\to\R\}$ given by
		\[\sum_{e\in C^+} \l(e)=\sum_{e\in C^-} \l(e),\]
		where $C=(C^+,C^-)$ runs over the set of all circuits of $\Mcal$.
              \end{definition}
\def\third{\second'}
Thus coherent colocalizations of $\Mcal$ correspond to the regions of $\second(\Mcal)$. If $\Mcal$ is realized by some vector configuration $\VC$ then one can define the \emph{secondary arrangement} $\third(\VC)$ as follows: for any circuit $C=(C^+,C^-)$ of $\Mcal$, there are unique (up to a multiplication by a common positive scalar) positive real numbers $\alpha_e$, $e\in \Cu$ such that

\[\sum_{e\in C^+} \alpha_e \v_e=\sum_{e\in C^-} \alpha_e \v_e.\]
The arrangement $\third(\Mcal)$ in $\R^E=\{\l:E\to\R\}$ consists of the hyperplanes
\[\sum_{e\in C^+} \alpha_e  \l(e)=\sum_{e\in C^-}  \alpha_e \l(e).\]
The colocalizations of $\Mcal$ in general position correspond to zonotopal tilings of $\Zon_\VC$ by the Bohne-Dress theorem and the coherent colocalizations correspond to the regions of $\second(\Mcal)$. The regions of the secondary arrangement $\third(\Mcal)$ correspond to the \emph{coherent} (or \emph{regular}) zonotopal tilings of $\Zon_\VC$. Note that when $\Mcal$ is graphical, we have $\alpha_e=1$ for all $e$ so the arrangements $\second(\Mcal)$ and $\third(\Mcal)$ actually coincide. Secondary arrangements are closely related to \emph{secondary polytopes} of Gel'fand-Kapranov-Zelevinsky~\cite{GKZ}.
	
	\begin{proposition}\label{prop:coherent}
		Let $G$ be a triangulation of an $n+2$-gon. Then $G$ is all-coherent (i.e., $T(G)=T_{a,b}$) if and only if every $G$-colocalization is coherent.
	\end{proposition}
	
	Before we prove the proposition, let us show how it implies Theorem~\ref{thm:coherent}. Let $V:=\{-a,-a+1,\dots,b-1,b,\leaf\}$ be the vertex set of $T(G)$, and let $x_{-a},x_{-a+1},\dots,x_{b},z$ be real numbers. Define the following hyperplane arrangement $\Acal_{a,b}$ in $\R^V$:
	\begin{equation}\label{eq:Acal}
	\begin{split}
	\Acal_{a,b}=\{z=0\}&\cup\{x_i=0\mid -a\leq i\leq b\}\\
			   &\cup \{x_j-x_i=0\mid -a\leq i<j\leq b\}\\
			   &\cup \{x_j-x_i+z=0\mid -a\leq i\leq 0\leq j\leq b\}.
	\end{split}
	\end{equation}
	The arrangement $\Acal_{a,b}$ is related to $\second(\Mcal_\GORO)$ via a simple change of coordinates, as we now explain. Choose an orientation $\GORO$ of $G$ such that the boundary of the $n+2$-gon is oriented clockwise (and the remaining edges are oriented arbitrarily). We introduce a linear map $\phi:\R^E\to\R^V$, where $E$ is the ground set of $\Mcal_\GORO$. For each $v\in V$ and $e\in E$, the $(e,v)$-th entry of the matrix of $\phi$ is zero unless $e$ is an edge of the triangle of $G$ around $v$. If $e$ is oriented clockwise around the boundary of this triangle, the matrix entry is $+1$, otherwise it is $-1$. One easily checks that each hyperplane in $\second(\Mcal_\GORO)$ contains the kernel of $\phi$, thus $\phi(\second(\Mcal_\GORO))$ yields a hyperplane arrangement inside $\R^V$ with the same number of regions. It is also straightforward to see that the hyperplane arrangements $\Acal_{a,b}$ and $\phi(\second(\Mcal_\GORO))$ are related by another coordinate change $(x_{-a},\dots,x_b,z)\mapsto \l$, where $\l:V\to\R$ is given by
	\[\l(i)=x_{i}-x_{i-1},\quad -a\leq i\leq b,\quad\text{and}\quad \l(\leaf)=z.\]
	Here we put $x_{-a-1}:=0$. Thus the regions of $\Acal_{a,b}$ correspond precisely to coherent $G$-colocalizations $\sigma$. Since by the above proposition, every $G$-colocalization is coherent, the number of them (which is the same as the number of maximal $G$-separated collections) equals the number of regions of $\Acal_{a,b}$. The formula~\eqref{eq:coherent} then follows from~\cite[Theorem~3.4]{Ath2} since $\Acal_{a,b}$ is a \emph{cone} over a \emph{graphical Shi arrangement}. The only thing left to do is to prove Proposition~\ref{prop:coherent}.

\def\affD{{\hat D}}
\def\affE{{\hat E}}

\begin{figure}

\begin{tabular}{c|c}
% \scalebox{0.8}{
% \begin{tikzpicture}[scale=2]
% 		\node[draw,circle] (A) at (-1,-1) {$+1$};
% 		\node[draw,circle] (B) at (+1,-1) {$+1$};
% 		\node[draw,circle] (C) at (+1,+1) {$+1$};
% 		\node[draw,circle] (D) at (-1,+1) {$+1$};
% 		\node[draw,circle] (E) at (0,0) {$-2$};
% 		\draw (A)--(E)--(B);
% 		\draw (C)--(E)--(D);
% \end{tikzpicture}}
\scalebox{0.7}{
\begin{tikzpicture}[scale=1.5]
		\node[draw,circle] (A) at (-3,-1) {$+1$};
		\node[draw,circle] (C) at (+3,-1) {$+1$};
		\node[draw,circle] (D) at (+3,+1) {$+1$};
		\node[draw,circle] (B) at (-3,+1) {$+1$};
\def\ndscl{0.8}
		\node[draw,circle,scale=\ndscl] (E1) at (-2,0) {$\l(1)$};
		\node[draw,circle,scale=\ndscl] (E2) at (-1,0) {$\l(2)$};
		\node[draw,circle,scale=\ndscl] (E3) at (-0,0) {$\l(3)$};
		\node[draw,circle,scale=\ndscl] (E4) at (1,0) {$\l(4)$};
		\node[draw,circle,scale=\ndscl] (E5) at (2,0) {$\l(5)$};
		\draw (A)--(E1)--(E2)--(E3)--(E4)--(E5)--(C);
		\draw (B)--(E1);
		\draw (D)--(E5);
		\node[anchor=north,scale=1] (X) at (0,-0.5) {$\l(1)+\l(2)+\dots+\l(5)=-2$};
\end{tikzpicture}}
&\scalebox{0.7}{
	\begin{tikzpicture}[scale=1.5]
		\node[draw,circle] (A) at (-2,0) {$+1$};
		\node[draw,circle] (B) at (-1,0) {$-2$};
		\node[draw,circle] (C) at (0,0) {$+3$};
		\node[draw,circle] (D) at (1,0) {$-2$};
		\node[draw,circle] (E) at (2,0) {$+1$};
		\node[draw,circle] (F) at (0,1) {$-2$};
		\node[draw,circle] (G) at (0,2) {$+1$};
		\draw (A) -- (B) --(C)--(D)--(E);
		\draw (C)--(F)--(G);
	\end{tikzpicture}}

\end{tabular}

	\caption{\label{fig:ADE} Dynkin diagrams of affine types $\affD_n$ (for $n=8$) and $\affE_6$ together with the choice of $\l$.}
\end{figure}
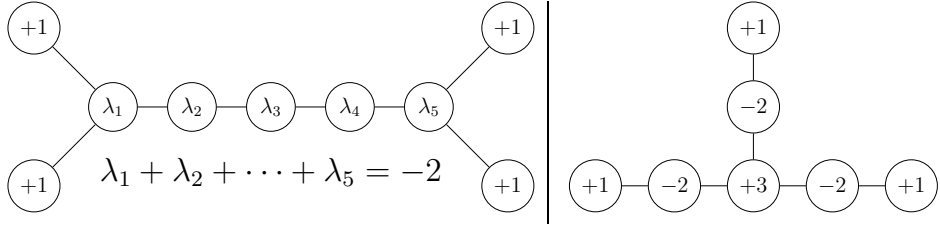

\begin{proof}[Proof of Proposition~\ref{prop:coherent}]
We start by showing that if $G$ is not all-coherent then there is a non-coherent $G$-colocalization. Since  a colocalization of a subgraph of $G$ can be extended to a $G$-colocalization (for example, by Theorem~\ref{thm:outerplanar}), it suffices to construct non-coherent $G$-colocalizations for the case when $T(G)$ is a Dynkin diagram of type either $\affD_n$, $n\geq 5$, or $\affE_6$, shown in Figure~\ref{fig:ADE}.\footnote{Since $G$ is a triangulation, we do not consider the affine Dynkin diagram $\affD_4$ since $T(G)$ can only have vertices of degree at most $3$.} The construction is similar to the one used in~\cite[Example~3.9]{GaPa}. 

Let $T(G)$ be either $\affD_n$ or $\affE_6$. To each vertex $v$ of $T(G)$ we will assign a real number $\l(v)$. For $\affE_6$, the values of $\l(v)$ are shown in Figure~\ref{fig:ADE} (thus $\l$ is obtained from Vinberg's \emph{additive function}~\cite{Vinberg} by changing signs of all vertices of the same color). For $\affD_n$, denote by $L$ the set of leaves of $\affD_n$ and by $U$ the set of all other vertices, so we have $|L|=4$ and $|U|=n-3$. We put $\l(v)=1$ for all $v\in L$ and for $v\in U$, we choose $\l(v)$ to be generic real numbers such that
\[\sum_{v\in U} \l(v)=-2.\]
For example, if $n=4$ then $U$ consists of just one vertex $v$ so we must have $\l(v)=-2$ which again recovers Vinberg's additive function.

Now that we have constructed $\l$, we will define $\sigma:\Conn(T(G))\to \{+,-,0\}$ as follows: for each subtree $T'\in \Conn(T(G))$, we put $\sigma(T')$ to be the sign of $\sum_{v\in T'} \l(v)$, so we put $\sigma(T')=0$ if this sum is zero, in which case we call $T'$ a \emph{zero subtree}. We give a complete description of zero subtrees for each of the cases. For $\affD_n$, denote $L=\{a,b,c,d\}$ so that $a,b$ have a common neighbor and $c,d$ have a common neighbor. For $i\neq j\in L$, denote by $iUj$ the subtree of $\affD_n$ with vertex set $\{i,j\}\cup U$.
\begin{itemize}
	% \item For $T(G)=\affD_4$, there are six zero subtrees and they have the form $iUj$ for all pairs $i\neq j\in L$.
	\item For $T(G)=\affD_n,n\geq 5$, there are six zero subtrees: 
	\[aUc,aUd,bUc,bUd,aUb,cUd.\]
	\item For $T(G)=\affE_6$, there are seven zero subtrees: all the six paths of length $4$ together with $T(G)$ itself.
\end{itemize}

We are going to specify a \emph{non-coherent perturbation} of $\sigma$ by assigning some carefully chosen signs to the zero subtrees of $T(G)$ in such a way that $\sigma$ would be a colocalization but not a coherent one. First note that no matter how we assign the signs to these subtrees, the result will be a colocalization (i.e., have no \emph{bad triples} as in Definition~\ref{dfn:bad_triple}). Indeed, in any bad triple $(T_1,T_2,T_3)$, only one of the trees can satisfy $\sigma(T_i)=0$ as it follows from the description above. If $\sigma(T_2)=0$ then $\l(T_2)=0$ but since $\l(T_2)=\l(T_1)+\l(T_3)$, we have $\sigma(T_1)=-\sigma(T_3)\neq 0$. If instead $\sigma(T_1)=0$ then $\l(T_1)=0$ and thus $\l(T_2)=\l(T_3)$ so $\sigma(T_2)=\sigma(T_3)\neq 0$. This shows that after we choose the signs for the zero subtrees of $\sigma$, it becomes a colocalization for $G$ in general position. 

Suppose that $T(G)=\affD_n$. Choose the signs for the four zero subtrees as follows:
\[\sigma(aUc)=\sigma(bUd)=+,\quad \sigma(aUd)=\sigma(bUc)=-.\]
Choose the remaining two signs $\sigma(aUb),\sigma(cUd)$ arbitrarily. Then $\sigma$ cannot be coherent because if it was defined by some labeling $\m$ of the vertices of $T(G)$ by real numbers then we would have 
\[\m(aUc)+\m(bUd)=\m(a)+\m(b)+\m(c)+\m(d)+2\m(U)=\m(aUd)+\m(bUc),\]
but on the other hand, the left hand side has to be positive and the right hand side has to be negative. 

\def\sign{ \operatorname{sign}}
We deal with the case $T(G)=\affE_6$ in an analogous fashion. There are six paths $T_1,T_2,\dots,T_6$ of length $4$ in $\affE_6$, we order them so that for any $i=1,2,\dots,6$, $T_i$ and $T_{i+1}$ have $3$ vertices in common. Here the indices are taken modulo $6$. We put 
\begin{equation}\label{eq:sigmas}
\sigma(T_1)=\sigma(T_3)=\sigma(T_5)=+;\quad \sigma(T_2)=\sigma(T_4)=\sigma(T_6)=-,
\end{equation}
and for the remaining zero subtree which coincides with $T(G)$, we choose $\sigma(T(G))$ to be arbitrary, say, $\sigma(T(G))=+$. Again, summing up both sides of~\eqref{eq:sigmas} shows that $\sigma$ is non-coherent. Thus both $\affD_n$ and $\affE_6$ admit non-coherent colocalizations, and therefore if all $G$-colocalizations are coherent then $T(G)$ does not contain either of these two trees as a minor and so $G$ must be all-coherent. We have shown one direction of Proposition~\ref{prop:coherent}. Note that the same method of proof does not work for the affine Dynkin diagrams $\affE_7$ and $\affE_8$, in fact, these graphs are all-coherent.

So let $G$ be an all-coherent graph and thus $T(G)=T_{a,b}$ for some $a,b$ with $a+b+2=n$. We would like to show that every $G$-colocalization $\sigma:\Conn(T)\to\{+,-\}$ is coherent. Our first goal is to explain that the only possible counterexamples to this are the ones of the form above. More precisely, by~\cite[Lemma~2.3.1]{BBY}, $\sigma$ is coherent if and only if for every $C_1,C_2,\dots,C_k\in\Cyc(G)$, there is no $k$-tuple $(a_1,\dots,a_k)$ of integers such that $\sigma(C_i)=\sign(a_i)$ and $\sum_{i=1}^k a_i \vect_{C_i}=0$, where $\vect_{C_i}\in\R^E$ is defined  by linearity via~\eqref{eq:vect_O}. If we allow repetitions, we may assume that $a_i=\pm1$ for all $i=1,2,\dots,k$. Let $T_1,T_2,\dots,T_k\in\Conn(T(G))$ be the subtrees of $T(G)$ corresponding to $C_1,\dots,C_k$. Thus showing the result amounts to showing the following:

\def\Trees{\mathcal{T}}
\begin{lemma}
	Suppose that we are given two multisets
	\[\Trees^+,\Trees^-\subset\Conn(T_{a,b})\] 
	that satisfy $\sigma(T)=+$ for all $T\in\Trees^+$ and $\sigma(T)=-$ for all $T\in\Trees^-$. Suppose in addition that every vertex of $T_{a,b}$ appears in $\Trees^+$ and in $\Trees^-$ the same number of times. Then we have $\Trees^+=\emptyset$ and $\Trees^-=\emptyset$.
\end{lemma}
\begin{proof}
We assume that the result of the lemma is true for all smaller values of $a$ and $b$ by induction (the base case being trivial). We also assume that the collections $\Trees^+,\Trees^-$ are minimal by size satisfying the above conditions.
\def\A{A}
\def\B{B}
For $i\leq j$, denote by $[i,j]\subset\Z$ the set $[i,j]=\{i,i+1,\dots,j\}$. Recall that the vertices of $T_{a,b}$ are the numbers in $[-a,b]$ together with an extra vertex $\leaf$. For $-a\leq i\leq 0\leq j\leq b$, we denote by $[i,\leaf,j]:=[i,j]\cup\{\leaf\}$ the corresponding subtree of $T_{a,b}$. Thus the set $\Conn(T_{a,b})$ consists of all sets that are either of the form $[i,j]$ for $-a\leq i\leq j\leq b$ or of the form $[i,\leaf,j]$ for $-a\leq i\leq 0\leq j\leq b$, together with one extra subtree $\{\leaf\}$. We call the subtrees in $\Trees^+$ \emph{positive} and the subtrees in $\Trees^-$ \emph{negative}. We also assume that $\leaf$ appears in at least one subtree, otherwise we can remove it and designate some other vertex to be $\leaf$. We split $\Trees^+\cup\Trees^-$ into two multisets $\A,\B$ as follows:
\[
  \resizebox{\textwidth}{!}{$\A=\{\text{positive subtrees containing $\leaf$}\}\cup \{\text{negative subtrees not containing $\leaf$}\};$}\]
\[
  \resizebox{\textwidth}{!}{$\B=\{\text{positive subtrees not containing $\leaf$}\}\cup \{\text{negative subtrees containing $\leaf$}\}.$}\]
It is clear that both sets are nonempty because we assumed that $\leaf$ appears in at least one (positive or negative) subtree. We will show that if $\A$ is nonempty then we must have $\sigma(\{\leaf\})=+$ and if $\B$ is nonempty then we must have $\sigma(\{\leaf\})=-$. This will immediately lead to a contradiction. Moreover, by symmetry, we only need to prove the first of these two claims. 

\def\D{D}
Suppose that $\A$ is nonempty. We are going to construct a certain directed graph $\D(\A)$ with vertex set $[-a,b]$ and with arrows colored red and blue. For each positive subtree $[i,\leaf,j]$ (with necessarily $i\leq 0\leq j$), draw a red arrow $i\to j$ in $\D(\A)$. For each negative subtree $[i,j]$ with $i\leq 0\leq j$, draw a blue arrow $j\to i$ in $\D(\A)$. We claim that for every red arrow $i\to j$ in $\D(\A)$, there is also a blue arrow $j\to q$ in $\D(\A)$, and vice versa, for any blue arrow $j\to i$ in $\D(\A)$, there is a red arrow $i\to q$ in $\D(\A)$. 

Indeed, suppose that $i\to j$ is a red arrow in $\D(\A)$, that is, the subtree $[i,\leaf,j]$ appears in $\Trees^+$. We will consider two cases: $j=b$ and $j<b$. If $j=b$ then there must be some negative subtree $T^-\in\Trees^-$ containing $b$. It has the form either $[q,b]$ for some $q\in [-a,b]$ or $[q,\leaf,b]$ for some $q\in[-a,0]$. The latter case is impossible since then we will either have $[q,\leaf,b]\subset [i,\leaf,b]$ or $[q,\leaf,b]\supset [i,\leaf,b]$, and in any case we can replace these two subsets by their difference which also belongs to $\Conn(T_{a,b})$. This would contradict the minimality of $(\Trees^+,\Trees^-)$. The former case is impossible by the same reasoning unless $q\leq 0$.  Thus for the case $j=b$ we have established a blue arrow from $j$ to $q\leq 0$. Suppose now that $0\leq j<b$. Then the rest of $(\Trees^+,\Trees^-)$ contributes more to $j+1$ than to $j$. There cannot be any positive subtree of the form $[j+1,q]$ because otherwise $[i,\leaf,j]\cup[j+1,q]$ would be a connected subtree which again is a contradiction since $(\Trees^+,\Trees^-)$ is minimal. Thus there must be a negative subtree of the form either $[q,j]$ or $[q,\leaf, j]$. By the same reasoning, the latter gives a contradiction, and for the former, we must have $q\leq 0$, otherwise the difference $[i,\leaf,j]- [q,j]$ would be connected. We have found a blue arrow from $j$ in $\D(\A)$. 

\def\Cycle{C}
A similar argument shows that for every blue arrow $j\to i$, there exists a red arrow $i\to q$ in $\D(\A)$. Hence we can find a directed cycle $\Cycle$ in $\D(\A)$ that alternates between red and blue arrows. We want to show that $\sigma(\{\leaf\})=+$. Suppose that we have $\sigma(\{\leaf\})=-$ instead. Then for every positive subtree $[i,\leaf,j]$, we know that $\sigma([i,j])=+$ because otherwise $([i,j],[i,\leaf,j],\{\leaf\})$ would be a bad triple for $\sigma$. Now, let us remove $\leaf$ from all the subtrees appearing in $\Cycle$. This will give a linear combination of subtrees of $T_{a,b}-\{\leaf\}$ that sums up to zero at every vertex which cannot exist by the induction hypothesis. We get a contradiction thus finishing the proof of the lemma. As a consequence, this also finishes the proofs of Proposition~\ref{prop:coherent} and Theorem~\ref{thm:coherent}. %\qed
\end{proof}
\let\qed\relax
\end{proof}	
\let\qed\relax
\end{proof}

\subsection{Regular matroids}
In this section, we briefly describe how to generalize our results to the class of \emph{regular} (unoriented) matroids. We refer the reader to~\cite{Oxley} for the background on (unoriented) matroids.

\begin{definition}
	We say that a matroid $M$ is \emph{regular} if it can be realized by a \emph{totally unimodular matrix}, that is, a matrix with all maximal minors equal to either $+1,0,$ or $-1$.
\end{definition}

We refer the reader to the discussion before~\cite[Proposition~7.9.3]{Book} for other equivalent definitions of regular matroids. Note that regular matroids are always orientable in an essentially unique realizable way:

\begin{proposition}[{\cite[Corollary~7.9.4]{Book}}]\label{prop:regular_orientation}
	If $M$ is a regular matroid then all orientations of $M$ are realizable. They differ only by reorientation.
\end{proposition}

The class of regular matroids is closed under taking duals and contains the class of graphical (as well as \emph{cographical}) matroids. Moreover, graphical and regular matroids admit the following nice characterizations by forbidden minors. Let $F_7$ denote the famous \emph{Fano matroid} on $7$ elements, and for a graph $G$, denote by $M(G)$ the associated \emph{graphical matroid}. Let $U_{2,4}$ denote the uniform matroid of rank $2$ on $4$ elements and let $M^*$ denote the dual matroid of $M$.

\begin{proposition}[\cite{Tutte1} and~\cite{Tutte2}]\label{prop:regular_minors}
	\leavevmode
	\begin{enumerate}
		\item A matroid $M$ is regular if and only if it does not contain $U_{2,4}$, $F_7$, and $F_7^*$ as minors.
		\item A matroid $M$ is graphical if and only if it does not contain $U_{2,4}$, $F_7$, $F_7^*$, $M^*(K_{3,3})$, and $M^*(K_5)$ as minors.
	\end{enumerate}
\end{proposition}

\begin{theorem}\label{thm:regular}
	  Let $\Mcal$ be an orientation of a regular matroid $M$. Then $\Mcal$ is pure if and only if $M=M(G)$ is a graphical matroid where $G$ is an outerplanar graph.
\end{theorem}
\begin{proof}
	Suppose that $M$ is regular but not graphical. Then by Proposition~\ref{prop:regular_minors}, it contains either $M^*(K_{3,3})$ or $M^*(K_5)$ as a minor. By Proposition~\ref{prop:regular_orientation}, $\Mcal$ contains $\Mcal^*_{\vec K_{3,3}}$ or $\Mcal^*_{\vec K_5}$ as a minor. One easily checks that both $K_{3,3}$ and $K_5$ contain $K_4$ as a minor and since $K_4$ is self-dual and not pure, we are done by Proposition~\ref{prop:if_pure_then_minors_are_pure}.
\end{proof}

\section{The rank $3$ case}\label{sect:rk_3}

The goal of this section is to prove one direction of Theorem~\ref{thm:purity_3d}, namely, that all positively oriented matroids of rank $3$ are pure. The case of the rank $3$ \emph{uniform} positively oriented matroid $C^{n,3}$ is done in Corollary~\ref{cor:c23narepure}.

The following lemma is well known, see, e.g.,~\cite{Postnikov} or~\cite[Example~3.3]{ARW}.
\begin{lemma}\label{lemma:positroids}
 Suppose that a simple oriented matroid $\Mcal$ of rank $3$ is isomorphic to a positively oriented matroid. Then $\Mcal$ is isomorphic to $\Mcal_\VC$ where the endpoints of vectors of $\VC$ all belong to the boundary of a convex $n$-gon lying in the plane $z=1$.  \qed
\end{lemma}
% \footnote{Note that this lemma is only true as long as $\Mcal$ is simple, see~\cite[Figure~2]{ARW} for a non-simple counterexample.} %this is not really the case
 
We call such vector configurations $\VC$ \emph{totally nonnegative}. If all endpoints of $\VC$ are vertices of that convex $n$-gon then we say that $\VC$ is \emph{totally positive}. Thus, Corollary~\ref{cor:c23narepure} shows that all three-dimensional totally positive vector configurations are pure and our goal is to generalize this result to totally nonnegative vector configurations:

\begin{theorem}\label{thm:purity_vc}
 Let $\VC\subset\R^3$ be a totally nonnegative vector configuration. Then the map $\Tiling\mapsto \Vert(\Tiling)$ is a bijection between fine zonotopal tilings of $\Zon_\VC$ and maximal \emph{by inclusion} $\VC$-separated collections of subsets. 
\end{theorem}
\begin{proof}

Let $[n]$ be the ground set of $\VC$. Since all endpoints of the vectors of $\VC$ belong to the plane $z=1$, we can restrict our attention to this plane. Moreover, by Lemma~\ref{lemma:parallel}, we can assume that the oriented matroid $\Mcal_\VC$ is simple. We let $\PC=(\x_1,\x_2,\dots,\x_n)\subset\R^2$ be the \emph{affine point configuration} corresponding to $\VC$, that is, the points of $\PC$ are the endpoints of the vectors in $\VC$. Two subsets $S,T\subset [n]$ are \emph{$\PC$-separated} if and only if 
\[\Conv(S-T)\cap\Conv(T-S)=\emptyset,\]
where $\Conv(R)$ denotes the convex hull of a set $\{\x_i\mid i\in R\}\subset \PC$. It is clear two subsets are $\PC$-separated if and only if they are $\VC$-separated.

As we have already mentioned, Corollary~\ref{cor:c23narepure} implies the result for the case when the points of $\PC$ are the vertices of a convex $n$-gon $P$, and we need to show purity for the case when the points of $\PC$ are either the vertices of $P$ or belong to the edges of $P$.

We proceed by induction on the number of points in $\PC$, and for each fixed number of points we proceed by ``reverse induction'' on the number of sides of $P$: we start with the polygon with $|\PC|$ sides, for which we already know the result, and then do the induction step from polygons with $p$ sides to polygons with $p-1$ sides. 

Assume that $\x_1,\x_2,\dots,\x_n$ are cyclically ordered on the boundary of $P$ which is a polygon with $p-1<n$ sides. Assume also that one of the sides contains the points $\x_1,\x_2,\dots,\x_r$ for some $r>2$ but does not contain any other points. This is always achievable by a cyclic shift of vertex indices. Let 
\[\x'_1:=\frac12(\x_1+\x_n),\]
thus we have a non-degenerate triangle with vertices $\x'_1,\x_2,\x_n$. Define $\PC':=\{\x'_1,\x_2,\dots,\x_n\}$. Then it is clear that $\conv(\PC')$ is a polygon with $p$ sides, so by the induction hypothesis, any maximal \emph{by inclusion} $\PC'$-separated collection corresponds to a zonotopal tiling of $\Zon_{\VC'}$, where $\VC'$ is the vector configuration corresponding to $\PC'$.

Let $\WS\subset 2^{[n]}$ be any maximal \emph{by inclusion} $\PC$-separated
collection. If $\WS$ is complete then
by~Proposition~\ref{prop:complete_properties},
part~(\ref{item:complete_colocalization}), $\sigma_\WS$ is a colocalization in
general position, and thus by Theorem~\ref{thm:max_size_matroid}, $\WS$ is maximal \emph{by size} so we are
done. Thus assume that $\WS$ is not complete, in which case $\sigma_\WS$ has
some zeroes in the image so we cannot conclude that it is a colocalization. Let
$\Mcal$ and $\Mcal'$ be the oriented matroids on the ground set $[n]$
corresponding to $\PC$ and $\PC'$ respectively with circuits $\Ccal$ and
$\Ccal'$, respectively. It is apparent from the definition that there is a weak
map $\Mcal'\leadsto\Mcal$, which implies the first part of the following lemma.

\begin{lemma}\label{lemma:circuit_diff}\leavevmode
\begin{enumerate}[\normalfont (1)]
 \item If any two sets are $\Mcal$-separated then they are $\Mcal'$-separated.
 \item If $S,T$ are $\Mcal'$-separated but not $\Mcal$-separated then there is a circuit $X\in\Ccal$ such that $1\in\Xu\subset [r]$, $X^+\subset T-S$ and $X^-\subset S-T$. 
 \item Conversely, all circuits that satisfy $X^+\subset T-S$ and $X^-\subset S-T$ must also satisfy $1\in\Xu\subset [r]$.
\end{enumerate}
\end{lemma}
\begin{proof}
 The first claim is obvious, the second and the third claims follow from the construction of $\PC'$, since $\Ccal-\Ccal'$ consists exactly of all circuits of $\Mcal$ containing $1$ whose support belongs to $[r]$.
\end{proof}

\begin{lemma}\label{lemma:flats_Cpm}
For any $C\in\Ccal$, we have
\begin{equation}\label{eq:flats_Cpm}
  C^+\subset[r] \Longleftrightarrow C^-\subset[r].
\end{equation}
\end{lemma}
\begin{proof}
Follows from the fact that $[r]$ is a flat of $\Mcal$.
\end{proof}

\begin{lemma}\label{lemma:sigma_orients_r}
The restriction $\WS([r]):=\{S\cap [r]\mid S\in\WS\}$ is a maximal \emph{by size} $\Mcal\mid_{[r]}$-separated collection. Moreover, $\WS$ contains 
% There exists a maximal \emph{by size} $\Mcal\mid_{[r]}$-separated collection $\WS_r$ such that $\WS$ contains
\begin{equation*}%\label{eq:*}
  \WS([r])\cup \{T\sqcup[r+1,n]\mid T\in \WS([r])\}.
\end{equation*}
% Let $C\in\Ccal$ be a circuit of $\Mcal$ such that $\Cu\subset[r]$. Then $\sigma_\WS(C)\neq0$. Moreover, if $\sigma_\WS(C)=+$ then $C^+,C^+\sqcup[r+1,n]\in \WS$. (Similarly, if $\sigma_\WS(C)=-$ then $C^-,C^-\sqcup[r+1,n]\in \WS$.)
\end{lemma}
\begin{proof}
% Let $\WS([r]):=\{S\cap [r]\mid S\in\WS\}$ be the restriction of $\WS$ to $[r]$. It is $\Mcal\mid_{[r]}$-separated. 
Since the rank of $\Mcal\mid_{[r]}$ is $2$, it is pure by Corollary~\ref{cor:c23narepure}, so let $\WS_r$ be a maximal by inclusion (and by size) $\Mcal\mid_{[r]}$-separated collection. Let $T\in\WS_r$. We claim that both $T$ and $T':=T\sqcup[r+1,n]$ are $\Mcal$-separated from $\WS$. Indeed, suppose that $S\in\WS$ is not $\Mcal$-separated from $T$. Let $X\in\Ccal$ be such that $X^+\subset S-T$ and $X^-\subset T-S$. Then $X^-\subset[r]$, and therefore by~\eqref{eq:flats_Cpm} we must have $X^+\subset[r]$. Thus $T$ is not $\Mcal\mid_{[r]}$-separated from $S\cap [r]$, a contradiction. Similarly, suppose that $S\in\WS$ is not $\Mcal$-separated from $T'$. Let $X\in\Ccal$ be such that $X^+\subset S-T'$ and $X^-\subset T'-S$. Then $X^+\subset[r]$, and therefore by~\eqref{eq:flats_Cpm} we must have $X^-\subset[r]$, which again leads to a contradiction.
% $T\subset[r]$ be a set that orients $C$ and is $\Mcal\mid_{[r]}$-separated from $\WS([r])$. We claim that $T$ is $\Mcal$-separated from $\WS$. Indeed, suppose otherwise that some $S\in\WS$ is not $\Mcal$-separated from $T$. Let $X\in\Ccal$ be such that $X^+\subset S-T$ and $X^-\subset T-S$. Then $X^-\subset[r]$, and therefore we must have $X^+\subset[r]$. But then $T$ is also not $\Mcal\mid_{[r]}$-separated from $S\cap [r]$, a contradiction. 
\end{proof}

% By Lemmas~\ref{lemma:circuit_diff} and~\ref{lemma:sigma_orients_r}, we have $\sigma_\WS(C)\neq 0$ for all $C\in\Ccal$

 Let $\WS'$ be a maximal \emph{by inclusion} $\Mcal'$-separated collection that contains $\WS$. Then by the induction hypothesis we know that $\WS'$ is complete and that $\sigma_{\WS'}$ is a colocalization in general position. Our goal is to show that $\WS$ is also complete. Suppose otherwise that $\sigma_\WS(C)=0$ for some $C\in\Ccal$. We may assume that $\sigma_{\WS'}(C)=+$, thus let $S\in \WS'$ be a set orienting $C$ positively. By Lemma~\ref{lemma:circuit_diff}, we see that $\Cu\not\subset[r]$, thus by~\eqref{eq:flats_Cpm}, we get
\begin{equation}\label{eq:Cpm_not_subset_r}
  C^+\not\subset[r] \quad\text{and}\quad C^-\not\subset[r].
\end{equation}
Since $S\notin \WS'$, there exists a set $T\in\WS$ that is not $\Mcal$-separated from $S$. By Lemma~\ref{lemma:circuit_diff}, there exists a circuit $X\in\Ccal$ satisfying
\begin{equation*}%\label{eq:*}
  1\in\Xu\subset[r],\quad X^+\subset T-S,\quad\text{and}\quad X^-\subset S-T.
\end{equation*}
By Lemma~\ref{lemma:sigma_orients_r}, we may replace $T$ with $T\cap [r]$. Thus we have $T\subset [r]$ and denote $T':=T\sqcup [r+1,n]$. 

Let $\Xu=\{1,i,j\}$ for some $1<i<j\leq r$. We consider two cases: $X=(\{1,j\},\{i\})$ and $X=(\{i\},\{1,j\})$. Assume first that $X=(\{1,j\},\{i\})$. Thus $1,j\in T-S$ and $i\in S-T$. By~\eqref{eq:Cpm_not_subset_r}, there exists $k\in C^+\setminus [r]$. Since $S$ orients $C$ positively and $T\subset[r]$, we get $k\in S-T$. Since $(\{1,j\},\{i,k\})\in\Ccal'$, we find that $S$ and $T$ are not $\Mcal'$-separated, which is a contradiction. Assume now that $X=(\{i\},\{1,j\})$. Thus $i\in T-S$ and $1,j\in S-T$. By~\eqref{eq:Cpm_not_subset_r}, there exists $k\in C^-\setminus [r]$. Since $S$ orients $C$ positively and $T'=T\sqcup [r+1,n]$, we get $k\in S-T'$. Since $(\{1,j\},\{i,k\})\in\Ccal'$, we find that $S$ and $T'$ are not $\Mcal'$-separated, which again yields a contradiction. 

We have shown that $\WS$ is complete. As explained above, this finishes the proof of Theorem~\ref{thm:purity_vc}.
% and Theorem~\ref{thm:purity_vc} now follows from Proposition~\ref{prop:complete_properties}, part~(\ref{item:complete_colocalization}). 
\end{proof}

\section{Classification results}\label{sect:classif}

In this section, we prove the remaining parts of Theorems~\ref{thm:uniform_classification} and~\ref{thm:purity_3d}, as well as some other results that we announced in the earlier sections. We start by showing Proposition~\ref{prop:rk_2_cork_1}:
\begin{proposition}\label{prop:rank_2}
 All simple oriented matroids of rank at most $2$ or corank at most $1$ are pure. 
\end{proposition}
\begin{proof}
 If $\Mcal$ has corank $1$ then it has just one pair of opposite circuits $\pm C=\pm (C^+,C^-)$ supported on the whole ground set. Thus any two subsets of $E$ are $\Mcal$-separated from each other except for $C^+$ and $C^-$. The purity of $\Mcal$ follows.
 
 If $\Mcal$ has corank $0$ then there are no circuits so any two subsets of $E$ are $\Mcal$-separated from each other.
 
 If $\Mcal$ (note that it is simple) has rank $2$ then it is isomorphic to the alternating matroid $C^{n,2}$ for some $n$. As we have noted before, in this case $\Mcal$-separation is the same thing as strong separation and thus the purity of $\Mcal$ is a special case of Theorem~\ref{thm:purity_ss}.
 
 If $\Mcal$ has rank $0$ or $1$ and is simple, it means that it has at most one element and the result is trivial. 
\end{proof}

Next, we analyze which of the six-element oriented matroids are pure. Note that for an oriented matroid $\Mcal$ with $|E|=6$ elements, if $\rk(\Mcal)=0,1,2,5,$ or $6$ then $\Mcal$ is pure by the above proposition. Also, for $\rk(\Mcal)\geq 4$ we only care about the case of $\Mcal$ being uniform. 
As we will see later in Lemma~\ref{lemma:minors_6_4}, all  uniform oriented matroids of rank $4$ and corank $2$ are isomorphic to $C^{6,4}$.
\begin{lemma}\label{lemma:six_elements}
\begin{enumerate}[\normalfont (1)]
 \item The alternating matroid $C^{6,4}$ is non-pure;
 \item There are $17$ isomorphism classes of simple oriented matroids of rank $3$ with $6$ elements. Eight of them (Figure~\ref{fig:positroids}) are positively oriented and therefore are pure. The other nine of them (Figure~\ref{fig:non_postiroids}) are not pure. 
\end{enumerate} 
\end{lemma}
\begin{proof}
 The first claim has already been mentioned in the end of the proof of Proposition~\ref{prop:complete_properties}. Recall that $\altn$ defines an oriented matroid isomorphic to $C^{n,n-2}$.  Take $\WS=\{\emptyset,[4],[6]-[2]\}\subset 2^{[6]}$. It is complete but it does not define a colocalization of $\altsixfour$ because it does not have Type III so it is not contained in any maximal by size $\altsixfour$-separated collection.
 
 The second claim is a computational fact. It is easy to list all the totally nonnegative point configurations, see Figure~\ref{fig:positroids}, and then for each of the remaining nine oriented matroids one needs to construct an $\Mcal$-separated collection that is not contained in any complete $\Mcal$-separated collection. We list this data in Figure~\ref{fig:non_postiroids}. Namely, for each of the nine oriented matroids $\Mcal$ that are not positively oriented, we specify a \emph{bad collection} $\WS_0$ and a \emph{bad circuit} $C\in\Ccal(\Mcal)$ with the following property: any maximal \emph{by inclusion} $\Mcal$-separated collection containing the bad collection $\WS_0$ cannot orient the bad circuit $C$ either negatively or positively. Let us give an example. Take $\Mcal=\IC(6,3,13)$ from Figure~\ref{fig:non_postiroids}. Then the bad circuit is $C=(6,124)$ and the bad collection is 
 \[\WS_0=\{456,1356,2345,12346\}.\]
 First, note that $\WS_0$ is $\Mcal$-separated, e.g., the sets $1356$ and $2345$ are $\Mcal$-separated because the segments $16$ and $24$ do not intersect each other in Figure~\ref{fig:non_postiroids}. 
 
 Now, there are four subsets of $[6]$ that orient $C$ positively, namely, $6,36,56,356$. Similarly, there are four subsets of $[6]$ that orient $C$ negatively: $124,1234,1245,12345$. One easily checks that for each such set $S$, there is at least one subset $T\in\WS_0$ that is \emph{not} $\Mcal$-separated from~$S$:

 \begin{center}
 \begin{tabular}{|c|c|c|c|c|c|c|c|c|}\hline
  $S$ &$6      $&$ 36    $&$ 56    $&$ 356   $&$ 124   $&$1234   $&$1245   $&$12345  $ \\\hline
  $T$ &$2345   $&$2345   $&$12346  $&$12346  $&$1356   $&$456    $&$1356   $&$456    $ \\\hline           
 \end{tabular}
  \end{center}
 
This shows that the oriented matroid $\IC(6,3,13)$ is not pure. The same argument applied to the other eight oriented matroids in Figure~\ref{fig:non_postiroids} finishes the proof of Lemma~\ref{lemma:six_elements}. 
\end{proof}

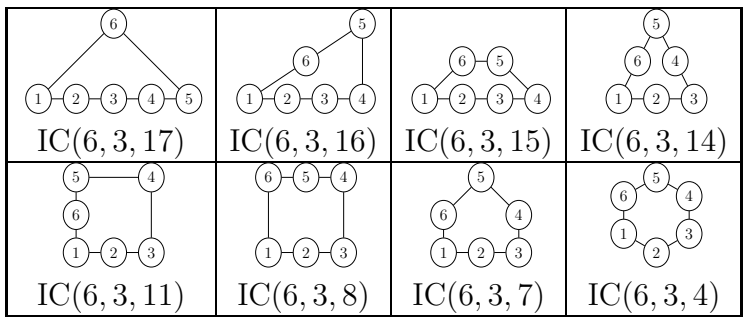
\begin{figure}
 \begin{tabular}{|c|c|c|c|}\hline
\scalebox{0.5}{
\begin{tikzpicture}
\node[draw,ellipse,black,fill=white] (node1) at (0.00,0.00) {$1$};
\node[draw,ellipse,black,fill=white] (node2) at (1.00,0.00) {$2$};
\node[draw,ellipse,black,fill=white] (node3) at (2.00,0.00) {$3$};
\node[draw,ellipse,black,fill=white] (node4) at (3.00,0.00) {$4$};
\node[draw,ellipse,black,fill=white] (node5) at (4.00,0.00) {$5$};
\node[draw,ellipse,black,fill=white] (node6) at (2.00,2.00) {$6$};
\draw[line width=0.25mm,black] (node1) -- (node2);
\draw[line width=0.25mm,black] (node2) -- (node3);
\draw[line width=0.25mm,black] (node3) -- (node4);
\draw[line width=0.25mm,black] (node4) -- (node5);
\draw[line width=0.25mm,black] (node5) -- (node6);
\draw[line width=0.25mm,black] (node6) -- (node1);
\end{tikzpicture}}
&
\scalebox{0.5}{
\begin{tikzpicture}
\node[draw,ellipse,black,fill=white] (node1) at (0.00,0.00) {$1$};
\node[draw,ellipse,black,fill=white] (node2) at (1.00,0.00) {$2$};
\node[draw,ellipse,black,fill=white] (node3) at (2.00,0.00) {$3$};
\node[draw,ellipse,black,fill=white] (node4) at (3.00,0.00) {$4$};
\node[draw,ellipse,black,fill=white] (node5) at (3.00,2.00) {$5$};
\node[draw,ellipse,black,fill=white] (node6) at (1.50,1.00) {$6$};
\draw[line width=0.25mm,black] (node1) -- (node2);
\draw[line width=0.25mm,black] (node2) -- (node3);
\draw[line width=0.25mm,black] (node3) -- (node4);
\draw[line width=0.25mm,black] (node4) -- (node5);
\draw[line width=0.25mm,black] (node5) -- (node6);
\draw[line width=0.25mm,black] (node6) -- (node1);
\end{tikzpicture}}
&
\scalebox{0.5}{
\begin{tikzpicture}
\node[draw,ellipse,black,fill=white] (node1) at (0.00,0.00) {$1$};
\node[draw,ellipse,black,fill=white] (node2) at (1.00,0.00) {$2$};
\node[draw,ellipse,black,fill=white] (node3) at (2.00,0.00) {$3$};
\node[draw,ellipse,black,fill=white] (node4) at (3.00,0.00) {$4$};
\node[draw,ellipse,black,fill=white] (node5) at (2.00,1.00) {$5$};
\node[draw,ellipse,black,fill=white] (node6) at (1.00,1.00) {$6$};
\draw[line width=0.25mm,black] (node1) -- (node2);
\draw[line width=0.25mm,black] (node2) -- (node3);
\draw[line width=0.25mm,black] (node3) -- (node4);
\draw[line width=0.25mm,black] (node4) -- (node5);
\draw[line width=0.25mm,black] (node5) -- (node6);
\draw[line width=0.25mm,black] (node6) -- (node1);
\end{tikzpicture}}
&
\scalebox{0.5}{
\begin{tikzpicture}
\node[draw,ellipse,black,fill=white] (node1) at (0.00,0.00) {$1$};
\node[draw,ellipse,black,fill=white] (node2) at (1.00,0.00) {$2$};
\node[draw,ellipse,black,fill=white] (node3) at (2.00,0.00) {$3$};
\node[draw,ellipse,black,fill=white] (node4) at (1.50,1.00) {$4$};
\node[draw,ellipse,black,fill=white] (node5) at (1.00,2.00) {$5$};
\node[draw,ellipse,black,fill=white] (node6) at (0.50,1.00) {$6$};
\draw[line width=0.25mm,black] (node1) -- (node2);
\draw[line width=0.25mm,black] (node2) -- (node3);
\draw[line width=0.25mm,black] (node3) -- (node4);
\draw[line width=0.25mm,black] (node4) -- (node5);
\draw[line width=0.25mm,black] (node5) -- (node6);
\draw[line width=0.25mm,black] (node6) -- (node1);
\end{tikzpicture}}\\
$\IC(6,3,17)$ &$\IC(6,3,16)$ &$\IC(6,3,15)$ &$\IC(6,3,14)$ \\\hline
\scalebox{0.5}{
\begin{tikzpicture}
\node[draw,ellipse,black,fill=white] (node1) at (0.00,0.00) {$1$};
\node[draw,ellipse,black,fill=white] (node2) at (1.00,0.00) {$2$};
\node[draw,ellipse,black,fill=white] (node3) at (2.00,0.00) {$3$};
\node[draw,ellipse,black,fill=white] (node4) at (2.00,2.00) {$4$};
\node[draw,ellipse,black,fill=white] (node5) at (0.00,2.00) {$5$};
\node[draw,ellipse,black,fill=white] (node6) at (0.00,1.00) {$6$};
\draw[line width=0.25mm,black] (node1) -- (node2);
\draw[line width=0.25mm,black] (node2) -- (node3);
\draw[line width=0.25mm,black] (node3) -- (node4);
\draw[line width=0.25mm,black] (node4) -- (node5);
\draw[line width=0.25mm,black] (node5) -- (node6);
\draw[line width=0.25mm,black] (node6) -- (node1);
\end{tikzpicture}}
&
\scalebox{0.5}{
\begin{tikzpicture}
\node[draw,ellipse,black,fill=white] (node1) at (0.00,0.00) {$1$};
\node[draw,ellipse,black,fill=white] (node2) at (1.00,0.00) {$2$};
\node[draw,ellipse,black,fill=white] (node3) at (2.00,0.00) {$3$};
\node[draw,ellipse,black,fill=white] (node4) at (2.00,2.00) {$4$};
\node[draw,ellipse,black,fill=white] (node5) at (1.00,2.00) {$5$};
\node[draw,ellipse,black,fill=white] (node6) at (0.00,2.00) {$6$};
\draw[line width=0.25mm,black] (node1) -- (node2);
\draw[line width=0.25mm,black] (node2) -- (node3);
\draw[line width=0.25mm,black] (node3) -- (node4);
\draw[line width=0.25mm,black] (node4) -- (node5);
\draw[line width=0.25mm,black] (node5) -- (node6);
\draw[line width=0.25mm,black] (node6) -- (node1);
\end{tikzpicture}}
&
\scalebox{0.5}{
\begin{tikzpicture}
\node[draw,ellipse,black,fill=white] (node1) at (0.00,0.00) {$1$};
\node[draw,ellipse,black,fill=white] (node2) at (1.00,0.00) {$2$};
\node[draw,ellipse,black,fill=white] (node3) at (2.00,0.00) {$3$};
\node[draw,ellipse,black,fill=white] (node4) at (2.00,1.00) {$4$};
\node[draw,ellipse,black,fill=white] (node5) at (1.00,2.00) {$5$};
\node[draw,ellipse,black,fill=white] (node6) at (0.00,1.00) {$6$};
\draw[line width=0.25mm,black] (node1) -- (node2);
\draw[line width=0.25mm,black] (node2) -- (node3);
\draw[line width=0.25mm,black] (node3) -- (node4);
\draw[line width=0.25mm,black] (node4) -- (node5);
\draw[line width=0.25mm,black] (node5) -- (node6);
\draw[line width=0.25mm,black] (node6) -- (node1);
\end{tikzpicture}}
&
\scalebox{0.5}{
\begin{tikzpicture}
\node[draw,ellipse,black,fill=white] (node1) at (-0.87,-0.50) {$1$};
\node[draw,ellipse,black,fill=white] (node2) at (-0.00,-1.00) {$2$};
\node[draw,ellipse,black,fill=white] (node3) at (0.87,-0.50) {$3$};
\node[draw,ellipse,black,fill=white] (node4) at (0.87,0.50) {$4$};
\node[draw,ellipse,black,fill=white] (node5) at (0.00,1.00) {$5$};
\node[draw,ellipse,black,fill=white] (node6) at (-0.87,0.50) {$6$};
\draw[line width=0.25mm,black] (node1) -- (node2);
\draw[line width=0.25mm,black] (node2) -- (node3);
\draw[line width=0.25mm,black] (node3) -- (node4);
\draw[line width=0.25mm,black] (node4) -- (node5);
\draw[line width=0.25mm,black] (node5) -- (node6);
\draw[line width=0.25mm,black] (node6) -- (node1);
\end{tikzpicture}}
\\
$\IC(6,3,11)$ &$\IC(6,3,8)$ &$\IC(6,3,7)$ &$\IC(6,3,4)$ \\\hline
 \end{tabular}
  \caption{\label{fig:positroids} Affine diagrams of the eight (isomorphism classes of) positively oriented matroids of rank $3$ with $6$ elements and their names from~\cite{Finschi}.}
\end{figure}

\begin{figure}
  \scalebox{0.8}{
 \begin{tabular}{|c|c|c|c|}\hline
  &
\scalebox{0.5}{
\begin{tikzpicture}
\node[draw,ellipse,blue,fill=white] (node1) at (0.00,0.00) {$1$};
\node[draw,ellipse,blue,fill=white] (node2) at (4.00,0.00) {$2$};
\node[draw,ellipse,black,fill=white] (node3) at (3.00,2.00) {$3$};
\node[draw,ellipse,blue,fill=white] (node4) at (2.00,4.00) {$4$};
\node[draw,ellipse,black,fill=white] (node5) at (1.00,2.00) {$5$};
\node[draw,ellipse,red,fill=white] (node6) at (2.00,1.33) {$6$};
\draw[line width=0.25mm,black] (node1) -- (node2);
\draw[line width=0.25mm,black] (node2) -- (node3);
\draw[line width=0.25mm,black] (node3) -- (node4);
\draw[line width=0.25mm,black] (node4) -- (node5);
\draw[line width=0.25mm,black] (node5) -- (node6);
\draw[line width=0.25mm,black] (node6) -- (node1);
\draw[line width=0.25mm,black] (node2) -- (node6);
\draw[line width=0.25mm,black] (node3) -- (node6);
\draw[line width=0.25mm,black] (node1) -- (node5);
\end{tikzpicture}}
&
\scalebox{0.5}{
\begin{tikzpicture}
\node[draw,ellipse,blue,fill=white] (node1) at (0.00,0.00) {$1$};
\node[draw,ellipse,black,fill=white] (node2) at (4.00,0.00) {$2$};
\node[draw,ellipse,blue,fill=white] (node3) at (4.00,2.00) {$3$};
\node[draw,ellipse,blue,fill=white] (node4) at (0.00,4.00) {$4$};
\node[draw,ellipse,black,fill=white] (node5) at (0.00,2.00) {$5$};
\node[draw,ellipse,red,fill=white] (node6) at (2.00,2.00) {$6$};
\draw[line width=0.25mm,black] (node1) -- (node2);
\draw[line width=0.25mm,black] (node2) -- (node3);
\draw[line width=0.25mm,black] (node3) -- (node4);
\draw[line width=0.25mm,black] (node4) -- (node5);
\draw[line width=0.25mm,black] (node5) -- (node6);
\draw[line width=0.25mm,black] (node2) -- (node6);
\draw[line width=0.25mm,black] (node3) -- (node6);
\draw[line width=0.25mm,black] (node5) -- (node1);
\draw[line width=0.25mm,black] (node4) -- (node6);
\end{tikzpicture}}
&
\scalebox{0.5}{
\begin{tikzpicture}
\node[draw,ellipse,blue,fill=white] (node1) at (0.00,0.00) {$1$};
\node[draw,ellipse,blue,fill=white] (node2) at (4.00,0.00) {$2$};
\node[draw,ellipse,black,fill=white] (node3) at (4.00,2.67) {$3$};
\node[draw,ellipse,blue,fill=white] (node4) at (2.00,4.00) {$4$};
\node[draw,ellipse,black,fill=white] (node5) at (0.00,2.67) {$5$};
\node[draw,ellipse,red,fill=white] (node6) at (2.00,1.33) {$6$};
\draw[line width=0.25mm,black] (node1) -- (node2);
\draw[line width=0.25mm,black] (node2) -- (node3);
\draw[line width=0.25mm,black] (node3) -- (node4);
\draw[line width=0.25mm,black] (node4) -- (node5);
\draw[line width=0.25mm,black] (node5) -- (node6);
\draw[line width=0.25mm,black] (node6) -- (node1);
\draw[line width=0.25mm,black] (node1) -- (node5);
\draw[line width=0.25mm,black] (node1) -- (node6);
\draw[line width=0.25mm,black] (node2) -- (node6);
\draw[line width=0.25mm,black] (node3) -- (node6);
\end{tikzpicture}}

\\
Name in~\cite{Finschi} & $\IC(6,3,13)$               &$\IC(6,3,12)$ &$\IC(6,3,10)$ \\
Bad circuit            & $(6,124)$                   & $(6,134)$   &  $(6,124)$  \\
Bad collection         & $456,1356,2345,12346$       & $13,126,1245,2356$            & $24,346,2356,1345$  \\\hline

  &
  \scalebox{0.5}{
\begin{tikzpicture}
\node[draw,ellipse,blue,fill=white] (node1) at (0.00,0.00) {$1$};
\node[draw,ellipse,black,fill=white] (node2) at (2.67,0.00) {$2$};
\node[draw,ellipse,blue,fill=white] (node3) at (4.00,0.00) {$3$};
\node[draw,ellipse,black,fill=white] (node4) at (2.67,4.00) {$4$};
\node[draw,ellipse,blue,fill=white] (node5) at (1.33,4.00) {$5$};
\node[draw,ellipse,red,fill=white] (node6) at (1.33,2.00) {$6$};
\draw[line width=0.25mm,black] (node1) -- (node2);
\draw[line width=0.25mm,black] (node2) -- (node3);
\draw[line width=0.25mm,black] (node3) -- (node4);
\draw[line width=0.25mm,black] (node4) -- (node5);
\draw[line width=0.25mm,black] (node6) -- (node1);
\draw[line width=0.25mm,black] (node4) -- (node6);
\draw[line width=0.25mm,black] (node2) -- (node5);
\draw[line width=0.25mm,black] (node1) -- (node5);
\end{tikzpicture}}
&
\scalebox{0.5}{
\begin{tikzpicture}
\node[draw,ellipse,blue,fill=white] (node1) at (0.00,0.00) {$1$};
\node[draw,ellipse,black,fill=white] (node2) at (2.00,0.00) {$2$};
\node[draw,ellipse,blue,fill=white] (node3) at (4.00,1.00) {$3$};
\node[draw,ellipse,blue,fill=white] (node4) at (2.00,4.00) {$4$};
\node[draw,ellipse,black,fill=white] (node5) at (0.00,1.00) {$5$};
\node[draw,ellipse,red,fill=white] (node6) at (2.00,2.00) {$6$};
\draw[line width=0.25mm,black] (node1) -- (node2);
\draw[line width=0.25mm,black] (node2) -- (node3);
\draw[line width=0.25mm,black] (node3) -- (node4);
\draw[line width=0.25mm,black] (node4) -- (node5);
\draw[line width=0.25mm,black] (node2) -- (node6);
\draw[line width=0.25mm,black] (node3) -- (node5);
\draw[line width=0.25mm,black] (node4) -- (node6);
\draw[line width=0.25mm,black] (node1) -- (node5);
\end{tikzpicture}}
&
\scalebox{0.5}{
\begin{tikzpicture}
\node[draw,ellipse,black,fill=white] (node1) at (0.00,0.00) {$1$};
\node[draw,ellipse,blue,fill=white] (node2) at (4.00,0.00) {$2$};
\node[draw,ellipse,black,fill=white] (node3) at (4.00,2.00) {$3$};
\node[draw,ellipse,blue,fill=white] (node4) at (2.00,4.00) {$4$};
\node[draw,ellipse,blue,fill=white] (node5) at (0.00,2.00) {$5$};
\node[draw,ellipse,red,fill=white] (node6) at (2.00,2.00) {$6$};
\draw[line width=0.25mm,black] (node1) -- (node2);
\draw[line width=0.25mm,black] (node2) -- (node3);
\draw[line width=0.25mm,black] (node3) -- (node4);
\draw[line width=0.25mm,black] (node4) -- (node5);
\draw[line width=0.25mm,black] (node5) -- (node6);
\draw[line width=0.25mm,black] (node3) -- (node6);
\draw[line width=0.25mm,black] (node1) -- (node5);
\draw[line width=0.25mm,black] (node1) -- (node4);
\draw[line width=0.25mm,black] (node2) -- (node4);
\end{tikzpicture}}
\\
Name in~\cite{Finschi} & $\IC(6,3,9)$               &$\IC(6,3,6)$ &$\IC(6,3,5)$ \\
Bad circuit            & $(6,135)$                   & $(6,134)$   &  $(6,245)$  \\
Bad collection         & $35,346,1234,2456$       & $13,126,2356,1245$            & $25,126,1356,1234$  \\\hline
 &
 \scalebox{0.5}{
\begin{tikzpicture}
\node[draw,ellipse,blue,fill=white] (node1) at (0.00,0.00) {$1$};
\node[draw,ellipse,blue,fill=white] (node2) at (4.00,0.00) {$2$};
\node[draw,ellipse,black,fill=white] (node3) at (4.00,3.00) {$3$};
\node[draw,ellipse,blue,fill=white] (node4) at (2.00,4.00) {$4$};
\node[draw,ellipse,black,fill=white] (node5) at (0.00,3.00) {$5$};
\node[draw,ellipse,red,fill=white] (node6) at (2.00,0.70) {$6$};
\draw[line width=0.25mm,black] (node1) -- (node2);
\draw[line width=0.25mm,black] (node2) -- (node3);
\draw[line width=0.25mm,black] (node3) -- (node4);
\draw[line width=0.25mm,black] (node4) -- (node5);
\draw[line width=0.25mm,black] (node1) -- (node3);
\draw[line width=0.25mm,black] (node2) -- (node5);
\draw[line width=0.25mm,black] (node1) -- (node5);
\end{tikzpicture}}
&
\scalebox{0.5}{
\begin{tikzpicture}
\node[draw,ellipse,black,fill=white] (node1) at (0.00,0.00) {$1$};
\node[draw,ellipse,blue,fill=white] (node2) at (4.00,0.00) {$2$};
\node[draw,ellipse,black,fill=white] (node3) at (4.00,1.33) {$3$};
\node[draw,ellipse,blue,fill=white] (node4) at (2.00,4.00) {$4$};
\node[draw,ellipse,blue,fill=white] (node5) at (0.00,1.33) {$5$};
\node[draw,ellipse,red,fill=white] (node6) at (2.00,2.67) {$6$};
\draw[line width=0.25mm,black] (node1) -- (node2);
\draw[line width=0.25mm,black] (node2) -- (node3);
\draw[line width=0.25mm,black] (node3) -- (node4);
\draw[line width=0.25mm,black] (node4) -- (node5);
\draw[line width=0.25mm,black] (node1) -- (node5);
\draw[line width=0.25mm,black] (node3) -- (node5);
\end{tikzpicture}}
&
\scalebox{0.5}{
\begin{tikzpicture}
\node[draw,ellipse,blue,fill=white] (node1) at (-1.90,-0.62) {$1$};
\node[draw,ellipse,blue,fill=white] (node2) at (-0.00,-2.00) {$2$};
\node[draw,ellipse,black,fill=white] (node3) at (1.90,-0.62) {$3$};
\node[draw,ellipse,blue,fill=white] (node4) at (1.18,1.62) {$4$};
\node[draw,ellipse,black,fill=white] (node5) at (-1.18,1.62) {$5$};
\node[draw,ellipse,red,fill=white] (node6) at (0.00,0.00) {$6$};
\draw[line width=0.25mm,black] (node1) -- (node2);
\draw[line width=0.25mm,black] (node2) -- (node3);
\draw[line width=0.25mm,black] (node3) -- (node4);
\draw[line width=0.25mm,black] (node4) -- (node5);
\draw[line width=0.25mm,black] (node1) -- (node3);
\draw[line width=0.25mm,black] (node1) -- (node4);
\draw[line width=0.25mm,black] (node2) -- (node5);
\draw[line width=0.25mm,black] (node2) -- (node4);
\draw[line width=0.25mm,black] (node3) -- (node5);
\draw[line width=0.25mm,black] (node1) -- (node5);
\end{tikzpicture}}

\\

Name in~\cite{Finschi} & $\IC(6,3,3)$               &$\IC(6,3,2)$ &$\IC(6,3,1)$ \\
Bad circuit            & $(6,124)$                   & $(6,245)$   &  $(6,124)$  \\
Bad collection         & $14,456,125,1356$       & $56,24,345,1346$            & $14,26,2345,1356$  \\\hline

 \end{tabular}
 }
  \caption{\label{fig:non_postiroids} Affine diagrams of the remaining nine (isomorphism classes of) oriented matroids of rank $3$ with $6$ elements. The red and blue vertices form a bad circuit.}
\end{figure}

\begin{lemma}\label{lemma:minors_6_4}
 If $\Mcal$ is a uniform oriented matroid with $\rk(\Mcal)\geq 4$ and $\cork(\Mcal)\geq 2$ then $\Mcal$ contains a minor isomorphic to $C^{6,4}$.
\end{lemma}
\begin{proof}
 We show this by induction on $\cork(\Mcal)$. Suppose $\cork(\Mcal)=2$. Since $\Mcal$ is uniform, its dual is a uniform oriented matroid of rank $2$ with at least $6$ elements. Any such oriented matroid (they are all isomorphic) contains $C^{6,2}$ as a minor, and therefore its dual contains a minor isomorphic to $C^{6,4}$. Now let $\cork(\Mcal)>2$. If every element $e\in E$ is a coloop then $\cork(\Mcal)$ would be zero, so suppose $e\in E$ is not a coloop. Then removing $e$ from $\Mcal$ preserves its rank but decreases its corank by $1$. Therefore $\Mcal-e$ contains $C^{6,4}$ as a minor by the induction hypothesis. We are done with the proof of the lemma.
\end{proof}

This lemma finishes the proof of parts~\eqref{item:uniform_classification:rk_2} and~\eqref{item:uniform_classification:cork_1} of  Theorem~\ref{thm:uniform_classification}. We already have shown one direction of Theorem~\ref{thm:purity_3d}, namely, that every positively oriented matroid of rank $3$ is pure (see Theorem~\ref{thm:purity_vc}). We also have shown the converse for $|E|= 6$ in Lemma~\ref{lemma:six_elements}. (For $|E|\leq 5$, each rank $3$ oriented matroid is already isomorphic to a positively oriented matroid, so there is nothing to prove.) It remains to show that if $|E|\geq 7$ then $\Mcal$ is either isomorphic to a positively oriented matroid or contains one of the six-element non-pure oriented matroids as a minor. In other words, \emph{we need to prove that a rank $3$ oriented matroid is isomorphic to a positively oriented matroid if and only if it does not contain an oriented matroid from Figure~\ref{fig:non_postiroids} as a minor}.

\begin{definition}
 An oriented matroid $\Mcal$ of rank $3$ is called \emph{almost positively oriented} if for every element $e\in E$, $\Mcal-e$ is isomorphic to a positively oriented matroid.
\end{definition}

The rest of the paper is devoted to the proof of the theorem below which clearly implies Theorem~\ref{thm:purity_3d} as well as Theorem~\ref{thm:uniform_classification}, part~\eqref{item:pure_3}.
\begin{theorem}\label{thm:minors}
 Suppose that $\Mcal$ is an almost positively oriented matroid with at least $7$ elements. Then $\Mcal$ is isomorphic to a positively oriented matroid.
\end{theorem}
% \begin{proof}

We prove this theorem via a series of lemmas. Throughout the proof, all oriented matroids are assumed to be simple and to have rank $3$.

\begin{definition}
 Let $\Mcal$ be an oriented matroid of rank $3$. Then any maximal \emph{by inclusion} subset $P\subset E$ of rank $2$ is called a \emph{line}. A line $P$ is called \emph{non-trivial} if $|P|>2$.
\end{definition}
We will mostly work with affine diagrams of rank $3$ oriented matroids (such as the ones in Figures~\ref{fig:positroids} and~\ref{fig:non_postiroids}) where the above defined lines can be represented by lines in the affine diagram.

\def\vert{{ \operatorname{vert}}}

% The following lemma is well known, see e.g.~\cite{Postnikov} or~\cite[Example~3.3]{ARW}.
% \begin{lemma}\label{lemma:positroids}
%  Suppose $\Mcal$ is isomorphic to a positively oriented matroid. Then $\Mcal$ can be realized by an affine diagram where all the points belong to the boundary of a convex polygon.  \qed
% \end{lemma}
% Note that this only works as long as $\Mcal$ is simple, see~\cite[Figure~2]{ARW}.

Recall from Lemma~\ref{lemma:positroids} that a positively oriented matroid $\Mcal$ of rank $3$ can be realized by an affine diagram where all the points belong to the boundary of a convex polygon. The number of vertices of this polygon is called \emph{the number of vertices of $\Mcal$} and denoted $\vert(\Mcal)$. The following corollary is immediate.

\begin{corollary}\label{cor:number_of_vertices}
 Suppose $\Mcal$ is isomorphic to a positively oriented matroid, then
 \[\vert(\Mcal)=|E|-\sum\limits_{P\text{ is a line of $\Mcal$}} (|P|-2).\]
 \qed
\end{corollary}

We will repeatedly use another well known fact.
\begin{proposition}[{\cite[Theorem 8.2.4]{Book} or \cite{GP}}]\label{prop:realizable_8}
 Any oriented matroid of rank $3$ on at most $8$ elements is realizable.\qed
\end{proposition}

\begin{lemma}\label{lemma:line_intersection}
 Let $\Mcal$ be an oriented matroid, and let $P_1\neq P_2$ be any two lines in $\Mcal$. Then 
 \[|P_1\cap P_2|\leq 1.\]
\end{lemma}
\begin{proof}
Since the rank function of $\Mcal$ is submodular, we get $\rk(P_1\cap P_2)\leq 1$. Since $\Mcal$ is assumed to be simple, this gives the result.
\end{proof}

\def\Order{{\mathcal{O}}}

\begin{lemma}\label{lemma:any_element_belongs_to_2_NT_lines}
 Let $\Mcal$ be an almost positively oriented matroid. Then any element of $\Mcal$ belongs to at most two non-trivial lines in $\Mcal$.
\end{lemma}
\begin{proof}
Let $E=\{a_1,a_2,\dots,a_n\}$ and suppose that $P_1,P_2,P_3$ are three non-trivial lines in $\Mcal$ that all contain $a_1$. By Lemma~\ref{lemma:line_intersection}, the sets $P_i-a_1$ are disjoint for  $i=1,2,3$, and each of them contains at least two elements by non-triviality, so without loss of generality we have
\[a_2,a_3\in P_1,\quad a_4,a_5\in P_2,\quad a_6,a_7\in P_3.\]
Consider the restriction $\Mcal'$ of $\Mcal$ to $\{a_1,\dots,a_7\}$. By Proposition~\ref{prop:realizable_8}, $\Mcal'$ is a realizable almost positively oriented matroid. Thus $\Mcal'-a_1$ is isomorphic to a positively oriented matroid, so reorient the elements $a_2,a_3,\dots,a_7$ so that $\Mcal'-a_1$ would be just positively oriented. We know that the points of $\Mcal'-a_1$ belong to the boundary of some convex polygon, and then $a_1$ just belongs to the intersection of the three lines $(a_2a_3)$, $(a_4a_5)$, and $(a_6a_7)$, where $(a_ia_j)$ is the line passing through the points $a_i$ and $a_j$ in the affine diagram of $\Mcal'$. There is a natural cyclic order $\Order$ on the points $a_2,a_3,\dots,a_7$ since they belong to the boundary of a convex polygon. Draw these points on the circle according to $\Order$ and then draw the matching $\{(a_2,a_3),(a_4,a_5),(a_6,a_7)\}$. Up to rotation and reflection, we will get one of the five matchings in Figure~\ref{fig:matchings}. 

\newdimen\R
\R=0.7cm
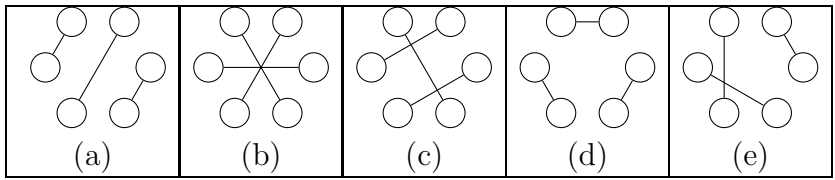
\begin{figure}
 
\begin{tabular}{|c|c|c|c|c|}\hline\label{item:matching:a}
\begin{tikzpicture}
 \draw (0:\R) node[circle,draw](a){$ $};
 \draw (60:\R) node[circle,draw](b){$ $};
 \draw (120:\R) node[circle,draw](c){$ $};
 \draw (180:\R) node[circle,draw](d){$ $};
 \draw (240:\R) node[circle,draw](e){$ $};
 \draw (300:\R) node[circle,draw](f){$ $};
 \draw (a)--(f);
 \draw (b)--(e);
 \draw (c)--(d); 
\end{tikzpicture}&
\begin{tikzpicture}
 \draw (0:\R) node[circle,draw](a){$ $};
 \draw (60:\R) node[circle,draw](b){$ $};
 \draw (120:\R) node[circle,draw](c){$ $};
 \draw (180:\R) node[circle,draw](d){$ $};
 \draw (240:\R) node[circle,draw](e){$ $};
 \draw (300:\R) node[circle,draw](f){$ $};
 \draw (a)--(d);
 \draw (b)--(e);
 \draw (c)--(f); 
\end{tikzpicture}&
\begin{tikzpicture}
 \draw (0:\R) node[circle,draw](a){$ $};
 \draw (60:\R) node[circle,draw](b){$ $};
 \draw (120:\R) node[circle,draw](c){$ $};
 \draw (180:\R) node[circle,draw](d){$ $};
 \draw (240:\R) node[circle,draw](e){$ $};
 \draw (300:\R) node[circle,draw](f){$ $};
 \draw (a)--(e);
 \draw (b)--(d);
 \draw (c)--(f); 
\end{tikzpicture}&
\begin{tikzpicture}
 \draw (0:\R) node[circle,draw](a){$ $};
 \draw (60:\R) node[circle,draw](b){$ $};
 \draw (120:\R) node[circle,draw](c){$ $};
 \draw (180:\R) node[circle,draw](d){$ $};
 \draw (240:\R) node[circle,draw](e){$ $};
 \draw (300:\R) node[circle,draw](f){$ $};
 \draw (a)--(f);
 \draw (b)--(c);
 \draw (e)--(d); 
\end{tikzpicture}&
\begin{tikzpicture}
 \draw (0:\R) node[circle,draw](a){$ $};
 \draw (60:\R) node[circle,draw](b){$ $};
 \draw (120:\R) node[circle,draw](c){$ $};
 \draw (180:\R) node[circle,draw](d){$ $};
 \draw (240:\R) node[circle,draw](e){$ $};
 \draw (300:\R) node[circle,draw](f){$ $};
 \draw (a)--(b);
 \draw (c)--(e);
 \draw (f)--(d); 
\end{tikzpicture}\\
(a)& \label{item:matching:b}(b)& \label{item:matching:c}(c)& \label{item:matching:d}(d)& \label{item:matching:e}(e)\\\hline
\end{tabular}
\caption{\label{fig:matchings} The five possible combinatorial types of matchings of six points on a circle.}
\end{figure}

\R=3.6cm
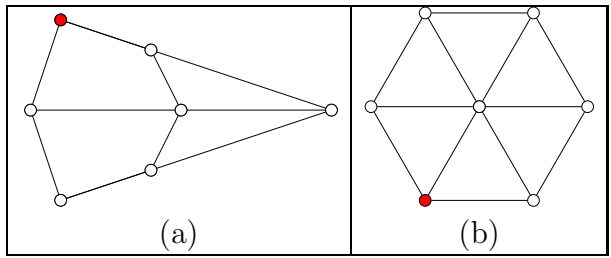
\begin{figure}
 
\begin{tabular}{|c|c|}\hline
\scalebox{0.4}{
\begin{tikzpicture}
 \draw (0,0) node[circle,draw](a){$ $} --
       (1,3) node[circle,draw](b){$ $} --
       (4,2) node[circle,draw](c){$ $} --
       (5,0) node[circle,draw](d){$ $} --
       (4,-2) node[circle,draw](e){$ $} --
       (1,-3) node[circle,draw](f){$ $} --cycle;
 \draw (10,0) node[circle,draw](g){$ $};
 \draw (b)--(g);
 \draw (a)--(g);
 \draw (f)--(g); 
  \draw (0,0) node[circle,draw,fill=white](a){$ $};
\draw  (1,3) node[circle,draw,fill=red](b){$ $};
\draw  (4,2) node[circle,draw,fill=white](c){$ $};
\draw  (5,0) node[circle,draw,fill=white](d){$ $};
\draw  (4,-2) node[circle,draw,fill=white](e){$ $};
\draw  (1,-3) node[circle,draw,fill=white](f){$ $};
 \draw (10,0) node[circle,draw,fill=white](g){$ $};
\end{tikzpicture}}&
\scalebox{0.4}{\begin{tikzpicture}
 \draw (0:\R) node[circle,draw](a){$ $};
 \draw (60:\R) node[circle,draw](b){$ $};
 \draw (120:\R) node[circle,draw](c){$ $};
 \draw (180:\R) node[circle,draw](d){$ $};
 \draw (240:\R) node[circle,draw,fill=red](e){$ $};
 \draw (300:\R) node[circle,draw](f){$ $};
 \draw (a)--(b)--(c)--(d)--(e)--(f)--(a);
 \draw (0:0)  node[circle,draw](g){$ $};
 \draw (a)--(d);
 \draw (b)--(e);
 \draw (c)--(f); 
 \draw (0:0)  node[circle,draw,fill=white](g){$ $};
\end{tikzpicture}}\\
\label{item:red_a}(a)&(b)\label{item:red_b}\\\hline
\end{tabular}
\caption{\label{fig:hexagons} Two possibilities for $\Mcal'$. Removing the red point produces a forbidden minor.}
\end{figure}

It is clear that for the matchings \hyperref[fig:matchings]{(c)}, \hyperref[fig:matchings]{(d)}, or~\hyperref[fig:matchings]{(e)} in Figure~\ref{fig:matchings}, the three lines will not intersect at the same point. For~\hyperref[fig:matchings]{(a)} or~\hyperref[fig:matchings]{(b)} this is possible and we clearly get only two possibilities for $\Mcal$, shown in Figure~\ref{fig:hexagons}. If we remove the red point in~\hyperref[fig:hexagons]{(a)}, we will get the oriented matroid $\IC(6,3,9)$ from Figure~\ref{fig:non_postiroids}. If we remove the red point in~\hyperref[fig:hexagons]{(b)}, we will get $\IC(6,3,10)$ from Figure~\ref{fig:non_postiroids}. Both of them are not isomorphic to a positively oriented matroid, and thus $\Mcal$ is not almost positively oriented, and we are done with the proof of the lemma.
\end{proof}

% There are some positively oriented matroids whose non-trivial reorientations still yield positively oriented matroids. In the next lemma, we 

\begin{definition}
 We say that an oriented matroid $\Mcal$ \emph{contains a pentagon} if there is a $5$-element subset $\Pi\subset E$ such that the restriction of $\Mcal$ to $\Pi$ is a uniform oriented matroid of rank $3$. 
\end{definition}
Recall that all uniform oriented matroids on five elements are isomorphic, because their duals are uniform oriented matroids of rank $2$ with five elements.

\begin{definition}
 We say that an oriented matroid $\Mcal$ \emph{is contained in two lines} if there are two lines $P_1$ and $P_2$ of $\Mcal$ whose union is $E$.
\end{definition}

\begin{lemma}\label{lemma:pentagon_or_two_lines}
 Let $\Mcal$ be an almost positively oriented matroid with at least $7$ elements. Then either $\Mcal$ contains a pentagon or $\Mcal-e$ is contained in two lines for some $e\in E$.
\end{lemma}
\begin{proof}
Assume that $\Mcal-e$ is not contained in two lines for any $e\in E$. Fix some element $e\in E$ and consider the positively oriented matroid $\Mcal-e$. If $\vert(\Mcal-e)\geq 5$ then we are done because the vertices of course form a uniform oriented matroid. Suppose now that $\vert(\Mcal-e)=4$, that is, the points of $\Mcal-e$ belong to the boundary of some quadrilateral. Since $\Mcal-e$ is not contained in two lines, we get that there are two sides of this quadrilateral such that they share a vertex and both of them contain points of $\Mcal-e$ in their interior. But then removing their shared vertex increases the number of vertices, so we are done with this case as well. 

The only case left is when $\vert(\Mcal-e)=3$ for all $e\in E$. Thus $\Mcal-e$ is a triangle with at least one point in the interior of each side (otherwise $\Mcal-e$ is contained in two lines). If the interior of one of the sides (say, connecting the vertices $f$ and $g$) contains at least two points then $\vert(\Mcal-\{e,f,g\})=5$ so $\Mcal-e$ contains a pentagon. Thus each side of the triangle contains exactly one point in its interior, and therefore $|E|=7$. 
We will show that this is impossible.

% , and it suffices to consider the smallest case $|E|=7$.
 % In this case, $\Mcal-e$ has to be a triangle with an extra point on each side, that is,
We see that $\Mcal-e$ is isomorphic to $\IC(6,3,14)$ from Figure~\ref{fig:positroids} for every $e$. Let $f\in E$ be a vertex of this triangle. It belongs to at least $2$ non-trivial lines in $\Mcal$ and by Lemma~\ref{lemma:any_element_belongs_to_2_NT_lines} $f$ belongs to exactly $2$ non-trivial lines in $\Mcal$, denoted $P_1$ and $P_2$. 
 % Let $f\in E$ be a vertex of the triangle and consider the two non-trivial lines $P_1$ and $P_2$ of $\Mcal$ that contain $f$.
 For $i=1,2$, $P_i$ has to contain exactly three points. Indeed, by non-triviality it contains at least three points, and if it contained at least four points then we would remove some other element $g\not\in P_i$ and get an oriented matroid $\Mcal-g$ with a line containing at least $4$ elements, but $\Mcal-g$ has to be isomorphic to $\IC(6,3,14)$ which does not have such a line. 

Observe that each $P_i$ contains at least two elements of $\Mcal-f$. Now, $P_i$ cannot contain a vertex of $\Mcal-f$ because otherwise this vertex would be contained in three non-trivial lines, contradicting Lemma~\ref{lemma:any_element_belongs_to_2_NT_lines}. We find that each of $P_1$ and $P_2$ contains at least two points which are not among the vertices of $\Mcal-f$. Since $\Mcal-f$ contains three such points in total, one of them has to be common to $P_1$ and $P_2$. On the other hand, the only point in $\Mcal$ belonging to both $P_1$ and $P_2$ is $f$, a contradiction.
% Finally, note that any two lines that pass through the midpoints of the edges in $\IC(6,3,14)$ already intersect each other by one of these points and thus the corresponding three lines cannot intersect together at $f$ in $\Mcal$. This completes the proof of the lemma.
 \end{proof}
\begin{remark}
 Note that both the property of containing a pentagon and of being contained in two lines can be stated for unoriented matroids, and then an example of a matroid that has neither of these two properties is the Fano plane. Moreover, removing any point from the Fano plane gives an unoriented matroid isomorphic to $\underline{\IC(6,3,14)}$, the underlying matroid of $\IC(6,3,14)$. In particular, this implies that the Fano  matroid is non-orientable.
\end{remark}

 We would like to restrict our attention to only almost positively oriented matroids that contain a pentagon. In order to do so, we need to eliminate the other option from Lemma~\ref{lemma:pentagon_or_two_lines}.
 
 \begin{lemma}\label{lemma:two_lines}
  Let $\Mcal$ be an almost positively oriented matroid with $|E|\geq 7$ that does not contain a pentagon. Then $\Mcal$ is isomorphic to a positively oriented matroid.
 \end{lemma}
 \begin{proof}
 We will split the proof into three cases: 
 \begin{enumerate}
  \item\label{item:two_lines} $\Mcal$ is contained in two lines;
  \item\label{item:big_line} there is a line $P$ of $\Mcal$ with $|P|\geq |E|-3$;
  \item\label{item:two_lines_and_e} for some $e\in E$, there are two lines $P_1$ and $P_2$ of $\Mcal$ whose union is $E-e$.
 \end{enumerate}
The proof in Case~\eqref{item:big_line} will depend on Case~\eqref{item:two_lines} while the proof in Case~\eqref{item:two_lines_and_e} will depend on Cases~\eqref{item:two_lines} and~\eqref{item:big_line}.

Consider the first case. If $E=P_1\cup P_2$ is a union of two lines then $\Mcal$ has no other non-trivial lines, because any such line would intersect either $P_1$ or $P_2$ by at least two elements (cf. Lemma~\ref{lemma:line_intersection}). Consider the cocircuit $C^{(1)}$ of $\Mcal$ with zero set $P_1$ and reorient all elements of $P_2$ so that $C^{(1)}$ would be a positive cocircuit. Now consider the cocircuit $C^{(2)}$ of $\Mcal$ with zero set $P_2$ and reorient all elements of $P_1$ so that $C^{(2)}$ would be a positive cocircuit. Since the restriction of $\Mcal$ to $P_i$ is a simple rank $2$ oriented matroid, we can assume that $P_1=\{a_1,a_2,\dots,a_n\}$ and $P_2=\{b_1,b_2,\dots,b_m\}$ are ordered in such a way that the circuits of $\Mcal\mid_{P_1}$ are $(\{a_i,a_k\},\{a_j\})$ for all $1\leq i<j<k\leq n$ and the circuits of $\Mcal\mid_{P_2}$ are $(\{b_i,b_k\},\{b_j\})$ for all $1\leq i<j<k\leq m$. 

Assume first that $P_1\cap P_2=\emptyset$. Consider the circuit $X$ with $\Xu=\{a_1,a_n,b_1,b_m\}$. Since $X$ has to be  orthogonal to $C^{(1)}$ and $C^{(2)}$, we have $X_{a_1}\neq X_{a_n}$ and $X_{b_1}\neq X_{b_m}$. Thus after a possible reversal of the order of $a_i$'s, we have 
\[X_{a_1}=+;\quad X_{a_n}=-;\quad X_{b_1}=-;\quad X_{b_m}=+.\]
Now choose any $1\leq i\leq n$ and $1\leq j\leq m$. Let $C^{ij}$ be the cocircuit with zero set $\{i,j\}$. Then for any $k<i$, the sign of $C^{ij}_{a_k}$ has to be the same, and for any $k>i$, $C^{ij}_{a_k}$ has to have the opposite sign, and similarly for $b_k$'s. Since $C^{ij}$ is orthogonal to $X$, we must have 
\[C^{ij}_{a_k}=\begin{cases}
                +,&\text{ if $k<i$;}\\
                -,&\text{ if $k>i$;}
               \end{cases};\quad
C^{ij}_{b_k}=\begin{cases}
                +,&\text{ if $k<j$;}\\
                -,&\text{ if $k>j$;}
               \end{cases}.\]

Consider a four-gon with vertices $a_1,a_n,b_m,b_1$ in clockwise order and points $a_1,\dots,a_n$ on one side and points $b_m,b_{m-1},\dots,b_1$ on the opposite side. We claim that the associated oriented matroid $\Mcal'$ coincides with $\Mcal$. It is clear that the underlying matroids are the same. Clearly they also have the same oriented cocircuits so we have finished dealing with the case $P_1\cap P_2=\emptyset$. The case $|P_1\cap P_2|=1$ is handled similarly, so we are done with Case~(\ref{item:two_lines}).
 
  Now consider Case~(\ref{item:big_line}): there is a line $P$ of $\Mcal$ with $|P|\geq |E|-3$, and in particular $|P|\geq 4$. If $|P|>|E|-3$ then we are done by Case~(\ref{item:two_lines}), thus $|P|=|E|-3$, so let $E-P=\{e,f,g\}$. Since every other line can intersect $P$ by at most one element, there are at most $3$ other non-trivial lines of $\Mcal$, and thus there is an element $h\in P$ not belonging to any of them. If there is another such element $h'$ then $h,h',e,f,g$ form a pentagon and we are done, so assume that $h$ is the only such element. In order to have $|P|\geq 4$ while only one element in $P$ not in any other line, we must have $|E|=7$ and for any two elements from $\{e,f,g\}$ there is a non-trivial line containing them and one other element in $P$. But then we have ``too many lines'': consider $\Mcal-h$, which is isomorphic to a positively oriented matroid on $6$ elements. Using Corollary~\ref{cor:number_of_vertices}, we can count that 
  \[\vert(\Mcal-h)=6-1-1-1-1=2,\]
  which means that $\Mcal-h$ is contained in a line even though we know it is not. This finishes Case~(\ref{item:big_line}).
  
  Consider Case~(\ref{item:two_lines_and_e}). We must have that $|P_1-P_2|\geq 3$ and $|P_2-P_1|\geq 3$, otherwise we would arrive at Case~\eqref{item:big_line}. Moreover, we may assume that $e$ does not belong to either $P_1$ or $P_2$, otherwise we would arrive at Case~\eqref{item:two_lines}. 
There are at most two non-trivial lines through $e$, each of them intersects $P_i$ in at most one element for $i=1,2$, and there are no other non-trivial lines in $\Mcal$. Thus there is at least one element $e_i$ in each of $P_i$, $i=1,2$, that is not contained in any other non-trivial line of $\Mcal$. Let $f_i\neq e_i$ be an element of $P_i-P_{3-i}$ ($i=1,2$) such that $e,f_1,f_2$ do not lie on the same line. We get that $e_1,e_2,f_1,f_2,e$ form a pentagon so we are done with the proof of the lemma.
\end{proof}

There is a reason why we care about $\Mcal$ containing a pentagon.

\begin{lemma}\label{lemma:pentagon_rigid}
 Suppose that a positively oriented matroid $\Mcal$ contains a pentagon. Then the only reorientation of $\Mcal$ that is also positively oriented is $\reorient{E}{\Mcal}$, that is, there are no non-trivial positively oriented reorientations of $\Mcal$.
\end{lemma}
\begin{proof}
 We prove this by induction on $|E|$. Let $|E|=5$. Then $\Mcal$ is itself a pentagon, and now suppose that $\reorient{A}{\Mcal}$ is also positively oriented. We may assume $1\leq |A|\leq 2$. Note that for every circuit $C$ of $\Mcal$, $A$ has to contain an even number of elements from $\Cu$. It is clear that for every one- or two-element subset of $E$, there is a circuit of $\Mcal$ that contains exactly one element of $A$, thus we are done with the base case. To show the induction step, consider any positively oriented matroid $\Mcal$ that contains a pentagon $\Pi$ and suppose that $\reorient{A}{\Mcal}$ is also positively oriented for some proper subset $A$ of $E$. Choose an element $e\in E-\Pi$. By the induction hypothesis, the intersection of $A$ with $E-e$ has to be either empty or equal to $E-e$. In any of these cases, one can easily find a circuit $C$ that contains $e$ and three other elements from $E-e$. For this circuit, the intersection $A\cap \Cu$ would have an odd number of elements, so we get a contradiction which finishes the proof of the lemma.
\end{proof}

Recall that the map $\chi_\Mcal:E^r\to\{+,-,0\}$ is the \emph{chirotope of $\Mcal$} defined in Section~\ref{sect:OM}.

\begin{lemma}\label{lemma:unique_order_reorient}
 Let $\Mcal$ be an almost positively oriented matroid with $|E|\geq 7$ that contains a pentagon. Then $\Mcal$ is positively oriented if and only if there is a cyclic order $\Order^*$ on $E$ such that for any $a,b,c\in E$ ordered in accordance with $\Order^*$, we have $\chi_\Mcal(a,b,c)\geq 0$. Moreover, if such $\Order^*$ exists, it is unique.
\end{lemma}
\begin{proof}
 Obviously, if $\Mcal$ is positively oriented then $\Order^*$ exists and, by Lemma~\ref{lemma:pentagon_rigid}, is unique. Conversely, if $\Mcal$ admits such an order then it is positively oriented by definition (see Section~\ref{sect:OM}), and thus by the previous claim such $\Order^*$ is unique.
\end{proof}

 Our second to last step in proving Theorem~\ref{thm:minors} is to reduce it to oriented matroids with at most $8$ elements.
 
 \begin{lemma}\label{lemma:reduction_to_8}
  Suppose that any almost positively oriented matroid with at most $8$ elements is isomorphic to a positively oriented matroid. Then the same statement holds for almost positively oriented matroids with arbitrary number of elements.
 \end{lemma}
 \begin{proof}
  Suppose $\Mcal$ is an almost positively oriented matroid with $|E|>8$. By the above discussion, we are assuming that $\Mcal$ contains a pentagon $\Pi$. Choose any two distinct elements $e,f\in E-\Pi$. By Lemma~\ref{lemma:pentagon_rigid}, there is an essentially unique reorientation of $\Mcal-\{e,f\}$ that makes it positively oriented. We claim that there is a unique reorientation of $\Mcal$ such that every element from $E-e$ is oriented the same way in $\Mcal$ and in the positive reorientation of $\Mcal-e$, and every element from $E-f$ is oriented the same way in $\Mcal$ and in the positive reorientation of $\Mcal-f$. Such an orientation is clearly unique, but the fact that it exists is a consequence of the fact that there is only one reorientation of $\Mcal-\{e,f\}$ that makes it positively oriented, and this orientation has to agree with the corresponding unique orientations that make $\Mcal-e$ and $\Mcal-f$ into positively oriented matroids. Let us explain this in more detail.
  
Let $\Order$ be the cyclic order on the elements of $\Mcal-\{e,f\}$ that comes from the boundary of the convex polygon that realizes $\Mcal-\{e,f\}$. In other words, $\Order$ is the unique cyclic order on $E-\{e,f\}$ such that for any three elements $a,b,c$ ordered in accordance with $\Order$, we have $\chi_\Mcal(a,b,c)\geq 0$ (see Section~\ref{sect:OM}). The order $\Order$ can be extended to a cyclic order $\Order'$ on $E-e$ (resp., $\Order''$ on $E-f$) such that for any $a,b,c\in E-e$ (resp., $a,b,c\in E-f$) ordered in accordance with $\Order'$ (resp., with $\Order''$), we have $\chi_\Mcal(a,b,c)\geq 0$. Since the cyclic orders $\Order'$ and $\Order''$ agree on $E-\{e,f\}$, there is a cyclic order $\Order^*$ on $E$ such that removing $e$ from $\Order^*$ results in $\Order'$ and removing $f$ from $\Order^*$ results in $\Order''$. If $e$ and $f$ are not adjacent in $\Order^*$ then such order is unique, otherwise it is unique up to a transposition of $e$ and $f$. What we would like to show is that for any $a,b,c\in E$ ordered in accordance with $\Order^*$, we have $\chi(a,b,c)\geq 0$. Clearly, this holds for any $a,b,c$ such that $\{e,f\}\not\subset \{a,b,c\}$. But note that if $\Order^*$ is not unique (i.e., if $e$ and $f$ are adjacent in $\Order^*$) then we can choose an element $g\in \Pi$ such that $\{e,f,g\}$ is a basis, because $\rk(\Pi ef)>2$, and then the sign of $\chi(e,f,g)$ will determine $\Order^*$ uniquely. Thus we have found a unique possible candidate for $\Order^*$ from Lemma~\ref{lemma:unique_order_reorient}.

Assume that we have some element $g\in E$ such that $e,f,g$ is ordered in accordance with $\Order^*$ and we are trying to show that $\chi(e,f,g)\geq 0$. Let $\Mcal'$ be the restriction of $\Mcal$ to $Q:=\Pi\cup\{e,f,g\}$. Our goal is to show that $\chi_{\Mcal'}(e,f,g)\geq 0$. Since $|Q|\leq 8$ and $\Mcal'$ is still almost positively oriented, by the assumption of the lemma we know that $\Mcal'$ is isomorphic to a positively oriented matroid. By the above discussion, the unique cyclic order of $\Mcal'$ from Lemma~\ref{lemma:unique_order_reorient} has to coincide with the restriction of $\Order^*$ to $Q$, in which case we clearly get $\chi_{\Mcal'}(e,f,g)\geq 0$. This finishes the proof of the lemma.
 \end{proof}

\def\QQQ#1{{\mathfrak{#1}}}
\def\A{\QQQ a}
\def\B{\QQQ b}
\def\C{\QQQ c}
\def\D{\QQQ d}
\def\E{\QQQ e}

 \def\R{\mathbb{R}}
The following lemma combined with Lemmas~\ref{lemma:reduction_to_8} and~\ref{lemma:two_lines} completes the proof of Theorem~\ref{thm:minors}.
\begin{lemma}\label{lemma:base_8}
 If $\Mcal$ is an almost positively oriented matroid on at most $8$ elements that contains a pentagon $\Pi$ then $\Mcal$ is isomorphic to a positively oriented matroid.
\end{lemma}
\begin{proof}
 By Proposition~\ref{prop:realizable_8}, $\Mcal$ is realizable, so it comes from some vector configuration $\VC\subset\R^3$. Reorient $\Pi$ in the unique way such that $\Mcal\mid_\Pi$ is positively oriented. We know that $\Mcal\mid_\Pi$ is then acyclic, and let $H\subset \R^3$ be any (affine) plane such that each vector from $\Pi$ belongs to $H$ after some rescaling by a positive real number. The endpoints of vectors from $\Pi$ form a convex pentagon. Let us say that a convex pentagon is \emph{good} if the sum of any two adjacent angles is greater than $180^\circ$. Otherwise, call such a pentagon \emph{bad}. See Figure~\ref{fig:pentagons} for an illustration.

 \definecolor{qqqqff}{rgb}{0.,0.,1.}
 
 \begin{figure}
  \begin{tabular}{cc}%\hline
   
 \scalebox{0.6}{
\begin{tikzpicture}[scale=0.8]
\clip(4.,-4.) rectangle (12.,3.);
\draw [domain=4.:12.] plot(\x,{(-8.911159073833751-0.04958677685950397*\x)/3.0909090909090917});
\draw [domain=4.:12.] plot(\x,{(--13.982126903900001-2.644628099173553*\x)/1.1900826446280988});
\draw [domain=4.:12.] plot(\x,{(--9.184680401202069-1.5364057851239623*\x)/-2.387898347107435});
\draw [domain=4.:12.] plot(\x,{(--16.26675284160913-1.66863719008264*\x)/2.6864818181818206});
\draw [domain=4.:12.] plot(\x,{(--27.32163923229286-2.561983471074379*\x)/-0.7933884297520652});
\begin{scriptsize}
% \draw [fill=qqqqff] (6.6322314049586755,-2.989421487603303) circle (2.5pt);
% \draw[color=qqqqff] (6.6322314049586755,-2.989421487603303) node[anchor=south west] {$A$};
\node[draw,ellipse,fill=white] at (6.6322314049586755,-2.989421487603303) {$\A$};
\node[draw,ellipse,fill=white] at (5.442148760330577,-0.34479338842975027) {$\B$};
% \draw[color=qqqqff] (5.442148760330577,-0.34479338842975027) node[anchor=south] {$B$};
% \draw [fill=qqqqff] (7.830047107438012,1.191612396694212) circle (2.5pt);
\node[draw,ellipse,fill=white] at (7.830047107438012,1.191612396694212) {$\C$};
% \draw [fill=qqqqff] (10.516528925619832,-0.4770247933884279) circle (2.5pt);
\node[draw,ellipse,fill=white] at (10.516528925619832,-0.4770247933884279) {$\D$};
% \draw [fill=qqqqff] (9.723140495867767,-3.039008264462807) circle (2.5pt);
\node[draw,ellipse,fill=white] at (9.723140495867767,-3.039008264462807) {$\E$};
% \draw[color=black] (-2.2603173703981994,-2.5991329075882796) node {$a$};
% \draw[color=black] (3.0439651239669367,5.364803981968444) node {$b$};
% \draw[color=black] (-2.2603173703981994,-5.033392862509392) node {$c$};
% \draw[color=black] (1.3309673779113405,5.364803981968444) node {$d$};
% \draw[color=black] (12.11984791885799,5.364803981968444) node {$e$};
\end{scriptsize}
\end{tikzpicture}}
&
\scalebox{0.6}{
\begin{tikzpicture}[scale=0.8]
\clip(6.,-6.) rectangle (19.,4.);
\draw [domain=6.:19.] plot(\x,{(-43.9064-0.24*\x)/8.68});
\draw [domain=6.:19.] plot(\x,{(-49.2732--3.56*\x)/-2.3});
\draw [domain=6.:19.] plot(\x,{(-8.0676--0.84*\x)/-2.34});
\draw [domain=6.:19.] plot(\x,{(--14.4204-0.9*\x)/-2.56});
\draw [domain=6.:19.] plot(\x,{(--36.3364-3.26*\x)/-1.48});
\begin{scriptsize}
\node[draw,ellipse,fill=white] at (8.74,-5.3) {$\A$};
% \draw[color=qqqqff] (8.88,-4.94) node {$A$};
\node[draw,ellipse,fill=white] at (10.22,-2.04) {$\B$};
% \draw[color=qqqqff] (10.36,-1.68) node {$B$};
\node[draw,ellipse,fill=white] at (12.78,-1.14) {$\C$};
% \draw[color=qqqqff] (12.92,-0.78) node {$C$};
\node[draw,ellipse,fill=white] at (15.12,-1.98) {$\D$};
% \draw[color=qqqqff] (15.26,-1.62) node {$D$};
\node[draw,ellipse,fill=white] at (17.42,-5.54) {$\E$};
% \draw[color=qqqqff] (17.56,-5.18) node {$E$};
% \draw[color=black] (-4.14,-4.6) node {$a$};
% \draw[color=black] (10.12,6.14) node {$b$};
% \draw[color=black] (-4.14,4.78) node {$c$};
% \draw[color=black] (-4.14,-6.74) node {$d$};
% \draw[color=black] (13.64,6.14) node {$e$};
\end{scriptsize}
\end{tikzpicture}
}\\
A good pentagon & A bad pentagon: $\hat \A+\hat \E<180^\circ$\\%\hline
  \end{tabular}
  \caption{\label{fig:pentagons} Examples of good and bad pentagons. Here $\hat \A$ denotes the angle at vertex $\A$ and so on.}
 \end{figure}
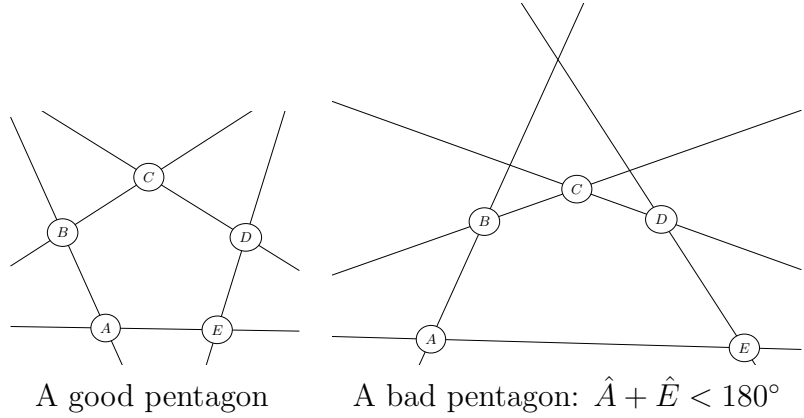

 We would like to choose an affine plane $H$ so that the vectors from $\Pi$ would form a good convex pentagon. Our first goal is to show that it is always possible.
 \begin{claim}
  For every $5$ vectors $\Pi$ in $\R^3$ that form a positively oriented uniform matroid, there exists an affine plane $H$ such that each vector of $\Pi$ belongs to $H$ after some rescaling by a positive real number, and moreover, the endpoints of rescaled vectors form a good pentagon in $H$.
 \end{claim}
 \begin{claimproof}
  Let $\vec a,\vec b,\vec c,\vec d,\vec e$ be the elements of $\Pi$ listed in cyclic order, and let $a,b,c,d,e$ be the rays spanned by the corresponding vectors. By $P(x,y)$ denote the plane through the origin spanned by the vectors $\vec x$ and $\vec y$. Choose any point $\A$ on $a$. Consider the lines $\ell_1=P(a,b)\cap P(d,e)$ and $\ell_2=P(a,e)\cap P(b,c)$. Both of these lines pass through the origin and it is clear that the cone $W$ spanned by $\vec a,\vec b,\vec c,\vec d,\vec e$ intersects the plane $P(\ell_1,\ell_2)$ spanned by $\ell_1$ and $\ell_2$ only in the origin. Let $H'$ be the affine plane through $\A$ parallel to $P(\ell_1,\ell_2)$. Then $H'$ intersects $W$ by a convex pentagon, so let $\A,\B,\C,\D,\E$ be the intersection points of $a,b,c,d,e$ with $H'$ respectively. By the choice of $H'$, the lines $\A\B$ and $\D\E$ are parallel. Moreover, the lines $\A\E$ and $\B\C$ are parallel. Let $\alpha:=\hat \A$ be the angle of $\A\B\C\D\E$ at vertex $\A$, so we have $0^\circ<\alpha<180^\circ$. Thus we have
  \[\hat \B=\hat \E=180^\circ-\alpha,\quad\hat \C+\hat \D=180^\circ+\alpha.\]
  Since $\hat \C,\hat \D<180^\circ$, we get that $\alpha<\hat \C,\hat \D$. Therefore
  \[\hat \A+\hat \B=\hat \A+\hat \E=180^\circ;\quad \hat \B+\hat \C,\hat \C+\hat \D,\hat \D+\hat \E>180^\circ.\]
  So the pentagon $\A\B\C\D\E$ is almost good. We would like to find a slight perturbation $H$ of $H'$ around $\A$ so that we would get $\hat \A+\hat \B,\hat \A+\hat \E>180^\circ $ without violating the remaining three inequalities. To do so, observe that the plane $H'$ is parallel to the lines $\ell_1$ and $\ell_2$ by construction. Orient $\ell_1$ (resp., $\ell_2$) in the direction of the vector $\A-\B$ (resp., $\A-\E$). Choose a point $X_i$ on $\ell_i$ far away in the positive direction for $i=1,2$. Let $H$ be the plane through $\A,X_1,X_2$. Then $H$ is indeed a slight perturbation of $H'$ and therefore the strict inequalities are still satisfied, provided that $X_1$ and $X_2$ are far enough. Letting $\A\B\C\D\E$ denote the corresponding pentagon inside $H$, we see that the points $X_1,\A,\B$ lie on a common line (in this order), and the same is true for $X_1,\D,\E$, as well as for $X_2,\A,\E$ and $X_2,\B,\C$. Thus the new pentagon $\A\B\C\D\E$ is good.
% The line  $(\A\B)$ 
% We may direct the line $\ell_1$ in the direction of 
% Now we can just set $H$ to be a slight perturbation of $H'$ around $\A$ so that we would get $\hat \A+\hat \B,\hat \A+\hat \E>180^\circ $ without violating the remaining three inequalities. We are done with the claim.
 \end{claimproof}

So now we are assuming that the intersection of the cone spanned by $\Pi$ with $H$ forms a good pentagon. Reorient $E-\Pi$ in a unique way so that all vectors in $\VC$ would belong to $H$ after a  positive rescaling. We claim that now $\Mcal$ is oriented ``the correct way'' meaning that for any element $e\in E-\Pi$, $\Mcal-e$ is already a positively oriented matroid. Suppose this is not the case, that is, for some $e\in E-\Pi$, $\Mcal':=\Mcal-e$ is not positively oriented but $\reorient{A}{\Mcal'}$ is positively oriented for some proper subset $A$ of $E-e$. By Lemma~\ref{lemma:pentagon_rigid}, we may assume that $A\cap \Pi=\emptyset$. Let $f\in A$ be any element. Since $f$ is represented by a point in $H$, one easily checks (using the fact that $\Pi$ is a good pentagon) that there is a circuit $C$ of $\Mcal'$ that involves $f$ and three elements from $\Pi$, denote them $a,b,c$, so that $C^+=\{f,a\}$ and $C^-=\{b,c\}$. Thus after reorienting by $A$, our circuit $C$ will contain an odd number of plus signs which contradicts the fact that $\reorient{A}{\Mcal'}$ is positively oriented. This shows that $\Mcal-e$ is positively oriented for any $e\in E-\Pi$. From this we can deduce some information about possible locations of points from $E-\Pi$ in $H$.

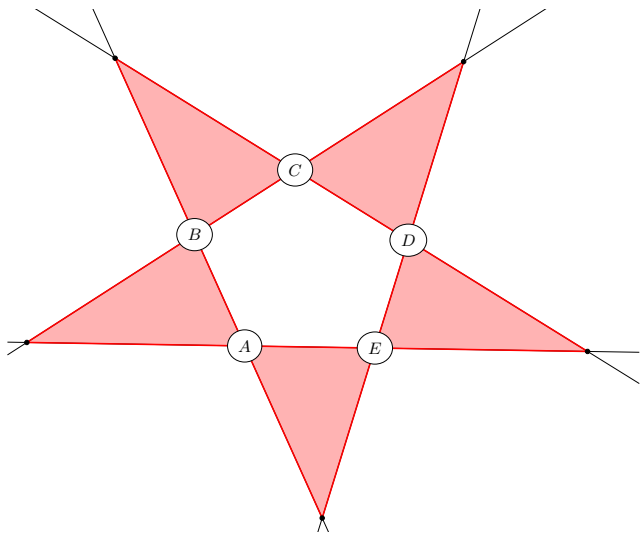
\begin{figure}
 \scalebox{0.7}{
 
\definecolor{zzttqq}{rgb}{1,0,0}
\definecolor{uuuuuu}{rgb}{0,0,0}
\definecolor{qqqqff}{rgb}{0.,0.,1.}
\begin{tikzpicture}[scale=0.8]
\clip(1.,-7.4) rectangle (16.,5.);
\fill[color=zzttqq,fill=zzttqq,fill opacity=0.3] (1.460783574454355,-2.9064570839053734) -- (5.442148760330577,-0.34479338842975027) -- (6.6322314049586755,-2.989421487603303) -- cycle;
\fill[color=zzttqq,fill=zzttqq,fill opacity=0.3] (3.556202665213042,3.846197934053662) -- (5.442148760330577,-0.34479338842975027) -- (7.830047107438012,1.191612396694212) -- cycle;
\fill[color=zzttqq,fill=zzttqq,fill opacity=0.3] (10.516528925619832,-0.4770247933884279) -- (7.830047107438012,1.191612396694212) -- (11.830326192770796,3.765445548453232) -- cycle;
\fill[color=zzttqq,fill=zzttqq,fill opacity=0.3] (10.516528925619832,-0.4770247933884279) -- (9.723140495867767,-3.039008264462807) -- (14.771682375779076,-3.120000914942668) -- cycle;
\fill[color=zzttqq,fill=zzttqq,fill opacity=0.3] (9.723140495867767,-3.039008264462807) -- (8.472248249723641,-7.0783478093032235) -- (6.6322314049586755,-2.989421487603303) -- cycle;
\draw [domain=1.:16.] plot(\x,{(-8.911159073833751-0.04958677685950397*\x)/3.0909090909090917});
\draw [domain=1.:16.] plot(\x,{(--13.982126903900001-2.644628099173553*\x)/1.1900826446280988});
\draw [domain=1.:16.] plot(\x,{(--9.184680401202069-1.5364057851239623*\x)/-2.387898347107435});
\draw [domain=1.:16.] plot(\x,{(--16.26675284160913-1.66863719008264*\x)/2.6864818181818206});
\draw [domain=1.:16.] plot(\x,{(--27.32163923229286-2.561983471074379*\x)/-0.7933884297520652});
\draw [color=red,line width=0.30mm] (1.460783574454355,-2.9064570839053734)-- (5.442148760330577,-0.34479338842975027);
\draw [color=red,line width=0.30mm] (5.442148760330577,-0.34479338842975027)-- (6.6322314049586755,-2.989421487603303);
\draw [color=red,line width=0.30mm] (6.6322314049586755,-2.989421487603303)-- (1.460783574454355,-2.9064570839053734);
\draw [color=red,line width=0.30mm] (3.556202665213042,3.846197934053662)-- (5.442148760330577,-0.34479338842975027);
\draw [color=red,line width=0.30mm] (5.442148760330577,-0.34479338842975027)-- (7.830047107438012,1.191612396694212);
\draw [color=red,line width=0.30mm] (7.830047107438012,1.191612396694212)-- (3.556202665213042,3.846197934053662);
\draw [color=red,line width=0.30mm] (10.516528925619832,-0.4770247933884279)-- (7.830047107438012,1.191612396694212);
\draw [color=red,line width=0.30mm] (7.830047107438012,1.191612396694212)-- (11.830326192770796,3.765445548453232);
\draw [color=red,line width=0.30mm] (11.830326192770796,3.765445548453232)-- (10.516528925619832,-0.4770247933884279);
\draw [color=red,line width=0.30mm] (10.516528925619832,-0.4770247933884279)-- (9.723140495867767,-3.039008264462807);
\draw [color=red,line width=0.30mm] (9.723140495867767,-3.039008264462807)-- (14.771682375779076,-3.120000914942668);
\draw [color=red,line width=0.30mm] (14.771682375779076,-3.120000914942668)-- (10.516528925619832,-0.4770247933884279);
\draw [color=red,line width=0.30mm] (9.723140495867767,-3.039008264462807)-- (8.472248249723641,-7.0783478093032235);
\draw [color=red,line width=0.30mm] (8.472248249723641,-7.0783478093032235)-- (6.6322314049586755,-2.989421487603303);
\draw [color=red,line width=0.30mm] (6.6322314049586755,-2.989421487603303)-- (9.723140495867767,-3.039008264462807);
\begin{scriptsize}
\node[draw,ellipse,fill=white] at (6.6322314049586755,-2.989421487603303) {$\A$};
% \draw[color=qqqqff] (6.740433944402697,-2.7193432757325318) node {$A$};
\node[draw,ellipse,fill=white] at (5.442148760330577,-0.34479338842975027) {$\B$};
% \draw[color=qqqqff] (5.553356558978205,-0.07471517655897866) node {$B$};
\node[draw,ellipse,fill=white] at (7.830047107438012,1.191612396694212)  {$\C$};
% \draw[color=qqqqff] (7.942537625845221,1.4579670172802397) node {$C$};
\node[draw,ellipse,fill=white] at (10.516528925619832,-0.4770247933884279)  {$\D$};
% \draw[color=qqqqff] (10.617218317054835,-0.20995184072126263) node {$D$};
\node[draw,ellipse,fill=white] at (9.723140495867767,-3.039008264462807)  {$\E$};
% \draw[color=qqqqff] (9.820824628099164,-2.7644221637866266) node {$E$};
% \draw[color=black] (-2.2603173703981994,-2.5991329075882796) node {$a$};
% \draw[color=black] (3.0439651239669367,5.364803981968444) node {$b$};
% \draw[color=black] (-2.2603173703981994,-5.033392862509392) node {$c$};
% \draw[color=black] (1.3309673779113405,5.364803981968444) node {$d$};
% \draw[color=black] (12.11984791885799,5.364803981968444) node {$e$};
\draw [fill=uuuuuu] (3.556202665213042,3.846197934053662) circle (1.5pt);
\draw [fill=uuuuuu] (11.830326192770796,3.765445548453232) circle (1.5pt);
\draw [fill=uuuuuu] (14.771682375779076,-3.120000914942668) circle (1.5pt);
\draw [fill=uuuuuu] (8.472248249723641,-7.0783478093032235) circle (1.5pt);
\draw [fill=uuuuuu] (1.460783574454355,-2.9064570839053734) circle (1.5pt);
\end{scriptsize}
\end{tikzpicture}}
\caption{\label{fig:good_pentagon} The areas of $H$ that are allowed to contain elements from $E-\Pi$ are the five (closed) triangles shaded in red.}
\end{figure}

\begin{claim}
 Suppose that the good pentagon $\Pi$ has vertices $\A,\B,\C,\D,\E$ as in Figure~\ref{fig:good_pentagon}. Then every point from $E-\Pi$ belongs to one of the shaded areas (including their boundary) in Figure~\ref{fig:good_pentagon}.
\end{claim}
\begin{claimproof}
 It is easy to see that if $e\in E-\Pi$ belongs to any other area then removing some other element $f\in E-\Pi$ (which exists due to $|E|\geq 7$) from $\Mcal$ would result in a non-positively oriented matroid which contradicts the above discussion.
\end{claimproof}

Before we continue the proof, let us revisit the oriented matroids on six elements (of rank $3$). By Lemma~\ref{lemma:six_elements}, we know exactly which of them are pure: the ones in Figure~\ref{fig:positroids} are pure and the ones in Figure~\ref{fig:non_postiroids} are not. An important observation is that for every oriented matroid $\Mcal$ in Figure~\ref{fig:non_postiroids}, there is a rank-preserving weak map $\Mcal\weakMap \IC(6,3,13)$. The conclusion is that \emph{an oriented matroid $\Mcal$ of rank $3$ with six elements is pure if and only if there is no rank-preserving weak map $\Mcal\weakMap\IC(6,3,13)$}. This illustrates the power of Conjecture~\ref{conj:weak_minors} in the rank $3$ case. 

\begin{claim}
 Let $e,f\in E-\Pi$ be any two distinct points. 
 \begin{enumerate}[(a)]
  \item\label{item:same_shaded_area} If $e$ and $f$ belong to the same shaded area in Figure~\ref{fig:good_pentagon} then the line through them does not intersect the interior of the pentagon $\A\B\C\D\E$.
  \item \label{item:different_shaded_areas} If they belong to different shaded areas then the segment $[e,f]$ intersects the closure of the pentagon $\A\B\C\D\E$. 
 \end{enumerate}  
\end{claim}
\begin{claimproof}
 For two points $X$ and $Y$, by $(XY)$ we denote the line that passes through them and by $[X,Y]$ denote the line segment that connects them. 
 
 First, if one of~(\ref{item:same_shaded_area}) or~(\ref{item:different_shaded_areas}) is violated then one can easily check that the restriction of $\Mcal$ to $\Pi e f$ is not isomorphic to a positively oriented matroid. Thus we may assume that $E=\Pi e f$.
 
 Let $[\A,\B]$ be the side of the pentagon adjacent to the shaded area that contains $e$. We are first going to prove both~(\ref{item:same_shaded_area}) and~(\ref{item:different_shaded_areas}) when $e$ is not the intersection point of lines $(\B\C)$ and $(\A\E)$ (we call this intersection point the \emph{outside vertex} of the shaded area containing $e$). For example, assume that $e\not\in (\B\C)$. Then $\Mcal-\E$ contains a pentagon. Clearly, if one of~(\ref{item:same_shaded_area}) or~(\ref{item:different_shaded_areas}) is violated then $\Mcal-\E$ is not positively oriented. Moreover, the same is true for $\reorient{f}{\Mcal-\E}$, so $\Mcal-\E$ cannot be isomorphic to a positively oriented matroid either because any reorientation would be trivial on the pentagon $\Pi e-\E$, and thus no matter how we orient $f$, we do not get a convex polygon. This leads to a contradiction.
 
 Similarly we deal with the case when $e$ is arbitrary and $f$ is not the outside vertex of the shaded area to which it belongs. The only case left is when $e$ and $f$ are both outside vertices of the corresponding shaded areas. Since $\Mcal$ is simple, these shaded areas cannot be the same, so we are done with the proof of~(\ref{item:same_shaded_area}). If the two shaded areas are  not adjacent (i.e., their closures do not intersect) then the segment $[e,f]$ intersects the closure of the pentagon by one of its sides, so~(\ref{item:different_shaded_areas}) holds. If the two shaded areas are adjacent, say, one of them is adjacent to $[\A,\B]$ and another one is adjacent to $[\B,\C]$ then $\Mcal-\D$ is isomorphic to $\IC(6,3,12)$ and thus cannot be isomorphic to a positively oriented matroid. This contradiction finishes the proof of~(\ref{item:different_shaded_areas}). 
\end{claimproof}

\begin{figure}
\definecolor{xdxdff}{rgb}{0.49019607843137253,0.49019607843137253,1.}
% \definecolor{zzttqq}{rgb}{0.6,0.2,0.}
% \definecolor{uuuuuu}{rgb}{0.26666666666666666,0.26666666666666666,0.26666666666666666}
\definecolor{qqqqff}{rgb}{0.,0.,1.}

\definecolor{zzttqq}{rgb}{1,0,0}
\definecolor{uuuuuu}{rgb}{0,0,0}
\definecolor{qqqqff}{rgb}{0.,0.,1.}
\scalebox{0.7}{
\begin{tikzpicture}[scale=0.9]
\clip(1.,-7.4) rectangle (16.,5.);
\fill[color=zzttqq,fill=zzttqq,fill opacity=0.3] (1.460783574454355,-2.9064570839053734) -- (5.442148760330577,-0.34479338842975027) -- (6.6322314049586755,-2.989421487603303) -- cycle;
\fill[color=zzttqq,fill=zzttqq,fill opacity=0.3] (3.556202665213042,3.846197934053662) -- (5.442148760330577,-0.34479338842975027) -- (7.830047107438012,1.191612396694212) -- cycle;
\fill[color=zzttqq,fill=zzttqq,fill opacity=0.3] (10.516528925619832,-0.4770247933884279) -- (7.830047107438012,1.191612396694212) -- (11.830326192770796,3.765445548453232) -- cycle;
\fill[color=zzttqq,fill=zzttqq,fill opacity=0.3] (10.516528925619832,-0.4770247933884279) -- (9.723140495867767,-3.039008264462807) -- (14.771682375779076,-3.120000914942668) -- cycle;
\fill[color=zzttqq,fill=zzttqq,fill opacity=0.3] (9.723140495867767,-3.039008264462807) -- (8.472248249723641,-7.0783478093032235) -- (6.6322314049586755,-2.989421487603303) -- cycle;
\draw [domain=1.:16.] plot(\x,{(-8.911159073833751-0.04958677685950397*\x)/3.0909090909090917});
\draw [domain=1.:16.] plot(\x,{(--13.982126903900001-2.644628099173553*\x)/1.1900826446280988});
\draw [domain=1.:16.] plot(\x,{(--9.184680401202069-1.5364057851239623*\x)/-2.387898347107435});
\draw [domain=1.:16.] plot(\x,{(--16.26675284160913-1.66863719008264*\x)/2.6864818181818206});
\draw [domain=1.:16.] plot(\x,{(--27.32163923229286-2.561983471074379*\x)/-0.7933884297520652});
\draw [color=red, line width=0.30mm] (1.460783574454355,-2.9064570839053734)-- (5.442148760330577,-0.34479338842975027);
\draw [color=red, line width=0.30mm] (5.442148760330577,-0.34479338842975027)-- (6.6322314049586755,-2.989421487603303);
\draw [color=red, line width=0.30mm] (6.6322314049586755,-2.989421487603303)-- (1.460783574454355,-2.9064570839053734);
\draw [color=red, line width=0.30mm] (3.556202665213042,3.846197934053662)-- (5.442148760330577,-0.34479338842975027);
\draw [color=red, line width=0.30mm] (5.442148760330577,-0.34479338842975027)-- (7.830047107438012,1.191612396694212);
\draw [color=red, line width=0.30mm] (7.830047107438012,1.191612396694212)-- (3.556202665213042,3.846197934053662);
\draw [color=red, line width=0.30mm] (10.516528925619832,-0.4770247933884279)-- (7.830047107438012,1.191612396694212);
\draw [color=red, line width=0.30mm] (7.830047107438012,1.191612396694212)-- (11.830326192770796,3.765445548453232);
\draw [color=red, line width=0.30mm] (11.830326192770796,3.765445548453232)-- (10.516528925619832,-0.4770247933884279);
\draw [color=red, line width=0.30mm] (10.516528925619832,-0.4770247933884279)-- (9.723140495867767,-3.039008264462807);
\draw [color=red, line width=0.30mm] (9.723140495867767,-3.039008264462807)-- (14.771682375779076,-3.120000914942668);
\draw [color=red, line width=0.30mm] (14.771682375779076,-3.120000914942668)-- (10.516528925619832,-0.4770247933884279);
\draw [color=red, line width=0.30mm] (9.723140495867767,-3.039008264462807)-- (8.472248249723641,-7.0783478093032235);
\draw [color=red, line width=0.30mm] (8.472248249723641,-7.0783478093032235)-- (6.6322314049586755,-2.989421487603303);
\draw [color=red, line width=0.30mm] (6.6322314049586755,-2.989421487603303)-- (9.723140495867767,-3.039008264462807);
\draw [domain=1.:16.] plot(\x,{(--14.581147043076905-2.607305908062249*\x)/-1.1363339346432202});
\draw [domain=1.:16.] plot(\x,{(--1.8022198264905693-1.564698060035737*\x)/2.8685214922407405});
\draw [domain=1.:16.] plot(\x,{(--4.9762501286100465-1.5273758689244332*\x)/0.5421049129694215});
\begin{scriptsize}
\node[draw,circle,fill=white] at (6.6322314049586755,-2.989421487603303) {$\A\strut$};
% \draw[color=qqqqff] (6.740433944402697,-2.7193432757325318) node {$A$};
\node[draw,circle,fill=white] at (5.442148760330577,-0.34479338842975027) {$\B\strut$};
% \draw[color=qqqqff] (5.553356558978205,-0.07471517655897866) node {$B$};
\node[draw,circle,fill=white] at (7.830047107438012,1.191612396694212) {$\C\strut$};
% \draw[color=qqqqff] (7.942537625845221,1.4579670172802397) node {$C$};
\node[draw,circle,fill=white] at (10.516528925619832,-0.4770247933884279) {$\D\strut$};
% \draw[color=qqqqff] (10.617218317054835,-0.20995184072126263) node {$D$};
\node[draw,circle,fill=white] at (9.723140495867767,-3.039008264462807) {$\E\strut$};
% \draw[color=qqqqff] (9.820824628099164,-2.7644221637866266) node {$E$};
% \draw[color=black] (-2.2603173703981994,-2.5991329075882796) node {$a$};
% \draw[color=black] (3.0439651239669367,5.364803981968444) node {$b$};
% \draw[color=black] (-2.2603173703981994,-5.033392862509392) node {$c$};
% \draw[color=black] (1.3309673779113405,5.364803981968444) node {$d$};
% \draw[color=black] (12.149900510894055,5.334751389932381) node {$e_3$};
\draw [fill=uuuuuu] (3.556202665213042,3.846197934053662) circle (1.5pt);
\draw [fill=uuuuuu] (11.830326192770796,3.765445548453232) circle (1.5pt);
\draw [fill=uuuuuu] (14.771682375779076,-3.120000914942668) circle (1.5pt);
\draw [fill=uuuuuu] (8.472248249723641,-7.0783478093032235) circle (1.5pt);
\draw [fill=uuuuuu] (1.460783574454355,-2.9064570839053734) circle (1.5pt);
\node[draw,circle,fill=white] at (3.763709912717935,-1.4247234275675662) {$e\strut$};
% \draw[color=xdxdff] (3.8704114049586718,-1.1566084898572506) node {$e$};
\node[draw,circle,fill=white] at (4.3058148256873565,-2.9520992964919994) {$g\strut$};
% \draw[color=xdxdff] (4.611358061607808,-2.7742643876784373) node {$g$};
% \draw[color=black] (8.107826882043568,5.394856574004507) node {$e_4$};
% \draw[color=black] (-2.230264778362136,1.7735192336589023) node {$f_2$};
% \draw[color=black] (1.5413355221637821,5.334751389932381) node {$g_2$};
\node[draw,circle,fill=white] at (4.73950837851255,-1.9569941358556888) {$f\strut$};
% \draw[color=uuuuuu] (4.947120646130722,-1.7426340345604812) node {$f$};
\end{scriptsize}
\end{tikzpicture}
}
 \caption{\label{fig:last_bad_case} The essentially unique oriented matroid $\Mcal_0$ on $8$ points that is not positively oriented but such that $\Mcal_0-e,\Mcal_0-f,$ and $\Mcal_0-g$ are.}
\end{figure}
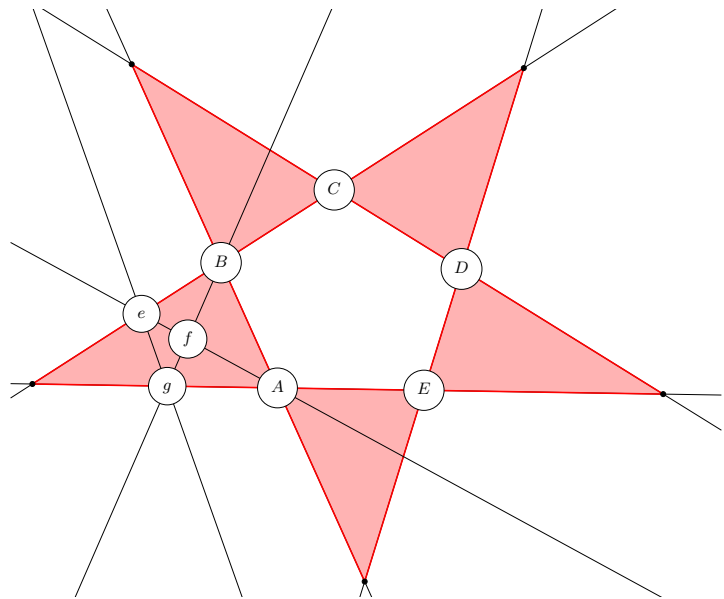

Notice that the above claim actually proves Lemma~\ref{lemma:base_8} when $|E|\leq 7$. Thus the only case left is when $|E|=8$, so we have $\Pi$ together with three other points $e,f,g$, each of them is located in one of the shaded areas in Figure~\ref{fig:good_pentagon} and any two of them satisfy the assertions of the above claim. A quick simple case analysis shows that if $e,f,g$ do not all belong to the same shaded area then $\Mcal$ is a positively oriented matroid. Moreover, if $e,f,g$ belong to the same shaded area which is adjacent, say, to $[\A,\B]$, in such a way that $\Mcal$ is not positively oriented but the lines $(ef),(fg),$ and $(eg)$ do not intersect the interior of the pentagon, then it is easy to see that there is a rank-preserving weak map from $\Mcal$ to the oriented matroid $\Mcal_0$ shown if Figure~\ref{fig:last_bad_case}. But removing $\C$ and $\D$ from $\Mcal_0$ makes it isomorphic to $\IC(6,3,12)$ from Figure~\ref{fig:non_postiroids}. Thus there is a rank-preserving weak map $(\Mcal-\{C,D\})\weakMap \IC(6,3,12)$ and, in turn, there is a rank-preserving weak map $\IC(6,3,12)\weakMap \IC(6,3,13)$ so it follows that in this case $\Mcal$ cannot be pure. We are done with the proof of Lemma~\ref{lemma:base_8} which, as we showed earlier, implies Theorem~\ref{thm:minors}.
\end{proof}

\subsection{Oriented matroids of rank $4$ and corank $2$}

\begin{figure}\scalebox{0.8}{
\begin{tikzpicture}[xscale=2,yscale=1.2]
 \node[red] (node111111) at (1,5) {$(1,1,1,1,1,1)$};
 \node[red] (node21111) at (1,4) {$(2,1,1,1,1)$};
 \node (node3111) at (0,3) {$(3,1,1,1)$};
 \node[red] (node2211) at (1,3) {$(2,2,1,1)$};
 \node (node2121) at (2,3) {$(2,1,2,1)$};
 \node (node321) at (1,2) {$(3,2,1)$};
 \node[red] (node222) at (2,2) {$(2,2,2)$};
 \node (node33) at (2,1) {$(3,3)$}; 
 \draw (node111111) -- (node21111);
 \draw (node3111) -- (node21111);
 \draw (node2211) -- (node21111);
 \draw (node2121) -- (node21111);
 \draw (node3111) -- (node321);
 \draw (node2211) -- (node321);
 \draw (node2121) -- (node321);
 \draw (node2211) -- (node222);
 \draw (node33) -- (node321);
\end{tikzpicture}}
\caption{\label{fig:rank_4_corank_2} The $8$ simple oriented matroids of rank $4$ and corank $2$ ordered by weak maps. Four of them (colored red) are not pure, the other four (colored black) are pure.}
\end{figure}
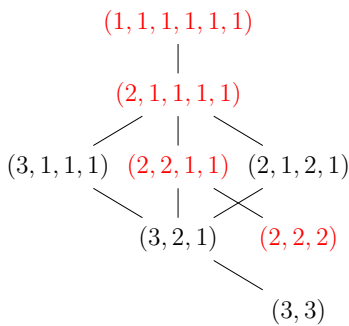

The smallest case not covered by Theorem~\ref{thm:purity_3d} and Proposition~\ref{prop:rk_2_cork_1} is when $\Mcal$ is an oriented matroid with $\rk(\Mcal)=4$ and $\cork(\Mcal)=2$. In this case, $\Mcal^*$ is of rank $2$. Moreover, $\Mcal$ is simple if and only if $\Mcal^*$ has no coloops and no \emph{coparallel} elements. Each oriented matroid of rank $2$ can be reoriented in an acyclic way, and then we can just record it in a composition $\alpha=(\alpha_1,\dots,\alpha_k)$ of $6$, where $\alpha_i$ is the size of the $i$-th parallelism class, and they are ordered according to the way they appear on an affine line. For example, the alternating matroid $C^{6,4}$ corresponds to the composition $(1,1,1,1,1,1)$, while the oriented matroid $\vec K_{2,3}$ corresponds to the composition $(2,2,2)$. There is a rank-preserving weak map between two oriented matroids of rank $2$ if and only if their corresponding compositions are refinements of each other, up to a cyclic shift. Having no coloops and coparallel elements translates into $\alpha_i\leq 3$ for all $i$. The $8$ possible compositions that we can get from $\Mcal^*$ are depicted in Figure~\ref{fig:rank_4_corank_2}, ordered by rank-preserving weak maps. According to our computations, the four oriented matroids labeled by $(1,1,1,1,1,1)$, $(2,1,1,1,1)$, $(2,2,1,1)$, and $(2,2,2)$ are not pure, and the other four labeled by $(3,1,1,1)$, $(2,1,2,1)$, $(3,2,1)$, and $(3,3)$ are pure. This further supports Conjecture~\ref{conj:weak_minors}.

\bibliographystyle{alpha}
\bibliography{separation}

\newcommand{\etalchar}[1]{$^{#1}$}
\def\cprime{$'$}
\begin{thebibliography}{BLVS{\etalchar{+}}99}

\bibitem[ARW13]{ARW}
Federico Ardila, Felipe Rinc{\'o}n, and Lauren Williams.
\newblock Positively oriented matroids are realizable.
\newblock {\em \arxiv{1310.4159}}, 2013.

\bibitem[Ath96]{Ath2}
Christos~A. Athanasiadis.
\newblock Characteristic polynomials of subspace arrangements and finite
  fields.
\newblock {\em Adv. Math.}, 122(2):193--233, 1996.

\bibitem[BBY17]{BBY}
Spencer Backman, Matthew Baker, and Chi~Ho Yuen.
\newblock Geometric bijections for regular matroids, zonotopes, and {E}hrhart
  theory.
\newblock {\em \arxiv{1701.01051}}, 2017.

\bibitem[BFZ05]{FZ3}
Arkady Berenstein, Sergey Fomin, and Andrei Zelevinsky.
\newblock Cluster algebras. {III}. {U}pper bounds and double {B}ruhat cells.
\newblock {\em Duke Math. J.}, 126(1):1--52, 2005.

\bibitem[BKS94]{BKS}
L.~J. Billera, M.~M. Kapranov, and B.~Sturmfels.
\newblock Cellular strings on polytopes.
\newblock {\em Proc. Amer. Math. Soc.}, 122(2):549--555, 1994.

\bibitem[BLVS{\etalchar{+}}99]{Book}
Anders Bj{\"o}rner, Michel Las~Vergnas, Bernd Sturmfels, Neil White, and
  G{\"u}nter~M. Ziegler.
\newblock {\em Oriented matroids}, volume~46 of {\em Encyclopedia of
  Mathematics and its Applications}.
\newblock Cambridge University Press, Cambridge, second edition, 1999.

\bibitem[Boh92]{Bohne}
Jochen Bohne.
\newblock Eine kombinatorische analyse zonotopaler raumaufteilungen.
\newblock {\em Univ., Diss.--Bielefeld}, 1992.

\bibitem[CH67]{CH}
Gary Chartrand and Frank Harary.
\newblock Planar permutation graphs.
\newblock {\em Ann. Inst. H. Poincar\'e Sect. B (N.S.)}, 3:433--438, 1967.

\bibitem[Cra69]{Crapo}
Henry~H. Crapo.
\newblock The {T}utte polynomial.
\newblock {\em Aequationes Math.}, 3:211--229, 1969.

\bibitem[DKK10]{DKK10}
Vladimir~I. Danilov, Alexander~V. Karzanov, and Gleb~A. Koshevoy.
\newblock On maximal weakly separated set-systems.
\newblock {\em J. Algebraic Combin.}, 32(4):497--531, 2010.

\bibitem[DKK14]{DKK14}
Vladimir~I Danilov, Alexander~V Karzanov, and Gleb~A Koshevoy.
\newblock Combined tilings and the purity phenomenon on separated set-systems.
\newblock {\em \arxiv{1401.6418}}, 2014.

\bibitem[FG16]{FG}
Miriam Farber and Pavel Galashin.
\newblock Weak separation, pure domains and cluster distance.
\newblock {\em \arxiv{1612.05387}}, 2016.

\bibitem[Fin01]{Finschi}
Lukas Finschi.
\newblock {\em A graph theoretical approach for reconstruction and generation
  of oriented matroids}.
\newblock PhD thesis, Swiss Federal Institute of Technology Zurich, 2001.

\bibitem[FZ02]{FZ}
Sergey Fomin and Andrei Zelevinsky.
\newblock Cluster algebras. {I}. {F}oundations.
\newblock {\em J. Amer. Math. Soc.}, 15(2):497--529 (electronic), 2002.

\bibitem[FZ03a]{FZ2}
Sergey Fomin and Andrei Zelevinsky.
\newblock Cluster algebras. {II}. {F}inite type classification.
\newblock {\em Invent. Math.}, 154(1):63--121, 2003.

\bibitem[FZ03b]{FZy}
Sergey Fomin and Andrei Zelevinsky.
\newblock {$Y$}-systems and generalized associahedra.
\newblock {\em Ann. of Math. (2)}, 158(3):977--1018, 2003.

\bibitem[FZ07]{FZ4}
Sergey Fomin and Andrei Zelevinsky.
\newblock Cluster algebras. {IV}. {C}oefficients.
\newblock {\em Compos. Math.}, 143(1):112--164, 2007.

\bibitem[Gal16]{Galashin}
Pavel Galashin.
\newblock Plabic graphs and zonotopal tilings.
\newblock {\em \arxiv{1611.00492}}, 2016.

\bibitem[Gio07]{Gioan1}
Emeric Gioan.
\newblock Enumerating degree sequences in digraphs and a cycle-cocycle
  reversing system.
\newblock {\em European J. Combin.}, 28(4):1351--1366, 2007.

\bibitem[Gio08]{Gioan2}
Emeric Gioan.
\newblock Circuit-cocircuit reversing systems in regular matroids.
\newblock {\em Ann. Comb.}, 12(2):171--182, 2008.

\bibitem[GKZ94]{GKZ}
I.~M. Gel'fand, M.~M. Kapranov, and A.~V. Zelevinsky.
\newblock {\em Discriminants, resultants, and multidimensional determinants}.
\newblock Mathematics: Theory \& Applications. Birkh\"auser Boston, Inc.,
  Boston, MA, 1994.

\bibitem[GP80]{GP}
Jacob~E. Goodman and Richard Pollack.
\newblock Proof of {G}r\"unbaum's conjecture on the stretchability of certain
  arrangements of pseudolines.
\newblock {\em J. Combin. Theory Ser. A}, 29(3):385--390, 1980.

\bibitem[GP16]{GaPa}
Pavel Galashin and Gaiane Panina.
\newblock Manifolds associated to simple games.
\newblock {\em J. Knot Theory Ramifications}, 25(12):1642003, 14, 2016.

\bibitem[GW94]{VCmatroid}
B.~G\"artner and E.~Welzl.
\newblock Vapnik-{C}hervonenkis dimension and (pseudo-)hyperplane arrangements.
\newblock {\em Discrete Comput. Geom.}, 12(4):399--432, 1994.

\bibitem[Kur30]{Kuratowski}
Casimir Kuratowski.
\newblock Sur le probl\`eme des courbes gauches en {T}opologie.
\newblock {\em Fundamenta mathematicae}, 15(1):271--283, 1930.

\bibitem[Liu16]{Liu}
Gaku Liu.
\newblock A counterexample to the extension space conjecture for realizable
  oriented matroids.
\newblock {\em \arxiv{1606.05033}}, 2016.

\bibitem[LV78]{LasVergnas}
Michel Las~Vergnas.
\newblock Extensions ponctuelles compatibles d'une g\'eom\'etrie combinatoire.
\newblock {\em C. R. Acad. Sci. Paris S\'er. A-B}, 286(21):A981--A984, 1978.

\bibitem[LZ98]{LZ}
Bernard Leclerc and Andrei Zelevinsky.
\newblock Quasicommuting families of quantum {P}l\"ucker coordinates.
\newblock In {\em Kirillov's seminar on representation theory}, volume 181 of
  {\em Amer. Math. Soc. Transl. Ser. 2}, pages 85--108. Amer. Math. Soc.,
  Providence, RI, 1998.

\bibitem[MS86]{MS}
Yu.~I. Manin and V.~V. Shekhtman.
\newblock Higher {B}ruhat orderings connected with the symmetric group.
\newblock {\em Funktsional. Anal. i Prilozhen.}, 20(2):74--75, 1986.

\bibitem[OPS15]{OPS}
Suho Oh, Alexander Postnikov, and David~E. Speyer.
\newblock Weak separation and plabic graphs.
\newblock {\em Proc. Lond. Math. Soc. (3)}, 110(3):721--754, 2015.

\bibitem[Oxl11]{Oxley}
James Oxley.
\newblock {\em Matroid theory}, volume~21 of {\em Oxford Graduate Texts in
  Mathematics}.
\newblock Oxford University Press, Oxford, second edition, 2011.

\bibitem[Pos06]{Postnikov}
Alexander Postnikov.
\newblock Total positivity, {G}rassmannians, and networks.
\newblock {\em \arxiv{math/0609764}}, 2006.

\bibitem[Rei99]{Reiner}
Victor Reiner.
\newblock The generalized {B}aues problem.
\newblock In {\em New perspectives in algebraic combinatorics ({B}erkeley,
  {CA}, 1996--97)}, volume~38 of {\em Math. Sci. Res. Inst. Publ.}, pages
  293--336. Cambridge Univ. Press, Cambridge, 1999.

\bibitem[Sau72]{Sauer}
N.~Sauer.
\newblock On the density of families of sets.
\newblock {\em J. Combinatorial Theory Ser. A}, 13:145--147, 1972.

\bibitem[Sco05]{Scott}
Joshua~S. Scott.
\newblock Quasi-commuting families of quantum minors.
\newblock {\em J. Algebra}, 290(1):204--220, 2005.

\bibitem[Sco06]{Scott2}
Joshua~S. Scott.
\newblock Grassmannians and cluster algebras.
\newblock {\em Proc. London Math. Soc. (3)}, 92(2):345--380, 2006.

\bibitem[She72]{Shelah}
Saharon Shelah.
\newblock A combinatorial problem; stability and order for models and theories
  in infinitary languages.
\newblock {\em Pacific J. Math.}, 41:247--261, 1972.

\bibitem[Sta07]{Stanley}
Richard~P. Stanley.
\newblock An introduction to hyperplane arrangements.
\newblock In {\em Geometric combinatorics}, volume~13 of {\em IAS/Park City
  Math. Ser.}, pages 389--496. Amer. Math. Soc., Providence, RI, 2007.

\bibitem[Tut58]{Tutte2}
W.~T. Tutte.
\newblock A homotopy theorem for matroids. {I}, {II}.
\newblock {\em Trans. Amer. Math. Soc.}, 88:144--174, 1958.

\bibitem[Tut59]{Tutte1}
W.~T. Tutte.
\newblock Matroids and graphs.
\newblock {\em Trans. Amer. Math. Soc.}, 90:527--552, 1959.

\bibitem[VC71]{VCold}
V.N. {Vapnik} and A.Ya. {Chervonenkis}.
\newblock {On the uniform convergence of relative frequencies of events to
  their probabilities.}
\newblock {\em {Theory Probab. Appl.}}, 16:264--280, 1971.

\bibitem[Vin71]{Vinberg}
{\`E}.~B. Vinberg.
\newblock Discrete linear groups that are generated by reflections.
\newblock {\em Izv. Akad. Nauk SSSR Ser. Mat.}, 35:1072--1112, 1971.

\bibitem[VK91]{VK}
V.~A. Voevodski{\u\i} and M.~M. Kapranov.
\newblock The free {$n$}-category generated by a cube, oriented matroids and
  higher {B}ruhat orders.
\newblock {\em Funktsional. Anal. i Prilozhen.}, 25(1):62--65, 1991.

\bibitem[Wag37]{Wagner}
K.~Wagner.
\newblock \"{U}ber eine {E}igenschaft der ebenen {K}omplexe.
\newblock {\em Math. Ann.}, 114(1):570--590, 1937.

\bibitem[Zie93]{Ziegler}
G{\"u}nter~M. Ziegler.
\newblock Higher {B}ruhat orders and cyclic hyperplane arrangements.
\newblock {\em Topology}, 32(2):259--279, 1993.

\end{thebibliography}

\end{document}